\documentclass[12pt]{amsart}
\usepackage{amssymb}
\usepackage{amsbsy}
\usepackage{amscd}
\usepackage[mathscr]{eucal}
\usepackage{verbatim}
\usepackage{version}

\usepackage{nicefrac}

\usepackage{pstricks}
\usepackage{pst-all}
\usepackage{color}

\usepackage[all]{xy}
\usepackage{mathdots}
\makeatletter
\def\cal{\mathcal}
\def\Bbb{\mathbb}
\def\frak{\mathfrak}

\newenvironment{pf*}[1]{\proof[#1]}{\endproof}
\newcommand{\rom}{\textup}
\makeatother

\newenvironment{aenume}{%
  \begin{enumerate}%
  }{\end{enumerate}}
\makeatletter
\makeatletter
\@ifclasslater{amsart}{1999/11/24}{}{
\renewcommand*\subjclass[2][1991]{%
  \def\@subjclass{#2}%
  \@ifundefined{subjclassname@#1}{%
    \ClassWarning{\@classname}{Unknown edition (#1) of Mathematics
      Subject Classification; using '1991'.}%
  }{%
    \@xp\let\@xp\subjclassname\csname subjclassname@#1\endcsname
  }%
}
\renewcommand{\subjclassname}{%
  \textup{1991} Mathematics Subject Classification}
\@xp\let\csname subjclassname@1991\endcsname \subjclassname
\@namedef{subjclassname@2000}{%
  \textup{2000} Mathematics Subject Classification}
}

\renewcommand{\MR}[1]{}

\makeatother
\newenvironment{NB}{
\color{red}{\bf NB}. \footnotesize 
}{}
\newenvironment{NB2}{
\color{blue}{\bf NB}. \footnotesize
}{}
\excludeversion{NB}
\excludeversion{NB2}

\newtheorem{Theorem}[equation]{Theorem}
\newtheorem{Corollary}[equation]{Corollary}
\newtheorem{Lemma}[equation]{Lemma}
\newtheorem{Proposition}[equation]{Proposition}

\theoremstyle{definition}
\newtheorem{Definition}[equation]{Definition}

\theoremstyle{remark}
\newtheorem{Remark}[equation]{Remark}

\numberwithin{equation}{section}

\newcommand{\thmref}[1]{Theorem~\ref{#1}}
\newcommand{\secref}[1]{\S\ref{#1}}
\newcommand{\lemref}[1]{Lem\-ma~\ref{#1}}
\newcommand{\propref}[1]{Proposition~\ref{#1}}

\newcommand{\subsecref}[1]{\S\ref{#1}}

\newcommand{\defref}[1]{Definition~\ref{#1}}
\newcommand{\remref}[1]{Remark~\ref{#1}}
\newcommand{\defeq}{:=}
\newcommand{\C}{{\Bbb C}}
\newcommand{\Z}{{\Bbb Z}}
\newcommand{\Q}{{\Bbb Q}}
\newcommand{\R}{{\Bbb R}}
\newcommand{\proj}{{\Bbb P}}

\newcommand{\SL}{\operatorname{\rm SL}}

\newcommand{\GL}{\operatorname{GL}}

\newcommand{\Hom}{\operatorname{Hom}}
\newcommand{\Ext}{\operatorname{Ext}}
\newcommand{\Ker}{\operatorname{Ker}}
\newcommand{\Coker}{\operatorname{Coker}}
\newcommand{\Ima}{\operatorname{Im}}

\newcommand{\rank}{\operatorname{rank}}

\newcommand{\pd}[2]{\frac{\partial#1}{\partial#2}}

\newcommand{\id}{\operatorname{id}}
\newcommand{\ve}{\varepsilon}

\newcommand{\linf}{{\ell_\infty}}
\newcommand{\shfO}{\mathcal O}
\newcommand{\dslash}{/\!\!/} 
\newcommand{\bp}{{{\widehat\proj}^2}}
\newcommand{\bC}{{\widehat\C}^2}
\newcommand{\bM}{{\widehat M}}
\newcommand{\tM}{{\widetilde M}}
\newcommand{\baM}{{\bar M}}

\newcommand{\ch}{\operatorname{ch}}
\newcommand{\Wedge}{{\textstyle \bigwedge}}

\newcommand{\Zin}{Z^{\text{\rm inst}}}

\newcommand{\bZin}{\widehat{Z}^{\text{\rm inst}}}
\newcommand{\bZ}{\widehat{Z}}

\newcommand{\Fin}{F^{\text{\rm inst}}}

\newcommand{\q}{\mathfrak q}
\newcommand{\bbeta}{\boldsymbol\beta}
\newcommand{\hT}{\widetilde T}

\newcommand{\rk}{\mathop{{\rm rk}}}

\newcommand{\td}{\mathop{\text{\rm td}}}

\newcommand{\Res}{\operatornamewithlimits{Res}}

\newcommand{\codim}{\mathop{\text{\rm codim}}\nolimits}

\newcommand{\Qcal}{\mathcal Q}

\makeatletter
\newcommand{\vechatom}{
    {\Vec{\omega}}
    \,\smash[b]{\hbox{\lower2\ex@\hbox{$\m@th\hat{\null}$}}}
}
\makeatother

\newcommand{\pt}{\operatorname{pt}}
\newcommand{\bX}{{\widehat X}}
\newcommand{\bMm}[1]{\widehat{M}^{#1}}

\DeclareMathOperator{\vdim}{\mathbf{dim}}
\newcommand{\bn}{\mathbf{n}}
\newcommand{\Flag}{\operatorname{Flag}}
\newcommand{\Gr}{\operatorname{Gr}}
\newcommand{\cM}{\mathcal M}
\newcommand{\cV}{\mathcal V}
\newcommand{\Vcal}{\cV}
\newcommand{\cE}{\mathcal E}
\newcommand{\cF}{\mathcal F}
\newcommand{\cL}{\mathcal L}
\newcommand{\cQ}{\mathcal Q}
\newcommand{\Xf}{X_\flat}
\newcommand{\Xs}{X_\sharp}
\newcommand{\Ff}{F_\flat}
\newcommand{\Fs}{F_\sharp}
\newcommand{\Fcal}{\cF}
\newcommand{\If}{I_\flat}
\newcommand{\Is}{I_\sharp}
\newcommand{\Vf}{V^\flat}
\newcommand{\Vs}{V^\sharp}

\newcommand{\cEf}{\cE_\flat}
\newcommand{\cEs}{\cE_\sharp}
\newcommand{\Trel}{\varTheta_{\operatorname{rel}}}
\newcommand{\fN}{{\mathfrak N}}
\newcommand{\fI}{{\mathfrak I}}
\newcommand{\topdeg}{\operatorname{top}}
\newcommand{\Dec}{\operatorname{Dec}}

\usepackage[dvips,bookmarks=false]{hyperref}

\begin{document}
\title[Perverse coherent sheaves on blow-up. III]
{Perverse coherent sheaves on blow-up. III.
\\ Blow-up formula from wall-crossing
\\
}
\author{Hiraku Nakajima}
\address{Research Institute for Mathematical Sciences,
Kyoto University, Kyoto 606-8502,
Japan}
\email{nakajima@kurims.kyoto-u.ac.jp}

\author{K\={o}ta Yoshioka}
\dedicatory{To the memory of the late Professor Masaki Maruyama}
\address{Department of Mathematics, Faculty of Science, Kobe University,
Kobe 657-8501, Japan}
\address{Max-Planck-Institut f\"{u}r Mathematik,
Vivatsgasse 7
53111 Bonn, Germany}
\email{yoshioka@math.kobe-u.ac.jp}


\begin{abstract}
  In earlier papers \cite{perv,perv2} of this series we constructed a
  sequence of intermediate moduli spaces $\{
  \bM^m(c)\}_{m=0,1,2,\dots}$ connecting a moduli space $M(c)$ of
  stable torsion free sheaves on a nonsingular complex projective
  surface $X$ and $\bM(c)$ on its one point blow-up $\bX$. They are
  moduli spaces of perverse coherent sheaves on $\bX$.
  In this paper we study how Donaldson-type invariants (integrals of
  cohomology classes given by universal sheaves) change from
  $\bM^m(c)$ to $\bM^{m+1}(c)$, and then from $M(c)$ to $\bM(c)$.
  As an application we prove that Nekrasov-type partition functions
  satisfy certain equations which determine invariants recursively in
  second Chern classes. They are generalization of the blow-up
  equation for the original Nekrasov deformed partition function for
  the pure $N=2$ SUSY gauge theory, found and used to derive the
  Seiberg-Witten curves in \cite{NY1}.
\end{abstract}

\maketitle
\section*{Introduction}

Let $X$ be a nonsingular complex projective surface and $p\colon
\widehat X\to X$ the blow-up at a point $0$. Let $C = p^{-1}(0)$ be
the exceptional divisor.
Let $c = (r,c_1,\ch_2)\in H^{\mathrm{ev}}(\bX)$ be a cohomological
data.
Let $\bM(c)$ be the moduli space of stable torsion free sheaves $E$ on
$\bX$ with $\ch(E) = c$ and $M(p_*(c))$ the corresponding moduli space
on $X$.
In \cite{perv,perv2} we constructed a sequence of intermediate moduli
spaces $\bM^m(c)$ connected by birational morphisms as
\begin{equation}
  \label{eq:flip}
\cdots \setbox5=\hbox{$\bMm{m,m+1}(c)$}{\rule{0mm}{\ht5}\hspace*{-\wd5}}
\begin{aligned}[m]
  \xymatrix@R=.5pc{ &
\bMm{m}(c) \ar[rd]^{\xi_{m}} \ar[ld] & & 
\bMm{m+1}(c) \ar[rd]^{\xi_{m+1}} \ar[ld]_{\xi^+_{m}} & 
& 
\ar[ld]
\\
\setbox5=\hbox{$\bMm{m,m+1}(c)$}{\rule{0mm}{\ht5}\hspace*{\wd5}}
&& \bMm{m,m+1}(c) & & \bMm{m+1,m+2}(c) &
}
\end{aligned}
\cdots
\tag{$*$}
\end{equation}
such that
\begin{enumerate}
\item $\bM^m(c) \cong \bM(c)$ if $m$ is sufficiently large, and
\item $\bM^0(c) \cong M(p_*(c))$ if $(c_1,[C]) = 0$ under the natural
  homomorphism given by $E\mapsto p_*(E)$. (See \propref{prop:Grassmann}
  for the statement when $0 < (c_1,[C]) < r$.)
\end{enumerate}
The diagram \eqref{eq:flip} is an example of those often appearing in
variations of GIT quotients \cite{Thaddeus}, and similar to ones for
moduli spaces of sheaves (by Thaddeus, Ellingsrud-G\"ottsche,
Friedman-Qin and others) when we move polarizations. See
\cite{perv,perv2} for more references on earlier works.

In this paper, we study how Donaldson type invariants (certain
integrals of cohomology classes given by universal sheaves) change
from $\bM^m(c)$ to $\bM^{m+1}(c)$. By a technical reason we restrict
ourselves to the case when $X$ is $\proj^2$ and $\bM^m(c)$ is replaced
by the moduli space of framed sheaves for which quiver description was
given in \cite{perv}. We conjecture that the results are universal,
i.e., independent of the choice of a surface.
Moreover we have a natural $(r+1)$-dimensional torus $\hT =
(\C^*)^2\times T^{r-1}$ action on $\bM^m(c)$ from the
$(\C^*)^2$-action on $\bp$ and the change of framing. Thus we can
consider equivariant Donaldson type invariants to which we can apply
Atiyah-Bott-Lefschetz fixed point formula to perform a further
computation. In this sense, we think our situation is most basic.

Our first main result says that the difference of invariants is given
by a variant of Mochizuki's weak wall-crossing formula \cite{Moc},
i.e., it is expressed as a sum of an integral over $\bM^m(c')$ with
smaller $c'$ (\thmref{thm:m=0}). Our argument closely follows
Mochizuki's, once \eqref{eq:flip} is understood as a variation of GIT
quotients.

Summing up the weak wall-crossing formula from $0$ to $m$, we get the
formula for the difference of $\bM(c)$ and $M(p_*(c))$ by integrals
over various $\bM^{m'}(c')$, as a result. We normalize the first Chern
class of $c'$ in the interval $[0,r-1]$ twisting by a line bundle in
order to apply $M(p_*(c'))\cong \bM^0(c')$ for $(c_1(c'),[C]) = 0$ and
its modification \propref{prop:Grassmann}. Then those integrals
themselves can be expressed by integrals over $M(p_*(c'))$ and ones
over even smaller $\bM^{m''}(c'')$. We apply the same argument for
$\bM^{m''}(c'')$.
We thus do this argument recursively to give an algorithm to express
$\bM(c)$ by a linear combination of integrals over $M(c_\flat)$ for
various $c_\flat$. Since this algorithm is complicated (see
Figure~\ref{fig:flow} for the flowchart), we do not try to write down
an explicit formula in general.
We instead focus on vanishing theorems for special cases when
integrands are not twisted too much along $C$. This is our second main
result.
See \secref{sec:vanish}.

Our motivation of study in this series is an application to the
Nekrasov partition function \cite{Nek}. Let us explain it briefly.
The Nekrasov partition function is the generating function of an
equivariant integral over $M(c)$.
One of the main conjecture on it states that the leading part
$\mathcal F_0$ of its logarithm is given by the Seiberg-Witten
prepotential, a certain period integral on the Seiberg-Witten curves.
The three consecutive coefficients (denoted by $H$, $A$, $B$) are also
important for the application to the wall-crossing formula for usual
or $K$-theoretic Donaldson invariants for projective surfaces with
$p_g = 0$ \cite{GNY,GNY2}.

When the integrand is (1) $1$, (2) slant products of Chern classes of
universal sheaves with the fundamental classes of $\C^2$, or (3) the
Todd class of $M(c)$, the authors proved that the partition function
satisfies functional equations, called the {\it blow-up equations},
which determine coefficients recursively in second Chern numbers of
$c$ \cite{NY1,NY2,NY3}. The functional equations induce a nonlinear
partial differential equation for $\mathcal F_0$, which has been known
as the {\it contact term equation\/} in the physics literature
\cite{LNS1,GM3}. In particular, the Seiberg-Witten prepotential
satisfies the same equation and hence is equal to $\mathcal F_0$. This
was our proof of the above mentioned conjecture.  There are other
completely independent proofs by \cite{NO,BE}.
But so far $H$, $A$, $B$ can be determined only from the blow-up
equation.

Nekrasov's partition functions have more variants by replacing the
integrand. Let us give three examples:

(a) We integrate Euler classes of vector bundles given by pushforward
of universal sheaves. They are called the {\it theories with
  fundamental matters\/} in the physics literature.

(b) When we integrate the Todd classes, we can cap with powers of the
first Chern classes of the same bundles. They are called {\it
  $5$-dimensional Chern-Simons terms}.

(c) We can also incorporate universal sheaves to which Adams operators
are applied. They are called {\it (higher) Casimir operators}. They
give coefficients appearing in the defining equation of the
Seiberg-Witten curve.

The blow-up equation was derived by analyzing relation between
integrals over $M(c)$ and $\bM(c)$.
Our vanishing results in \secref{sec:vanish} enable us to generalize
our proof for those variants\footnote{The proof in \cite{NO} can be
  generalized to those variants. See \cite{Ta} for the theory with
  $5$-dimensional Chern-Simons terms, and \cite{MN} for higher Casimir
  operators, but not with Todd genus.}. In this paper, we explain it
for theories with $5$-dimensional Chern-Simons terms and Casimir
operators. The case for the theories with matters will be given
elsewhere \cite{NY4}.

The paper is organized as follows. In \secref{sec:1st} we state our
results after preparing the necessary notations.
In \secref{sec:vanish} we prove several versions of vanishing theorems
as applications of the results in \secref{sec:1st}.
In \secref{sec:SW} we study the Nekrasov partition function for
theories with $5$-dimensional Chern-Simons terms. The blow-up equation
is derived. This section is expository since the derivation of
Nekrasov's conjecture was already given in \cite{GNY2} assuming the
vanishing theorems.

The actual proof starts from \secref{sec:quiver}. We review the quiver
description of the framed moduli spaces obtained in \cite{perv} and
the analysis of the wall-crossing in \cite{perv2}, and add a few
things. The quiver description is necessary to define master spaces.
In \secref{sec:mater spaces} we define enhanced master spaces.  We
follow Mochizuki's method \cite{Moc}, but give the construction in
detail for the sake of a reader.
In \secref{sec:wallcrossing} we prove \thmref{thm:m=0}, the variant of
Mochizuki's weak wall-crossing formula. Again the proof is the same as
Mochizuki's.

\subsection*{Acknowledgments}

The first named author is supported by the Grant-in-Aid for Scientific
Research (B) (No.~19340006), Japan Society for the Promotion of
Science.
The second named author is supported by the Grant-in-aid for
Scientific Research (B) (No.~18340010), (S) (No.~19104002), JSPS and
Max Planck Institute for Mathematics.
Both authors thank Takuro Mochizuki for explanations of his results,
and the referee for many valuable comments.
\begin{NB2}
  Added. 2010/04/03
\end{NB2}%

\section{Main result}\label{sec:1st}

\subsection*{Notations}

Let $[z_0:z_1:z_2]$ be the homogeneous coordinates on $\proj^2$ and
$\linf = \{ z_0 = 0\}$ the line at infinity.
Let $p\colon \bp\to \proj^2$ be the blow-up of $\proj^2$ at $[1:0:0]$.
Then $\bp$ is the closed
subvariety of $\proj^2\times\proj^1$ defined by
\begin{equation*}
  \left\{ ([z_0:z_1:z_2],[z:w])\in\proj^2\times\proj^1 \mid
    z_1 w = z_2 z \right\},
\end{equation*}
where the map $p$ is the projection to the first factor.  We denote
$p^{-1}(\linf)$ also by $\linf$ for brevity. Let $C$ denote the
exceptional divisor given by $z_1 = z_2 = 0$. Let $\shfO$ denote the
structure sheaf of $\bp$, $\shfO(C)$ the line bundle associated with
the divisor $C$, and $\shfO(mC)$ its $m^{\mathrm{th}}$ tensor product
$\shfO(C)^{\otimes m}$ when $m > 0$, $\left(\shfO(C)^{\otimes
    -m}\right)^\vee$ if $m < 0$, and $\shfO$ if $m=0$. And we use the
similar notion $\shfO(mC+n\linf)$ for tensor products of $\shfO(mC)$
and tensor powers of the line bundle corresponding to $\linf$ or its
dual.

The structure sheaf of the exceptional divisor $C$ is denoted by
$\shfO_C$. If we twist it by the line bundle $\shfO_{\proj^1}(n)$ over
$C \cong\proj^1$, we denote the resulted sheaf by $\shfO_C(n)$. Since
$C$ has the self-intersection number $-1$, we have
$\shfO_C\otimes \shfO(C) = \shfO_C(-1)$.

For $c\in H^*(\bp)$, its degree $0$, $2$, $4$-parts are denoted by
$r$, $c_1$, $\ch_2$ respectively. If we want to specify $c$, we denote
by $r(c)$, $c_1(c)$, $\ch_2(c)$. 

For brevity, we twist the push-forward homomorphism $p_*$ by Todd
genera of $\proj^2$ and $\bp$ as in \cite[\S3.1]{perv2} so that it is
compatible with the Riemann-Roch formula.

We also use the following notations frequently:
\begin{itemize}
\item $\vee$ is the involution on the $K$-group given by taking the
  dual of a vector bundle.
\item $\Delta(E) \defeq c_2(E) - \frac{r-1}{2r} c_1(E)^2$,
$\Delta(c) \defeq -\ch_2(c) + \frac1{2r(c)} c_1(c)^2$.
\item $C_m$ denotes $\shfO_C(-m-1)$.
\item $e_m\defeq \ch(\shfO_C(-m-1))$.
  \begin{NB}
    $e_m = e^{(m+1)[C]} - e^{m[C]}$ from the exact sequence
    $0 \to \shfO(mC) \to \shfO((m+1)C) \to \shfO_C(-(m+1)C) \to 0$.
  \end{NB}
\item $\pt$ is a single point in $X$, $\bX$ or sometimes an abstract
  point. Its Poincar\'e dual in $H^4(X)$ or $H^4(\bX)$ is also denoted
  by the same notation.
\begin{NB}
\item $\chi(E,F) := \sum_{i=-\infty}^{\infty} (-1)^i \dim \Ext^i(E,F)
= \sum_{i=-\infty}^{\infty} (-1)^i \dim \Hom(E,F[i])$.
\item $h^0(E, F) := \dim \Hom(E,F)$.
\item $\chi(E) := \chi(\shfO_{\bX}, E)$.
\item $h^0(E) := h^0(\shfO_{\bX}, E)$.
\item $V_0(E) := H^1(E(-\linf))$, $V_1(E) := H^1(E(C-\linf))$.
\item For a sheaf $\cE$ on $\bp\times \bM^m(c)$, say the universal
  sheaf, we set
  $\Vcal_0(\cE) := R^1 q_{2*}(\cE\otimes q_1^*\shfO(-\linf))$,
  $\Vcal_1(\cE) := R^1 q_{2*}(\cE\otimes q_1^*\shfO(C-\linf))$,
  where $q_1$, $q_2$ are projections to the first and second factors
  of $\bp\times \bM^m(c)$.
\end{NB}
\item For an integer $N$ let $\underline{N} = \{ 1,2,\dots, N\}$.
\end{itemize}

\begin{NB}
Serre duality:
\begin{equation*}
\begin{split}
  & \Ext^i(E,\shfO_C(-1)) \cong \Ext^{2-i}(\shfO_C(-1),E\otimes K)^\vee
  \cong \Ext^{2-i}(K^\vee\otimes\shfO_C(-1),E)^\vee
  \cong \Ext^{2-i}(\shfO_C,E)^\vee,
\\
  & \Ext^i(\shfO_C(-1),E) \cong \Ext^{2-i}(E, \shfO_C(-1)\otimes K)^\vee
  \cong \Ext^{2-i}(E, \shfO_C(-2))^\vee,
\end{split}
\end{equation*}
as $K|_C \cong \shfO_C(-1)$.

Riemann-Roch:
\begin{gather*}
\ch(\shfO_C(-1)) = [C] - \frac12 \pt,
\\
\td{\bX} = 1 + \frac12 c_1(\bX) + \frac1{12}(c_1(\bX)^2 + c_2(\bX)).
\end{gather*}
Therefore
\begin{equation*}
\begin{split}
  & \chi(E,\shfO_C(-1)) = \int_{\bX} \ch(E)^\vee \ch(\shfO_C(-1))\td \bX
  = - (c_1(E), [C]),
\\
  & \chi(\shfO_C(-1),E) = \int_{\bX} \ch(E)\ch(\shfO_C(-1))^\vee \td \bX
  = - (c_1(E), [C]) - \rk E,
\end{split}
\end{equation*}
as $c_1(\bX)\cdot [C] = \pt$.
\end{NB}

For a sheaf $E$ on $\bp$, we denote $H^1(E(-\linf))$,
$H^1(E(C-\linf))$ by $V_0(E)$, $V_1(E)$ respectively (and
simply by $V_0$, $V_1$ if there is no fear of confusion).
In this paper we mainly treat sheaves $E$ with $H^i(E(-\linf)) = 0 =
H^i(E(C-\linf))$ for $i\neq 1$. This is clear after we will recall the
quiver description of framed moduli spaces in \secref{sec:quiver}:
$V_\alpha(E)$ appears as a vector space on the vertex $\alpha$, and
any sheaf in this paper corresponds to a representation of the quiver.
Under this assumption we have
\begin{equation*}
  \begin{split}
    & \dim V_0 = \dim H^1(E(-\linf)) = -(\ch_2(E),[\bp]) + \frac12
    (c_1(E),[C]),
    \\
    & \dim V_1 = \dim H^1(E(C-\linf)) = -(\ch_2(E),[\bp]) - \frac12
    (c_1(E),[C])
  \end{split}
\end{equation*}
by Riemann-Roch.

Let $\bM$ be a moduli scheme (or stack) and $q_1$, $q_2$ be
projections to the first and second factors of $\bp\times\bM$. For a
sheaf $\cE$ (e.g., the universal sheaf) on $\bp\times\bM$, let
\begin{itemize}
\item $\Vcal_0(\cE) \defeq R^1 q_{2*}(\cE\otimes q_1^*\shfO(-\linf))$,
\item $\Vcal_1(\cE) \defeq R^1 q_{2*}(\cE\otimes q_1^*\shfO(C-\linf))$.
\end{itemize}
Let $\Ext^\bullet_{q_2}$ denotes the derived functor of the composite
functor $q_{2*}\circ{\mathcal H}om$. We often consider
$\Ext^\bullet_{q_2}(\cE,C_m)$, where $C_m$ is considered as a sheaf on
$\bp\times\bM$ via the pull-back by $q_1$.
\begin{NB2}
  Corrected according to the referee's remark 2010/04/01
\end{NB2}%

\subsection{Framed moduli spaces}

A {\it framed sheaf\/} $(E,\Phi)$ on $\proj^2$ is a pair of
\begin{itemize}
\item a coherent sheaf $E$, which is locally free in a neighborhood
  of $\linf$, and
\item an isomorphism $\Phi\colon E|_{\linf}\to \shfO_{\linf}^{\oplus
    r}$, where $r$ is the rank of $E$.
\end{itemize}
An isomorphism of framed sheaves $(E,\Phi)$, $(E',\Phi')$ is an
isomorphism $\xi\colon E\to E'$ such that $\Phi'\circ \xi|_{\linf} =
\Phi$. When $r=0$, we understand that a framed sheaf is an ordinary
sheaf of rank $0$ whose support does not intersect with $\linf$.
\begin{NB2}
  Editted according to the referee's remark 2010/04/01
\end{NB2}%
We have the corresponding definition of a framed sheaf on the blow-up
$\bp$.

\begin{Definition}\label{def:m-stable}
  Let $m\in \Z_{\ge 0}$.
  A framed sheaf $(E,\Phi)$ on $\bp$ is called {\it $m$-stable\/} if 
  \begin{enumerate}
  \item $\Hom(E,\shfO_C(-m-1)) = 0$,
    \begin{NB}
      i.e., $\Hom(E(-mC),\shfO_C(-1)) = 0$,
    \end{NB}
  \item $\Hom(\shfO_C(-m),E) = 0$, and
    \begin{NB}
      i.e., $\Hom(\shfO_C,E(-mC)) = 0$,
    \end{NB}
  \item $E$ is torsion free outside $C$.
  \end{enumerate}
\end{Definition}

Though it is not obvious from the definition, an $m$-stable sheaf must
have $r > 0$. (See \cite[\S2.2]{perv}.)
\begin{NB2}
  Added according to the referee's suggestion. 2010/04/01
\end{NB2}%

We have a smooth fine moduli scheme $\bM^m(c)$ of $m$-stable framed
sheaves $(E,\Phi)$ with $\ch(E) = c\in H^*(\bp)$ such that
$(c,[\linf]) = 0$. It is of dimension $2r(c)\Delta(c)$.
\begin{NB2}
  Added according to the referee's remark. 2010/04/01
\end{NB2}%
(See \thmref{thm:quiver}.) Let $\cE$ be the universal sheaf on $\bp\times
\bM^m(c)$, which is unique thanks to the framing unlike the case of
ordinary moduli spaces.
\begin{NB}
Let $\Vcal_0(\cE) \defeq R^1 q_{2*}(\cE\otimes q_1^*\shfO(-\linf))$,
$\Vcal_1(\cE) \defeq R^1 q_{2*}(\cE\otimes q_1^*\shfO(C-\linf))$,
where $q_1$, $q_2$ are projections to the first and second factors of
$\bp\times \bM^m(c)$. We use the same notation for more general
sheaves $\cE$ on $\bp\times \bM^m(c)$.
\end{NB}%

As special cases with $m = 0$ and $m$ sufficiently large, we have fine
moduli schemes $M(c)$ and $\bM(c)$ of framed torsion free sheaves
$(E,\Phi)$ on $\proj^2$ and $\bp$ respectively. For $M(c)$, we take
$c\in H^*(\bp)$ with $(c_1,[C]) = 0$. (See \cite[\S7]{perv} or
\cite[\S3.1, \S3.9]{perv2}.)
They are connected by a sequence of birational morphisms as explained
in the introduction. See \subsecref{subsec:wall-crossing}.

In fact, $M(c)$ was studied earlier in \cite[Chap.~2,3]{Lecture}
(denoted there by ${\mathcal M}(r,n)$). We need to recall one
important property: We have a projective morphism $\pi\colon M(c)\to
M_0(c)$, where $M_0(c)$ is the Uhlenbeck (partial) compactification of
the moduli space $M_0^{\mathrm{reg}}(c)$ of framed locally free
sheaves $(E,\Phi)$. In [loc.\ cit.] $M_0(c)$ was constructed via the
quiver description, and bijective to
\begin{equation*}
  \bigsqcup_{c'} M_0^{\mathrm{reg}}(c')\otimes S^{\Delta(c) - \Delta(c')}(\C^2)
\end{equation*}
set-theoretically. Here $S^{n}(\C^2)$ denotes the $n^{\mathrm{th}}$
symmetric product of $\C^2$.

For any $m$, we still have a projective morphism $\widehat\pi\colon
\bM^m(c)\to M_0(p_*(c))$. This follows from the quiver description
(\thmref{thm:quiver}) or \cite[\S3.2]{perv2}. It is compatible with
the diagram \eqref{eq:flip} and induced from a projective morphism
$\bM^{m,m+1}(c)\to M_0(p_*(c))$.

\subsection{Grassmann bundle structure}

As we mentioned above, we have $M^0(c) \cong M(p_*(c))$ when
$(c_1,[C]) = 0$.
For $0 < (c_1,[C]) < r$, we have a similar relation as follows.
We need to consider $\bM^1(c)$ with $0 > (c_1,[C]) > -r$ instead after
twisting by the line bundle $\shfO(C)$.

\begin{Proposition}[\protect{\cite[\S3.10]{perv2}}]\label{prop:Grassmann}
  Suppose $0 < n \defeq -(c_1,[C]) < r$. There is a variety $\widehat
  N(c,n)$ relating $\bM^1(c)$ and $\bM^1(c-ne_0)$ through a diagram
\begin{equation*}
\xymatrix@R=.5pc{
& \widehat N(c,n) \ar[ld]_{f_1} \ar[rd]^{f_2} &
\\
\bM^{1}(c) & & \bM^{1}(c-ne_0)
}
\end{equation*}
satisfying the followings\textup:

\textup{(1)} $f_1$ is surjective and birational.

\textup{(2)} $f_2$ is the Grassmann bundle 
\(
   \Gr(n,\Ext^1_{q_2}(\shfO_C(-1),\cE'))
\)
of $n$-planes in the vector bundle $\Ext^1_{q_2}(\shfO_C(-1),\cE')$ of
rank $r$ over $\bM^1(c-ne_0)$.

\textup{(3)} We have a short exact sequence
\begin{equation*}
   0 \to (\id_{\bp}\times f_2)^*\cE' \to (\id_{\bp}\times f_1)^*\cE 
   \to \shfO_C(-1)\boxtimes\mathcal S\to 0.
\end{equation*}
Here $\cE$, $\cE'$ are the universal sheaves for $\bM^1(c)$ and
$\bM^1(c-ne_0)$ respectively, and $\mathcal S$ is the universal rank $n$
subbundle of $\Ext^1_{q_2}(\shfO_C(-1),\cE')$ over
$\Gr(n,\Ext^1_{q_2}(\shfO_C(-1),\cE'))$.
\end{Proposition}

Remark that $\Ext^i_{q_2}(\shfO_C(-1),\cE') = 0$ for $i=0,2$ by the
remark after \lemref{lem:Ext+} below. Hence
$\Ext^1_{q_2}(\shfO_C(-1),\cE')$ is a vector bundle, and its rank is
$r$ by Riemann-Roch.

We have $(c_1(c-ne_0),[C]) = (c_1,[C]) + n = 0$. Therefore
$\bMm{1}(c-ne_0)$ becomes $M(p_*(c))$ after crossing the wall between
$0$-stability and $1$-stability.
\begin{NB2}
  Editted according to the referee's remark 2010/04/01
\end{NB2}%

\begin{NB}
  I am not sure that we really need this description. It is probably
  enough to explain $\bM^0(c) = \emptyset$ if $(c_1,[C]) < 0$. Then
  the wall-crossing formula will take care.
\end{NB}

\subsection{Torus action and equivariant homology
  groups}\label{subsec:equivhom}

Let $T$ be the maximal torus of $\SL_r(\C)$ consisting of diagonal
matrices and let $\hT = \C^*\times\C^*\times T$. We have a
$\hT$-action on $\bM^m(c)$ induced from the $\C^*\times\C^*$-action on
$\bp$ given by
\begin{equation*}
  ([z_0:z_1:z_2],[z:w])\mapsto
  ([z_0:t_1 z_1: t_2 z_2],[t_1z:t_2 w])
\end{equation*}
and the change of the framing $\Phi$.
See \cite[\S5]{perv2}.
It was defined exactly as in the case of framed moduli spaces of
torsion free sheaves, given in \cite[\S3]{NY1}. The action is
compatible with one on $M_0(c)$, i.e., $\widehat\pi$ is
$\hT$-equivariant.
All the constructions, which we have explained so far, are canonically
$\hT$-equivariant. For example, we have the canonical $\hT$-action on
the universal sheaf $\cE$.

Let $H^{\hT}_*(X)$ be the $\hT$-equivariant Borel-Moore homology group
of a $\hT$-space $X$ with rational coefficients. Let $H_{\hT}^*(X)$ be
the $\hT$-equivariant cohomology group with rational
coefficients. They are defined for $X$ satisfying a reasonable
condition, say an algebraic variety with an algebraic
$\hT$-action. See, for example, \cite[App.~C]{NY2}. They are modules
over the equivariant cohomology group $H_{\hT}^*(\pt)$ of a point,
isomorphic to the symmetric product of the dual of the Lie algebra,
which we denote by $S(\hT)$.

The projective morphism $\widehat\pi\colon \bM^m(c)\to M_0(p_*(c))$
induce a homomorphism 
\[
   \widehat\pi_*\colon H^{\hT}_*(\bM^m(c))
   \to H^{\hT}_*(M_0(p_*(c))).
\]
We denote this homomorphism $\widehat\pi_*$ by $\int_{\bM^m(c)}$,
since we also use similar push-forward homomorphisms from homology
groups of various moduli schemes or stacks and want to emphasize the
domain.

On the other hand, the target space $M_0(p_*(c))$ is not at all
important. We can compose the push-forward homomorphism for the
inclusion $M_0(p_*(c))\subset M_0(c')$ for $\Delta(c')\ge\Delta(p_*(c))$.
Then $\int_{\bM^m(c)}$ takes values in $H^{\hT}_*(M_0(c'))$.
We can also make $\int_{\bM^m(c)}$ with values in $\mathfrak S(\hT)$,
the quotient field of $S(\hT)$ as follows:
Recall that $\hT$ has the unique fixed point $0$ in $M_0(p_*(c))$
\cite[Prop.2.9(3)]{NY1}. We compose $\int_{\bM^m(c)}$ with the inverse
$\iota_{0*}^{-1}$ of the push-forward homomorphism $\iota_{0*}$ for
the inclusion $\{0\}\to M_0(p_*(c))$ by using the localization theorem
for the equivariant homology group, which says $\iota_{0*}$ becomes an
isomorphism after taking tensor products with $\mathfrak S(\hT)$ over
$H^*_{\hT}(\pt) = S(\hT)$.
This is compatible with the above inclusion.
See \cite[\S4]{NY1} for more detail.

\subsection{Weak wall-crossing formula}\label{subsec:statement}
We now state our first main result in this subsection.

Let $\Phi(\cE)\in H_{\hT}^*(\bM^m(c))$ be an equivariant
\begin{NB}
multiplicative   
(no longer assumed)
\end{NB}%
cohomology class on $\bM^m(c)$ defined from a sheaf $\cE$ on
$\bp\times\bM^m(c)$ by taking a slant product by a cohomology class on
$\bp$, or taking a cohomology group, for example,
{\allowdisplaybreaks
\begin{multline}\label{eq:Phi}
  \Phi(\cE) \defeq
      \exp\left[
      \sum_{p=1}^\infty \left\{
        t_p \ch_{p+1}(\cE)/[C] + \tau_p \ch_{p+1}(\cE)/[\bC]
      \right\}
    \right],
\\
  \text{or }
  \Phi(\cE) \defeq 
  \prod_{f=1}^{N_f} e(\Vcal_a(\cE)\otimes e^{m_f})
  \quad \text{$a=0$ or $1$},
\end{multline}
where} $t_p$, $\tau_p$ are variables and the exponential defines formal
power series in $t_p$, $\tau_p$ in the first case, and
$m_1$, \dots, $m_{N_f}$ are variables for the equivariant cohomology
$H^*_{(\C^*)^{N_f}}(\pt)$ of the $N_f$-dimensional torus of a point, and
$e^{m_f}$ is the corresponding equivariant line bundle.
For $\cE$ we typically take the universal sheaf, or its variant.
For the latter $\Phi(\cE) = \prod_{f=1}^{N_f} e(\Vcal_a(\cE)\otimes
e^{m_f})$, we need to enlarge $\hT$ to $\hT\times (\C^*)^{N_f}$ but
keep the notation $\hT$ for brevity. And $e(\ )$ denotes the
equivariant Euler class.

\begin{Remark}
  The notation $N_f$ is taken from physics literature. It is the
  number of {\it flavors}. But we denote the rank by $r$, though it is
  denoted by $N_c$ (number of {\it colors\/}) in physics literature.
\begin{NB2}
  Added according to the referee's remark 2010/04/01
\end{NB2}%
\end{Remark}

The above examples of $\Phi$ are {\it multiplicative}, i.e.,
$\Phi(\cE\oplus \cE') = \Phi(\cE)\Phi(\cE')$. This condition is useful
when we will study the vanishing theorem in \secref{sec:vanish}. But
we do not assume it in general.

For $j\in \Z_{>0}$ we consider the $j$-dimensional torus $(\C^*)^j$
acting trivially on moduli schemes. We denote the $1$-dimensional
weight $n$ representation of the $i^{\mathrm{th}}$ factor by $e^{n
  \hbar_i}$.
\begin{NB}
(Denoted by $I_n$ previously.)  
\end{NB}%
The equivariant cohomology $H_{(\C^*)^j}^*(\pt)$ of the point is
identified with $\C[\hbar_1,\dots,\hbar_j]$.
\begin{NB2}
  We keep as it is, since $H^*$ is a direct {\it sum\/} of $H^k$ for
  all $k$, not direct {\it product}. We do not use $e^{n \hbar_i}$
  directly, as it appears either in $\Phi$ or $e(\ )$. 2010/04/03
\end{NB2}%
In the following formula we invert variables $\hbar_1$,\dots,
$\hbar_j$. See \subsecref{subsec:Euler} for the precise definition.
Also we identify $\Phi(\cE)$ with the homology class
$\Phi(\cE)\cap[\bM^{m+1}(c)]$ and apply the push-forward homomorphism
$\int_{\bM^{m+1}(c)}$.
\begin{NB2}
  Editted according to the referee's remark 2010/04/01
\end{NB2}%

\begin{Theorem}\label{thm:m=0}
\begin{multline*}
  \int_{\bM^{m+1}(c)} \Phi(\cE) - \int_{\bM^{m}(c)} \Phi(\cE)
\\
  = 
  \sum_{j=1}^\infty
   \int_{\bM^{m}(c - je_m)}
   \Res_{\hbar_{j}=0} \cdots \Res_{\hbar_{1}=0}
   \left[
   \Phi(\cEf\oplus\bigoplus_{i=1}^j  C_m\boxtimes e^{-\hbar_i})
   \Psi^{j}(\cEf)\right],
\end{multline*}
where
$\cEf$ is the universal sheaf for $\bM^m(c-je_m)$ and
{\allowdisplaybreaks
\begin{gather*}
  \Psi^{j}(\cEf) \defeq
    \frac1{j!}
      \frac{
    \prod_{1\le i_1\neq i_2 \le j} 
   ({-\hbar_{i_1}+\hbar_{i_2}})
    }{
      \prod_{i=1}^j
      e(\fN(\cEf,C_m)
      \otimes e^{-\hbar_i})
      \, e(\fN(C_m,\cEf)
      \otimes e^{\hbar_i})
    }
    ,
\\
  \fN(\cEf,C_m) \defeq - \sum_{a=0}^2 (-1)^a \Ext^a_{q_2}(\cEf,C_m),
\quad
  \fN(C_m,\cEf) \defeq - \sum_{a=0}^2 (-1)^a \Ext^a_{q_2}(C_m,\cEf).
\end{gather*}
\textup(}Note that $\Psi^j(\cEf)$ depends on $j$, but not on $c-je_m$
if we consider $\cEf$ as a variable.\textup)
\end{Theorem}

\begin{NB}
This looks simpler than \thmref{thm:2nd}.
\end{NB}

The proof will be given in \subsecref{subsec:2nd}.

\subsection{Blow-up formula}\label{subsec:blowup}

Recall that $\bM^m(c)$ is isomorphic to the framed moduli space
$\bM(c)$ of torsion free sheaves on $\bp$ if $m$ is sufficiently
large. 
Using \propref{prop:Grassmann}, \thmref{thm:m=0} and twist by the line
bundle $\shfO(C)$, we can express $\int_{\bM(c)}\Phi(\cE)$ as a sum of
various $\int_{M(c')}\Phi'(\cE)$'s for some $c'$, $\Phi'$.
Unfortunately the procedure, which we will explain below in detail, is
recursive in nature and rather cumbersome. See Figure~\ref{fig:flow}
for the flowchart. In particular, we do not solve the recursion and
do not give the explicit formula.

\subsubsection{}\label{subsub:1st}

For a sequence $\vec{j} = (j_0,j_1,\dots, j_{m-1})\in \Z_{\ge 0}^{m}$, we
define $\Psi^{\vec{j}}_n$ recursively starting from $\Psi^{\vec{j}}_m
= 1$ by
\begin{equation*}
   \Psi^{\vec{j}}_n(\bullet)
\defeq
   \Psi^{\vec{j}}_{n+1}(\bullet\oplus C_n\boxtimes e^{-\hbar_i^n})
   \times
    \frac1{j_n!}
      \frac{
    \prod_{1\le i_1\neq i_2 \le j_n} 
   ({-\hbar_{i_1}^n+\hbar_{i_2}^n})
    }{
      \prod_{i=1}^{j_n}
      e(\fN(\bullet,C_n)
      \otimes e^{-\hbar_i^n})
      \, e(\fN(C_n,\bullet)
      \otimes e^{\hbar_i^n})
    }
   ,
\end{equation*}
where $\hbar^n_1$, \dots, $\hbar^n_{j_n}$ are variables. Then we set
$\Psi^{\vec{j}} = \Psi^{\vec{j}}_0$. By \thmref{thm:m=0} we get
\begin{equation}\label{eq:1st}
  \int_{\bM^m(c)} \Phi(\cE)
  = 
  \sum_{\vec{j}}
   \int_{\bM^0(c - \sum_n j_n e_n)}
   \overrightarrow{\Res_{\vec{\hbar}=0}}
   \,
   \Phi(\cEf\oplus\bigoplus_{n=0}^{m-1}
   \bigoplus_{i=1}^{j_n} C_n\otimes e^{\hbar^n_{i}})
   \Psi^{\vec{j}}(\cEf),
\end{equation}
where $\overrightarrow{\Res}_{\vec{\hbar}=0}$ is the iterated residues
\begin{equation*}
 \overrightarrow{\Res_{\vec{\hbar}=0}} \defeq
 \Res_{\hbar^0_{j_0} = 0} \cdots \Res_{\hbar^0_{1} = 0}
 \Res_{\hbar^1_{j_1} = 0} \cdots \Res_{\hbar^1_{1} = 0}
 \;
 \cdots
 \Res_{\hbar^{m-1}_{j_{m-1}} = 0} \cdots \Res_{\hbar^{m-1}_{1} = 0}.
\end{equation*}

Since $\bM(c)$ is isomorphic to $\bM^m(c)$ for a sufficiently large
$m$, an integral over $\bM(c)$ can be written in terms of integrals
over $\bM^0(c')$ with various $c'$ thanks to this formula.

Note that if $(c_1,[C]) \ge 0$, we have
\begin{equation}\label{eq:dim_est}
  \begin{split}
  & \dim \bM^0(c-\sum_n j_n e_n) 
\\
  = \; &  \dim \bM^m(c) - \sum r(2n+1)j_n 
  -\sum j_n \left(\sum j_n + 2(c_1,[C])\right)
  < \dim \bM^m(c)
  \end{split}
\end{equation}
if $(j_1,j_2,\dots) \neq 0$. 
\begin{NB}
  Consider the ADHM data for $\bM^m(c)$ and we have vector spaces $W$,
  $V_0$, $V_1$. Then
  \begin{equation*}
   \dim \bM^m(c) 
    = \dim W(\dim V_0 + \dim V_1) - (\dim V_0 - \dim V_1)^2.
  \end{equation*}
  The sheaf $C_n = \shfO_C(-n-1)$ corresponds to the data $V_0 =
  \C^n$, $V_1 = \C^{n+1}$. Hence the data for $\bM^0(c-\sum j_n e_n)$
  has $\dim V'_0 = \dim V_0 - \sum nj_n$, $\dim V'_1 = \dim V_1 - \sum
  (n+1)j_n$. Therefore
\begin{equation*}
  \begin{split}
   & \dim \bM^0(c-\sum j_n e_n) =
   \dim W(\dim V_0 + \dim V_1 - \sum (2n+1)j_n )
   - (\dim V_0 - \dim V_1 + \sum j_n)^2
\\
   =\; &
   \dim W(\dim V_0 + \dim V_1) - \sum (2n+1)j_n \dim W
   - (\dim V_0 - \dim V_1)^2 - \left(\sum j_n\right)^2
   - 2(\dim V_0 - \dim V_1)\sum j_n
\\
  =\; &
  \dim \bM^0(c) - \sum r(2n+1)j_n 
  -\sum j_n \left(\sum j_n + 2(c_1,[C])\right).
  \end{split}
\end{equation*}

\end{NB}

\subsubsection{}\label{subsub:recursive}

We usually consider the moduli space $\bM(c)$ with $0 \le (c_1,[C]) <
r$. This is always achieved by tensoring a power of the line bundle
$\shfO(C)$. And then we can hope to relate $\int_{\bM(c)} \Phi(\cE)$
to $\int_{M(p_*(c))}\Phi(\cE)$ thanks to \propref{prop:Grassmann}.
But look at \eqref{eq:1st}. The right hand side of
\eqref{eq:1st} contains integrals over $\bM^0(c')$ ($c' = c-\sum_n j_n
e_n$) for which we have $0\le (c_1(c'),[C])$, but not necessarily $<
r$. Thus we need to tensor a line bundle again.

Since keeping track the precise form of the formula is a rather
tiresome work, we {\it redefine\/} a term in the right hand side of
\eqref{eq:1st} as $\int_{\bM^0(c)} \Phi(\cE)$, and start from it. We
can assume $(c_1,[C])\ge 0$, as we explained.

\subsubsection{}\label{subsub:c_1=0}

First consider the case $(c_1,[C]) = 0$. We have an isomorphism
$\Pi\colon \bM^0(c) \cong M(p_*(c))$ by $(E,\Phi)\mapsto
(p_*(E),\Phi)$, where the higher direct image sheaves $R^{>0}p_*(E)$
vanishes. Moreover its inverse is given by $(F,\Phi)\mapsto
(p^*F,\Phi)$, and $L^{<0}p^* F = 0$.
(See \cite[Prop.~3.3 and \S1]{perv2}.)
If we denote the universal sheaf for $M(p_*(c))$ by $\mathcal F$, the
universal sheaf $\cE$ for $\bM^0(c)$ is equal to
$(p\times\Pi)^*(\mathcal F)$.
Therefore we have
\begin{equation*}
  \int_{\bM^0(c)} \Phi(\cE)
  = \int_{\bM^0(c)} \Phi((p\times\Pi)^*(\mathcal F)).
\end{equation*}
Since $L^{<0}(p\times\Pi)^*(\mathcal F)$ vanishes, this holds in the
level of $K$-group.

We may have expressions $\shfO(C)$ or $[C]$, which do not come from
$M(p_*(c))$ in the expression $\Phi(\cE)$, but we can use the
projection formula to rewrite the right hand side as
\begin{equation*}
  \int_{M(p_*(c))} \Phi'(\mathcal F).
\end{equation*}
for a possibly different cohomology class $\Phi'(\bullet)$.

\subsubsection{}\label{subsub:O(C)}

Now we may assume $(c_1,[C]) > 0$.
We have
\begin{equation}\label{eq:tensor}
  \int_{\bM^0(c)} \Phi(\cE)
  = \int_{\bM^1(c e^{[C]})} \Phi(\cE(-C))
\end{equation}
by the isomorphism $\bM^0(c)\ni (E,\Phi)\mapsto (E(C),\Phi)\in
\bM^1(ce^{[C]})$. Two universal sheaves for $\bM^0(c)$,
$\bM^1(ce^{[C]})$ are denoted by the same notation, but the latter is
twisted by $\shfO(C)$ from the former under this isomorphism, and it
is the reason why we have $\Phi(\cE(-C))$. 

We have $-r < (c_1(ce^{[C]}),[C]) = (c_1,[C]) - r$. If this is
negative, in other words, if we have $(c_1,[C]) < r$, we go to
the step which will be explained in \subsecref{subsub:Grassmann}. So
we assume $(c_1(ce^{[C]}),[C]) \ge 0$.
We now {\it redefine\/} the right hand side of \eqref{eq:tensor} as
$\int_{\bM^1(c)} \Phi(\cE)$ and return back to \subsecref{subsub:1st}
and apply \eqref{eq:1st} with $m=1$.

We repeat this procedure until all terms are integrals over
$\bM^1(c')$ with $-r < (c_1(c'),[C]) < 0$, or $\bM^0(c'')$ with
$c_1(c'') = 0$. From the dimension estimate \eqref{eq:dim_est}, the
procedure ends after finite steps.

For $\Phi(\cE) = \prod_{f=1}^{N_f} e(\Vcal_a(\cE)\otimes e^{m_f})$,
this process requires a care, since $\Vcal_a(\cE(-C)) = R^1
q_{2*}(\cE(-C)\otimes q_1^*(\shfO(aC-\linf)))$ may not be a vector
bundle on $M^0(c)$, so $e(\Vcal_a(\cE(-C))\otimes e^{m_f})$ does not
make sense, and we cannot apply \eqref{eq:1st} with $m=1$.
We overcome this difficulty by replacing $e(\Vcal_a(\cE(-C))\otimes
e^{m_f})$ by a product of $e(\Vcal_a(\cE)\otimes e^{m_f})$ and a
certain class, which is well-defined on $M^0(c)$. See the proof of
\thmref{thm:matter_gap} for detail.
\begin{NB2}
Changed. 2010/04/02

Earlier version:

Therefore we replace $R^1 q_{2*}$ by the alternating sum
$-R^0q_{2*} + R^1 q_{2*} - R^2 q_{2*}$ and consider equivariant Euler
class as in \subsecref{subsec:Euler} by inverting $m_f$.
\end{NB2}%

\subsubsection{}\label{subsub:Grassmann}

Now we {\it redefine\/} the right hand side of \eqref{eq:tensor}
as $\int_{\bM^1(c)} \Phi(\cE)$ and consider it under the assumption
$-r < (c_1,[C]) < 0$.

Let $n \defeq -(c_1,[C])$. By \propref{prop:Grassmann} we have
\begin{equation*}
  \begin{split}
  & \int_{\bM^1(c)} \Phi(\cE)
 = \int_{\widehat N(c,n)} \Phi(
  (\id_{\bp}\times f_1)^*\cE)
\\
  =\; &
  \int_{\widehat N(c,n)} 
  \Phi\left(
      (\id_{\bp}\times f_2)^*(\cE) \oplus C_0\boxtimes\mathcal S
      \right),
\end{split}
\end{equation*}
where we denote the universal bundle over
$\bM^1(c-ne_0)$ by $\cE$ for brevity. Since
$f_2\colon \widehat N(c,n)\to \bM^1(c-ne_0)$ is the Grassmann bundle
of $n$-planes in $\Ext^1_{q_2}(C_0,\cE)$, we can pushforward
to $\bM^1(c-ne_0)$ to get
\begin{equation*}
  \int_{\widehat N(c,n)} 
  \Phi\left(
      (\id_{\bp}\times f_2)^*(\cE) \oplus C_0\boxtimes\mathcal S
      \right)
  = \int_{\bM^1(c-ne_0)} {}'\Phi(\cE),
\end{equation*}
where
\begin{equation*}
  {}'\Phi(\bullet) \defeq
  \left.
    \int_{\Gr(n,r)} 
    \Phi\left(\bullet\oplus 
    (C_0\boxtimes\mathcal S)
    \right)
   \right|_{c(\underline{\C^r})= c(\Ext^1_{q_2}(C_0,\cE))}.
\end{equation*}
We need to explain the notation. We consider the Grassmannian
$\Gr(n,r)$ of $n$-planes in $\C^r$, and $\int_{\Gr(r-n,r)}$ is the
pushforward
\(
   H^*_{\GL(r)}(\Gr(r-n,r)) \to H^*_{\GL(r)}(\pt).
\)
The $\bullet$ is a variable living in the $K$-group
$K(\bp\times\pt)$.
The universal subbundle of the trivial bundle
$\underline{\C^r}$ is denoted by $\mathcal S$. And
$C_0\boxtimes\mathcal S$ is a sheaf
on $\bp\times\Gr(r,n)$. We consider $\Gr(r,n)$ as a moduli space
and $\bullet\oplus C_0\boxtimes\mathcal S$ is a universal sheaf,
and apply the function $\Phi$.
Finally $\left.(\ )\right|_{c(\underline{\C^r})=
  c(\Ext^1_{q_2}(C_0,\cE))}$ means that we substitute the Chern
classes of $\Ext^1_{q_2}(C_0,\cE))$ to the equivariant Chern classes
of $\underline{\C^r}$ in $H^*_{\GL(r)}(\pt)$.

We now redefine $\int_{\bM^1(c)}\Phi(\cE)$ as $\int_{\bM^1(c-ne_0)}
{}'\Phi(\cE)$, and return to \subsecref{subsub:1st}. Since $\dim
\bM^1(c-ne_0) < \dim \bM^1(c)$, this procedure eventually stop.

\begin{figure}[htbp]
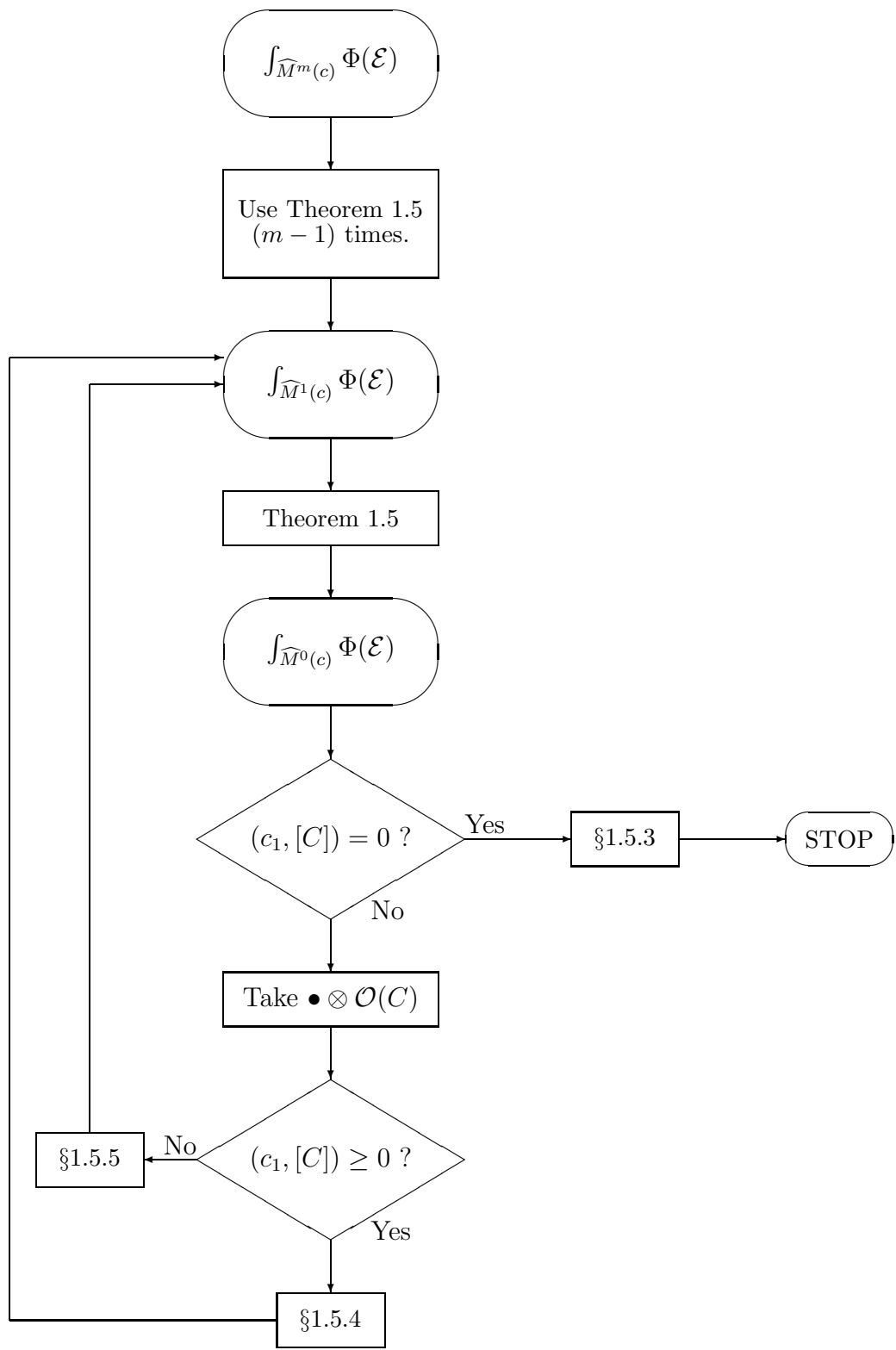

  \centering
\input test.pic
\caption{flowchart}
\label{fig:flow}
\end{figure}

\begin{NB}
  The following earlier version has a gap, for the discussion on the
  twist by $\shfO(mC)$ with $m > 1$.

We consider a sequence $\mathfrak{J} = (\vec{j}^1,\vec{j}^2,\cdots)$
of $\vec{j}^k = (j^k_0,j^k_1,\dots)\in\Z_{\ge 0}^\infty$ with the
following recursive properties and auxiliary sequence $(c^1, c^2,
\cdots) \in H^*(\bp)^{\infty}$, $(m^1,m^2,\cdots)\in \Z_{\ge
  0}^\infty$:
\begin{enumerate}
\item For $k= 1$, no restriction on $\vec{j}^1$. Set $c^1 = \sum_m
  j^1_m e_m$, $m^1 = 0$.
\item Suppose that the conditions
  on $\vec{j}^k$ is imposed, and $c^k$, $m^k$ are defined.
\item Consider $(c_1 - c^k,[C])$. If it is $0$, then
  $\vec{j}^{k+1}\equiv \vec{j}^{k+2} \equiv \cdots \defeq 0$. We set
  $c^{k+1} \equiv c^{k+2} \equiv \cdots \defeq c^k$, $m^{k+1} \equiv
  m^{k+2} \equiv \cdots \defeq 0$ and stop the recursion.
\item If $0 < (c_1 - c^k,[C]) < r$, then we impose the condition
  $j^{k+1}_m = 0$ for $m \ge 1$. We set $m^{k+1} = 1$, 
\(
   c^{k+1} = 
   - (c_1,[C])[C] - \frac12( r + (c_1 - c^k,[C]))\pt
   + j^{k+1}_0 e_0.
\)
Go to (2) after increasing $k$ by $1$.

\item If $(c_1 - c^k,[C]) \ge r$, then let $m^{k+1} \defeq \lfloor
  (c_1 - c^k,[C]) / r \rfloor$. We impose the condition $j^{k+1}_m = 0$
  for $m \ge m^{k+1}$. We set 
\( 
   c^{k+1} \defeq c^k e^{m^{k+1}[C]} - (r + c_1) (e^{m^{k+1}[C]} - 1)
     + \sum j^{k+1}_m e_m.
\)
 Go to (2) after increasing $k$ by $1$.
\item The recursion stops at the case (3) for some $k$.
\end{enumerate}
We introduce variables $\hbar^{k,n}_i$ ($k=1,2,\dots$,
$n=0,1,\dots,m^{k+1}-1$, $i=1,2,\dots,j_n$) and
define $\Phi^{k}$ recursively as
\begin{enumerate}
\item For $k=1$ we set
\(
  \Phi^{1}(\bullet) 
  = \Phi(\bullet\oplus\bigoplus_{n=0}^{m-1}
   \bigoplus_{i=1}^{j_n} C_n\otimes e^{\hbar^{n,1}_{i}})
   \Psi^{\vec{j}^1}(\bullet)
\)
defined as in \S\ref{subsub:1st}.
\item Suppose that $\Phi^{k}(\bullet)$ is defined.
\item If $(c_1^k,[C]) = 0$, then we put
   $\Phi^{k+1}(\bullet) \defeq \Phi^{k}(\bullet)$ and stop.
\item If $0 < n\defeq (c_1^k,[C]) < r$, then
\begin{equation*}
  \begin{split}
  & {}'\Phi^{k}(\bullet) \defeq
  \left.
    \int_{\Gr(r-n,r)} 
    \Phi^k\left(\left(\bullet\oplus 
    (C_0\boxtimes\mathcal S)
    \right)
    \otimes\shfO(-C)
    \right)
   \right|_{c(\underline{\C^r})= c(\fN(C_0,\bullet
     ))},
\\
  & \Phi^{k+1}(\bullet) \defeq
  {}'\Phi^{k}\left(\bullet\,
      \oplus \bigoplus_{i=1}^{j_0^{k+1}}
      C_0\otimes e^{\hbar^{k+1,0}_{i}}
    \right) 
  \Psi^{{j}^{k+1}_0}(\bullet),
  \end{split}
\end{equation*}
where $\mathcal S$ is the universal subbundle of the trivial bundle
$\underline{\C^r}$ over the Grassmannian $\Gr(r-n,r)$ of
$(r-n)$-planes in $\C^r$, $\int_{\Gr(r-n,r)}$ is the pushforward
\(
   H^*_{\GL(r)}(\Gr(r-n,r)) \to H^*_{\GL(r)}(\pt)
\)
and $\left.(\ )\right|_{c(\underline{\C^r})= c(\fN(C_0,\bullet))}$ means
that we substitute the Chern classes of $\fN(C_0,\bullet)$ to
the equivariant Chern classes of $\underline{\C^r}$ in $H^*_{\GL(r)}(\pt)$.

\item If $(c_1^k,[C]) \ge r$, then 
\begin{equation*}
  \begin{split}
  & {}'\Phi^{k}(\bullet)
  \defeq
  \Phi^{k}(\bullet\otimes\shfO(-m^{k+1}C)),
\\
  & \Phi^{k+1}(\bullet) \defeq
  {}'\Phi^{k}\left(\bullet\,
      \oplus \bigoplus_{n=0}^{m^{k+1}-1} \bigoplus_{i=1}^{j_n^{k+1}}
      C_n\otimes e^{\hbar^{k+1,n}_{i}}
    \right) 
  \Psi^{\vec{j}^{k+1}}(\bullet),
  \end{split}
\end{equation*}
where $m^{k+1} = \lfloor (c_1 - c^k,[C]) / r \rfloor$ as above.
\end{enumerate}

We then set
$c^{\mathfrak{J}} = c^k$,
$\Phi^{\mathfrak{J}} = \Phi^k$ for $k$ with the case (3).

\begin{Theorem}\label{thm:blowup}
We have
  \begin{equation*}
     \int_{\bM(c)} \Phi(\cE)
     = \sum_{\mathfrak J} \int_{M^0(c-c^{\mathfrak{J}})}
     \Phi^{\mathfrak J}(\cEf).
  \end{equation*}
\end{Theorem}

\begin{proof}
We want to prove
\begin{equation}\label{eq:k}
  \int_{\bM(c)} \Phi(\cE)
  = \sum_{\vec{j}^1} \sum_{\vec{j}^2} \cdots
  \sum_{\vec{j}^k} \int_{\bM^0(c - c^k)} 
  \Phi^k(\cEf)
\end{equation}
by a recursion on $k$.

\begin{enumerate}
\item For $k=1$, the assertion is nothing but \eqref{eq:1st} with
  sufficiently large $m$.
  \begin{NB2}
\begin{equation*}
  \int_{\bM(c)} \Phi(\cE)
  = 
  \sum_{\vec{j}^1}
   \int_{\bM^0(c - c^1)} \Phi^1(\cEf)
  =
  \sum_{\vec{j}^1}
   \int_{\bM^0(c - \sum_m j^1_m e_m)} \Phi(\cEf)
   \overrightarrow{\Res_{\vec{\hbar}=0}} \Psi^{\vec{j}^1}(\cEf),
\end{equation*}
  \end{NB2}

\item Suppose that \eqref{eq:k} holds for $k$. We want to prove
  \begin{equation*}
    \int_{\bM^0(c - c^k)} \Phi^k(\cEf)
    = \sum_{\vec{j}^{k+1}} \int_{\bM^0(c - c^{k+1})} 
    \Phi^{k+1}(\cEf)
    .
  \end{equation*}
\item If $(c_1-c^k,[C]) = 0$, this is a trivial identity from the
  definitions of $\vec{j}^{k+1}$, $c^{k+1}$, $m^{k+1}$,
  $\Phi^{k+1}(\cEf)$.

\item If $0 < n\defeq (c_1 - c^k,[C]) < r$, we use \propref{prop:Grassmann}
  after taking the tensor product $\shfO(C)$:
\begin{equation*}
  \begin{split}
  & \int_{\bM^0(c-c^k)} \Phi^k(\cEf)
  = \int_{\bM^1((c-c^k)e^{[C]})} \Phi^k(\cEf(-C))
\\
  =\; &\int_{\widehat N((c-c^k)e^{[C]},r-n)} \Phi^k(
  (\id_{\bp}\times f_1)^*\cEf\otimes \shfO(-C))
\\
  =\; &
  \int_{\widehat N((c-c^k)e^{[C]},r-n)} 
  \Phi^k\left(\left(
      (\id_{\bp}\times f_2)^*(\cEf) \oplus C_0\boxtimes\mathcal S
      \right)
  \otimes \shfO(-C)\right)
\\
  =\; &
  \int_{M^1((c-c^k)e^{[C]}-(r-n)e_0)} 
  {}'\Phi^k(\cEf).
  \end{split}
\end{equation*}
\begin{NB2}
\begin{equation*}
  \begin{split}
    & (c-c^k)e^{[C]}-(r-n)e_0 = c - c(1 - e^{[C]}) - c^k e^{[C]} - (r-n)e_0
\\
  =\; & c - (r + c_1)(1 - e^{[C]}) - c^k e^{[C]} - (r-n)e_0
\\
  =\; & c + (n + c_1)e_0 - c^k e^{[C]} 
  = c + n [C] - c_1(c^k) 
  + \left(- \frac{n}2 + (c_1,[C]) + \frac{r}2 - (c_1(c^k),[C])\right)\pt
\\
  =\; & c + (c_1,[C]) [C] + \frac{r+n}2 \pt
  = c + (c_1,[C])[C] + \frac12( r + (c_1 - c^k,[C]))\pt.
  \end{split}
\end{equation*}
\end{NB2}%
  We apply \eqref{eq:1st} for the right hand side to get the assertion
  for $k+1$.

\item If $(c_1 - c^k,[C]) \ge r$, then we have
  \begin{equation*}
    \int_{\bM^0(c-c^k)} \Phi^k(\cEf)
    = \int_{\bM^{m^{k+1}}((c-c^k)e^{m^{k+1}[C]})} \Phi^k(\cEf(-m^{k+1}C))
  \end{equation*}
  We apply \eqref{eq:1st} for the right hand side to get the assertion
  for $k+1$.
  \begin{NB2}
  We have
  \begin{equation*}
    (c-c^k)e^{m^{k+1}[C]} - \sum j^{k+1}_m e_m
    =  c - \left(c^k e^{m^{k+1}[C]} - (r + c_1) (e^{m^{k+1}[C]} - 1)
     + \sum j^{k+1}_m e_m\right),
  \end{equation*}
  as $c (e^{m^{k+1}[C]} - 1) = (r + c_1) (e^{m^{k+1}[C]} - 1)$.
  \end{NB2}%
\end{enumerate}
\end{proof}
\end{NB}

\begin{NB}
{\bf The following subsection is included in the previous one......}  
\subsection{Twisting by a line bundle}

Thus we get
\begin{equation*}
  \int_{\bM^0(c)} \Phi 
  - \int_{\bM(c)} \Phi
  = \sum \text{ of integrals over $\bM^0(c')$ with various 
    $c' = c - \sum_{m=1}^\infty p_m e_m$}.
\end{equation*}

We need to have $0 \ge (c_1,[C])$ to relate $\bM^0(c)$ to
framed moduli spaces on $\proj^2$. 
\begin{NB2}
$\bM^0(c) = \emptyset$ if $(c_1,[C]) < 0$.
\end{NB2}%
Since
\(
   (c_1(c'),[C]) = (c_1,[C]) + \sum_{m=0}^\infty p_m
\)
need not to satisfy this condition, we use an isomorphism
\(
  \bM^0(c') \cong \bM^{1}(c'e^{[C]});
  (E,\Phi)\mapsto (E(C),\Phi)
\)
to go back the next chamber.
\begin{NB2}
I am not sure that we can go to
$\bM^{n}(c'e^{n[C]})$ with
 $n \defeq \lfloor\nicefrac{(c_1(c'),[C])}r - 1\rfloor + 1$.
\end{NB2}

Since the pullback of the universal sheaf on $\bM^1(c'e^{[C]})$ is
$\cE(C)$ on $\bM^0(c')$, we need to understand the relation between
$\Phi(\cE(C))$ and $\Phi(\cE)$.
\end{NB}

\subsection{Example}\label{subsec:Example}

Consider the case $c \in H^*(\bp)$ with $r(c) = r$, $c_1(c) = 0$,
$(\Delta(c),[\bp]) = 1$.
\begin{NB}
  We have $\dim V_0 = \dim V_1 = 1$.
\end{NB}
In \thmref{thm:m=0} the wall-crossing term appears only in the case
$m=0$, $j=0$. Therefore
\begin{equation}\label{eq:temp}
   \int_{\bM(c)} \Phi(\cE) - \int_{\bM^0(c)} \Phi(\cE)
   = \int_{\bM^0(c-e_0)} 
   \Res_{\hbar_1=0}
   \frac{
   \Phi(\cEf\oplus C_0\boxtimes e^{-\hbar_1})
   }
   {
      e(\fN(\cEf,C_0)
      \otimes e^{-\hbar_1})
      \, e(\fN(C_0,\cEf)
      \otimes e^{\hbar_1})
    }
\end{equation}
In the quiver description \thmref{thm:quiver} for $\bM^0(c-e_0)$ we
have $V_0 = \C$, $V_1 = 0$ and hence
\begin{equation*}
  \bM^0(c-e_0) \cong \proj^{r-1}.
\end{equation*}
This also follows from \propref{prop:Grassmann}. In fact, $f_1$ is an
isomorphism in this case.
We also see that
\( 
  \cEf\cong 
  \Ker\left[
  \shfO_{\proj}^{\oplus r}\to \shfO_{\proj}(1)\boxtimes
  \shfO_C
  \right].
\)
Then we have
\begin{equation*}
      \fN(\cEf,C_0) \cong \shfO_\proj(-1), \quad
      \fN(C_0,\cEf) \cong \shfO_\proj(1)^{\oplus 2} \oplus \mathcal S,
\end{equation*}
where $\shfO_\proj(1)$ is the hyperplane bundle of $\proj=\proj^{r-1}$
and $\mathcal S$ is the universal subbundle, i.e., kernel of
$\shfO^{\oplus r}_\proj\to \shfO_\proj(1)$. This also follows from
Lemmas~\ref{lem:Ext-},\ref{lem:Ext+}.
Therefore
\begin{equation*}
  \begin{split}
  & e(\fN(\cEf,C_0)\otimes e^{-\hbar_1})
  = - c_1(\shfO_{\proj}(1)) - \hbar_1,
\\
  & e(\fN(C_0,\cEf) \otimes e^{\hbar_1})
  = (c_1(\shfO_{\proj}(1)) + \hbar_1)^2 e(\mathcal S\otimes e^{\hbar_1})
  = \hbar_1^{r}(c_1(\shfO_{\proj}(1)) + \hbar_1).
  \end{split}
\end{equation*}
\begin{NB}
We observe that the second term is equal to
\(
  \frac{\hbar_1^r(c_1(\shfO_{\proj}(1)) + \hbar_1)^2}
  { \hbar_1 + c_1(\shfO_{\proj}(1))}.
\)
\end{NB}%

For $\Phi$, we consider a simplest nontrivial case.
Let $\mu(C)$ be the cohomology class on $\bM^m(c)$ given by
\begin{equation}\label{eq:mu(C)}
  \mu(C) \defeq \Delta(\cE)/[C],
\end{equation}
where $/$ denotes the slant product
\(
  /\colon H^d_{\hT}(\bp\times\bM^m(c))\otimes H_i^{\hT}(\bp))
  \to H^{d-i}_{\hT}(\bM^m(c)).
\)
This is the $\mu$-map appearing in the usual Donaldson invariants.
We have
\begin{equation*}
  \Delta(\cEf\oplus C_0\boxtimes e^{-\hbar_1})/[C]
  = - c_1(\shfO_{\proj}(1)) - \hbar_1 + \ve_1+\ve_2,
\end{equation*}
\begin{NB2}
  Editted according to the referee's remark 2010/04/01
\end{NB2}%
where $\ve_1$, $\ve_2$ are generators of
$\operatorname{Lie}(\C^*\times\C^*)$ corresponding to $t_1$, $t_2$.
\begin{NB}
Note
\begin{equation*}
  \ch(\cEf\oplus C_0\boxtimes e^{-\hbar_1})
  = \ch(\cEf) + \ch(C_0\boxtimes e^{-\hbar_1})
  = \sum_{\alpha=1}^r e^{a_\alpha} - \ch(\shfO_{\proj}(1))\ch(C_{-1})
  + \ch(C_0) e^{-\hbar_1}.
\end{equation*}
Therefore
\begin{equation*}
  \begin{split}
  & \Phi(\cEf\oplus C_0\boxtimes e^{-\hbar_1})
  = -\ch_2(\cEf\oplus C_0\boxtimes e^{-\hbar_1})/[C]
\\
  =\;& c_1(\shfO_{\proj}(1))\ch_1(C_{-1})/[C] + \hbar_1\ch_1(C_0)/[C]
  + \ch_2(C_{-1})/[C] - \ch_2(C_0)/[C]
\\
  =\;& - c_1(\shfO_{\proj}(1)) - \hbar_1
  + \ve_1 + \ve_2
  \end{split}
\end{equation*}
as
\begin{equation*}
  \ch_1(C_m)/[C] = -1, \qquad
  \ch_2(C_m)/[C] = -(m+\frac12)(\ve_1+\ve_2).
\end{equation*}
\end{NB}

We also have
\begin{equation*}
  \Vcal_1(\cEf\oplus C_0\boxtimes e^{-\hbar_1}) \cong e^{-\hbar_1}.
\end{equation*}
Hence
\begin{equation*}
  \prod_{f=1}^{N_f} e(\Vcal_1(\cEf\oplus C_0\boxtimes e^{-\hbar_1})\otimes e^{m_f})
  = \prod_{f=1}^{N_f} (m_f - \hbar_1).
\end{equation*}

We assume $2r-N_f \ge 1$ and take
\(
  \Phi(\cE) = \mu(C)^{2r-N_f} \prod_{f=1}^{N_f} e(\Vcal_1(\cE)\otimes e^{m_f}).
\)
Then the right hand side of \eqref{eq:temp} becomes
\begin{equation*}
  \int_{\proj^{r-1}} \Res_{\hbar_1=0}
  \frac{
    (- c_1(\shfO_{\proj}(1)) - \hbar_1 + \ve_1 + \ve_2)^{2r-N_f}
  }{-\hbar_1^r (c_1(\shfO_\proj(1)) + \hbar_1)^2}
  \prod_{f=1}^{N_f} (m_f - \hbar_1).
\end{equation*}
By the degree reason, this must be a constant in $\ve_1$, $\ve_2$,
$m_f$. Therefore we may set all $0$. Then this is equal to
\begin{equation*}
  -\int_{\proj^{r-1}} \Res_{\hbar_1=0}
  \hbar_1^{N_f - r}
    (c_1(\shfO_{\proj}(1)) + \hbar_1)^{2r-N_f-2}
    = -\binom{2r-N_f-2}{r-1}.
\end{equation*}
If $r=2$, $N_f=0$, the answer is $-2$. This is a simplest case of the
blow-up formula, which was used to define Donaldson invariants for
$c_2$ in the unstable range.

\section{Applications -- Vanishing theorems}\label{sec:vanish}

As we mentioned above, the wall-crossing formula only gives us a
recursive procedure to give the blow-up formula.
In this section, we concentrate on a rather special $\Phi(\cE)$ and
derive certain vanishing theorems. They turn out to be enough for
applications to the instanton counting.

\subsection{Theory with matters}
Let $\mu(C)$ be as in \eqref{eq:mu(C)}. We consider
\begin{equation*}
  \Phi(\cE) =
  \prod_{f=1}^{N_f} e(\Vcal_0(\cE)\otimes e^{m_f})
  \times
  \exp \left( 
    t \mu(C) 
  \right),
\end{equation*}
and study the coefficient $\Phi_d(\cE)$ of $t^d$ with small $d$. We
assume $N_f\le 2r$ hereafter.

\begin{Theorem}\label{thm:matter_gap}
Suppose $(c_1,[C]) = 0$. Then
\begin{equation*}
    \int_{\bM^m(c)} \Phi(\cE) 
    =   \int_{M(p_*(c))} \prod_{f=1}^{N_f} e(\Vcal(\cE)\otimes e^{m_f})
    + O(t^{k}),
\end{equation*}
\begin{NB2}
  Editted according to the referee's remark 2010/04/01
\end{NB2}%
where $k = \max(r+1, 2r - N_f)$. Here $\Vcal(\cE) = R^1
q_{2*}(\cE\otimes q_1^*(\shfO(-\linf)))$ is defined from the universal
sheaf $\cE$ on $\bp\times M(p_*(c))$ as in the case of $\Vcal_0(\cE)$,
$\Vcal_1(\cE)$.
\end{Theorem}

This, in particular, means that $\Phi_d(\cE) = 0$ for
$d=1,\dots,k-1$. When $N_f = 0$, this vanishing was shown in
\cite[\S6]{NY1} by the dimension counting argument. The key point was
that $\dim M(c) = 2r\Delta$, and hence the smaller moduli spaces have
codimension greater than or equal to $2r$. Once the wall-crossing
formula is established as in the previous section, the remaining
argument below is similar,
\begin{NB2}
  Expanded. 2010/04/01
\end{NB2}%
and the bound $2r - N_f$ comes from the
fact that the `virtual fundamental class' $\prod_{f=1}^{N_f}
e(\Vcal(\cE)\otimes e^{m_f})\cap [M(p_*(c))]$ has dimension
$(2r-N_f)\Delta$.

\begin{proof}
Let us compute the cohomological degrees of the both sides of the
equality in \thmref{thm:m=0}, where we say
\(
   \int_{\bM^{m}(c)} \heartsuit
\)
{\it has degree $k$\/} if it is contained in
\(
   H_{2k}^{\hT}(M_0(p_*(c))).
\)
We have
\begin{multline}\label{eq:deg}
  \deg \int_{\bM^{m+1}(c)} \Phi_d(\cE)
  = \deg \int_{\bM^m(c)} \Phi_d(\cE)
\\
  = \dim \bM^m(c) - N_f \dim V_0(c) - d
  = - (2r - N_f) (\ch_2(c),[\bp]) - d.
\end{multline}
On the other hand, we can write 
\begin{equation*}
   \int_{\bM^m(c-je_m)} 
   \Phi(\cEf\oplus\bigoplus_{i=1}^j  C_m\boxtimes e^{-\hbar_i})
   \Psi^j(\cEf)
   = \int_{\bM^m(c-je_m)} \prod_{f=1}^{N_f} e(\Vcal_0(\cEf)\otimes e^{m_f})
   \cup\heartsuit
\end{equation*}
for some cohomology class $\heartsuit$. Therefore its degree is at most
\begin{equation*}
  \begin{split}
   & \dim \bM^m(c-je_m) - N_f \dim V_0(c-je_m)
\\
   =\; & - (2r - N_f) (\ch_2(c),[\bp]) - j(m(2r-N_f) + r + j).
  \end{split}
\end{equation*}
\begin{NB}
  \begin{equation*}
    \begin{split}
   \mathrm{LHS} &= r \left(v_0(c-je_m)+v_1(c-je_m)\right) 
   - \left(v_0(c-je_m)-v_1(c-je_m)\right)^2
   - N_f v_0(c-je_m)
\\
   & = r (v_0(c)+v_1(c)) - N_f v_0(c) - rj(2m+1) - j^2 + j mN_f 
   = \mathrm{RHS}
    \end{split}
  \end{equation*}
\end{NB}
Since
\(
   j(m(2r-N_f) + r + j) \ge r+1,
\)
\begin{NB}
(the equality holds if and only if $j=1$, $m(2r-N_f) = 0$)
\end{NB}
it is zero if $d\le r$.

\begin{NB}
Here is Kota's argument for
\begin{description}
\item[(a)] $d \le r$. 
\end{description}

We assign $\deg \hbar_i = 1$ and $0$ for all other variables. Let us
compute the degree of $\Psi^j_{d'}(\cEf)$, the coefficient of $t_1^{d'}$ in
$\Phi^j(\cEf)$ as in \thmref{thm:m=0}. Namely
\begin{gather*}
  \Psi^j_{d'}(\cEf) =
    \frac1{j! d'!} 
      \prod_{i=1}^j
      \frac{
        \prod_{f=1}^{N_f} e(\Vcal_0(C_m)\otimes e^{m_f - \hbar_i})
        \left(
          - \hbar_i
          + (m + \frac12)(\ve_1+\ve_2)
        \right)^{d'}
    }{
      e(\fN(\cEf,C_m)
      \otimes e^{-\hbar_i})
      \, e(\fN(C_m,\cEf)
      \otimes e^{\hbar_i})
    }
    \prod_{1\le i_1\neq i_2 \le j} 
   ({-\hbar_{i_1}+\hbar_{i_2}})
   .
\end{gather*}
Note that $\cEf$ is the universal sheaf for $\bM^m(c - je_m)$. Then
\begin{gather*}
  \deg e(\fN(\cEf,C_m)\otimes e^{-\hbar_i})
  = mr + (c_1,[C]) + j,
\quad
  \deg e(\fN(C_m,\cEf) \otimes e^{\hbar_i})
  = (m+1)r + (c_1,[C]) + j,
\\
  \deg e(\Vcal_0(C_m)\otimes e^{m_f - \hbar_i}) = m,
\quad
  \deg \left(
          - \hbar_i
          + (m + \frac12)(\ve_1+\ve_2)
        \right)^{d'} = d'.
\end{gather*}
Therefore in total,
\begin{equation*}
   \deg \Psi^j_{d'}(\cEf)
   \begin{NB2}
   = j (m N_f + d' + j - 1 - 
   ((2m+1) r + 2(c_1,[C]) + 2j))
   \end{NB2}%
   = j (d' + m (N_f - 2r) - r - 2(c_1,[C]) - j - 1).
\end{equation*}

Under the assumption $c_1 = 0$ and {\bf (a)} $d\le r$, we have
$d'\le d$ and
\begin{equation*}
   \deg \Psi^j_{d'} \le - j(j+1) \le - j - 1.
\end{equation*}
Therefore $\Phi^j_{d'}$ becomes $0$ if we take the iterated
residues. Hence we do not have contributions to the wall-crossing term.
\end{NB}

In order to prove the vanishing for $d\le 2r - N_f - 1$, we need a
refinement of the general machinery in \subsecref{subsec:blowup}.
We need to look at each step in the flowchart (Figure~\ref{fig:flow})
more closely.

The first step \subsecref{subsub:1st} has no problem. In
\eqref{eq:1st} we have
\begin{equation*}
    \Phi(\cEf\oplus\bigoplus_{n=0}^{m-1}
      \bigoplus_{i=1}^{j_n} C_n\otimes e^{\hbar^{n,1}_{i}})
    = \Phi(\cEf)\prod_{n=0}^{m-1}\prod_{i=1}^{j_n} 
     \Phi(C_n\otimes e^{\hbar^{n,1}_{i}}),
\end{equation*}
as we have a decomposition of a vector bundle 
\(
  \Vcal_0(\cE \oplus\bigoplus_{n=0}^{m-1}
      \bigoplus_{i=1}^{j_n} C_n\otimes e^{\hbar^{n,1}_{i}})
= \Vcal_0(\cE)\oplus
\bigoplus_{n=0}^{m-1}
      \bigoplus_{i=1}^{j_n} \Vcal_0(C_n\otimes e^{\hbar^{n,1}_{i}}),
\)
and the Euler class has a multiplicative with respect to the Whitney
sum.

In \S\ref{subsub:O(C)}, \ref{subsub:Grassmann} we consider the tensor
product $\cE(-C)$, where $\cE$ is the universal family on the
moduli space of $1$-stable sheaves.
This causes a trouble because $R^1 q_{2*}(\cE(-C-\linf))$ is not a
vector bundle, as mentioned ealier.
\begin{NB2}
  Editted. 2010/04/02
\end{NB2}%
We need a closer look.

By \cite[Lemma~7.3]{perv} the natural homomorphism $H^1(E(-\linf))\to
H^1(E(C-\linf))$ is surjective for a $0$-stable framed sheaf
$(E,\Phi)$. Therefore $H^1(E(-C-\linf))\to H^1(E(-\linf))$ is
surjective for a $1$-stable framed sheaf $(E,\Phi)$.
Let us give a direct proof since we need to understand the kernel.
Suppose $E$ is $1$-stable. Since $\Hom(E,\shfO_C(-2)) = 0$, we have
$E\otimes\shfO_C/\mathrm{torsion} = \bigoplus \shfO_C(a_i)$ with $a_i
\ge -1$. Therefore $H^1(E\otimes\shfO_C) = 0$. Since we have an exact
sequence
\(
  0 = H^2(\operatorname{Tor}_1(E,\shfO_C)) \to
  H^1(E\otimes^L\shfO_C) \to H^1(E\otimes\shfO_C), 
\) 
it implies $H^1(E\otimes^L\shfO_C) = 0$, and hence
$H^1(E(-C-\linf))\to H^1(E(-\linf))$ is surjective.

Since the kernel of this surjective homomorphism is
$H^0(E\otimes^L\shfO_C)$, we have
\begin{equation*}
  \begin{split}
    e(\Vcal_0(\cE(-C))\otimes e^{m_f})
   &=e(\Vcal_0(\cE)\otimes e^{m_f})\,
   e(q_{2*}(\cE\otimes^L\shfO_C)\otimes e^{m_f})
\\
   &= e(\Vcal_0(\cE)\otimes e^{m_f})\,
   c_a(q_{2*}([\cE]\otimes[\shfO_C])
   \otimes e^{m_f})
  \end{split}
\end{equation*}
on $\bM^1(c)$, where
\(
  a = \dim H^1(E(-C-\linf)) - \dim H^1(E(-\linf))
  = (c_1(E),[C]) + r
\)
and we replace $\cE$, $\shfO_C$, $q_{2*}$ by their $K$-theory classes
and the $K$-theory pushforward in the last expression.

Now $\Vcal_0(\cE)$ is a vector bundle over $\bM^0(c)$, $\bM^1(c)$
\begin{NB2}
  $\bM^0(c)$, $\bM^1(c)$ are correct. But we add an explanation that
  $\Vcal_0(\cE)$ is a rector bundle. 2010/04/03
\end{NB2}%
and master spaces from the quiver description in \secref{sec:quiver}.
\begin{NB2}
  Changed from `or' to `and'.
  2010/04/02
\end{NB2}%
Therefore $e(\Vcal_0(\cE)\otimes e^{m_f})$ (and also
$c_a(q_{2*}([\cE]\otimes[\shfO_C]) \otimes e^{m_f})$) are well-defined
so we replace $e(\Vcal_0(\cE(-C))\otimes e^{m_f})$ by the right hand
side and continue the flow in Figure~\ref{fig:flow}. We may still need
to treat $\otimes\shfO(C)$ in a subsequent process in the
flowchart. Then we again get $e(\Vcal_0(\cE(-C))\otimes e^{m_f})$, so
use the same procedure to replace by the right hand side.

\begin{NB}
Next assume 
\begin{description}
\item[(a')] $N_f \le r - 2$ and $d \le 2r - N_f - 1$,
\end{description}
instead of {\bf (a)} above.
If $m\ge 1$, we have
\begin{equation*}
   j(m(2r-N_f)+r+j)
   \ge 2r - N_f \ge d + 1.
\end{equation*}
\begin{NB2}
In fact, the first inequality is strict.
\end{NB2}
Therefore the right hand side vanishes in this case.
If $m=0$ and $j\ge r$, we have
\begin{equation*}
  j (r + j) \ge 2 r^2 \ge d + 1,
\end{equation*}
\begin{NB2}
  The second inequality is strict.
\end{NB2}%
and hence the right hand side vanishes also in this case.

\begin{NB2}
If $m \ge 1$, we have
\(
   \deg \Psi^j_d \le - j - 1
\)
and we have no contribution to the wall-crossing term.
If $m = 0$ and $j\ge r$, we have
\begin{equation*}
   \deg \Psi^j_d(\cEf)
   = j (d - r - j - 1) 
   \le j(d - 2r - 1) 
   \le j(-N_f - 2) \le - j - 1.
\end{equation*}
Hence we have no contribution.
\end{NB2}

Therefore we may assume $m=0$, $j\le r - 1$. Since $0 < (c_1(c -
je_0),[C]) = j < r$, we can apply \propref{prop:Grassmann} after using the
isomorphism $\bM^0(c-je_0)\cong \bM^1(c'_j);
(E_\flat,\Phi)\mapsto (E_\flat(C),\Phi)$
($c'_j = (c - je_0) e^{[C]}$). From (1) we have
\begin{equation*}
   \int_{\bM^0(c-je_0)} \Phi(\cEf)
   \Res_{\hbar_{j}=0} \cdots \Res_{\hbar_{1}=0}\Psi^j(\cEf)
   = \int_{\widehat N(c'_j,r-j)} f_1^* (\Phi(\cEf)
   \Res_{\hbar_{j}=0} \cdots \Res_{\hbar_{1}=0}
   \Psi^j(\cEf)).
\end{equation*}
Note that $\cEf$ is the universal sheaf on $\bM^0(c-je_0)$, and hence
$\cEf(C) = \cEf\otimes q_1^*\shfO(C)$ is one on $\bM^1(c'_j)$. Using
the short exact sequence
\begin{NB2}
\begin{equation*}
   0 \to (\id_{\bp}\times f_2)^*\cE' 
   \to (\id_{\bp}\times f_1)^*(\cEf(C))
   \to C_0\boxtimes\mathcal S\to 0
\end{equation*}
\end{NB2}%
in (3),
\begin{NB2}
we have
\begin{equation*}
  \begin{CD}
\\
     0 @>>> \mathcal S @>>> R^1q_{2*}(\cE'(-C)) 
       @>>> R^1q_{2*}(\cEf) @>>> 0
\\
       @.    @VVV @| @VVV @.
\\
     0 @>>> \Ext^1_{q_2}(C_0,\cE') @>>> R^1q_{2*}(\cE'(-C))
       @>>> R^1q_{2*}(\cE') @>>> 0
  \end{CD}
\end{equation*}
and hence,
\end{NB2}%
we get
\begin{equation*}
   f_1^* e(\Vcal_0(\cEf) \otimes e^{m_f}) =
   f_2^* e(\Vcal_0(\cE') \otimes e^{m_f})
   \, e((f_2^* \Ext^1_{q_2}(C_0,\cE')/\mathcal S) \otimes e^{m_f}),
\end{equation*}
And
\begin{gather*}
  f_1^* \fN(\cEf,C_0)
  \begin{NB2}
  = 
  f_2^* \fN(\cE'(-C),C_0) + \fN(C_{-1},C_0)\otimes \mathcal S^*
  \end{NB2}%
  =
  f_2^* \fN(\cE',C_1) + \fN(C_0,C_1)\otimes \mathcal S^*,
\\
  f_1^* \fN(C_0,\cEf)
  \begin{NB2}
  = 
  f_2^* \fN(C_0, \cE'(-C)) + \fN(C_0,C_{-1})\otimes \mathcal S
  \end{NB2}%
  =
  f_2^* \fN(C_1, \cE') + \fN(C_1,C_0)\otimes \mathcal S,
\end{gather*}
where $\cE'$ is the universal sheaf on $\bM^1(c'_j - (r-j)e_0)
= \bM^1(c + j\pt)$. Therefore we have
\begin{equation*}
   \int_{\bM^0(c-je_0)} \Phi(\cEf)
   \Res_{\hbar_{j}=0} \cdots \Res_{\hbar_{1}=0}
   \Psi^j(\cEf)
   = \int_{\bM^1(c + j\pt)} \prod_{f=1}^{N_f} e(\Vcal_0(\cEf')\otimes e^{m_f})
   \cup\heartsuit
\end{equation*}
for a cohomology class $\heartsuit$. Its degree is at most
\begin{equation*}
  \dim \bM^1(c + j\pt) - N_f v_0(c + j\pt)
  = (2 r - N_f) \left( - (\ch_2(c),[\bp]) - j\right).
\end{equation*}
From the assumption {\bf (a')}, this is smaller than $- (2 r - N_f)
(\ch_2(c),[\bp]) - d$, hence has no contribution to the wall-crossing
formula.
\end{NB}

As a result we can write
\begin{equation}\label{eq:error}
   \int_{\bM^m(c)} \Phi_d(\cE) 
   = \sum_{\mathfrak J}\int_{\bM^0(c-c^{\mathfrak{J}})}
    \prod_{f=1}^{N_f} 
    e(\Vcal_0(\cEf)\otimes e^{m_f})
    \Omega^{\mathfrak J}_d(\cEf)
\end{equation}
\begin{NB2}
  $\bM(c)\to \bM^m(c)$.
\end{NB2}%
for various $c^{\mathfrak J}$ with $c_1(c-c^\mathfrak J) = 0$ and
cohomology classes $\Omega^{\mathfrak J}_d(\cEf)$.
\begin{NB}
Let us compute the cohomological degrees of the both sides of the
equality in \thmref{thm:m=0}, where we say
\(
   \int_{\bM^{m}(c)} \heartsuit
\)
{\it has degree $k$\/} if it is contained in
\(
   H_{2k}^{\hT}(M_0(p_*(c))).
\)
\end{NB}%
The left hand side has degree as in \eqref{eq:deg}.
\begin{NB}
We have
\begin{equation*}
  \deg \int_{\bM(c)} \Phi_d(\cEf)
  = \dim \bM(c) - N_f v_0(c) - d
  = - (2r - N_f) (\ch_2(c),[\bp]) - d.
\end{equation*}
\end{NB}%
On the other hand, the degree of the right hand side is at most
\begin{equation*}
   \dim M(p_*(c-c^{\mathfrak J})) 
   - N_f \rank \Vcal_0(\cEf)
   = (2r - N_f) (\Delta(c-c^{\mathfrak J}),[\bp]).
\end{equation*}
If $c^{\mathfrak J}$ is nonzero, then it is at most $(2r -
N_f)\{(\Delta(c),[\bp]) - 1\}$, since $(\Delta(c-c^{\mathfrak
  J}),[\bp])$ is an integer and we have \eqref{eq:dim_est}.
\begin{NB2}
  Editted according to the referee's remark 2010/04/01
\end{NB2}%
Therefore there is no contribution to the wall-crossing formula if $d
< 2r-N_f$.
For $c^{\mathfrak J} = 0$, we get $\int_{\bM^0(c)} \Phi_d(\cE)$, but
it is equal to $\delta_{d0}\int_{M(p_*(c))} \prod_{f=1}^{N_f}
e(\Vcal(\cE)\otimes e^{m_f})$ as $\bM^0(c)\to M(p_*(c))$ is an
isomorphism and $\mu(C) = 0$ on $\bM^0(c)$.
\begin{NB2}
  Expanded according to Kota's comment. 2010/04/01
\end{NB2}%
\end{proof}

For a slightly modified version
\begin{equation*}
  \Phi'(\cE) =
  \prod_{f=1}^{N_f} e(\Vcal_1(\cE)\otimes e^{m_f})
  \times
  \exp \left( 
    t\mu(C)
  \right),
\end{equation*}
the second part of the argument works, we get
\begin{Theorem}\label{thm:matter_gap2}
Suppose $(c_1,[C]) = 0$. Then
\begin{equation*}
    \int_{\bM^m(c)} \Phi'(\cE) 
    =   \int_{M(p_*(c))} \prod_{f=1}^{N_f} e(\Vcal(\cE)\otimes e^{m_f})
     + O(t^{2r-N_f}).
\end{equation*}
Moreover the coefficient of $t^{2r-N_f}$ is
\[
   -\binom{2r-N_f-2}{r-1}\int_{M(p_*(c)+\pt)}\prod_{f=1}^{N_f}
   e(\Vcal(\cE)\otimes e^{m_f}),
\]
if $N_f < 2r$.
\end{Theorem}

For the last assertion, it is enough to calculate the case
$(\Delta(c),[\bp]) = 1$ by the same argument. 
(See the proof of \thmref{thm:structure} for more detail.)
\begin{NB2}
  Added according to the referee's suggestion. 2010/04/01
\end{NB2}%
Hence \subsecref{subsec:Example} gives us the answer.

\begin{NB}
It seems that the bound $d \le r - 1$ cannot be proved.
Feb. 26
\end{NB}

\begin{NB}
We have
\begin{equation*}
   \deg \int_{\bM^m(c)} \Phi'_d(\cE)
   = - (2r - N_f) (\ch_2(c),[\bp]) - d.
\end{equation*}
We also have
\begin{equation*}
  \begin{split}
   & \deg \int_{\bM^m(c-je_m)} \Phi'_d(\cEf) 
   \Res_{\hbar_{j}=0} \cdots \Res_{\hbar_{1}=0}
   \Psi^{\prime j}(\cEf)
   =
    \deg \int_{\bM^m(c-je_m)} \prod_{f=1}^{N_f} e(\Vcal_1(\cEf')\otimes 
    e^{m_f})
    \cup\heartsuit
\\
   \ge \; &
      \dim \bM^m(c-je_m) - N_f v_1(c-je_m)
   = - (2r - N_f) (\ch_2(c),[\bp]) - j(m(2r-N_f) + r + j - N_f).
  \end{split}
\end{equation*}
Therefore we have no contribution if $d\le r-N_f$ as
\(
   j (m (2r - N_f) + r + j - N_f) \ge r.
\)
But this is a stronger assumption.

Next consider the case
\begin{description}
\item[(a')] $d \le 2r - N_f - 1$.
\end{description}
If $m\ge 1$, we still have
\begin{equation*}
  j (m (2r - N_f) + r + j - N_f) \ge 3r - 2N_f + 1 \ge d + 1.
\end{equation*}
\begin{NB2}
  Unfortunately this estimate is not enough, and the argument is incomplete.
  
  Feb. 24.
\end{NB2}

If $m = 0$ and $j\ge r$, we still have
\begin{equation*}
  j ( r + j - N_f) \ge r(2r - N_f) \ge 2r - N_f \ge d + 1,
\end{equation*}
and hence it is enough to consider the case $m=0$, $j\le r-1$. In this
case, we have
\begin{equation*}
   \Vcal_1(\cEf) \cong \Vcal_0(\cE'),
\end{equation*}
and the rest of the argument is the same.
\end{NB}

\begin{NB}
We use the complex \eqref{eq:Ext(C_m,X)} for $m=0$, $X = C_1$:
  \begin{equation*}
     \Hom(C_0,C_1) = 0,
     \Ext^1(C_0,C_1) = \Coker\left[ V_1 \xrightarrow
     { \left(
       \begin{smallmatrix}
         B_1 \\ B_2
       \end{smallmatrix}
       \right) }
     \C^2\otimes V_0\right],
    \Ext^2(C_0,C_1) = V_0,
  \end{equation*}
where $V_0 = H^1(\shfO_C(-2))$, $V_1$, $B_1$, $B_2$ are the ADHM data for $X = C_1$:

\end{NB}

Next consider the $c_1\neq 0$ case.
\begin{Theorem}
  Suppose $0 < n\defeq (c_1,[C]) < r$. Then
  \begin{equation*}
    \int_{\bM^m(c)} \Phi'(\cE) 
    = 
    O(t^{n(r-n)}).
  \end{equation*}
\end{Theorem}

In fact, we have
\begin{equation*}
   \deg \int_{\bM^m(c)} \Phi_d'(\cE)
  = (2r - N_f) \dim V_1(c) + n(r - n) - d,
\end{equation*}
and all terms in the right hand side of \eqref{eq:error} has degrees
at most $(2r - N_f) \dim V_1(c)$ as $\dim V_1(c-c^{\mathfrak J}) \le
\dim V_1(c)$.
\begin{NB2}
  Editted according to the referee's remark 2010/04/01
\end{NB2}%
So the same argument works.

Let us state what we observed in the above proof as a general
structure theorem. Let $\Phi(\cE)$ be a multiplicative class in the
universal family $\cE$. Then
\begin{Theorem}\label{thm:structure}
  Let us fix $c_1$ with $0\le - (c_1,[C]) < r$.
  There exists a class $\Omega_j(\cE,t)$, which is a polynomial in
  $c_i(\cE)/[0]$ \rom($i=2,\dots,r$\rom) with coefficients in
  $H^*_{\C^*\times\C^*}(\pt)[[t]] = \C[\ve_1,\ve_2][[t]]$, and independent of
  $\Delta(c)$ such that
\begin{equation}\label{eq:Omega}
  \int_{\bM(c)} \Phi(\cE)\exp(t\mu(C))
  = \sum_{j\ge 0} \int_{M(p_*(c)+j\pt)} \Phi(\cE) \Omega_j(\cE,t).
\end{equation}
Moreover $\Omega_j(\cE,t)$ is unique if $\Phi(\cE)\neq 0$ for
$H^*_{\hT}(M(r,0,0)) = H^*_{\hT}(\pt) = S(\hT)$.
\end{Theorem}

For the theory with matters, the coefficients of $\Omega_j(\cE,t)$ are
in $H^*_{\C^*\times\C^*\times(\C^*)^{N_f}}(\pt)[[t]] =
\C[\ve_1,\ve_2,m_1,\dots,m_{N_f}][[t]]$.

\begin{proof}
As in the derivation of \eqref{eq:error}, we obtain a formula as
above, where $M(p_*(c)+j\pt)$ is replaced by $\bM^0(p^*(p_*(c) +
j\pt))$ and $\Omega_j(\cE,t)$ is a polynomial in Chern classes of
$q_{2*}([\cE]\otimes[\shfO_C(m)])$, $q_{2*}([\cE]^\vee
\otimes[\shfO_C(m)])$ of various $m$.
This $\Omega_j(\cE,t)$ is independent of $\Delta(c)$, as
\begin{itemize}
\item $\Psi^j(\cEf)$ in \thmref{thm:m=0} depends only on $j$,
\item the choice, whether we perform twist by $\shfO(C)$ or not, is
  determined by $c_1$, and
\item the Grassmannian in \propref{prop:Grassmann} is determined by
$r$ and $c_1$.
\end{itemize}
Thus it only remains to show that we can further replace $\Omega_j$ so
that it is a polynomial in $c_i(\cE)/[0]$ ($i=2,\dots, r$).

By the Grothendieck-Riemann-Roch theorem, these classes can be
expressed by $c_i(\cE)/[0] = c_i(q_{2*}([\cE] \otimes[\shfO_C]))$.
Note that
\(
  q_{2*}([\cE] \otimes[\shfO_C])
  = -\sum_{a=0}^2 (-1)^a \Ext^a(\shfO_C(-1),\cE).
\)
If we cross back the wall from $\bM^0$ to $\bM^1$,
\begin{NB2}
  Editted according to the referee's suggestion. 2010/04/01
\end{NB2}%
we have $\Ext^0_{q_2}(\shfO_C(-1),\cE) = 0 =
\Ext^2_{q_2}(\shfO_C(-1),\cE)$ on $\bM^1$ by the remark after
\lemref{lem:Ext+} below. Therefore $\Ext^1_{q_2}(\shfO_C(-1),\cE)$ is a
vector bundle of rank $r$, and $c_i(\cE)/[0]$ vanishes for $i > r$.
Since the difference of the integrals over $\bM^1$ and $\bM^0$ are
expressed by integrals over small moduli spaces, we can eventually
express $\Omega_j$ as a polynomial in $c_i(\cE)/[0]$ ($i=2,\dots, r$)
by a recursion.

Let us show the uniqueness of $\Omega_j(\cE,t)$ by a recursion on $j$.
Let us take the smallest possible $\Delta(c)$ with
$\bM(c)\neq\emptyset$, i.e., the case when $c+(c_1,[C])e_0 = (r,0,0)$
(cf.\ \propref{prop:Grassmann}). Then we only have the term with $j=0$
in the right hand side of \eqref{eq:Omega}. In this case, $p_*(c) =
(r,0,0)$ and the moduli space $M(r,0,0)$ is a single point. Therefore
$\Omega_0(\cE,t)$ is determined by \eqref{eq:Omega}.

Now suppose that $\Omega_j(\cE,t)$ with $j<n$ are determined. Then we
take $c$ whose $\Delta(c)$ is $n$ larger than the previous smallest
$\Delta(c)$. Then $j$ in the summation in \eqref{eq:Omega} runs from
$0$ to $n$. Moreover $M(p_*(c)+n\pt) = M(r,0,0)$ is a single point
again. Therefore $\Omega_n(\cE,t)$ is determined from \eqref{eq:Omega}
and $\Omega_j(\cE,t)$ with $j<n$.
\end{proof}

Suppose that the degree of $\int_{\bM(c)} \Phi(\cE)$ is of a form
$\gamma\Delta(c) + a(c_1,[C]) + b$ for some constants $\gamma$, $a$
and $b$, depending only on $r$. We further assume that $\gamma > 0$,
as in the case $\gamma = 2r - N_f > 0$ for the theory with matters.
Then
\begin{equation*}
  \begin{split}
  & \deg \int_{\bM(c)}\Phi(\cE)\mu(C)^d = \gamma\Delta(c) + a(c_1,[C]) + b - d,
\\
  & \deg \int_{M(p_*(c)+j\pt)}\Phi(\cE)\Omega_j(\cE,t) 
  \le \gamma(\Delta(p_*(c)) - j) + b,
  \end{split}
\end{equation*}
Since we are fixing $(c_1,[C])$ as in \thmref{thm:structure}, the two
degree cannot match if $j$ is too large. This means that we only need
to calculate finitely many $\Omega_j(\cE,t)$ to determine
$\int_{\bM(c)}\Phi(\cE) \mu(C)^d$ for a fixed $d$. (In practice, the
maximal $j$ can be calculated explicitly.)

\begin{NB}
\subsubsection{$c_1\neq 0$ case}
We have
\begin{multline*}
  \deg \int_{\bM^{m+1}(c)} \Phi_d(\cE)
  = \deg \int_{\bM^m(c)} \Phi_d(\cE)
\\
  = \dim \bM^m(c) - N_f \dim V_0(c) - d
\\
  \begin{NB2}
    = r(v_0(c) + v_1(c)) - (v_0(c) - v_1(c))^2 - N_f v_0(c) - d
    = (2r - N_f)v_1(c) + (r - N_f)(v_0(c) - v_1(c)) - (v_0(c) - v_1(c))^2 - d
  \end{NB2}
  \\
  = - (2r - N_f) \dim V_1 + (r - N_f - (v_0(c) - v_1(c)))(v_0(c) -
  v_1(v)) - d
\end{multline*}

\begin{multline*}
  \deg \int_{\bM^{m+1}(c)} \Phi_d'(\cE)
  = \deg \int_{\bM^m(c)} \Phi_d'(\cE)
\\
  = \dim \bM^m(c) - N_f \dim V_1(c) - d
\\
  \begin{NB2}
    = r(v_0(c) + v_1(c)) - (v_0(c) - v_1(c))^2 - N_f v_1(c) - d
    = (2r - N_f)v_1(c) + r(v_0(c) - v_1(c)) - (v_0(c) - v_1(c))^2 - d
  \end{NB2}
  \\
  = - (2r - N_f) \dim V_1 + (r - (c_1,[C]))(c_1,[C]) - d.
\end{multline*}
\end{NB}

\subsection{$K$-theory version}\label{subsec:K}
We derive the wall-crossing formula for a $K$-theoretic integration
via the Grothendieck-Riemann-Roch formula.

Let $\td(\alpha)$ be the Todd class of a $K$-theory class $\alpha$ on
various moduli spaces. The Todd class of the tangent bundle $T_M$ of a
variety $M$ is denoted by $\td M$. We have
\begin{equation*}
   \td \bM^m(c) = \td (\fN(\cE,\cE)).
\end{equation*}
For integers $d$, $l$ and $a=0$ or $1$, we consider
\begin{equation}\label{eq:Kth}
   \Phi(\cE) = \td (\fN(\cE,\cE)) \exp(l c_1 (\Vcal_a(\cE)))
   \exp(-d \ch_2(\cE)/[C]).
\end{equation}
\begin{NB}
  There is a correction term if $(c_1,[C])\neq 0$. But it is
  immaterial, as it is `constant' over $\bM^m(c)$.
\end{NB}%
By a discussion in \cite[several paragraphs preceding Def.~2.1]{NY3},
$-\ch_2(\cE)/[C]$ is the first Chern class of an equivariant line
bundle up to a (rational) cohomology class in $H^2_{\hT}(\pt)$, which
is zero if $(c_1,[C]) = 0$. Since this difference is immaterial in the
following discussion (in particular in \thmref{thm:Kvanish}), we
identify $-\ch_2(\cE)/[C]$ with the equivariant line bundle, which we
denote by $\mu(C)$. (This is the same as $\mu(C)$ in the previous
subsection up to a class in $H^2_{\hT}(\pt)$.)
Then the equivariant Riemann-Roch theorem (see e.g., \cite{Jos}) we have
\begin{equation*}
  \int_{\bM^m(c)} \Phi(\cE) 
  = \tau\left( \widehat\pi_*(\mu(C)^{\otimes d}
    \otimes \det\Vcal_a(\cE)^{\otimes l})\right),
\end{equation*}
where $\tau$ is the equivariant Todd homomorphism $\tau\colon
K^{\hT}(M_0(p_*(c)))\to H^{\hT}_*(M_0(p_*(c)))$ for the Uhlenbeck
partial compactification $M_0(p_*(c))$, and $\widehat\pi_*$ is the
push-forward homomorphism in the equivariant $K$-theory.

We have
\begin{equation*}
   \td (\fN(\cE,\cE)) 
   = \td(\fN(\cE_1,\cE_1))
   \td(\fN(\cE_1,\cE_2))\td(\fN(\cE_2,\cE_1))
   \td(\fN(\cE_2,\cE_2))
\end{equation*}
if $\cE = \cE_1\oplus \cE_2$. Then from \thmref{thm:m=0} we get
\begin{equation*}
  \int_{\bM^{m+1}(c)} \Phi(\cE)
  - \int_{\bM^{m}(c)} \Phi(\cE)
  = 
  \sum_{j=1}^\infty
   \int_{\bM^{m}(c - je_m)} \Phi(\cEf)\cup
   \Res_{\hbar_{j}=0} \cdots \Res_{\hbar_{1}=0}\Psi^{j}(\cEf),
\end{equation*}
where
\begin{gather*}
  \Psi^{j}(\bullet) \defeq
    \frac1{j!}
    \left[
      \prod_{i=1}^j
      \frac{
      \exp\left(l c_1(\Vcal_a(C_m \otimes e^{- \hbar_i}))
      - d \ch_2(C_m\otimes e^{-\hbar_i})/[C]\right)
    }{
      e^K(\fN(\bullet,C_m)
      \otimes e^{-\hbar_i})
      \, e^K(\fN(C_m,\bullet)
      \otimes e^{\hbar_i})
    }
    \prod_{1\le i_1\neq i_2 \le j} 
    ( 1 - e^{\hbar_{i_1}-\hbar_{i_2}})
   \right].
\end{gather*}

Here $e^K$ is the (Chern character of) $K$-theoretic Euler class:
\begin{equation*}
   e^K(\alpha) = e(\alpha) \td(\alpha)^{-1}
   = \sum_{p=0}^\infty (-1)^p \ch (\Wedge^p \alpha^\vee), 
\end{equation*}
where $e(\alpha)$ is the usual Euler class as before.
\begin{NB2}
  Editted according to the referee's suggestion. 2010/04/01
\end{NB2}%

\begin{NB}
  \begin{equation*}
    \td(\fN(C_m,C_m)) = \td( - \shfO) = 1
  \end{equation*}
\end{NB}

Strictly speaking, we need to consider the completion
$\C[\hbar_i^{-1},\hbar_i]]$ for the coefficient rings of the localized
equivariant homology groups of moduli spaces, as for example,
$e^{-\hbar_i}$ is not allowed. Here $\C[\hbar_i^{-1},\hbar_i]]$ is the
algebra of formal power series $\sum a_j \hbar^j$ such that $\{ j <
0\mid a_j\neq 0\}$ is finite.
The modification appears only at \subsecref{subsec:Euler} and the
beginning of \subsecref{subsec:fixedptformula}, and the rest of the
proof remains unchanged.

Observe that $\hbar_i$ appears always as a function in
$x_i\defeq e^{-\hbar_i} - 1$ in the above formula. We change
the coefficient ring from $\C[\hbar_i^{-1},\hbar_i]]$ to
$\C[x_i^{-1},x_i]]$.

We have
\begin{equation*}
   \Res_{\hbar_i=0} f(e^{-\hbar_i}-1) = - \Res_{x_i=0} \frac{f(x_i)}{x_i+1}
\end{equation*}
as $d x_i = - e^{-\hbar_i} d\hbar_i = -(x_i + 1)d\hbar_i$. Therefore

\begin{Theorem}\label{thm:K}
\begin{equation*}
  \int_{\bM^{m+1}(c)} \Phi(\cE)
  - \int_{\bM^{m}(c)} \Phi(\cE)
  = 
  \sum_{j=1}^\infty
   \int_{\bM^{m}(c - je_m)} \Phi(\cEf)\cup
   \Res_{x_{j}=0} \cdots \Res_{x_{1}=0}\Psi^{j}(\cEf),
\end{equation*}
where
\begin{gather*}
  \Psi^{j}(\bullet) \defeq
    \frac1{j!}
      \prod_{i=1}^j
      \frac{
      \exp\left(l c_1(\Vcal_a(C_m \otimes (1 + x_i)))
      - d \ch_2(C_m\otimes (1+x_i))/[C]\right)
    }{
      e^K(\fN(\bullet,C_m)
      \otimes (1+x_i))
      \, e^K(\fN(C_m,\bullet)
      \otimes \frac1{1 + x_i})
      (-(1+x_i))
    }
\\
    \times \prod_{1\le i_1\neq i_2 \le j} 
    \frac{x_{i_1} - x_{i_2}}{1 + x_{i_1}}
   .
\end{gather*}
\end{Theorem}

In view of \propref{prop:Grassmann} the following is useful to replace
the $K$-theoretic integration on $\bM^1(c)$ by one on $\widehat
N(c,n)$.
\begin{Lemma}\label{lem:higher}
  Consider the diagram in \propref{prop:Grassmann}. We have
  \begin{equation*}
    {\mathbf R}f_{1*}(\shfO_{\widehat N(c,n)}) = \shfO_{\bM^1(c)}.
  \end{equation*}
\end{Lemma}

\begin{proof}
  Since $f_1$ is a proper birational morphism between smooth
  varieties, this is a well-known result (see e.g.,
  \cite[\S5.1]{KollarMori}).
\end{proof}

In the remaining of this section we study the vanishing theorem for
small $d$.

\begin{Theorem}\label{thm:Kvanish}
Assume $0\le l\le r$.

\textup{(1)}
Suppose $(c_1,[C]) = 0$. If $0\le al + d\le r$,
\begin{equation*}
  \int_{\bM(c)} \Phi(\cE)
  = \int_{M(p_*(c))} \td(\fN(\cE,\cE)) \exp(l c_1 (\Vcal(\cE))),
\end{equation*}
where $\Vcal(\cE)$ is the vector bundle defined as in
\thmref{thm:matter_gap}.

\textup{(2)} Suppose $a=1$ and $0 < (c_1,[C]) < r$.
If $0 < d \le \min(r+(c_1,[C]) - l,r-1)$,
\begin{equation*}
  \int_{\bM(c)} \Phi(\cE) = 0.
\end{equation*}
\end{Theorem}

This result was conjectured in \cite[(1.37),(1.43)]{GNY2}.
\begin{NB}
  In \cite{GNY2}, we consider the case $0 > (c_1,[C]) > -r$.
\end{NB}

\begin{proof}
We study factors of $\Psi^j(\cEf)$ more closely.
Note that $\cEf$ is the universal sheaf for $\bM^m(c - je_m)$, hence
\begin{equation*}
    \rank \fN(C_m,\cEf) = (m+1)r + (c_1,[C]) + j,
\quad
    \rank \fN(\cEf,C_m) = mr + (c_1,[C]) + j.
\end{equation*}
If $\{ \alpha_a\}$, $\{ \beta_b \}$ are Chern roots of
$\fN(C_m,\cEf)$, $\fN(\cEf,C_m)$ respectively,
we have
\begin{equation}\label{eq:eK}
  \begin{split}
  & \frac{1}{e^K(\fN(C_m,\cEf)\otimes \frac1{1 + x_i})}
  = \frac{\exp({\sum_a \alpha_a})}{(-x_i)^{(m+1)r + (c_1,[C]) + j}}
    \prod_a \frac1{1 + \frac{1 - e^{\alpha_a}}{x_i}},
\\
  & \frac{1}{e^K(\fN(\cEf,C_m)\otimes (1 + x_i))}
  =
  \left(1 + \frac{1}{x_i}\right)^{mr + (c_1,[C]) + j}
  \prod_b \frac1{1 + \frac{1 - e^{-\beta_b}}{x_i}}.
  \end{split}
\end{equation}
On the other hand, we have
\begin{equation*}
  \exp(l c_1(\Vcal_a(C_m \otimes (1 + x_i))))
  = \exp(l c_1(\Vcal_a(C_m))) (1 + x_i)^{l(m+a)}
\end{equation*}
as $\rank \Vcal_a(C_m) = m+a$. Also
\begin{equation*}
  \exp(-d \ch_2(C_m\otimes (1+x_i))/[C])
  =   \exp(-d \ch_2(C_m)/[C])(1+x_i)^d,
\end{equation*}
as $\ch_1(C_m)/[C] = -1$.

Let us expand $\Psi^j(\cEf)$ into formal Laurent power series in
$x_1$. Note that we have the remaining factor
\begin{equation*}
  \frac1{(1 + x_1)}\prod_{i_2\neq 1} \frac{x_1 - x_{i_2}}{1 + x_1}
  =   \frac{\prod_{i_2\neq 1} ({x_1 - x_{i_2}})}{(1 + x_1)^j}
\end{equation*}
from $\Psi^j(\bullet)$ in \thmref{thm:K}.
This term can be absorbed into the second equality of
\eqref{eq:eK} as
\begin{equation*}
  \frac{\prod_{i_2\neq 1} ({x_1 - x_{i_2}})}{(1 + x_1)^j}
  \left(1 + \frac{1}{x_1}\right)^{mr + (c_1,[C]) + j}
  = \frac1{x_1}
  \left(1 + \frac{1}{x_1}\right)^{mr + (c_1,[C])}
  \prod_{i_2\neq 1} ( 1 - \frac{x_{i_2}}{x_1}).
\end{equation*}
Note that $mr + (c_1,[C])\ge 0$, and hence $(1 + \frac1{x_1})^{mr +
  (c_1,[C])}$ is a polynomial in $x_1^{-1}$.
We also write
\begin{equation*}
  (1+x_1)^{l(m+a)+d} = x_1^{l(m+a)+d} \left(1 + \frac1{x_1}\right)^{l(m+a)+d}
\end{equation*}
as a Laurent polynomial in $x_1^{-1}$. Note that we have
$l(m+a)+d \ge la + d\ge 0$ by our assumption.
Therefore we have
\begin{equation*}
  \Psi^j(\cEf) 
  = \frac1{x_1^N} f(x_1^{-1})
\end{equation*}
for some formal power series $f(x_1^{-1})$ in $x_1^{-1}$ with
\begin{equation*}
  \begin{split}
  N &= (m+1)r + (c_1,[C]) + j + 1 - l(m+a) - d
\\
    &= m(r - l) + (r - d) + (c_1,[C]) - la + j + 1.
  \end{split}
\end{equation*}
Since $0\le r-l$ and $0\le m$,
\begin{NB2}
  Editted according to the referee's suggestion. 2010/04/01
\end{NB2}%
the first term $m(r-l)$ is nonnegative. We also have $j\ge
1$. Therefore we have $N\ge 2$ if $d + la \le r + (c_1,[C])$.
This shows that there are no wall-crossing term, i.e.,
\begin{equation*}
   \int_{\bM(c)} \Phi(\cE) = \int_{\bM^0(c)}\Phi(\cE).
\end{equation*}

If $(c_1,[C]) = 0$, we have an isomorphism $\Pi\colon \bM^0(c)\to
M(p_*(c))$ given by $\Pi(E,\Phi) = (p_*(E),\Phi)$, $\Pi^{-1}(F,\Phi) =
(p^*(F),\Phi)$ (see \subsecref{subsub:c_1=0} for a precise
statement). Therefore the tangent bundles $\fN(\cE,\cE)$ for $\bM(c)$
and $M(p_*(c))$ are isomorphic to each other. Vector bundles
$\Vcal_0(\cE)$ for $\bM(c)$ and $M(p_*(c))$ are isomorphic to each
other from the description of $\Pi$ in \subsecref{subsub:c_1=0}. We
also have $\ch_2(\cE)/[C] = 0$. These show (1).

To show (2), we consider $\bM^0(c)\cong \bM^1(ce^{[C]})$ given by
$(E,\Phi)\mapsto (E(C),\Phi)$. Then \propref{prop:Grassmann} is
applicable. We then have
\begin{equation*}
   \int_{\bM^0(c)} \Phi(\cE)
   = \int_{\bM^1(ce^{[C]})} \td(\fN(\cE,\cE))
   \exp(l c_1(\Vcal_0(\cE))) 
   \exp(-d \ch_2(\cE(-C))/[C])
\end{equation*}
where $\cE$ in the right hand side is the universal sheaf
for $\bM^1(ce^{[C]})$.
As we explained in the beginning of this subsection, we may replace
this integral by
\begin{equation*}
  \tau(\widehat\pi_*(\mu(C)^{\otimes d}\otimes
  \det\Vcal_0(\cE)^{\otimes l})).
\end{equation*}
The difference between $\ch_2(\cE(-C))/[C]$ and $\ch_2(\cE)/[C]$ is
immaterial, as it is a class pulled back from $M_0(p_*(c))$.
\begin{NB2}
  Editted according to the referee's comment. 2010/04/02
\end{NB2}%
By \lemref{lem:higher} and the projection formula, we can replace the
above by the push-forward from $N(ce^{[C]},n)$ with $n = r -
(c_1,[C])$, where $\mu(C)$, $\Vcal_0(\cE)$ are replaced by their
pull-backs by $f_1$.
From the exact sequence in \propref{prop:Grassmann}(3) we have
\begin{equation*}
  f_1^* \Vcal_0(\cE) = f_2^* \Vcal_0(\cE')
\end{equation*}
where $\cE'$ is the universal sheaf for $\bM^0(ce^{[C]} - ne_0)$.
We also have
\begin{equation*}
  f_1^*\mu(C) = \det \mathcal S
\end{equation*}
up to the pull-back of a line bundle from $\bM^0(ce^{[C]} - ne_0)$ by $f_2$.
Therefore it is enough to show that
\begin{equation*}
  f_{2*}(\det\mathcal S^{\otimes d}) = 0 \quad\text{for $0 < d < r$}.
\end{equation*}
Since $f_2$ is a Grassmann bundle of $n$-planes in a rank $r$ vector
bundle, this vanishing is a special case of Bott vanishing theorem
\cite{Bott} or a direct consequence of Kodaira vanishing theorem.
\begin{NB}
  Note first that $\det \mathcal S^\vee$ is $f_2$-ample. This shows
  $f_{2*}(\det \mathcal S^{\otimes d}) = 0$ for $d > 0$. The relative
  tangent bundle is $\mathcal S^{\vee} \otimes
  \Ext^1_{q_2}(\shfO_C(-1),\cE')/\mathcal S$. Therefore the relative
  canonical bundle is
  \begin{equation*}
    \det \mathcal S^{\otimes r-n} \otimes
    (\det \Ext^1_{q_2}(\shfO_C(-1),\cE')^\vee\otimes\det\mathcal S)^{\otimes n}
    = \det\mathcal S^{\otimes r} \otimes
    \det \Ext^1_{q_2}(\shfO_C(-1),\cE')^{\vee \otimes n}.
  \end{equation*}
  Now consider $d c_1(\mathcal S) = rc_1(\mathcal S) - (r-d)
  c_1(\mathcal S)$ with $r-d > 0$.
\end{NB}%
\end{proof}

\subsection{Casimir operators}\label{subsec:Casimir}

We generalize the vanishing result in the previous subsection to the
case when we integrate certain $K$-theoretic classes given by
universal sheaves on moduli spaces. 

Let $\psi^p$ be the $p^{\mathrm{th}}$ Adams operation (see e.g.,
\cite[I.\S6]{FL}). We will use it for $p < 0$ defined by
\(
   \psi^p(x) = \psi^{-p}(x^\vee)
\)
.
For indeterminates $\vec{\tau} = (\cdots, \tau_{-2},\tau_{-1},
\tau_1,\tau_2,\cdots)$ and $\vec{t} = (\cdots, t_{-2},t_{-1},
t_1,t_2,\cdots)$ we consider a generalization of \eqref{eq:Kth} with
$a=0$:
\begin{multline*}
  \Phi(\cE)
  = \td(\mathfrak N(\cE,\cE))\exp(lc_1(\Vcal_0(\cE))-d\ch_2(\cE)/[C])
\\
  \times
  \exp\left(\sum_{p\in\Z\setminus\{0\}}
    \tau_p \ch\left( \psi^p(\cE)/[\bC]\right)
    + t_p \ch \left(\psi^p(\cE)\otimes \shfO_C(-1)/[\bC]\right)
  \right),
\end{multline*}
where the $K$-theoretic slant product $\bullet/[\bC]$ is defined by
\(
  \bullet/[\bC] = q_{2*} (\bullet\otimes q_1^* p^*(K_{\C^2}^{1/2}))
\)
with $p\colon \bC\to\C^2$.
We have two remarks for the definition. First we define the
push-forward $q_{2*}$ by the localization formula, i.e., the sum of
the fixed point contributions, since $q_2$ is {\it not\/} proper. This
is a standard technique in instanton counting and its meaning was
explained in detail in \cite[\S4]{NY1}. Second $K_{\C^2}^{1/2}$ is the
trivial line bundle together twisted by the square root of the
character of $(\C^*)^2$, which we consider as a character of its
double cover.
Having the above two remarks in mind, we see that the above integral
is essentially defined by the equivariant $K$-theory push-forward as
before.

We expand $\int_{\bM^m(c)}\Phi(\cE)$ in $t_p$, $\tau_p$ and consider
coefficients
\begin{equation*}
  \int_{\bM^m(c)}\Phi(\cE)
  = 
  \sum_{\vec{n},\vec{m}} \prod_{p\neq 0} \tau_p^{n_p} t_p^{m_p}
  \int_{\bM^m(c)}\Phi_{\vec{n},\vec{m}}(\cE),
\end{equation*}
where $\vec{n} = (\cdots,n_{-1},n_1,\cdots)$, $\vec{m}
= (\cdots,m_{-1},m_1,\cdots)$.
If we set
\begin{equation*}
  \left(\pd{}{\vec{\tau}}\right)^{\vec{n}}
  \defeq \prod_{p\neq 0} \left(\pd{}{\tau_p}\right)^{n_p},
\quad
  \vec{n}\,! = \prod_{p\neq 0} n_p!,
\end{equation*}
we have
\begin{equation*}
  \begin{split}
  & \int_{\bM^m(c)}\Phi_{\vec{n},\vec{m}}(\cE)
  = \left.\frac1{\vec{n}\,!\; \vec{m}\,!}
  \left(\pd{}{\vec{\tau}}\right)^{\vec{n}}
  \left(\pd{}{\vec{t}}\right)^{\vec{m}}
  \int_{\bM^m(c)}\Phi(\cE)\right|_{\vec{\tau}=\vec{t}=0}
\\
  =\; &
  \frac1{\vec{n}\,!\; \vec{m}\,!} \int_{\bM^m(c)}
  \begin{aligned}[t]
  & \td(\mathfrak N(\cE,\cE))\exp(lc_1(\Vcal_0(\cE))-d\ch_2(\cE)/[C])
\\
  & \times 
  \ch\left(
    \bigotimes_p
    \left(\psi^p(\cE)/[\bC]\right)^{\otimes n_p}
    \otimes 
    \left(\psi^p(\cE)\otimes \shfO_C(-1)/[\bC]\right)^{\otimes m_p}
    \right).
  \end{aligned}
  \end{split}
\end{equation*}

\begin{Proposition}\label{prop:Casimirvanish}
Suppose $(c_1,[C]) = 0$.

  \textup{(1)} Assume the followings\textup:
  \begin{aenume}
    \item $0\le d + \sum_{p<0} p n_p + p m_p$, and
    \item $d + \sum_{p>0} p n_p + p m_p \le r$.
  \end{aenume}
  Then the wall-crossing term is zero, i.e.,
  \begin{equation*}
    \int_{\bM^m(c)}\Phi_{\vec{n},\vec{m}}(\cE)
    = \int_{\bM^0(c)}\Phi_{\vec{n},\vec{m}}(\cE).
  \end{equation*}

  \textup{(2)} We further assume $m_p\neq 0$ for some $p$. Then
  \begin{equation*}
    \int_{\bM^m(c)}\Phi_{\vec{n},\vec{m}}(\cE)
    = 0.
  \end{equation*}
\end{Proposition}

\begin{proof}
(1)  We note that
  \begin{equation*}
    \psi^p(\cEf\oplus C_m\otimes(1+x_i)) 
    = \psi^p(\cEf) + \psi^p(C_m) \otimes(1+x_i)^p,
  \end{equation*}
as the Adams operation is a homomorphism. Then the proof exactly goes
as before.
\begin{NB}
  The assumption $d + \sum p n_p + p m_p$ is needed in stead of
$a l + d \ge 0$.
\end{NB}

(2) Recall that $\bM^0(c)\cong M(p_*(c))$ under $(E,\Phi)\mapsto
(p_*(E),\Phi)$. Then $\psi^k(\cE)\otimes\shfO_C(-1)/[\bC]$ vanishes,
since $p_*(\shfO_C(-1)) = 0$.
\end{proof}

\begin{NB}
Here is Kota's note.

\section{K-theoretic cases}

We set $x:=e^t-1$.
Then ${\Bbb C}[[t]]={\Bbb C}[[x]]$.

We understand the denominator of the
Todd class of ${\cal O}(-\alpha) e^{-t}$ as
\begin{equation}
\begin{split}
\frac{1}{1-e^{t+\alpha}}= &
\frac{1}{1-(1+x)e^{\alpha}}\\
=&\frac{1}{1-e^{\alpha}-xe^{\alpha}}\\
= & -\frac{1}{xe^{\alpha}}
\frac{1}{1+\frac{1-e^{-\alpha}}{x}} \in {\Bbb C}((x))[[\alpha]].
\end{split}
\end{equation}
So
the denominator of the
Todd class of ${\cal O}(\alpha) e^{t}$ is regarded as

\begin{equation}
\begin{split}
\frac{1}{1-e^{-(t+\alpha)}}&=
\frac{-e^{t+\alpha}}{1-e^{t+\alpha}}\\
& =\frac{x+1}{x}
\frac{1}{1+\frac{1-e^{-\alpha}}{x}} \in {\Bbb C}((x))[[\alpha]]
\end{split}
\end{equation}

These are contained in
${\Bbb C}(x)[[\alpha]]$.

\begin{NB2}
\begin{equation}
\begin{split}
e^{-t}= & \frac{1}{x+1}=1-x+x^2-\cdots \in {\Bbb C}((x)) \\
 \not = & x^{-1}(1-x^{-1}+x^{-2}-\cdots) \in {\Bbb C}((1/x)).
\end{split}
\end{equation}
Since $\deg(1/(1+x))=\infty$,
we need to be careful to compute the residue.
More precisely, in order to generalize the argument
in the homological case, it is better to get rid of these terms.
\end{NB2}

\begin{equation}
dt=\frac{dt}{dx}dx=e^{-t}dx=\frac{1}{1+x}dx.
\end{equation}

\begin{equation}
\begin{split}
\prod_{i=1}^p \frac{1}{1-e^{t+\alpha_i}}= &
\frac{1}{(-x)^p e^{\sum_i \alpha_i}}
\prod_i \frac{1}{1+\frac{1-e^{-\alpha_i}}{x}}\\
\prod_{i=1}^q \frac{1}{1-e^{-t+\beta_i}}= &
\left(\frac{1+x}{x} \right)^q
\prod_i \frac{1}{1+\frac{1-e^{\beta_i}}{x}}
\end{split}
\end{equation}

\begin{equation}\label{eq:mu}
\begin{split}
e^{d\mu(C)}= & e^{d(t-\alpha)},\; \alpha \in H^2(M(c-C_0)).\\
\det H^1(E_c)^{\otimes l}= & \det H^1(E_{c-C_0})^{\otimes l}e^{lmt}.
\end{split}
\end{equation}

For ${\frak N}(E_{c-C_0}(-mC),C_0)e^t=(\sum_i e^{-\beta_i})e^t$,
\begin{equation}
\begin{split}
\frac{\td({\frak N}(E_{c-C_0}(-mC),C_0)e^t)}
{e({\frak N}(E_{c-C_0}(-mC),C_0)e^t)}
=& \prod_{i=1}^{rm+1+(c_1(c),C)} \frac{1}{1-e^{-t+\beta_i}}\\
=& \left(\frac{1+x}{x} \right)^{rm+1+(c_1(c),C)}
\prod_i \frac{1}{1+\frac{1-e^{\beta_i}}{x}}
\end{split}
\end{equation}

For
${\frak N}(C_0,E_{c-C_0}(-mC))e^{-t}=(\sum_i e^{-\alpha_i})e^{-t}$,
\begin{equation}
\begin{split}
\frac{\td({\frak N}(C_0,E_{c-C_0}(-mC))e^{-t})}
{e({\frak N}(C_0,E_{c-C_0}(-mC))e^{-t})}
=& \prod_{i=1}^{rm+r+1+(c_1(c),C)} \frac{1}{1-e^{t+\alpha_i}}\\
=& \frac{1}{(-x)^{rm+r+1+(c_1(c),C)} e^{\sum_i \alpha_i}}
\prod_i \frac{1}{1+\frac{1-e^{-\alpha_i}}{x}}.
\end{split}
\end{equation}

Assume that $c_1(c)=0$.
Then $(1+x)^{rm+1+(c_1(c),C)}dt=
(1+x)^{rm+(c_1(c),C)}dx \in {\Bbb C}[x]dx$.
For $0 \leq l \leq r$, we have
\begin{equation}
(rm+r+1+(c_1(c),C))+1-d-lm
\geq 2.
\end{equation}
Hence the wall-crossing terms are 0.

\subsubsection{The contribution of iterated residue}
\begin{NB2}
We don't need this for $d \leq r$.
\end{NB2}

We set $x_i:=e^{t_i}-1$.

For ${\frak N}(E_{c-nC_m}(-mC),C_0)e^{t_i}=(\sum_j e^{-\beta_j})e^{t_i}$,
\begin{equation}
\begin{split}
\frac{\td({\frak N}(E_{c-nC_m}(-mC),C_0)e^{t_i})}
{e({\frak N}(E_{c-nC_m}(-mC),C_0)e^{t_i})}
=& \prod_{j=1}^{rm+1+(c_1(c),C)} \frac{1}{1-e^{-t_i+\beta_j}}\\
=& \left(\frac{1+x_i}{x_i} \right)^{rm+n+(c_1(c),C)}
\prod_j \frac{1}{1+\frac{1-e^{\beta_j}}{x_i}}
\end{split}
\end{equation}

For
${\frak N}(C_0,E_{c-nC_m}(-mC))e^{-t_i}=
(\sum_j e^{-\alpha_j})e^{-t_i}$,
\begin{equation}
\begin{split}
\frac{\td({\frak N}(C_0,E_{c-nC_m}(-mC))e^{-t_i})}
{e({\frak N}(C_0,E_{c-nC_m}(-mC))e^{-t_i})}
=& \prod_{j=1}^{rm+r+n+(c_1(c),C)} \frac{1}{1-e^{t_i+\alpha_j}}\\
=& \frac{1}{(-x_i)^{rm+r+n+(c_1(c),C)} e^{\sum_j \alpha_j}}
\prod_j \frac{1}{1+\frac{1-e^{-\alpha_j}}{x_i}}.
\end{split}
\end{equation}

\begin{equation}
\prod_{i \ne j}(1-e^{t_i-t_j})=
-\prod_{i< j}\frac{(x_i-x_j)^2}{(1+x_i)(1+x_j)}.
\end{equation}

Assume that $c_1(c)=0$.
\begin{equation}
\begin{split}
& \prod_{i=1}^n (1+x_i)^d
\prod_{i=1}^n\frac{1}{(-x_i)^{r(m+1)+n}}
\prod_{i=1}^n\left( \frac{1+x_i}{x_i} \right)^{rm+n}
\prod_{i \ne j}(1-e^{t_i-t_j})\\
=& -
\prod_{i=1}^n (1+x_i)^d
\prod_{i=1}^n\frac{(1+x_i)^{rm+1}}{(-1)^{r(m+1)+n} x_i^{2rm+r+2n}}
\prod_{i< j}(x_i-x_j)^2.
\end{split}
\end{equation}

\begin{multline}
 \prod_{i=1}^n 
\frac{\td({\frak N}(E_{c-nC_m}(-mC),C_0)e^{t_i})}
{e({\frak N}(E_{c-nC_m}(-mC),C_0)e^{t_i})}
\frac{\td({\frak N}(C_0,E_{c-nC_m}(-mC))e^{-t_i})}
{e({\frak N}(C_0,E_{c-nC_m}(-mC))e^{-t_i})}\\
e^{d\mu(C)}\det H^1(E_c)^{\otimes l}
\prod_{i \ne j}(1-e^{t_i-t_j})
\prod_i dt_i\\
= -
\prod_{i=1}^n (1+x_i)^{d+ml}
\prod_{i=1}^n\frac{(1+x_i)^{rm+1}}{(-1)^{r(m+1)+n} x_i^{2rm+r+2n}}
\prod_{i< j}(x_i-x_j)^2 \\
F(A_1/x_1,...,A_n/x_n)\prod_{i=1}^n \frac{dx_i}{1+x_i}
\end{multline}
where $A_i$ are cohomology classes of positive degree
and $F(y_1,...,y_n)$ is a formal power series whose coefficients
are cohomology classes. 

We expand $\prod_{i< j}(x_i-x_j)^2$ as
\begin{equation}
\prod_{i< j}(x_i-x_j)^2=\sum_{q_1,...,q_n} a_{q_1,...,q_n}
\prod_i x_i^{q_i}.
\end{equation}
Since $\sum q_i=n(n-1)$,
$q:=\min\{q_1,...,q_n \} \leq n-1$.
Then
\begin{equation}
2rm+r+2n-d-q_i-lm-rm \leq
m(r-l)+r+n-1-d+2. 
\end{equation}
Hence if $d \leq r+n-1$, then
the wall crossing term is zero.

\subsection{}
\begin{NB2}
The following argument only works for $r=2$, since I don't know how to
treat the contribution of $n \geq 2$ at this moment.
\end{NB2}

Moreover we further assume that $l<r$.
Then
for $m \geq 1$ and $d \leq 2r-l$,
we have
\begin{equation}
rm+r+1+(c_1(c),C)+1-d-lm
\geq 2.
\end{equation}
Assume that $m=0$.
Then 
\begin{equation}
\begin{split}
\frac{\td({\frak N}(E_{c-C_0}(-mC),C_0)e^t)}
{e({\frak N}(E_{c-C_0}(-mC),C_0)e^t)}
=& \prod_{i=1}^{rm+1+(c_1(c),C)} \frac{1}{1-e^{-t+\beta_i}}\\
=& \left(\frac{1+x}{x} \right)^{rm+1+(c_1(c),C)}
\prod_i \frac{1}{1+\frac{1-e^{\beta_i}}{x}}
\end{split}
\end{equation}
where $\beta_1=\lambda$ and
$\beta_i \in H^2(M(c'))$ for $i \geq 2$.
For
${\frak N}(C_0,E_{c-C_0}(-mC))e^{-t}=(\sum_i e^{-\alpha_i})e^{-t}$,
\begin{equation}
\begin{split}
\frac{\td({\frak N}(C_0,E_{c-C_0}(-mC))e^{-t})}
{e({\frak N}(C_0,E_{c-C_0}(-mC))e^{-t})}
=& \prod_{i=1}^{rm+r+1+(c_1(c),C)} \frac{1}{1-e^{t+\alpha_i}}\\
=& \frac{1}{(-x)^{rm+r+1+(c_1(c),C)} e^{\sum_i \alpha_i}}
\prod_i \frac{1}{1+\frac{1-e^{-\alpha_i}}{x}}
\end{split}
\end{equation}
where $\alpha_1=-\lambda$ and 
$\alpha_i \in H^2(M(c'))$ for $i \geq 2$.
In \eqref{eq:mu}, we have
$\alpha=\lambda \mod H^2(M(c'))$.
Hence $-d\alpha-\sum_i \alpha_i+l\lambda=(-d+1+l)\lambda
\mod H^2(M(c'))$.
We consider the residue of 
\begin{equation}
e^{(-d+1+l)\lambda}\frac{(1+x)^d}{x^{r+2}}
\left(\frac{1-e^{\lambda}}{x}\right)^k
\end{equation}
at $x=0$.
If the residue is not zero, then
$d \geq r+k+1$.
If $r<d \leq \min\{2r-l,r+l \}$, then
$-r<-d+1+l \leq -d+1+l+k \leq l-r<0$.
Since ${\bf R} \pi_*({\cal O}(-i\lambda))=0$ for $0<i<r$,
the contribution of the wall-crossing term is 0.

The contribution of iterated residue will be treated in a similar way:

It is sufficient to treat the case where $m=0$.
We set $x_2:=e^{t_2}-1$.
We consider the residue of 
\begin{equation}
\begin{split}
& e^{(-d+1+l)(-t_2)}\frac{(1+x_1)^d}{x_1^{r+2}}
\left(\frac{1-e^{-t_2}}{x_1}\right)^k \frac{1}{x_2^{r+4}}\\
=& 
(1+x_2)^{d-k-1-l}\frac{(1+x_1)^d}{x_1^{r+2}}
\left(\frac{-x_2}{x_1}\right)^k \frac{1}{x_2^{r+4}}
\end{split}
\end{equation}
at $x_1=0$.
If the residue is not zero, then
$d \geq r+k+1$.
If $r<d \leq \min\{2r-l,r+l \}$, then
$r>d-1-l$ and $d-1-l-k \geq r-l>0$.
Since $dt_2=dx_2/(x_2+1)$,
we have 

\begin{equation}
\begin{split}
\mathrm{Res}_{x_2=0}
\mathrm{Res}_{x_1=0}
\left[(1+x_2)^{d-k-1-l}\frac{(1+x_1)^d}{x_1^{r+2}}
\left(\frac{x_2}{x_1}\right)^k \frac{dx_1 dx_2}{x_2^{r+4}}
\right]=0.
\end{split}
\end{equation}

\end{NB}

\section{Partition function and Seiberg-Witten curves}
\label{sec:SW}

In this section we explain an application of the vanishing theorems to
Nekrasov partition functions. Here a reader is supposed to be familiar
with \cite{NY1,NY2,NY3,GNY2}.

\subsection{Partition function}

Let us fix $l\in\Z$.
We define a partition function as the generating function of integrals
considered in \S\S\ref{subsec:K}, \ref{subsec:Casimir}:
\begin{multline}\label{eq:partition}
  \Zin_l(\ve_1,\ve_2,\vec{a};\Lambda,\vec{\tau})
  \defeq \sum_c (\Lambda^{2r} e^{-(r+l)(\ve_1+\ve_2)/2})^{(\Delta(c),[\proj^2])}
\\
  \times
  \int_{M(c)} \td M(c) \exp(l c_1(\Vcal(\cE)))
  \exp\left(\sum_{p\in\Z\setminus\{0\}}
    \tau_p \ch\left( \psi^p(\cE)/[\C^2]\right)
  \right),
\end{multline}
where the rank $r = r(c)$ is fixed. Here $\vec{a} = (a_1,\dots, a_r)$
($\sum a_\alpha = 0$) is the vector given by generators $a_i$ of
$H^*_{T}(\pt)$ and $\ve_1$, $\ve_2$ ones of $H^*_{(\C^*)^2}(\pt)$, and
the integrals, more coefficients of monomials in $\tau_p$'s, 
\begin{NB2}
  Added according to the referee's suggestion. 2010/04/01
\end{NB2}%
take values in the quotient field $\mathfrak S(\hT)$ of
$H^*_{\hT}(\pt)$ as explained in \subsecref{subsec:equivhom}. (More
precisely the quotient field of the representation ring $R(\hT) = \Z[
e^{\pm \ve_1}, e^{\pm \ve_2}, e^{\pm a_\alpha}]$ as in
\subsecref{subsec:K} since this is a $K$-theoretic partition
function.)
\begin{NB2}
  Editted according to the referee's suggestion. 2010/04/03
\end{NB2}
And the $K$-theoretic slant product $\bullet/[\C^2]$ is defined by
$q_{2*}(\bullet\otimes q_1^*(K_{\C^2}^{1/2}))$ as in
\subsecref{subsec:Casimir}.

We have
\begin{multline}\label{eq:deriv}
  \left(\pd{}{\vec{\tau}}\right)^{\vec{n}}
  \Zin_l(\ve_1,\ve_2,\vec{a};\Lambda,\vec{\tau}=0)
  = 
\sum_c (\Lambda^{2r} e^{-(r+l)(\ve_1+\ve_2)/2})^{(\Delta(c),[\proj^2])}
\\
  \times
  \int_{M(c)} 
  \ch\left(
   \bigotimes_p \left(\psi^{p}(\cE)/[\C^2]\right)^{\otimes n_p}\right)
  \td M(c) \exp(l c_1(\Vcal(\cE)))
  .
\end{multline}
Therefore $\Zin_l(\ve_1,\ve_2,\vec{a};\Lambda,\tau)$ gives us
integrals of any tensor products of various Adams operators applied to
the universal sheaves.

We consider fixed points of the $\hT$-action on $M(c)$ as in
\cite[\S2]{NY1}: they are parametrized by $r$-tuples of Young diagrams
$\vec{Y} = (Y_1,\dots, Y_r)$ with $|\vec{Y}| = \sum |Y_\alpha| =
(\Delta(c),[\proj^2])$ corresponding to direct sums of monomial ideals
in $\C[x,y]$.
The character of the fiber of $\Vcal(\cE)$ at the fixed point
$\vec{Y}$ is given by
\begin{equation*}
  \ch(\Vcal(\cE)|_{\vec{Y}})
  (\ve_1,\ve_2,\vec{a})
  = \sum_{\alpha=1}^{r} e^{a_\alpha} \sum_{s\in Y_\alpha}
  e^{-l'(s)\ve_1-a'(s)\ve_2},
\end{equation*}
where $a'(s)$, $l'(s)$ are as in \cite[\S2]{NY1}. We also have
\begin{equation*}
   \exp(lc_1(\Vcal(\cE)|_{\vec{Y}}))
   = \exp\left[
     l \sum_{\alpha=1}^{r}\sum_{s\in Y_\alpha}
      (a_\alpha -l'(s)\ve_1-a'(s)\ve_2)
     \right].
\end{equation*}
Therefore we have
{\allowdisplaybreaks
\begin{equation}
  \begin{split}
  & \Zin_l(\ve_1,\ve_2,\vec{a};\Lambda,\vec{\tau})
  = \sum_{\vec{Y}} 
  \frac{(\Lambda^{2r} e^{-r(\ve_1+\ve_2)/2})^{|\vec{Y}|}}
  {\displaystyle\prod_{\alpha,\beta} 
     n^{\vec{Y}}_{\alpha,\beta}(\ve_1,\ve_2,\vec{a})}
\\
  &\qquad\times
   \exp\left[
     l \sum_{\alpha=1}^{r}\sum_{s\in Y_\alpha}
      (a_\alpha -l'(s)\ve_1-a'(s)\ve_2 - \frac{\ve_1+\ve_2}2)
     \right]
\\
  &\qquad\times
   \exp\left[
     \sum_{p,\alpha} \tau_p \frac
     {e^{p a_\alpha} \left\{1 - (1-e^{-p\ve_1})(1-e^{-p\ve_2})
         \sum_{s\in Y_\alpha} e^{-pl'(s)\ve_1-pa'(s)\ve_2} \right\}}
     {(e^{\ve_1/2}-e^{-\ve_1/2})(e^{\ve_2/2}-e^{-\ve_2/2})}
     \right]
   ,
  \end{split}
\end{equation}
where} $n^{\vec{Y}}_{\alpha,\beta}(\ve_1,\ve_2,\vec{a})$ is the
alternating sum of characters of exterior powers of the cotangent
space of $M(r,n)$ at the fixed point $\vec{Y}$. Its explicit formula
was given in \cite[\S1.2]{NY3}, where it was denoted by
$n^{\vec{Y}}_{\alpha,\beta}(\ve_1,\ve_2,\vec{a};\bbeta)$, and we put
$\bbeta = 1$.

We have
\begin{equation}\label{eq:SerreDuality}
    \Zin_l(\ve_1,\ve_2,\vec{a};\Lambda,\vec{\tau})
    = \Zin_{-l}(-\ve_1,-\ve_2,-\vec{a};\Lambda,{}'\vec{\tau}),
\end{equation}
where ${}'\vec{\tau}$ is given by ${}'{\tau}_p = \tau_{-p}$. This
symmetry is a simple consequence of \cite[the displayed formula one
below (1.33)]{GNY2}, or the Serre duality.
\begin{NB2}
  The referee requested an explanation of the Serre duality. But we
  think that a detailed explanation is unnecessary. 2010/04/01
\end{NB2}

Let $d\in\Z_{\ge 0}$. We consider a similar partition function on the
blow-up:
\begin{multline*}
  \bZin_{l,k,d} (\ve_1,\ve_2,\vec{a};\Lambda,\vec{\tau},\vec{t}) 
  \defeq \sum_c
  (\Lambda^{2r} e^{-(r+l)(\ve_1+\ve_2)/2})^{(\Delta(c),[\proj^2])}
  \\
  \times 
  \begin{aligned}[t]
  & \int_{\bM(c)} \td \bM(c)
  \exp(lc_1(\Vcal_0(\cE))-d\ch_2(\cE)/[C])
\\
  & \qquad\times
  \exp\left(\sum_{p\in\Z\setminus\{0\}}
    \tau_p \ch\left( \psi^p(\cE)/[\bC]\right)
    + t_p \ch \left(\psi^p(\cE)\otimes \shfO_C(-1)/[\bC]\right)
  \right),
  \end{aligned}
\end{multline*}
where we also fix $(c_1(c),[C]) = -k$ in this case.

This is related to the partition function \eqref{eq:partition} by
{\allowdisplaybreaks
\begin{equation}\label{eq:blow-up2}
  \begin{split}
  & \bZin_{l,k,d}(\ve_1,\ve_2,\vec{a},\Lambda,\vec{\tau},\vec{t})
\\
=\; &
\sum_{\substack{\vec{k}=(k_\alpha)\in\Z^r \\ \sum k_\alpha = k}}
\frac{(e^{(\ve_1+\ve_2)(d-(r+l)/2)}\Lambda^{2r})^{(\vec{k}^2)/2}
e^{(d-l/2)(\vec{k},\vec{a})}}
{\prod_{\vec{\alpha} \in \Delta} 
l_{\vec{\alpha}}^{\vec{k}}(\ve_1,\ve_2,\vec{a})}
\\
& \quad\times
\exp \left[l\left(
\frac{1}{6}(\ve_1+\ve_2)\sum_\alpha k_\alpha^3+
\frac{1}{2}\sum_\alpha k_\alpha^2 a_\alpha \right) \right]
\\
& \quad\quad\times
\Zin_l(\ve_1,\ve_2-\ve_1,\vec{a}+\ve_1 \vec{k};
\Lambda e^{\frac{d-(r+l)/2}{2r}\ve_1},
e^{-\ve_1/2}(\vec{\tau}+(e^{\ve_1}-1)\vec{t}))
\\
& \quad\quad\quad\times
\Zin_l(\ve_1-\ve_2,\ve_2,\vec{a}+\ve_2 \vec{k};
\Lambda e^{\frac{d-(r+l)/2}{2r}\ve_2},
e^{-\ve_2/2}(\vec{\tau}+(e^{\ve_2}-1)\vec{t})),
  \end{split}
\end{equation}
where} $l_{\vec{\alpha}}^{\vec{k}}(\ve_1,\ve_2,\vec{a})$ is a function
given in \cite[(2.3)]{NY3}.

Let us briefly explain how this formula is proved.
It is a consequence of the Atiyah-Bott-Lefschetz fixed point formula
applied to the $\hT$-action on $\bM(c)$. The fixed points are
parametrized by $(\vec{k},\vec{Y}^1,\vec{Y}^2)$, where
$\vec{k}\in\Z^r$ corresponds to a line bundle $\shfO(k_\alpha C)$, and
$\vec{Y}^1$, $\vec{Y}^2$ are Young diagrams corresponding to monomial
ideals in the toric coordinates at the $\C^*\times\C^*$-fixed point
$p_1 = ([1:0:0],[1:0])$ and $p_2 = ([1:0:0],[0:1])$ in $\bp$. (See
\cite[\S3]{NY1} for more detail.) The structure of the above formula,
i.e., the sum over $\vec{k}$ of the product of two partition functions
comes from this description of the fixed point set.
The shift of variables in the partition functions come from
study of tangent bundles, universal sheaves at fixed points. All these
are done in \cite[\S3]{NY1}, \cite[\S1.7]{GNY2}, except the expression
$e^{-\ve_a/2}(\vec{\tau} + (e^{\ve_a}-1)\vec{t})$ ($a=1,2$) appears
for variables for the Adams operators.

If we replace the Adams operator $\psi^p$ by the degree $p$ part of
the Chern character, the expression was given in \cite[\S4]{NY2},
where we just need to change variables as $\vec{\tau}+\ve_a\vec{t}$.
In our situation, $\vec{t}$ is multiplied by $\ch(\shfO_C(-1))|_{p_a}
= e^{\ve_a} - 1$ instead of $\ve_a$.
The factor $e^{-\ve_a/2}$ appears as the `square root' of
$K_{\C^2}\otimes K_{\bC}^{-1}$ at the fixed point $p_a$, since the
$K$-theoretic slant product $\bullet/[\bC]$ was defined as
$q_{2*}(\bullet\otimes q_1^*p^*(K_{\C^2}^{1/2}))$, not as
`$q_{2*}(\bullet\otimes q_1^*(K_{\bC}^{1/2}))$'.
\begin{NB2}
  Corrected according to the referee's comment. 2010/04/01
\end{NB2}%
\begin{NB}
  We have
  \begin{equation*}
    \begin{split}
     & \tau_p \frac
     {e^{p (a_\alpha + \ve_1 k_\alpha)}
       \left\{1 - (1-e^{-p\ve_1})(1-e^{-p(\ve_2-\ve_1)})
         \sum_{s\in Y_\alpha} e^{-pl'(s)\ve_1-pa'(s)(\ve_2-\ve_1)} \right\}}
     {(1 - e^{-\ve_1})(1 -e^{-(\ve_2-\ve_1)})} e^{-(\ve_1+\ve_2)/2}
\\
  =\; &
   e^{-\ve_1/2} \tau_p \frac
     {e^{p (a_\alpha + \ve_1 k_\alpha)}
       \left\{1 - (1-e^{-p\ve_1})(1-e^{-p(\ve_2-\ve_1)})
         \sum_{s\in Y_\alpha} e^{-pl'(s)\ve_1-pa'(s)(\ve_2-\ve_1)} \right\}}
     {(e^{\ve_1/2} - e^{-\ve_1/2})(e^{(\ve_2-\ve_1)/2} -e^{-(\ve_2-\ve_1)/2})}
    \end{split}
  \end{equation*}
\end{NB}%
(We avoid $K_{\bC}^{1/2}$, which cannot be defined.)

We have
  \begin{equation*}
    \bZin_{l,k,d}(\ve_1,\ve_2,\vec{a};\Lambda,\vec{\tau}=0,\vec{t})
    =
    \bZin_{-l,k,r-d}(-\ve_1,-\ve_2,-\vec{a};\Lambda,
    \vec{\tau}=0,-{}'\vec{t}).
  \end{equation*}
  This is proved exactly as in \cite[the last displayed formula in
  \S1.7.1]{GNY2} and \eqref{eq:SerreDuality}, or the Serre duality. 
  \begin{NB}
  We also see from
  \begin{equation*}
    \begin{split}
    & \sum (-1)^i H^i(\bp,E\otimes \shfO_C(-1))
    = \sum (-1)^i H^i(\bp,E\otimes (\shfO(C) - \shfO))
\\
    = \; & \sum (-1)^i H^{2-i}(\bp,E^\vee \otimes
    (\shfO(-C) - \shfO)\otimes K_{\bp})^\vee
\\
    =\; & - \sum (-1)^i H^{2-i}(\bp,E^\vee \otimes
    (\shfO(C) - \shfO)\otimes \shfO(-C)\otimes K_{\bp})^\vee
\\
    =\; & - \sum (-1)^{2-i} H^{2-i}(\bp,E^\vee \otimes
    \shfO_C(-1))^\vee.
    \end{split}
  \end{equation*}
It seems that we do not have a clear symmetry for $\vec{\tau}$.
  \end{NB}

  We define the perturbation part, see \cite[\S4.2]{NY3} and
  \cite[\S1.7.2]{GNY2} for more details. We set
{\allowdisplaybreaks
\begin{equation*}\label{gammati}\begin{split}
\gamma_{\ve_1,\ve_2}(x;\Lambda)&:=
\begin{aligned}[t]
\frac{1}{2\ve_1\ve_2}& \left(
-\frac{1}{6}\left(x+\frac{1}{2}(\ve_1+\ve_2)\right)^3 
 +x^2\log\Lambda\right)
\\
&+\sum_{n \geq 1}\frac{1}{n}
\frac{e^{-nx}}{(e^{n\ve_1}-1)(e^{n\ve_2}-1)},
\end{aligned}
\\
\widetilde{\gamma}_{\ve_1,\ve_2}(x;\Lambda)
&:=\gamma_{\ve_1,\ve_2}(x;\Lambda)+
\frac{1}{\ve_1 \ve_2} \left(\frac{\pi^2 x}{6}-\zeta(3)\right)\\
&\qquad \qquad +\frac{\ve_1+\ve_2}{2\ve_1 \ve_2} 
\left( x \log \Lambda+\frac{\pi^2}{6} \right)+
\frac{\ve_1^2+\ve_2^2+3\ve_1 \ve_2}{12 \ve_1 \ve_2} \log\Lambda
\end{split}
\end{equation*}
for} $(x,\Lambda)$ in a neighborhood of
$\sqrt{-1}\R_{>0}\times\R_{>0}$. 

We define the full partition function by
\begin{equation*}
  \begin{split}
  & Z_l(\ve_1,\ve_2,\vec{a};\Lambda,\vec{\tau})
  \defeq \exp\left[ 
    -\sum_{\alpha\neq\beta}
      \widetilde\gamma_{\ve_1,\ve_2}(a_\alpha - a_\beta;\Lambda)
    - l \sum_{\alpha=1}^r \frac{a_\alpha^3}{6\ve_1\ve_2}
    \right]
    \Zin_l(\ve_1,\ve_2,\vec{a};\Lambda,\vec{\tau}),
\\    
  & 
  \begin{aligned}[t]
   & \bZ_{l,k,d}(\ve_1,\ve_2,\vec{a};\Lambda,\vec{\tau},\vec{t})
\\
  & \qquad \defeq \exp\left[ 
    -\sum_{\alpha\neq\beta}
      \widetilde\gamma_{\ve_1,\ve_2}(a_\alpha - a_\beta;\Lambda)
    - l \sum_{\alpha=1}^r \frac{a_\alpha^3}{6\ve_1\ve_2}
    \right]
    \bZin_{l,k,d}(\ve_1,\ve_2,\vec{a};\Lambda,\vec{\tau},\vec{t}).
  \end{aligned}
  \end{split}
\end{equation*}

Using the difference equation satisfied by the perturbation terms (see
\cite[\S4.2]{NY3} and \cite[\S1.7.2]{GNY2}), we can absorb the factors
in \eqref{eq:blow-up2} coming from line bundles $\shfO(k_\alpha C)$
into the partition function to get
\begin{equation}\label{eq:blow-up3}
\begin{split}
  &\bZ_{l,k,d}(\ve_1,\ve_2,\vec{a};\Lambda,\vec{\tau},\vec{t})
\\
  =\; & 
  \exp\left[\left\{
    -\frac{\left(4\left(d + l\left(-\nicefrac12+\nicefrac{k}r\right)\right)
        - r\right)
      (r-1)}{48}
   + \frac{k^3 l}{6r^2}
  \right\}
    (\ve_1+\ve_2)\right]
\\
&
   \times
   \!\!\sum_{\{\vec{l}\} = -\frac{k}{r}}\!\!
   \begin{aligned}[t]
     & Z_l(\ve_1,\ve_2-\ve_1,\vec{a}+\ve_1\vec{l}; e^{\frac{\ve_1}{2r}
       \left\{d+l\left(-\frac12+\frac{k}r\right)-\frac{r}2\right\}}
     \Lambda, e^{\frac{k\ve_1}r}\star e^{-\frac{\ve_1}2}
     (\vec{\tau}+(e^{\ve_1}-1)\vec{t}))
     \\
     &\times Z_l(\ve_1-\ve_2,\ve_2,\vec{a}+\ve_2\vec{l};
     e^{\frac{\ve_2}{2r}
       \left\{d+l\left(-\frac12+\frac{k}r\right)-\frac{r}2\right\}}
     \Lambda,e^{\frac{k\ve_2}r}\star e^{-\frac{\ve_2}2}
     (\vec{\tau}+(e^{\ve_2}-1)\vec{t})),
   \end{aligned}
\end{split}
\end{equation}
where $\vec{l}$ runs over the set $\{ \vec{l} =
(l_\alpha)_{\alpha=1}^r\in\Q^r \mid \sum l_\alpha = 0, l_\alpha \equiv
-k/r\bmod\Z\}$, and $e^{k\ve_a/r}\star\vec{\tau}$ is defined as
\begin{equation*}
    e^{k\ve_a/r}\star\vec{\tau}
    = (\cdots, e^{-2k\ve_a/r}\tau_{-2},e^{-k\ve_a/r}\tau_{-1},
    e^{k\ve_a/r}\tau_1,e^{2k\ve_a/r}\tau_2,\cdots).
\end{equation*}
We shift from the previous $\vec{k}$ to $\vec{l}$ by $l_\alpha =
k_\alpha - k/r$. The effect of this shift was calculated in
\cite[(1.36)]{GNY2} when $\vec{\tau} = \vec{t} = 0$, and we have used
it here.
And $e^{k\ve_a/r}\star$ comes from $\shfO(k_\alpha C) =
\shfO((l_\alpha+k/r)C)$. The part $l_\alpha$ is absorbed into the
shift $\vec{a}+\ve_a\vec{l}$, but we need the remaining contribution
from $k/r$.

\begin{Remark}\label{rem:full}
  We do not make precise to which ring the full partition functions
  belong, as a function in $\Lambda$. We just use them formally to
  make a formula shorter as above. This applies all formulas below
  until they (more precisely their leading coefficients) are
  identified with one defined via Seiberg-Witten curves, which are 
  really functions defined over an appropriate open set in $\Lambda$.
\end{Remark}
\begin{NB2}
An explanation added according to the referee comment. 2010/04/03
\end{NB2}%

\subsection{Regularity at $\ve_1 = \ve_2 = 0$}

We assume $0\le l \le r$ hereafter.
\begin{NB}
  I allow $l=r$. Oct.10, 2009, HN
\end{NB}

\thmref{thm:Kvanish}(1) means
\begin{equation}\label{eq:blow-up4}
  \bZ_{l,0,d}(\ve_1,\ve_2,\vec{a};\Lambda,\vec{\tau}=0,\vec{t}=0)
  = Z_l(\ve_1,\ve_2,\vec{a};\Lambda,\vec{\tau}=0)
\end{equation}
for $0\le d\le r$. This was conjectured in \cite[(1.37)]{GNY2}.
Combined with \eqref{eq:blow-up2}, we see that coefficients of
$\Zin(\ve_1,\ve_2,\vec{a};\Lambda,\vec{\tau}=0)$ in $\Lambda^n$ are
determined recursively if the above holds two different values of $d$,
as explained in \cite[\S5.2]{NY2}.
The equation \eqref{eq:blow-up4}, the left hand side replaced by
\eqref{eq:blow-up3}, is called the {\it blow-up equation}. It gives
a strong constraint on the partition function
$Z(\ve_1,\ve_2,\vec{a};\Lambda,\vec{\tau}=0)$.

As an application, in [loc.\ cit., Prop.~1.38] we proved the
followings under [loc.\ cit., (1.37)]:
\begin{align}
  & \Zin_l(\ve_1,-2\ve_1,\vec{a};\Lambda,\vec{\tau}=0) =
    \Zin_l(2\ve_1,-\ve_1,\vec{a};\Lambda,\vec{\tau}=0) 
    \quad\text{if $l\neq r$},
    \label{eq:sym}
\\
  & \text{$\ve_1\ve_2\log Z_l(\ve_1,\ve_2,\vec{a};\Lambda,\vec{\tau}=0)$
is regular at $\ve_1=\ve_2=0$.}
    \label{eq:regular}
\end{align}
(More precisely only the proof of \eqref{eq:sym} was given in [loc.\
cit.,(1.37)]. The proof of \eqref{eq:regular} was omitted since it is
the same as \cite[Th.~4.4]{NY3}.)
\begin{NB2}
  Added according to the referee's suggestion. 2010/04/01
\end{NB2}

\begin{Remark}
Though it was not stated explicitly in [loc.~cit.], the second
assertion holds even if $l=r$. On the other hand, the first one
follows from \eqref{eq:blow-up4} for $d$ and $r+l-d$, which must be
{\it different}. Therefore $l\neq r$ is required.
\end{Remark}

Let us apply the same argument to \propref{prop:Casimirvanish}. We
expand $\Zin_l$ as before:
\begin{equation}\label{eq:tauexpand}
    \Zin_l(\ve_1,\ve_2,\vec{a};\Lambda,\vec{\tau})
    = \sum_{\vec{n}} \prod_{p\neq 0} \tau_p^{n_p}
     \Zin_{\vec{n}}(\ve_1,\ve_2,\vec{a};\Lambda)
\end{equation}
with $\vec{n} = (\cdots,n_{-1},n_1,\cdots)$. Thus
\begin{equation*}
  \Zin_{\vec{n}}(\ve_1,\ve_2,\vec{a};\Lambda)
  = \frac1{\vec{n}\,!}
  \left(\pd{}{\vec{\tau}}\right)^{\vec{n}}
  \Zin_l(\ve_1,\ve_2,\vec{a};\Lambda,\vec{\tau}=0).
\end{equation*}

By \propref{prop:Casimirvanish} we have
\begin{equation}\label{eq:blow-up5}
  \left(\pd{}{\vec{t}}\right)^{\vec{n}} \bZ_{l,0,d}(\ve_1,\ve_2,\vec{a};
  \Lambda,\vec{\tau}=0,\vec{t}=0)
  = 0
\end{equation}
if $\vec{n}$ is nonzero and satisfies 
\begin{equation}
  \label{eq:const}
 -\sum_{p<0} p n_p \le d \le r - \sum_{p>0} pn_p.
\end{equation}
After substituting \eqref{eq:blow-up3} to the left hand side, we also
call this as the {\it blow-up equation}.

For simplicity we assume $\vec{n}$ is supported on either $\Z_{>0}$ or
$\Z_{<0}$, i.e., $n_p = 0$ for any $p < 0$ or $n_p = 0$ for any $p >
0$. We say $\vec{n}$ is {\it positive\/} in the first case, and {\it
  negative\/} in the second case.

We define `$\log Z(\ve_1,\ve_2,\vec{a};\Lambda,\vec{\tau})$' as
follows. Since the perturbative part is already the exponential of
something, we only need to define `$\log\Zin$'. Then observe that
\begin{equation*}
  \Zin(\ve_1,\ve_2,\vec{a};\Lambda,\vec{\tau})
  = \exp\left(
         \sum_{p,\alpha} \frac
     {\tau_p e^{p a_\alpha}}{(e^{\ve_1/2}-e^{-\ve_1/2})(e^{\ve_2/2}-e^{-\ve_2/2})}
    \right)
    \times (1 + O(\Lambda)).
\end{equation*}
Therefore $\log\Zin$ can be defined as the sum of
\(
         \sum_{p,\alpha} \nicefrac
     {\tau_p e^{p a_\alpha}}{(e^{\ve_1/2}-e^{-\ve_1/2})(e^{\ve_2/2}-e^{-\ve_2/2})}
\)
and a formal power series in $\Lambda$.
\begin{NB}
In \cite[the third displayed formula in \S5]{NY2}, 
$/[\C^2]$ must be $/[\proj^2]$.
\end{NB}

\begin{NB}
Earlier version:

Let us apply the same argument to \propref{prop:Casimirvanish}, which
implies
\begin{equation}\label{eq:blow-up5'}
  \pd{}{t_p} \bZ_{l,0,d}(\ve_1,\ve_2,\vec{a};
  \Lambda,\vec{\tau}=0,\vec{t}=0)
  = 0
\end{equation}
for $|p| \le d \le r$ if $p < 0$, and $0\le d \le r-p$ if $p > 0$.
\end{NB}

\begin{Proposition}\label{prop:regular}
\begin{NB}
  \textup{(1)} If $-(r+l)/2 \le p \le (r-l)/2$ and $l\neq r$, we have
  \begin{equation*}
    \pd{}{\tau_p}Z_l(\ve_1,-2\ve_1,\vec{a};\Lambda,\vec{\tau}=0)
  = \pd{}{\tau_p}Z_l(2\ve_1,-\ve_1,\vec{a};\Lambda,\vec{\tau}=0).
  \end{equation*}
  \begin{NB2}
    I slightly change Kota's condition to allow the equality 
    $p=(r-l)/2$. (This holds only if $r-l$ is even.)
  \end{NB2}
\end{NB}%
\textup{(1)} Suppose that $\vec{n}$ satisfies the followings\textup:
\begin{aenume}
\item If $\sum n_p$ is odd, $-(r+l)/2 \le \sum_{p < 0} p n_p$
  \textup(negative case\textup) or
  $\sum_{p > 0} p n_p \le (r-l)/2$ \textup(positive case\textup).
\item If $\sum n_p$ is even, the strict inequality holds.
\end{aenume}
Then
\begin{equation*}
  Z_{\vec{n}}(\ve_1,-2\ve_1,\vec{a};\Lambda)
  = Z_{\vec{n}}(2\ve_1,-\ve_1,\vec{a};\Lambda).
\end{equation*}

  \textup{(2)} If $-r < \sum_{p<0} p n_p$ \textup(negative case\textup) or
  $\sum_{p>0} pn_p < r$ \textup(positive case\textup),
  \begin{equation*}
  \ve_1\ve_2
　 \left(\pd{}{\vec{\tau}}\right)^{\vec{n}}
  \log Z_l(\ve_1,\ve_2,\vec{a};\Lambda,\vec{\tau}=0)
  \end{equation*}
  is regular at $\ve_1=\ve_2=0$.
\end{Proposition}

\begin{NB}
  Copy of the proof for
  $Z_l(\ve_1,\ve_2,\vec{a};\Lambda,\vec{\tau}=0)$.

\begin{proof}
By \eqref{eq:SerreDuality} we may assume $l\le 0$. By the assumption
we have \eqref{eq:blow-up4} for $d=0$ and $d=r+l$. Note that 
we have $0\neq r+l$ as $l\neq -r$.

We take the difference of both sides of \eqref{eq:blow-up5'} with
$d=r+m$, $0$ after setting $\ve_2 = -\ve_1$. We have
\begin{equation*}
\begin{split}
  & \left(Z_n(\ve_1,-2\ve_1,\vec{a}) - Z_n(2\ve_1,-\ve_1,\vec{a})\right)
  \left(e^{(r+m)n \ve_1/2} - e^{-(r+m)n \ve_1/2}\right)
\\
  =\; &
  -\sum_{\substack{(\vec{k},\vec{k})/2+l+l' = n\\
      l\neq n, l'\neq n}}
\begin{aligned}[t]
  & \frac{e^{r{(\vec{k},\vec{a})}/{2}}
   Z_{l'}(\ve_1,-2\ve_1,\vec{a}+\ve_1\vec{k})
   Z_l(2\ve_1,-\ve_1,\vec{a}-\ve_1\vec{k})}
  {\prod_{\vec\alpha\in\Delta} l^{\vec{k}}_{\vec{\alpha}}(\ve_1,-\ve_1,\vec{a})}  \\
  & \qquad
  \times\exp\left[ \frac{m\bbeta}2 \sum_\alpha k_\alpha^2
    a_\alpha\right] \\
  & \qquad
  \times\left(e^{(r+m){(\vec{k},\vec{a})}/{2}}
    e^{{(r+m)(l'-l)}\ve_1/{2}}-
    e^{{-(r+m)(\vec{k},\vec{a})}/{2}}e^{{-(r+m)(l'-l)}\ve_1/{2}}\right),
\end{aligned}
 \end{split}
\end{equation*}
where we expand $\Zin_m$ as
\begin{equation*}
    \Zin_m(\ve_1,\ve_2,\vec{a};\Lambda,\bbeta=1)
    = \sum_n Z_n(\ve_1,\ve_2,\vec{a}) \Lambda^{2rn}.
\end{equation*}

\begin{NB2}
{\bf Correction}: In the proof of \cite[Lemma~4.3]{NY3} we expand
$\Zin = \sum_n \left(
  \q\bbeta^{2r}e^{-r\bbeta(\ve_1+\ve_2)/2}\right)^n
  Z_n$.
But this was wrong and we expand it as above.
\end{NB2}

Let us show that
\(
Z_n(\ve_1,-2\ve_1,\vec{a})
   = Z_n(2\ve_1,-\ve_1,\vec{a})
\)
by using the induction on $n$. It holds for $n = 0$ as $Z_0 = 1$.
Suppose that it is true for $l,m< n$. Then the right hand side of the
above equation vanishes, as terms with $(\vec{k},l,l')$ and $(-\vec{k},
l', l)$ cancel thanks to \cite[Lemma~4.1(1) and (4.2)]{NY3}, and the
term $(0,l,l)$ is $0$. Therefore it is also true for $n$.
\end{proof}
\end{NB}

\begin{proof}
  Since the proof of (2) is the same as that of \eqref{eq:regular}, we
  only prove (1).

We expand $\Zin_l$ as
\begin{equation*}
    \Zin_l(\ve_1,\ve_2,\vec{a};\Lambda,\vec{\tau})
    = \sum_{N=0}^\infty Z_N(\ve_1,\ve_2,\vec{a},\vec{\tau}) \Lambda^{2rN}.
\end{equation*}
Note that
\begin{equation*}
  Z_0(\ve_1,\ve_2,\vec{a},\vec{\tau})
  = \exp\left[
     \sum_{p,\alpha} \tau_p \frac
     {e^{p a_\alpha}}
     {(e^{\ve_1/2}-e^{-\ve_1/2})(e^{\ve_2/2}-e^{-\ve_2/2})}
    \right].
\end{equation*}
Therefore the assertion is true for $Z_0$. We prove the assertion by
the induction on $N$.

\begin{NB}
Original proof (special case):

From \eqref{eq:blow-up5'} we have
{\allowdisplaybreaks
\begin{equation}\label{eq:blow-up6}
  \begin{split}
   0 &=
   \pd{}{\tau_p} Z_n(\ve_1,-2\ve_1,\vec{a},\vec{\tau}=0)
   e^{n(d-(r+l)/2)\ve_1}
\\
    & \quad
   - \pd{}{\tau_p} Z_n(2\ve_1,-\ve_1,\vec{a},\vec{\tau}=0)
   e^{-n(d-(r+l)/2)\ve_1}
\\
    &+
   \pd{}{\tau_p} Z_0(\ve_1,-2\ve_1,\vec{a},\vec{\tau}=0)
   Z_n(2\ve_1,-\ve_1,\vec{a},\vec{\tau}=0)
   e^{n(d-(r+l)/2)\ve_1}
\\
    &\quad-
   Z_n(\ve_1,-2\ve_1,\vec{a},\vec{\tau}=0)
   \pd{}{\tau_p} Z_0(2\ve_1,-\ve_1,\vec{a},\vec{\tau}=0)
   e^{-n(d-(r+l)/2)\ve_1}
\\
   & 
   + 
  \sum_{\substack{(\vec{k},\vec{k})/2+n_1+n_2 = n\\
      n_1, n_2 < n}}\!\!
  \frac{e^{(\vec{k},\vec{a})(d-l/{2})} e^{(n_1-n_2)(d-(r+l)/2)\ve_1}}
  {\prod_{\vec\alpha\in\Delta} l^{\vec{k}}_{\vec{\alpha}}(\ve_1,-\ve_1,\vec{a})}
    \exp\left[ \frac{l}2 \sum_\alpha k_\alpha^2
    a_\alpha\right] 
\\
  & \quad\quad\times 
  \Bigl(\pd{}{\tau_p} 
   Z_{n_1}(\ve_1,-2\ve_1,\vec{a}+\ve_1\vec{k},\vec{\tau}=0)
   Z_{n_2}(2\ve_1,-\ve_1,\vec{a}-\ve_1\vec{k},\vec{\tau}=0)
\\
   & \quad\quad\quad\quad
   - Z_{n_1}(\ve_1,-2\ve_1,\vec{a}+\ve_1\vec{k},\vec{\tau}=0)
   \pd{}{\tau_p}
   Z_{n_2}(2\ve_1,-\ve_1,\vec{a}-\ve_1\vec{k},\vec{\tau}=0)\Bigr).
  \end{split}
\end{equation}
}
\begin{NB2}
  Modulo $t_q$ ($q\neq p$) and $t_p^2$, we have
\begin{equation*}
  \begin{split}
   &Z_n(\ve_1,-\ve_1,\vec{a},\vec{\tau}=0)
\\
   = \; &
   Z_n(\ve_1,-2\ve_1,\vec{a},(e^{\ve_1/2} - e^{-\ve_1/2})\vec{t})
   Z_{0}(2\ve_1,-\ve_1,\vec{a}-\ve_1\vec{k},-(e^{\ve_1/2} - e^{-\ve_1/2})\vec{t})
   e^{n(d-(r+l)/2)\ve_1}
\\
   & \quad
   + 
   Z_0(\ve_1,-2\ve_1,\vec{a},(e^{\ve_1/2} - e^{-\ve_1/2})\vec{t})
   Z_n(2\ve_1,-\ve_1,\vec{a},-(e^{\ve_1/2} - e^{-\ve_1/2})\vec{t})
   e^{-n(d-(r+l)/2)\ve_1}
\\
   & 
   + 
  \sum_{\substack{(\vec{k},\vec{k})/2+n_1+n_2 = n\\
      n_1, n_2 < n}}\!\!
  \frac{e^{(\vec{k},\vec{a})(d-l/{2})} e^{(n_1-n_2)(d-(r+l)/2)\ve_1}}
  {\prod_{\vec\alpha\in\Delta} l^{\vec{k}}_{\vec{\alpha}}(\ve_1,-\ve_1,\vec{a})}
    \exp\left[ \frac{l}2 \sum_\alpha k_\alpha^2
    a_\alpha\right] 
\\
  & \quad\times
   Z_{n_1}(\ve_1,-2\ve_1,\vec{a}+\ve_1\vec{k},(e^{\ve_1/2} - e^{-\ve_1/2})\vec{t})
   Z_{n_2}(2\ve_1,-\ve_1,\vec{a}-\ve_1\vec{k},-(e^{\ve_1/2} - e^{-\ve_1/2})\vec{t}).
  \end{split}
\end{equation*}
\end{NB2}

We consider \eqref{eq:blow-up6} when $d$ is replace by $r+l-d$. Then
$d - (r+l)/2$ is replaced by $-\{ d - (r+l)/2\}$. We also replace
$\vec{k}$ by $-\vec{k}$ and $(n_1,n_2)$ by $(n_2,n_1)$. Note
\begin{equation*}
  \frac{e^{r(\vec{k},\vec{a})/2}}
  {\prod_{\vec{\alpha}\in\Delta} l^{\vec{k}}_{\vec{\alpha}}(\ve_1,-\ve_1,\vec{a})}
  = \frac{e^{-r(\vec{k},\vec{a})/2}}
  {\prod_{\vec{\alpha}\in\Delta} l^{-\vec{k}}_{\vec{\alpha}} (\ve_1,-\ve_1,\vec{a})}
\end{equation*}
by \cite[Lem.~4.1 and (4.2)]{NY3}. Therefore the replacement just
changes the sign of the sum of the terms except the first two in above
by the induction hypothesis and \eqref{eq:sym}. Taking sum of
\eqref{eq:blow-up6} for $d$ and $r+l-d$, we get
\begin{multline*}
     \pd{}{\tau_p} Z_n(\ve_1,-2\ve_1,\vec{a},\vec{\tau}=0)
        (e^{n(d-(r+l)/2)\ve_1} + e^{-n(d-(r+l)/2)\ve_1})
\\
    =
   \pd{}{\tau_p} Z_n(2\ve_1,-\ve_1,\vec{a},\vec{\tau}=0)
           (e^{n(d-(r+l)/2)\ve_1} + e^{-n(d-(r+l)/2)\ve_1}).
\end{multline*}
Thus we have the assertion for $Z_n$.

Thus we only need to check that \eqref{eq:blow-up6} holds for both $d$
and $r + l - d$.
\begin{NB2}
  Since we took sum of two equations, we do not need to assume that
  `$d$ and $r+l-d$ are different'.
\end{NB2}%
Suppose $p > 0$. Since $p \le (r-l)/2$ by our assumption, we can take
an integer $d$ with
\(
   p + l \le d \le r - p
\)
Thus the assumptions for \propref{prop:Casimirvanish} $0\le d\le
r-p$ and $0 \le r + l - d \le r - p$ both hold. The same argument works
for the $p < 0$ case.
\end{NB}%

We further expand as
\begin{equation*}
  Z_N(\ve_1,\ve_2,\vec{a},\vec{\tau})
  = \sum_{\vec{n}} \prod_{p\neq 0} \tau_p^{n_p}
  Z_{N,\vec{n}}(\ve_1,\ve_2,\vec{a})
\end{equation*}
as in \eqref{eq:tauexpand}.

Fix $\vec{n}$ and $N$ and consider the coefficient of $\Lambda^{2rN}
\prod \tau_p^{n_p}$ in \eqref{eq:blow-up2}. Setting $\ve_2 = -\ve_1$,
we have
\begin{equation}\label{eq:blow-up7}
  \begin{split}
  0 = 
   \sum_{\substack{(\vec{k},\vec{k})/2 + N_1 + N_2 = N \\
    \vec{n}_1+\vec{n}_2 = \vec{n}}}
  & (-1)^{\sum n_{2,p}}
  \frac{e^{(d - l/2)(\vec{k},\vec{a})} e^{(N_1 - N_2)(d - (r+l)/2)\ve_1}}
  {\prod_{\alpha\in\Delta} l^{\vec{k}}_\alpha(\ve_1,-\ve_1,\vec{a})}
  \exp\left(\frac{l}2\sum k_\alpha^2 a_\alpha\right)
\\  
  & \times 
  Z_{N_1,\vec{n}_1}(\ve_1,-2\ve_1,\vec{a}+\ve_1\vec{k})
  Z_{N_2,\vec{n}_2}(2\ve_1,-\ve_1,\vec{a}-\ve_1\vec{k})
  \end{split}
\end{equation}
if $\vec{n}\neq 0$
\begin{NB2}
Added. 2010/04/01  
\end{NB2}%
by the blow-up equation (\eqref{eq:blow-up5} and
\eqref{eq:blow-up2}). Here the summation is over $\vec{k}$,
$\vec{n}_1$, $\vec{n}_2$, $N_1$, $N_2$ and we write $\vec{n}_2 =
(\cdots, n_{2,-1}, n_{2,1}, \cdots)$. We assume
\begin{equation*}\label{eq:inequ1}
  -\sum_{p < 0} p n_p \le d \le r - \sum_{p > 0} p n_p.
\end{equation*}

We suppose that the same equality holds for $r + l - d$. So we assume
\begin{equation*}\label{eq:inequ2}
  l + \sum_{p > 0} p n_p \le d \le r + l + \sum_{p < 0} p n_p.
\end{equation*}
By our assumption, there exists $d$ satisfying both
inequalities, e.g., we take
\(
  d = r - \sum_{p > 0} p n_p
\)
in the positive case, 
\(
r + l + \sum_{p < 0} p n_p
\)
in the negative case.
Moreover, we may assume $d\neq (r+l)/2$ if $\sum n_p$ is even.

Substituting $r + l - d$ into $d$ in \eqref{eq:blow-up7}, and
replacing $\vec{k}$ by $-\vec{k}$, $(N_1, N_2)$ by $(N_2, N_1)$ and
$(\vec{n}_1,\vec{n}_2)$ by $(\vec{n}_2,\vec{n}_1)$,
\begin{NB2}
  Corrected. 2010/04/01
\end{NB2}%
we get
\begin{equation*}
  \begin{split}
  0 = 
   \sum_{\substack{(\vec{k},\vec{k})/2 + N_1 + N_2 = N \\
    \vec{n}_1+\vec{n}_2 = \vec{n}}}
  & (-1)^{\sum n_{1,p}}
  \frac{e^{(d - l/2 - r)(\vec{k},\vec{a})} e^{(N_1 - N_2)(d - (r+l)/2)\ve_1}}
  {\prod_{\alpha\in\Delta} l^{-\vec{k}}_\alpha(\ve_1,-\ve_1,\vec{a})}
  \exp\left(\frac{l}2\sum k_\alpha^2 a_\alpha\right)
\\  
  & \times 
  Z_{N_1,\vec{n}_1}(2\ve_1,-\ve_1,\vec{a}+\ve_1\vec{k})
  Z_{N_2,\vec{n}_2}(\ve_1,-2\ve_1,\vec{a}-\ve_1\vec{k}).
  \end{split}
\end{equation*}
Note
\begin{equation*}
  \frac{e^{r(\vec{k},\vec{a})/2}}
  {\prod_{\vec{\alpha}\in\Delta} l^{\vec{k}}_{\vec{\alpha}}(\ve_1,-\ve_1,\vec{a})}
  = \frac{e^{-r(\vec{k},\vec{a})/2}}
  {\prod_{\vec{\alpha}\in\Delta} l^{-\vec{k}}_{\vec{\alpha}} (\ve_1,-\ve_1,\vec{a})}
\end{equation*}
by \cite[Lem.~4.1 and (4.2)]{NY3}. Note also that
$\sum n_{1,p} = \sum n_{p} - \sum n_{2,p}$ and hence we can replace
$(-1)^{\sum n_{1,p}}$ by $(-1)^{\sum n_{2,p}}$ in the above formula.
Then the only difference of the above two equations are variables
for $Z_{N_1,\vec{n}_1}$ and $Z_{N_2,\vec{n}_2}$. By the induction hypothesis
those are also equal if $N_1$, $N_2 < N$. Therefore we have
{\allowdisplaybreaks
\begin{equation}\label{eq:touse}
  \begin{split}
  0 &=
  \sum_{\vec{n}_1 + \vec{n}_2 = \vec{n}}
  (-1)^{\sum n_{2,p}}
  \begin{aligned}[t]
  \Bigl( 
  & e^{N(d - (r+l)/2)\ve_1}
  Z_{N,\vec{n}_1}(\ve_1,-2\ve_1,\vec{a})
  Z_{0,\vec{n}_2}(2\ve_1,-\ve_1,\vec{a})
\\
  & +
  e^{-N(d - (r+l)/2)\ve_1}
  Z_{0,\vec{n}_1}(\ve_1,-2\ve_1,\vec{a})
  Z_{N,\vec{n}_2}(2\ve_1,-\ve_1,\vec{a})
\\
  & -
  e^{N(d - (r+l)/2)\ve_1}
  Z_{N,\vec{n}_1}(2\ve_1,-\ve_1,\vec{a})
  Z_{0,\vec{n}_2}(\ve_1,-2\ve_1,\vec{a})
\\
  & -
  e^{-N(d - (r+l)/2)\ve_1}
  Z_{0,\vec{n}_1}(2\ve_1,-\ve_1,\vec{a})
  Z_{N,\vec{n}_2}(\ve_1,-2\ve_1,\vec{a}) \Bigr),
  \end{aligned}
\\
  &=
  \sum_{\vec{n}_1 + \vec{n}_2 = \vec{n}}
  (-1)^{\sum n_{2,p}}
  \begin{aligned}[t]
  \Bigl[
  & e^{N(d - (r+l)/2)\ve_1}  Z_{0,\vec{n}_2}(2\ve_1,-\ve_1,\vec{a})
\\ 
  & \qquad\times
  \left\{Z_{N,\vec{n}_1}(\ve_1,-2\ve_1,\vec{a})
  - Z_{N,\vec{n}_1}(2\ve_1,-\ve_1,\vec{a})\right\}
\\
  & +
  e^{-N(d - (r+l)/2)\ve_1}   Z_{0,\vec{n}_1}(\ve_1,-2\ve_1,\vec{a})
\\
  & \qquad\times
  \left\{
  Z_{N,\vec{n}_2}(2\ve_1,-\ve_1,\vec{a})
  - Z_{N,\vec{n}_2}(\ve_1,-2\ve_1,\vec{a})
  \right\}\Bigr],
  \end{aligned}
  \end{split}
\end{equation}
if} $\vec{n} \neq 0$.
\begin{NB2}
  A condition added. 2010/04/02
\end{NB2}

We now prove the assertion by induction on $\vec{n}$. The case
$\vec{n} = 0$ is treated already in \eqref{eq:sym}.
\begin{NB2}
Corrected according to the referee's comment. 2010/04/02

  Earlier version.

First suppose $\vec{n} = 0$. Then $\vec{n}_1 = \vec{n}_2 = 0$, and we
have
\begin{equation*}
  0 = 
  \left(e^{N(d - (r+l)/2)\ve_1} - e^{-N(d - (r+l)/2)\ve_1}\right)
  \left\{Z_{N,0}(\ve_1,-2\ve_1,\vec{a}) - Z_{N,0}(2\ve_1,-\ve_1,\vec{a})
  \right\}.
\end{equation*}
Since we take $d \neq (r+l)/2$ in this case, we have the assertion for
$Z_{N,0}$.
\end{NB2}

Now we assume that the assertion holds for smaller $\vec{n}$. Note
that the assumption on $\vec{n}$ implies that on smaller ones.
\begin{NB}
  If $\vec{n}$ is not either positive nor negative, this is not true
  in general. So if $\vec{n}$ is odd and satisfies the inequality,
  smaller even $\vec{n}$ may not satisfy the strict inequality.
\end{NB}%
Then the only remaining terms in \eqref{eq:touse} are either
$\vec{n}_1 = 0$ or $\vec{n}_2 = 0$. Therefore we have
\begin{equation*}
  0 = 
    \left(e^{N(d - (r+l)/2)\ve_1} - 
      (-1)^{\sum n_p} e^{-N(d - (r+l)/2)\ve_1}\right)
  \left\{
    Z_{N,\vec{n}}(\ve_1,-2\ve_1,\vec{a}) - Z_{N,\vec{n}}(2\ve_1,-\ve_1,\vec{a})
  \right\}.
\end{equation*}
Hence we have the assertion for $Z_{N,\vec{n}}$. Note that we take
$d\neq (r+l)/2$ when $\sum n_p$ is even, so the above is a nontrivial
equality.
\begin{NB}

 (2) Let
\begin{equation*}
  \Fin(\ve_1,\ve_2,\vec{a},\Lambda,\vec{\tau}) \defeq
  \ve_1 \ve_2 \log \Zin(\ve_1,\ve_2,\vec{a},\Lambda,\vec{\tau}).
\end{equation*}
We divide the right hand side of \eqref{eq:blow-up2} by
$\Zin(\ve_1,\ve_2,\vec{a};\Lambda,\vec{\tau}=0)$ after setting
$\vec{\tau} = 0$:
\begin{equation*}
  \begin{split}
   & \sum_{\substack{\vec{k}=(k_\alpha)\in\Z^r \\ \sum k_\alpha = k}}
    \frac{(e^{(\ve_1+\ve_2)(d-(r+l)/2)}\Lambda^{2r})^{(\vec{k}^2)/2}
      e^{(d-l/2)(\vec{k},\vec{a})}}
    {\prod_{\vec{\alpha} \in \Delta} 
      l_{\vec{\alpha}}^{\vec{k}}(\ve_1,\ve_2,\vec{a})}
\\
& \quad\times
\exp \left[l\left(
\frac{1}{6}(\ve_1+\ve_2)\sum_\alpha k_\alpha^3+
\frac{1}{2}\sum_\alpha k_\alpha^2 a_\alpha \right) \right]
\\
& \quad\quad\times
\exp\Biggl[
\begin{aligned}[t]
& \frac{\Fin(\ve_1,\ve_2-\ve_1,\vec{a}+\ve_1 \vec{k};
\Lambda e^{\frac{d-(r+l)/2}{2r}\ve_1},
(e^{\ve_1/2}-e^{-\ve_1/2})\vec{t})}{\ve_1(\ve_2-\ve_1)}
\\
& \; +
\frac{\Fin(\ve_1-\ve_2,\ve_2,\vec{a}+\ve_2 \vec{k};
\Lambda e^{\frac{d-(r+l)/2}{2r}\ve_2},
(e^{\ve_2/2}-e^{-\ve_2/2})\vec{t})}{(\ve_1-\ve_2)\ve_2}
\\
& \; -
\frac{\Fin(\ve_1,\ve_2,\vec{a};\Lambda,\vec{\tau}=0)}{\ve_1\ve_2}
\Biggr].
\end{aligned}
\end{split}
\end{equation*}
The last three lines in the exponential is expressed as
\begin{equation*}
  \begin{split}
    & \frac{1}{\ve_2-\ve_1}
    \left[\frac{\Fin(\ve_1,\ve_2-\ve_1,\vec{a}+\vec{k}\ve_1;
        \Lambda e^{\frac{d-r/2}{2r}\ve_1},(e^{\ve_1/2}-e^{-\ve_1/2})\vec{t})
        -\Fin(\ve_1,\ve_2,\vec{a},\Lambda,\vec{\tau}=0)}{\ve_1}
    \right. 
\\
    & \left.
      -\frac{\Fin(\ve_1-\ve_2,\ve_2,\vec{a}+\vec{k}\ve_2,
        \Lambda e^{\frac{d-r/2}{2r}\ve_2},(e^{\ve_2/2}-e^{-\ve_2/2})\vec{t})
        -\Fin(\ve_1,\ve_2,\vec{a},\Lambda,\vec{\tau}=0)}
      {\ve_2} 
    \right]
  \end{split}
\end{equation*}
To be continued....
\end{NB}%
\end{proof}

\begin{Remark}
  \begin{NB2}
    Remark added. 2010/04/02
  \end{NB2}%
  We need the vanishing \eqref{eq:blow-up7} for the case when $\sum
  n_p$ is odd, but it is enough to suppose that the right hand side of
  \eqref{eq:blow-up7} is the same for $d$ and $r+l-d$ when $\sum n_p$
  is even. In particular, if we use \eqref{eq:blow-up4} instead of
  \eqref{eq:blow-up5}, the above argument works even for $\vec{n}
  = 0$. This is nothing but the proof of \eqref{eq:sym} in \cite{GNY2}.
\end{Remark}

\subsection{Contact term equations}

We expand as
\begin{multline*}
  \ve_1\ve_2\log Z_l(\ve_1,\ve_2,\vec{a};\Lambda,\vec{\tau})
\\
  = F_0(\vec{a};\Lambda,\vec{\tau})
  + (\ve_1+\ve_2) H(\vec{a};\Lambda,\vec{\tau})
  + \ve_1\ve_2 A(\vec{a};\Lambda,\vec{\tau})
  + \frac{\ve_1^2+\ve_2^2}3 B(\vec{a},\Lambda,\vec{\tau}) + \cdots,
\end{multline*}
where we only consider
$\left.(\partial/\partial\vec{\tau})^{\vec{n}}
\right|_{\vec{\tau}=0}$ applied to this function, with $\vec{\tau}$
in the range in \propref{prop:regular}(2) so that singular terms do
not appear.

By \eqref{eq:sym} $H(\vec{a};\Lambda,\vec{\tau}=0)$ comes only from
the perturbative part. Thus we have $H(\vec{a};\Lambda,\vec{\tau}=0) =
-\pi\sqrt{-1}\langle\vec{a},\rho\rangle$, where $\rho$ is the half sum
of the positive roots. More generally by \propref{prop:regular}(1) we
have
\begin{equation*}
  \left(\pd{}{\vec{\tau}}\right)^{\vec{n}}
  H(\vec{a};\Lambda,\vec{\tau}=0) = 0
\end{equation*}
if $\vec{n}$ satisfies the condition there.

We introduce a new coordinate system for $\vec{a}$ by $a^i = a_1 + a_2
+ \cdots + a_i$ ($i=1,\dots, r-1$) as in \cite[\S1]{NY1}. We also
define $k^i$ for $\vec{k}$ in the same way.

Let
\begin{equation}\label{eq:tau}
  \tau_{ij} 
  = -\frac1{2\pi\sqrt{-1}}
  \frac{\partial^2 F_0(\vec{a};\Lambda,\vec{\tau}=0)}
  {\partial a^i\partial a^j}.
\end{equation}
\begin{NB2}
  Corrected. The referee requests an explanation to which this
  element is belonged. This is postponed below. 2010/04/10
\end{NB2}%
Let $\Theta_E(\vec{\xi}|\tau)$ be the Riemann theta function defined by
\begin{equation*}
  \Theta_E(\vec{\xi}|\tau)
  = \sum_{\vec{k}\in\Z^{r-1}} \exp\left(
    \pi\sqrt{-1}\sum_{i,j} \tau_{ij} k^i k^j
    + 2\pi\sqrt{-1} \sum_i k^i(\xi^i+\frac12)
    \right).
\end{equation*}

We substitute \eqref{eq:blow-up3} into the left hand side of
\eqref{eq:blow-up4}, and take the limit of $\ve_1, \ve_2\to 0$. Using
$H(\vec{a};\Lambda,\vec{\tau}=0) =
-\pi\sqrt{-1}\langle\vec{a},\rho\rangle$, we obtain
\begin{multline*}
  \exp(B - A) =
      \exp\left[
        -\frac1{8 r^2}
        \left(d - \frac{r+l}2\right)^2
        \frac{\partial^2 F_0}{(\partial\log\Lambda)^2}
      \right]
      \times 
\\
      \Theta_E\left(\left.
    - \frac1{2\pi\sqrt{-1}}\frac1{2r}
     \left(d - \frac{r+l}{2}\right)
     \frac{\partial^2 F_0}{\partial\log\Lambda\partial\vec{a}}\,
    \right|\tau \right)
\end{multline*}
for $0\le d\le r$ as in \cite[\S4]{NY3}. Here $F_0$, $A$, $B$ are
evaluated at $(\vec{a};\Lambda,\vec{\tau}=0)$.

\begin{NB}
Copy from \cite{NY3}:

Let us derive equations for $F_0$. We use
{\allowdisplaybreaks
\begin{equation*}
\begin{split}
   & \frac{F_0(\vec{a}+\ve_1\vec{k}; t_1^{\left(d-r/2\right)}\q,\bbeta)}
   {\ve_1(\ve_2-\ve_1)}
    +\frac{F_0(\vec{a}+\ve_2\vec{k}; t_2^{\left(d-r/2\right)}\q,\bbeta)}
    {(\ve_1-\ve_2)\ve_2}
\\
  & \qquad\qquad\qquad =
  \frac{1}{\ve_1\ve_2}F_0
  - \left[
    \frac{\partial^2 F_0
    }{(\partial\log \q)^2} \frac{\bbeta^2}2\left(d-\frac{r}2\right)^2
    +
    \frac{\partial^2 F_0
    }{\partial\log \q\partial a^l}\bbeta\left(d-\frac{r}2\right) k^l
    +
    \frac{\partial^2 F_0
    }{\partial a^l\partial a^m}
      \frac{k^l k^m}{2}
  \right]
  + \cdots,
\\
   & \frac{\ve_2 H(\vec{a}+\ve_1\vec{k}; t_1^{\left(d-r/2\right)}\q,\bbeta)}
   {\ve_1(\ve_2-\ve_1)}
    +\frac{\ve_1 H(\vec{a}+\ve_2\vec{k};t_2^{\left(d-r/2\right)}\q,\bbeta)}
   {(\ve_1-\ve_2)\ve_2}
\\
  & \qquad\qquad =
  \frac{\ve_1+\ve_2}{\ve_1\ve_2} H
  + \left[\frac{\partial H
    }{\partial\log \q} \bbeta\left(d-\frac{r}2\right)
    +
    \frac{\partial H
    }{\partial a^l} k^l
  \right]
  + \cdots
  = \frac{\ve_1+\ve_2}{\ve_1\ve_2} H
  - \pi\sqrt{-1}\langle\vec{k},\rho\rangle  + \cdots
  ,
\\
   & \frac{\ve_2^2 G(\vec{a}+\ve_1\vec{k};t_1^{\left(d-\frac1r\right)}\q,\bbeta)}
  {\ve_1(\ve_2-\ve_1)}
    +\frac{\ve_1^2 G(\vec{a}+\ve_2\vec{k};t_2^{\left(d-\frac1r\right)}\q,\bbeta)}
  {(\ve_1-\ve_2)\ve_2}
  =
  \frac{(\ve_1+\ve_2)^2-\ve_1\ve_2}{\ve_1\ve_2} G
  + \cdots,
\\
   & F_1(\vec{a}+\ve_1\vec{k}; t_1^{\left(d-\frac1r\right)}\q,\bbeta)
    + F_1(\vec{a}+\ve_2\vec{k}; t_2^{\left(d-\frac1r\right)}\q,\bbeta)
   = 
  2 F_1
  + \cdots,
\end{split}
\end{equation*}
where} $F_0$, $H$, $G$, $F_1$ and their derivatives are evaluated at
$\vec{a}, \q$ in the right hand sides.
\end{NB}

In particular, the right hand side is independent of $d$. Dividing by
the expression for $d=(r+l)/2$ (if $r+l$ is even), $d=(r+l-1)/2$, (if
$r+l$ is odd), we have
\begin{equation}\label{eq:contact}
\begin{split}
   & \frac{\Theta_E\left(\left.
    -\frac1{4\pi\sqrt{-1}r}
     \left(d - \frac{r+l}{2}\right)
     \frac{\partial^2 F_0}{\partial\log\Lambda\partial\vec{a}}\,
    \right|\tau \right)}
   {\Theta_E(0|\tau)}
   = \exp\left[
        \frac1{8 r^2}
        \left(d - \frac{r+l}2\right)^2
        \frac{\partial^2 F_0}{(\partial\log\Lambda)^2}
      \right],
\\
   & \frac{\Theta_E\left(\left.
     -\frac1{4\pi\sqrt{-1} r}
     \left(d - \frac{r+l}{2}\right)
     \frac{\partial^2 F_0}{\partial\log\Lambda\partial\vec{a}}\,
    \right|\tau \right)}
   {\Theta_E\left(\left.
    \frac1{8\pi\sqrt{-1} r}
     \frac{\partial^2 F_0}{\partial\log\Lambda\partial\vec{a}}\,
    \right|\tau \right)}
   = \exp\left[
        \frac1{8 r^2}
        \left\{ \left(d - \frac{r+l}2\right)^2 - \frac14\right\}
        \frac{\partial^2 F_0}{(\partial\log\Lambda)^2}
     \right]
\end{split}
\end{equation}
according to either $r+l$ is even or odd.
This equation is called the {\it contact term equation}. This equation
determines $F_0(\vec{a};\Lambda,\vec{\tau}=0)$ recursively in the
expansion with respect to $\Lambda$, starting from the perturbation
part.

Here we remind again that our full partition function, and hence
$\tau_{ij}$ and $\Theta_E(\vec{\xi}|\tau)$, etc, do not have the
rigorous meaning. (See \remref{rem:full}.) The rigorous form of the
contact term equation is given by rewriting it as an equation for the
instanton part. This was given in \cite[(7.5)]{NY1}, for the
homological version of Nekrasov partition function, but we do not give
here since it is not enlightening.
\begin{NB2}
  An explantion added according to the referee's comment. 2010/04/03
\end{NB2}

Similarly as the limit of \eqref{eq:blow-up5}, we obtain
\begin{multline}\label{eq:contact22}
  0 = - \frac1{2r}\left(d - \frac{r+l}2\right)
  \pd{}{\log\Lambda}
    \left(\pd{}{\vec{\tau}}\right)^{\vec{n}}\!F_0
    (\vec{a};\Lambda,\vec{\tau}=0)
\\
  - \frac1{2\pi\sqrt{-1}}\sum_i 
    \left.\frac{\partial\log\Theta_E}{\partial \xi^i}\right|_{
      \vec{\xi} = - \frac1{4\pi\sqrt{-1}r}
      \left(d - \frac{r+l}2\right)
      \frac{\partial^2F_0}{\partial \log\Lambda\partial\vec{a}}}
    \pd{}{a^i}
      \left(\pd{}{\vec{\tau}}\right)^{\vec{n}}\!F_0
    (\vec{a};\Lambda,\vec{\tau}=0)
\end{multline}
when $\vec{n}$ is nonzero, either positive or negative and satisfies
\eqref{eq:const} and (\ref{prop:regular})(1)(a),(b).
\begin{NB}
Similarly as the limit of \eqref{eq:blow-up5'}, we obtain
\begin{multline}\label{eq:contact2}
  0 = - \frac1{2r}\left(d - \frac{r+l}2\right)
  \frac{\partial^2 F_0}{\partial\log\Lambda\partial\tau_p}
\\
  - \frac1{2\pi\sqrt{-1}}\sum_i 
    \left.\frac{\partial\log\Theta_E}{\partial \xi^i}\right|_{
      \vec{\xi} = - \frac1{4\pi\sqrt{-1}r}
      \left(d - \frac{r+l}2\right)
      \frac{\partial^2F_0}{\partial \log\Lambda\partial\vec{a}}}
    \frac{\partial^2 F_0}{\partial a^i\partial \tau_p}
\end{multline}
for $|p| \le d \le r$ if $p < 0$, $0\le d \le r-p$ if $p > 0$ and
$-(r+l)/2 \le p \le (r-l)/2$ (as we need to use
\propref{prop:regular}).
\end{NB}%
We call this as the {\it contact term equation\/} for
$(\partial/\partial\vec{\tau})^{\vec{n}}F_0$. This derivation of the
contact term equation from the blow-up equation can be done in the
same way as in \cite[\S5.3]{NY2}.
\begin{NB2}
  The reference added according to the referee's suggestion. 2010/04/01
\end{NB2}

\begin{NB}
{\allowdisplaybreaks
\begin{equation*}
\begin{split}
   & \frac{F_0(\vec{a}+\ve_1\vec{k}; t_1^{\left(d-r/2\right)}\q,
     (e^{\ve_1/2}-e^{-\ve_1/2})\vec{t})}
   {\ve_1(\ve_2-\ve_1)}
    +\frac{F_0(\vec{a}+\ve_2\vec{k}; t_2^{\left(d-r/2\right)}\q,
      (e^{\ve_2/2}-e^{-\ve_2/2})\vec{t})}
    {(\ve_1-\ve_2)\ve_2}
\\
  & \qquad\qquad\qquad =
  \frac{1}{\ve_1\ve_2}F_0
  - \left[
    \frac{\partial^2 F_0
    }{(\partial\log \q)^2} \frac12\left(d-\frac{r}2\right)^2
    +
    \frac{\partial^2 F_0
    }{\partial\log \q\partial a^l}\left(d-\frac{r}2\right) k^l
    +
    \frac{\partial^2 F_0
    }{\partial a^l\partial a^m}
      \frac{k^l k^m}{2}
  \right]
\\
  & \qquad\qquad\qquad\qquad
  - \sum_p \left[
    \left(d - \frac{r}2\right)
    \frac{\partial^2 F_0}{\partial \log\q\partial \tau_p}
    + 
    \frac{\partial^2 F_0}{\partial a^i\partial \tau_p} k^i
    \right] t_p
  + \cdots,
\end{split}
\end{equation*}
}
\end{NB}

\begin{NB}
  Suppose $l = r-1$. Then the range of $p$ is $-r+1$, $-r+2$, \dots,
  $-1$. If $p = -r+1$, then $|p|\le d\le r$ holds if and only if $d =
  r-1$ and $r$. But $d - (r+l)/2 = -1/2$ and $1/2$, hence two
  equations \eqref{eq:contact2} for $d=r-1$, $r$ are equivalent. So we
  only get one nontrivial equation. 

  If $p \ge -r+2$, we also have the equation for $d = r-2$, i.e., $d -
  (r+l)/2 = -3/2$. So we at least have two equations.
\end{NB}
\begin{NB}
  If $r-l$ is even, the range of $p$ is $-(r+l)/2, -(r+l)/2 + 1,
  \dots, (r-l)/2$. If $p = -(r+l)/2$, then the inequality is $(r+l)/2
  \le d\le r$. If $p = (r-l)/2$, then $0\le d\le (r+l)/2$. Note that
  from the symmetry $d\leftrightarrow (r+l)-d$ of \eqref{eq:contact2},
  we have the inequality for $l\le d \le r$ and $0\le d\le r+l$
  respectively.  Taking the intersection, we have the above holds
  $l\le d\le r$ for any $p$ in the range.

  Next suppose $r-l$ is odd. The range of $p$ is from $-(r+l-1)/2$ to
  $(r-l-1)/2$. Then the inequality is $(r+l-1)/2 \le d \le r$ and
  $0\le d \le (r+l+1)/2$. From the symmetry, we again have $l\le d\le r$.
\end{NB}

Since $\partial \log\Theta_E/\partial \xi^i$ is divisible by $\Lambda$,
\begin{NB}
  If we differentiate the theta by $\xi_\alpha$, we get
  $\sum_{\vec{k}} k_\alpha \exp(..)$. This is $0$ if $\vec{k} = 0$.
\end{NB}%
this equation determines $(\partial/\partial\vec{\tau})^{\vec{n}}F_0$
recursively in the expansion with respect to $\Lambda$ starting from
the constant term:
\begin{equation}\label{eq:constterm}
  \begin{split}
   & 
   \left(\pd{}{\vec{\tau}}\right)^{\vec{n}}\!F_0
    (\vec{a};\Lambda=0,\vec{\tau}=0)
   =
   \left.
   \left(\pd{}{\vec{\tau}}\right)^{\vec{n}}\!
   \sum_{p,\alpha} \tau_p e^{p a_\alpha}
   \right|_{\vec{\tau}, \Lambda = 0}
\\
   =\; & 
   \begin{cases}
   \sum_\alpha \left.e^{p a_\alpha}\right|_{\Lambda=0} & 
   \text{if $n_p = 1$ and $n_q = 0$ for $q\neq p$},
\\
   0 & \text{otherwise}.
   \end{cases}
  \end{split}
\end{equation}
In particular, we have
\begin{equation*}
    \left(\pd{}{\vec{\tau}}\right)^{\vec{n}}\!F_0
    (\vec{a};\Lambda,\vec{\tau}=0)
  = 0
\end{equation*}
unless $n_p = 1$ and $n_q = 0$ for $q\neq p$. The remaining
$\partial F_0/\partial \tau_p$ will be determined in the next subsection.

\begin{NB}
Thus we have
\begin{equation}
  \begin{split}
    &
  0 = \frac1{2r}
  \frac{\partial^2 F_0}{\partial\log\Lambda\partial\tau_p}
  - \frac1{2\pi\sqrt{-1}}\sum_i 
    \left.\frac{\partial\log\Theta_E}{\partial \xi^i}\right|_{
      \vec{\xi} = \frac1{4\pi\sqrt{-1}r}
      \frac{\partial^2F_0}{\partial \log\Lambda\partial\vec{a}}}
    \frac{\partial^2 F_0}{\partial a^i\partial \tau_p},
\\
    &
  0 = \frac1{4r}
  \frac{\partial^2 F_0}{\partial\log\Lambda\partial\tau_p}
  - \frac1{2\pi\sqrt{-1}}\sum_i 
    \left.\frac{\partial\log\Theta_E}{\partial \xi^i}\right|_{
      \vec{\xi} = \frac1{8\pi\sqrt{-1}r}
      \frac{\partial^2F_0}{\partial \log\Lambda\partial\vec{a}}}
    \frac{\partial^2 F_0}{\partial a^i\partial \tau_p},
  \end{split}
\end{equation}
according to either $r+l$ is even or odd.
\begin{NB2}
  I take $d = (r+l)/2 - 1$ and $(r+l-1)/2$.
\end{NB2}
\end{NB}

\begin{NB}
By \eqref{eq:deriv} and the above, we have
\begin{multline*}
  \sum_c (\Lambda^{2r} e^{-(r+l)(\ve_1+\ve_2)/2})^{(\Delta(c),[\proj^2])}
  \int_{M(c)} 
  \prod_i \ch\left(\psi^{\mu_i}(\cE)/[0]\right)
  \td M(c) \exp(l c_1(\Vcal(\cE)))
\\
  = \prod_i \left.\left(
    (1 - e^{-\ve_1})(1 - e^{-\ve_2})\pd{}{\tau_{\mu_i}}
  \right)
  \Zin_l(\ve_1,\ve_2,\vec{a};\Lambda,\vec{\tau})\right|_{\vec{\tau}=0}
\\
  = \prod_i \left(
    (1 - e^{-\ve_1})(1 - e^{-\ve_2})\pd{\Fin}{\tau_{\mu_i}}
  \right)
  \Zin_l(\ve_1,\ve_2,\vec{a};\Lambda,\vec{\tau} = 0)
\end{multline*}
\end{NB}

Since higher derivatives vanish, we have
\begin{multline*}
  \left.
    \frac{
   \left(
    (e^{\ve_1/2} - e^{-\ve_1/2})(e^{\ve_2} - e^{-\ve_2/2})
    \right)^{\sum_p n_p}
      \left(\pd{}{\vec{\tau}}\right)^{\vec{n}}
  Z_l(\ve_1,\ve_2,\vec{a};\Lambda,\vec{\tau}=0)
  }
  {
    Z_l(\ve_1,\ve_2,\vec{a};\Lambda,\vec{\tau} = 0)
    }
  \right|_{\ve_1=\ve_2=0}
\\
= 
   \prod_p \left(
     \pd{F_0}{\tau_p} (\vec{a};\Lambda,\vec{\tau}=0)\right)^{n_p}.
\end{multline*}
By \eqref{eq:deriv} this is equal to
\begin{multline}\label{eq:deriv2}
  \frac1
  {\Zin_l(\ve_1,\ve_2,\vec{a};\Lambda,\vec{\tau} = 0)}
    \sum_c (\Lambda^{2r} e^{-(r+l)(\ve_1+\ve_2)/2})^{(\Delta(c),[\proj^2])}
\\
  \times
  \left.
  \int_{M(c)} 
  \ch\left(
   \bigotimes_p \left(\psi^{p}(\cE)/[0]\right)^{\otimes n_p}\right)
  \td M(c) \exp(l c_1(\Vcal(\cE)))
  \right|_{\ve_1=\ve_2=0},
\end{multline}
where $\bullet/[0]$ is the slant product
$q_{2*}(\bullet\otimes q_1^*(\C_0)) = 
(e^{\ve_1/2} - e^{-\ve_1/2})(e^{\ve_2/2} - e^{-\ve_2/2})\bullet/[\C^2]$
with the skyscraper sheaf $\C_0$ at the origin, which is
nothing but the restriction to the origin.

\subsection{Seiberg-Witten prepotential}
We give a quick review of the Seiberg-Witten prepotential for the
theory with $5$-dimensional Chern-Simons term in this subsection.
See \cite[App.~A]{GNY2} for more detail.
(We change the notation slightly: $m$ in [loc.\ cit.] is our $l$.
$a_i$ is our $a_\alpha$. $\bbeta$ is set to be $1$.)

We consider a family of hyperelliptic curves parame\-trized by $\vec{U}
= (U_1,\dots, U_{r-1})$:
\begin{equation*}
\begin{split}
  & C_{\vec{U},l} : Y^2 = P(X)^2 - 4 (-X)^{r+l}\Lambda^{2r}
\\
  & \qquad P(X) = 
     X^r + U_1 X^{r-1} + U_2 X^{r-2} + \cdots + U_{r-1}X + (-1)^r
\end{split}
\end{equation*}
for $-r < l < r$, $l\in\Z$.
Note that we set $\bbeta = 1$ from [loc.\ cit.] for brevity.
We call them {\it Seiberg-Witten curves}. We define the {\it
  Seiberg-Witten differential\/} by
\begin{equation*}
  dS = \frac1{2\pi \sqrt{-1}
  }\, {\log X}\,  \frac{2XP'(X) - (r+l) P(X)}{2XY} dX.
\end{equation*}

We choose $z_\alpha$ ($\alpha=1,\dots,r$) so that $X_\alpha =
e^{-\sqrt{-1}z_\alpha}$ are zeroes of $P(X) = 0$. Then we take
the cycles $A_\alpha$, $B_\alpha$ ($\alpha=2,\dots, r$)
in a way explained in [loc.\ cit.].

We define $a_\alpha$, $a^D_\alpha$ by
\begin{equation*}
  a_\alpha = \int_{A_\alpha} dS, \qquad
  a^D_\alpha = \int_{B_\alpha} dS.
\end{equation*}
We then invert the role of $a_\alpha$ and $U_p$, so we consider
$a_\alpha$ as variables and $U_p$ are functions in $a_\alpha$. Here we use
\(
  a_\alpha = -\sqrt{-1}z_\alpha + O(\Lambda)
\)
[loc.\ cit., (A.1)].

Then one can show that there exists a locally defined function $\cF_0
= \cF_0(\vec{a};\Lambda)$ such that
\begin{equation*}
  a^D_\alpha = -\frac1{2\pi\sqrt{-1}}\pd{\cF_0}{a_\alpha}.
\end{equation*}
This defines $\cF_0$ up to a function (in $\Lambda$) independent of
$\vec{a}$. This ambiguity is fixed by specifying
$\partial\cF_0/\partial\log\Lambda$ and the perturbation part of
$\cF_0$. See [loc.\ cit.] for detail.

It was proved in [loc.\ cit., (A.27), (A.33)] that $\cF_0$ satisfies
the contact term equations \eqref{eq:contact} where the period matrix
$(\tau)$ is defined by the same formula as \eqref{eq:tau} by replacing
$F_0$ by $\cF_0$. Moreover, it was also proved that $\cF_0$ has the
same perturbation part as $F_0$ [loc.\ cit., Prop.~A.6]. Therefore the
recursive structure of \eqref{eq:contact} implies that $\cF_0 = F_0$,
i.e., the leading part of the Nekrasov partition function is equal to
the Seiberg-Witten prepotential.

In [loc.~cit., (A.25)], it was proved
\begin{equation}\label{eq:contact3}
    0 = \frac1{2r}\pd{U_p}{\log\Lambda}
\\
  + \frac1{2\pi\sqrt{-1}}\sum_i 
    \left.\frac{\partial\log\Theta_E}{\partial \xi^i}\right|_{
      \vec{\xi} = - \frac1{4\pi\sqrt{-1}r}
      \frac{\partial^2F_0}{\partial \log\Lambda\partial\vec{a}}}
    \pd{U_p}{a^i}
\end{equation}
under the assumption $r+l$ is even. This is the same as the equation
\eqref{eq:contact22} with $d-(r+l)/2 = \pm 1$.
\begin{NB}
  $d = (r+l)/2 \pm 1$
\end{NB}
When $r+l$ is odd, we assume that $U_p$ satisfies \eqref{eq:contact22}
with $d=(r+l+1)/2$ for a moment, and gives a proof later.
\begin{NB2}
  Corrected according to the referee's question. 2010/04/01
\end{NB2}%

The initial condition for $U_p$ is
\begin{equation*}
   U_p = (-1)^p e_p(X_1,\dots, X_r)
   = (-1)^p e_p(e^{a_1},\dots, e^{a_r}) \quad\text{at $\Lambda=0$},
\end{equation*}
where $e_p$ is the $p^{\mathrm{th}}$ elementary symmetric
polynomial. Noticing that \eqref{eq:contact22} holds for polynomials
in $\partial F_0/\partial \tau_p$, we see that $(-1)^pU_p$ and
$\partial F_0/\partial \tau_p$ are related exactly in the same way as
an elementary symmetric polynomial and a power sum by
\eqref{eq:constterm}.

\begin{Theorem}
  \begin{equation*}
    \pd{F_0}{\tau_p}(\vec{a};\Lambda,\vec{\tau}=0)
    = X_1^p + \cdots + X_r^p
  \end{equation*}
holds for $-(r+l)/2\le p \le (r-l)/2$.
\end{Theorem}

\begin{NB}
  The condition $|p| \le d \le r$ if $p < 0$, $0\le d \le r-p$ if $p >
  0$ for $d = (r+l)/2 \pm 1$ holds. More precisely, first we take $d =
  (r+l)/2 - 1$. Then the condition is $|p|\le (r+l)/2 - 1$ for $p<0$
  and $p\le (r-l)/2 + 1$ for $p>0$. Thus the range is $-(r+l)/2 + 1\le
  p \le (r-l)/2 + 1$. Second we take $d = (r+l)/2 + 1$. Then the range
  is $-(r+l)/2 - 1 \le p \le (r-l)/2 - 1$. Thus the union of two ranges
  contains $-(r+l)/2\le p \le (r-l)/2$.
\end{NB}

Note that all $U_p$'s are written by polynomials in
$\pd{F_0}{\tau_p}(\vec{a};\Lambda,\vec{\tau}=0)$ in the above range, as
$X_1\cdots X_r = 1$.
Thanks to \eqref{eq:deriv2} the polynomials can be replaced by those
in Adams operators.
We thus have
\begin{multline*}
  U_p 
  = 
    \frac{(-1)^p}
  {\Zin_l(\ve_1,\ve_2,\vec{a};\Lambda,\vec{\tau} = 0)}
    \sum_c (\Lambda^{2r} e^{-(r+l)(\ve_1+\ve_2)/2})^{(\Delta(c),[\proj^2])}
\\
  \times
  \left.
  \int_{M(c)} 
  \ch\left(
   \Wedge^p (\cE/[0])\right)
  \td M(c) \exp(l c_1(\Vcal(\cE)))
  \right|_{\ve_1=\ve_2=0}
\end{multline*}
for $0 < p \le (r-l)/2$. For $(r-l)/2 \le p < r$, the equation holds if
we replace $\Wedge^p (\cE/[0])$ by $\Wedge^{r-p} (\cE/[0])^\vee$.
\begin{NB}
  $0 > p - r \ge - (r+l)/2$.
\end{NB}%
\begin{NB}
  In particular, we have two expressions when $p=(r-l)/2$.
\end{NB}%
\begin{NB}
  Here we have used that $\bullet/[0]$ commutes with the Adams
  operator $\psi^p$ as it is nothing but the restriction to $0$.
\end{NB}%
This equation gives a moduli theoretic description of the coefficients
$U_p$ in the Seiberg-Witten curves.

\begin{NB}
  The following attempt for $\Wedge^p(\cE/[0])$ does not work:

As we pointed out, all $U_p$'s can be written by polynomials in
$\pd{F_0}{\tau_p}(\vec{a};\Lambda,\vec{\tau}=0)$ in the above range,
but we can write them more directly as follows.

Let $\Wedge_\tau \alpha = \sum_{p=0}^\infty \tau^p \Wedge^p \alpha$.
This is multiplicative, i.e., $\Wedge_\tau (\alpha+\beta)
= \Wedge_\tau\alpha\times \Wedge_\tau\beta$. Then we define
  \begin{multline*}
    \Zin_l(\ve_1,\ve_2,\vec{a};\Lambda,\tau,\tau')
    \defeq \sum_c
    \left(\Lambda^{2r} e^{-(r+l)(\ve_1+\ve_2)/2}\right)^{(\Delta(c),[\proj^2])}
\\
  \times
  \int_{M(c)} \td M(c) \exp(l c_1(\Vcal(\cE)))\,
  \Wedge_\tau \cE/[\C^2]
  \otimes\Wedge_{\tau'}\cE^\vee/[\C^2]
\end{multline*}
and also
$Z_l(\ve_1,\ve_2,\vec{a};\Lambda,\tau,\tau')$ as
before.
We consider the similar partition function
$\bZ(\ve_1,\ve_2,\linebreak[1]\vec{a};\Lambda,\tau,\tau',t,t')$ on the
blow-up as before.
Thanks to the multiplicativity of $\Wedge_\tau$, we get the blow-up
formula \eqref{eq:blow-up3} in this setting also, where the last variable
of the $Z_l$ in the right hand side becomes
$e^{\frac{k\ve_\alpha}r} e^{-\frac{\ve_\alpha}2}\left((\tau,\tau')
+(e^{\ve_\alpha}-1)(t,t')\right)$ ($\alpha=1,2$).
The proof of \propref{prop:Casimirvanish} also works for this version
to get
\begin{equation*}
   \left(
     \pd{}{t}
   \right)^n
   \left(
     \pd{}{t'}
   \right)^{n'}
   \bZ_{l,0,d}(\ve_1,\ve_2,\vec{a};
  \Lambda,\vec{\tau}=0,\vec{t}=0)
  = 0
\end{equation*}
for $n' \le d \le r - n$ as in \eqref{eq:blow-up5}.
  Let $\Wedge_\tau \bullet = \sum_{p=0}^\infty \tau^p \Wedge^p \bullet$. We
  then consider

Then
  \begin{equation*}
    \begin{split}
    & \Wedge_\tau\left(\cEf\oplus (1+x_i)C_m\right) 
    = \Wedge_\tau \cEf \otimes\left(\Wedge_\tau (1+x_i)C_m\right)
    = \Wedge_\tau \cEf \otimes \sum_p
    \left\{(1+x_i) \tau\right\}^p \Wedge^p C_m
\\
   =\; &
   \Wedge_\tau \cEf \otimes \Wedge_{(1+x_i)\tau} C_m.
    \end{split}
  \end{equation*}
Here we have used that $(1+x_i)$ is a line bundle. Similarly we have
\begin{equation*}
  \Wedge_{\tau'}\left(\cEf\oplus (1+x_i)C_m\right)^\vee
  = \Wedge_{\tau'} \cEf^\vee \otimes \Wedge_{(1+x_i)^{-1} \tau'} C_m^\vee.
\end{equation*}
\end{NB}

It remains to show that $U_p$ satisfies \eqref{eq:contact22}
with $d = (r+l+1)/2$ in the case $r+l$ is odd.
Let us give a sketch of the argument. We use the notation in [loc.\
cit.], e.g., $E(X_1,X_2)$ is the prime form, $\omega_{\infty_+ - 0_-}$
is the meromorphic differential with the vanishing $A$-periods having
poles $0_-$ and $\infty_+$ of residue $-1$, $+1$ respectively, etc.
\begin{NB2}
  Added according to the referee's suggestion. 2010/04/01
\end{NB2}%

By \cite[Cor.~2.11]{Fay} we have
\begin{multline*}
  \frac{
    \Theta_E^2(\frac12 \int^{2X}_{0_-+\infty_+}\vec\omega)
  }
  {\Theta_E^2(\frac12 \int_{0_-}^{\infty_+}\vec\omega)}
    \frac{E(0_-,\infty_+)}{E(X,0_-) E(X,\infty_+)}
\\
  = \omega_{\infty_+-0_-}(X)
  + 2\times\frac1{2\pi\sqrt{-1}}\sum_i
      \left.\frac{\partial\log\Theta_E}{\partial \xi^i}\right|_{
      \vec{\xi} = \frac12\int_{0_-}^{\infty_+}\vec\omega}
    \omega_i(X).
\end{multline*}

By [loc.\ cit., (A.18)], we have
\begin{equation*}
  \frac{(P(X) - Y)dX}{2XY}
  - \frac1{2r} \sum_{p=1}^{r-1} \pd{U_p}{\log\Lambda} \frac{X^{r-p-1}dX}{Y}
  = \omega_{\infty_+ - 0_-}(X).
\end{equation*}
Therefore it is enough to show
\begin{equation}
  \label{eq:want}
      \frac{(P(X) - Y)dX}{2XY}
    = 
      \frac{
    \Theta_E^2(\frac12 \int^{2X}_{0_-+\infty_+}\vec\omega)
  }
  {\Theta_E^2(\frac12 \int_{0_-}^{\infty_+}\vec\omega)}
  \frac{E(0_-,\infty_+)}{E(X,0_-) E(X,\infty_+)} 
  .
\end{equation}

As in [loc.\ cit., \S A.7] we take the branched double cover $p\colon
\widehat C_{\vec{U},l}\to C_{\vec{U},l}$ given by $W\mapsto X =
W^2$. It is given by
\begin{equation}\label{eq:cover}
  \begin{split}
  & Y^2 = P(W^2)^2 - 4 (-W^2)^{r+l}\Lambda^{2r}
\\
  =\; & \left( P(W^2) - 2(\sqrt{-1} W)^{r+l}\Lambda^r\right)
    \left( P(W^2) + 2(\sqrt{-1} W)^{r+l}\Lambda^r\right).
  \end{split}
\end{equation}
\begin{NB}
  Note that the formula \cite[(A.28)]{GNY2} was wrong.
\end{NB}

We consider the Szeg\"o kernel for $\widehat C_{\vec{U},l}$ given by
\begin{equation*}
  \Psi_{\widehat E}(W_1,W_2) = 
  \frac{\widehat\Theta_{\widehat E}(\int_{W_1}^{W_2} \vechatom)}
  {\widehat\Theta_{\widehat E}(0) E(W_1,W_2)},
\end{equation*}
where $\widehat\Theta_{\widehat E}$ is the Riemann theta function for
the curve $\widehat C_{\vec{U},l}$.
As in [loc.\ cit., the first two displayed formulas in \S A.6] we have
\begin{equation*}
  \Psi_{\widehat E}(W,\infty_+)^2
  = - \frac{Y - P(W^2)}{2 Y} dW \,
  \left.d\left(\frac1{W_2}\right)\right|_{W_2=\infty_+}.
\end{equation*}
\begin{NB}
  Note that $Y \sim -P(W)$ at $W\to \infty_+$, as $w = 0$.
\end{NB}%
\begin{NB2}
  The referee asks the branch of $\sqrt[4]{\psi_E(X)}$ in
  \cite[A.6]{GNY2}. But in the second line of the first displayed
  formula in \cite[A.6]{GNY2} (and \cite[p.13 (17)]{Fay}), we only
  have the ambiguity of the sign. We think that it is not appropriate
  to explain the detail in this paper, as it is about a formula in
  other papers. 2010/04/02
\end{NB2}%
From the defining equation \eqref{eq:cover} of the double cover,
\begin{NB2}
  I put the equation number for the defining equation. 2010/04/02
\end{NB2}%
this has zero of order $2(r+l)$ at
$0_+$ and of order $2(r-l-1)$ at $\infty_-$.
\begin{NB}
  \begin{equation*}
    \frac{Y-P(W^2)}{Y} = - \frac{4(-W^2)^{r+l}\Lambda^{2r}}{Y(Y + P(W^2))}
  \end{equation*}
and $Y = P(W^2) = (-1)^r$ at $0_+$. We also have
\begin{equation*}
  \frac{Y-P(W^2)}{Y} dW = - \frac{4(-W^2)^{r+l+1}\Lambda^{2r}}{Y(Y + P(W^2))}
  \frac{dW}{W^2}
\end{equation*}
and $Y\sim P(W^2) \sim W^{2r}$ at $\infty_-$.
\end{NB}%
Therefore
\begin{equation*}
   \operatorname{div}\widehat\Theta_{\widehat E}(W - \infty_+)
   = (r+l)\cdot 0_+ + (r - l - 1) \cdot \infty_-.
\end{equation*}
(See \cite[pp.16, 17]{Fay} for basic properties of $E(W_1,W_2)$.)
\begin{NB2}
  Added according to the referee's request. 2010/04/01
\end{NB2}%

If we identify the half-integer characteristic $\widehat E$ with a
vector in $\C^{2r-1}$ so that $\widehat{\Theta}_{\widehat E}(\xi) =
\widehat{\Theta}(\xi-\widehat{E})$, we have
\begin{equation}\label{eq:hate}
  \widehat E = 
  (r+l)\cdot 0_+ + (r - l - 1) \cdot \infty_- - \infty_+ - \widehat\Delta,
\end{equation}
where $\widehat\Delta$ is the Riemann's divisor class
(\cite[Th.~1.1]{Fay}).

By [loc.\ cit., Lem.~A.32] we have
\begin{equation}\label{eq:A32}
  \widehat E = p^* E - [0, c_*, 0],
\end{equation}
where $p^*\colon J_0(C_{\vec{U},l})\to J_0(\widehat C_{\vec{U},l})$.

On the other hand, we have
\begin{equation}\label{eq:hattheta}
  \widehat\Delta - p^*\Delta
  = 0_- + \infty_-
  + p^*\left(\frac12 \int_{0_-}^{\infty_+} \vec{\omega}\right)
  + [0, c_*, 0]
\end{equation}
by \cite[Prop.~5.3]{Fay}.

From (\ref{eq:hate},\ref{eq:A32},\ref{eq:hattheta}) we get
\begin{equation*}
  E = \frac{r+l-1}2\cdot 0_+ + \frac{r-l-1}2\cdot \infty_- - 0_- 
  - \frac12 \int_{0_-}^{\infty_+}\vec\omega - \Delta,
\end{equation*}
where we have used $0_+ - \infty_- = \infty_+ - 0_-$.
Therefore $\Theta_E(\frac12 \int^{2X}_{0_- + \infty_+}\vec\omega)$ has
zero of order $(r+l-1)/2$ at $0_+$, and of order $(r-l-1)/2$ at
$\infty_-$ again by \cite[Th.~1.1]{Fay}.

On the other hand, the left hand side of \eqref{eq:want} has
zero of order $r+l-1$ at $0_+$, and of order $r-l-1$ at $\infty_-$.
\begin{NB}
We have 
\begin{equation*}
  P(X) - Y = \frac{4(-X)^{r+l}\Lambda^{2r}}{Y+P(X)}.
\end{equation*}
At $0_+$, we have $Y = P(0) = (-1)^r \neq 0$. Therefore
\begin{equation*}
  \frac{(P(X) - Y)dX}{2XY}
  \sim (-X)^{r+l-1}\Lambda^{2r} dX
\end{equation*}
has zero of order $r+l-1$.
At $\infty_-$, we have $Y \sim P(X) \sim X^r$. Hence
  \begin{equation*}
    \frac{(P(X) - Y)dX}{2XY}
    \sim (-X)^{-r+l+1}\Lambda^{2r} \frac{dX}{X^2}
  \end{equation*}
has zero of order $r-l-1$.
\end{NB}%
Both sides of \eqref{eq:want} have poles at $0_-$, $\infty_+$ with residues
$-1$, $1$ respectively.
\begin{NB}
  We have $E(X,X') \sim (X'-X)/\sqrt{dX}\sqrt{dX'}$. Therefore
$E(X,0_-) \sim -X$.
\end{NB}%
(See \cite[Cor.~2.11]{Fay}.)
\begin{NB2}
  Added according to the referee's suggestion. 2010/04/01
\end{NB2}%
Therefore we have \eqref{eq:want}.
\begin{NB}
  The ratio of both sides of \eqref{eq:want} has no zeroes or
  poles. Therefore it must be constant. But the constant must be $1$
  since the residue matches.
\end{NB}

Although it is not necessary, let us also sketch how to prove
\eqref{eq:contact22} for more general $d$.

We first assume $r+l$ is even and generalize \eqref{eq:contact3} as
\begin{equation}\label{eq:contact4}
    0 = 
    \frac{d}{2r}\pd{U_p}{\log\Lambda}
\\
  + \frac1{2\pi\sqrt{-1}}\sum_i 
    \left.\frac{\partial\log\Theta_E}{\partial \xi^i}\right|_{
      \vec{\xi} = - \frac{d}{4\pi\sqrt{-1}r}
      \frac{\partial^2F_0}{\partial \log\Lambda\partial\vec{a}}}
    \pd{U_p}{a^i}
\end{equation}
for $|d| \le (r-l)/2$. This equation is nothing but
\eqref{eq:contact22} with $d - (r+l)/2$ replaced by $d$. In terms of
the $d$ in \eqref{eq:contact22}, the condition is $l\le d\le r$. This
is exactly one under which we have proved \eqref{eq:contact22} for all
$p$ from the vanishing theorem.

To show \eqref{eq:contact4} we use [loc.\ cit., the first displayed
formula in \S A.6.2], which is \cite[Cor.~2.19 (43)]{Fay}.
\begin{NB2}
  The reference for \cite{Fay} corrected. 2010/04/01
\end{NB2}%
We replace
$d$ by $d+1$ and take $x_0 = X$, $y_0 = X'$, $x_1 = \cdots = x_d =
0_-$, $y_1 = \cdots = y_d = \infty_+$. We then get
\begin{multline*}
  \frac{\Theta_E\left(\int^{d \infty_+ + X'}_{d 0_- + X} \vec{\omega}\right)}
   {\Theta_E(0)E(X,X')}
   \left( \frac{E(X,0_-) E(\infty_+,X')}
   {E(X,\infty_+) E(0_-,X')}\right)^d\!
  \left(E(0_-,\infty_+) \left.\sqrt{dX_1}\right|_{\scriptscriptstyle X_1=0}
    \left.\frac{\sqrt{dX_2}}{X_2}\right|_{\scriptscriptstyle X_2=\infty}\right)^{-d^2}
\\
  = 
  \Psi_E(X,X')
  - \frac{\Psi_E(0_-,X')\Psi_E(X,\infty_+)}{\Psi_E(0_-,\infty_+)}
  \frac{1 - (X/X')^d}{1 - (X/X')}.
\end{multline*}
\begin{NB2}
  The sign is corrected according to the referee's comment. 2010/04/08
\end{NB2}%
This is proved exactly as in [loc.\ cit., \S A.6.2], so the detail is
omitted.
\begin{NB2}
  I add a comment. 2010/04/08
\end{NB2}%
We multiply both sides $E(X,X')$, differentiate with respect to $X'$
and set $X'=X$.
We get
\begin{equation}\label{eq:contact5}
  \sum_{\alpha=2}^r 
  \pd{\log \Theta_E}{\xi_\alpha}(d \int^{\infty_+}_{0_-}\vec{\omega})
  \omega_\alpha(X)
  + d\times \omega_{\infty_+ - 0_-}(X)
  = d \frac{\Psi_E(X,0_-)\Psi_E(X,\infty_+)}{\Psi_E(0_-,\infty_+)}
\end{equation}
as in \cite[IIIb \S3, p.226]{Mumford}. When $d=1$, this is nothing but
[loc.\ cit., (A.24)]. From the argument in [loc.\ cit., \S A.6.1] we
get \eqref{eq:contact4}.

\begin{NB}
Let
\begin{equation*}
  \Psi_E(X_1,X_2)
  \defeq \frac{\Theta_E(\int^{X_2}_{X_1}\vec{\omega}|\tau)}
  {\Theta_E(0)E(X_1,X_2)}.
\end{equation*}
This is called the Szeg\"o kernel.

By \cite[Cor.~2.19 (43)]{Fay} we have
\begin{equation*}
\begin{split}
  & \frac{\Theta_E(\sum_{i=1}^d y_i - \sum_{i=1}^d x_i)}
   {\Theta_E(0)}
   \frac{\prod_{i<j} E(x_i,x_j) E(y_j,y_i)}{\prod_{i,j} E(x_i,y_j)}
   = 
   \det\left(\frac{\Theta_E(y_j - x_i)}{\Theta_E(0) E(x_i,x_j)}\right)
  =
  \det\left(\Psi_E(x_i,y_j)\right).
\end{split}
\end{equation*}
Let us replace $d$ by $d+1$ and take $\{ x_i\}_{i=0}^d$, $\{
y_i\}_{i=0}^d$. Then we take the limit of this equation when all $x_i$
(resp.\ $y_j$) goes to $0_-$ (resp.\ $\infty_+$) for $i=1,\dots,d$
and set $X=x_0$, $X'=y_0$.
As 
\(
   E(x_i,x_j) = \frac{(x_i - x_j)}{\sqrt{dx_i}\sqrt{dx_j}}
   \left(1 + O(x_i - x_j)^2\right),
\)
we have
\begin{equation*}
  \begin{split}
   & \pm \frac{\det\left(\Psi_E(x_i,y_j)\right)}
   {\prod_{0<i<j} E(x_i,x_j) E(y_j,y_i)}
   \left(\left.\sqrt{dx}\right|_{x=0}\left.\frac{\sqrt{dy}}y\right|_{y=\infty}
     \right)^{-d(d-1)}
   \to
\\
   & 
         \left|
     \begin{array}{c|cccc}
       \Psi_E(X,X') & \Psi_E(X,\infty_+) & \partial_y \Psi_E(X,\infty_+)
       & \cdots & \frac1{(d-1)!} \partial_y^{d-1} \Psi_E(X,\infty_+)
      \\
       \hline
       \Psi_E(0_-,X') & & & &
       \\
      \partial_x \Psi_E(0_-,X') & & & &
        \\
        \vdots & & \multicolumn{3}{c}{
          \left( \left.\frac1{i! j!}\partial_x^i \partial_y^j
              (\Psi_E)(x,y)\right|_{\substack{x=0_-\\ y=\infty_+}}
          \right)_{0\le i,j\le d-1}}
        \\
      \frac1{(d-1)!} \partial_x^{d-1} \Psi_E(0_-,X') & & & &
       \end{array}
       \right|.
  \end{split}
\end{equation*}
Therefore
\begin{equation}\label{eq:thetaid}
  \begin{split}
  & \pm\frac{\Theta_E\left(\int^{d \infty_+ + X'}_{d 0_- + X} \vec{\omega}\right)}
   {\Theta_E(0)E(X,X')}
   \left( \frac{E(X,0_-) E(\infty_+,X')}
   {E(X,\infty_+) E(0_-,X')}\right)^d
  \frac1{E(0_-,\infty_+)^{d^2}}
     \left(\left.\sqrt{dx}\right|_{x=0}\left.\frac{\sqrt{dy}}y\right|_{y=\infty}
     \right)^{-d(d-1)}
\\
  =\;&      
   \left|
     \begin{array}{c|cccc}
       \Psi_E(X,X') & \Psi_E(X,\infty_+) & \partial_y \Psi_E(X,\infty_+)
       & \cdots & \frac1{(d-1)!} \partial_y^{d-1} \Psi_E(X,\infty_+)
      \\
       \hline
       \Psi_E(0_-,X') & & & &
       \\
      \partial_x \Psi_E(0_-,X') & & & &
        \\
        \vdots & & \multicolumn{3}{c}{
          \left( \left.\frac1{i! j!}\partial_x^i \partial_y^j
              (\Psi_E)(x,y)\right|_{\substack{x=0_-\\ y=\infty_+}}
          \right)_{0\le i,j\le d-1}}
        \\
      \frac1{(d-1)!} \partial_x^{d-1} \Psi_E(0_-,X') & & & &
       \end{array}
       \right|.
  \end{split}
\end{equation}

Recall that the Szeg\"o kernel of the hyperelliptic curve is
explicitly given by
\begin{equation*}
\begin{split}
\Psi_E(X_1,X_2)
  &= \frac{\Theta_E(\int^{X_2}_{X_1}\vec{\omega}|\tau)}
  {\Theta_E(0)E(X_1,X_2)}
  = \frac12 \left(
     \sqrt[4]{\frac{\psi_E(X_1)}{\psi_E(X_2)}}
     +
     \sqrt[4]{\frac{\psi_E(X_2)}{\psi_E(X_1)}}
     \right)
     \frac{\sqrt{dX_1 dX_2}}{X_2-X_1}
\\
 &= 
\frac{Y_2 \prod (X_1 - X_\alpha^+) + Y_1 \prod (X_2 - X_\alpha^+)}
   {2(X_2 - X_1)}
   \sqrt{\frac{dX_1 dX_2}{Y_1 Y_2 \prod(X_1 - X_\alpha^+)(X_2 - X_\alpha^+)}},
\end{split}
\end{equation*}
where $E$ is the prime form and
\begin{equation*}
  \psi_E(X) = \frac{\prod (X - X_\alpha^+)}{\prod (X - X_\alpha^-)}
  = \frac{P(X) - 2(-\sqrt{-1}\bbeta\Lambda)^rX^{(r+m)/2}}
  {P(X) + 2(-\sqrt{-1}\bbeta\Lambda)^rX^{(r+m)/2}}.
\end{equation*}
See \cite[p.12 Example]{Fay}.

And we have \cite[the third displayed formula in \S A.6.2]{GNY2}
\begin{equation*}
  \psi_E(X) = 
   \begin{cases}
   1 + O(X^{(r+m)/2}) & \text{as $X\to 0$},
   \\
   1 + O(X^{-(r-m)/2}) & \text{as $X\to \infty$}.
   \end{cases}
\end{equation*}
This implies
\begin{equation*}
  \left. \frac1{i! j!}\partial_x^i \partial_y^j
              (\Psi_E)(x,y)\right|_{\substack{x=0_-\\ y=\infty_+}}
            = \delta_{ij} \left.\sqrt{dx}\right|_{x=0}
            \left.\frac{\sqrt{dy}}y\right|_{y=\infty}
\end{equation*}
if $d \le \max(r+l,r-l)/2$, as in \cite[(A.26)]{GNY2}.
\begin{NB2}
Let us check an identity in \cite[(A.26)]{GNY2} (after setting
$Y_j = 1/y_j$ for coordinates around $\infty_+$):
\begin{equation*}
  \begin{split}
    & \left.\prod_{j=1}^d Y_j^{-1} \det\left(\frac1{\frac1{Y_j} - x_i}
      \right)\frac1{\prod_{i<j} E(x_i,x_j)
        E(Y_j,Y_i)}\right|_{\substack{ x_i = 0_- \\ Y_j = 1/\infty_+}}
    \\
    =\; & \left.\det\left(\frac1{1- x_i{Y_j} }
      \right)\frac1{\prod_{i<j} E(x_i,x_j)
        E(Y_j,Y_i)}\right|_{\substack{ x_i = 0_- \\ Y_j = 1/\infty_+}}
    \\
    =\; & (-1)^{d(d-1)/2}\det\left[\left.
        \frac1{i!j!}\partial_x^i \partial_Y^j (\frac1{1 - xY})_{0\le
          i,j\le d-1}\right|_{x=Y=0}\right] \left(
      \left.dx\right|_{x=0} \left.dY\right|_{Y=0}\right)^{d(d-1)/2}
\\
    =\; &
    \left( \left.dx\right|_{x=0}
      \left.\frac{dx}{x^2}\right|_{x=\infty}\right)^{d(d-1)/2}
  \end{split}
\end{equation*}
It seems the sign was correct. Anyway $\sqrt{dy_i}$ has the sign
ambiguity. But we can check the sign is $+$ by considering the limit
$\bbeta\to 0$.
\end{NB2}%
We also have
\begin{equation*}
  \begin{split}
  & \frac1{i!} \partial_y^i \Psi_E(X,\infty_+)
  = \frac12 \left(\sqrt[4]{\psi_E(X)} + \frac1{\sqrt[4]{\psi_E(X)}}\right)
    \sqrt{dX}\left.\frac{\sqrt{dy}}y\right|_{y=\infty}
    X^i
  = \Psi_E(X,\infty_+) X^i,
\\
  & \frac1{i!} \partial_x^i \Psi_E(0_-,X')
  = \frac12 \left(\sqrt[4]{\psi_E(X')} + \frac1{\sqrt[4]{\psi_E(X')}}\right)
    \left.\sqrt{dx}\right|_{x=0}\sqrt{dX'}
    \frac1{(X')^{i+1}}
  = \Psi_E(0_-,X') \frac1{(X')^i}
  \end{split}
\end{equation*}
if $d \le \min(r+l,r-l)/2$. Plugging into \eqref{eq:thetaid}, we get
\begin{equation*}
  \begin{split}
  & \pm \frac{\Theta_E\left(\int^{d \infty_+ + X'}_{d 0_- + X} \vec{\omega}\right)}
   {\Theta_E(0)E(X,X')}
   \left( \frac{E(X,0_-) E(\infty_+,X')}
   {E(X,\infty_+) E(0_-,X')}\right)^d
  \left(E(0_-,\infty_+) \left.\sqrt{dx}\right|_{x=0}
    \left.\frac{\sqrt{dy}}y\right|_{y=\infty}\right)^{-d^2}
\\
  =\; &
  \begin{NB2}
  \Psi_E(X,X') - \sum_{i=0}^{d-1} \frac1{i!}\frac1{i!}
  \frac{\partial^i_x \Psi_E(0_-,X')\partial^i_y\Psi_E(X,\infty_+)}
  {\Psi_E(0_-,\infty_+)}
  =
  \end{NB2}%
  \Psi_E(X,X')
  - \frac{\Psi_E(0_-,X')\Psi_E(X,\infty_+)}{\Psi_E(0_-,\infty_+)}
  \frac{1 - (X/X')^d}{1 - (X/X')}.
  \end{split}
\end{equation*}
\begin{NB2}
  The sign is corrected. 2010/04/08
\end{NB2}%
Multiplying both sides by $E(X,X')$ and taking the limit $X'\to X$, we
find the sign must be $+$ from \cite[(A.26)]{GNY2}, i.e., 
\begin{equation*}
  \left(E(0_-,\infty_+) \left.\sqrt{dx}\right|_{x=0}
    \left.\frac{\sqrt{dy}}y\right|_{y=\infty}\right)^{d^2}
  = \frac{\Theta_E\left(d\int^{\infty_+}_{0_-} \vec{\omega}\right)}
   {\Theta_E(0)}.
\end{equation*}
\begin{NB2}
  We check the sign of \cite[(A.26)]{GNY2} as above.
\end{NB2}%

We now multiply both sides by $E(X,X')$, differentiate with respect to
$X'$ and set $X'=X$. We get
\begin{equation*}
  \sum_{\alpha=2}^r 
  \pd{\log \Theta_E}{\xi_\alpha}(d \int^{\infty_+}_{0_-}\vec{\omega})
  \omega_\alpha(X)
  + d\times \omega_{\infty_+ - 0_-}(X)
  = d \frac{\Psi_E(X,0_-)\Psi_E(X,\infty_+)}{\Psi_E(0_-,\infty_+)}.
\end{equation*}
Here we have used
\begin{equation*}
  \begin{split}
   &\left( \frac{E(X,0_-)}{E(X,\infty_+)} \right)^d
       \partial_X \left( \frac{E(\infty_+,X)}{E(0_-,X)}\right)^d
\\
   =\; & 
   d \left( \frac{E(X,0_-)}{E(X,\infty_+)} \right)^d
   \left( \frac{E(\infty_+,X)}{E(0_-,X)}\right)^d
   \left( \frac{E(\infty_+,X)}{E(0_-,X)}\right)^{-1}
      \partial_X \left( \frac{E(\infty_+,X)}{E(0_-,X)}\right)
   = 
   d \times \omega_{\infty_+ - 0_-}(X)
  \end{split}
\end{equation*}
(see \cite[IIIb, Lemma, p.226]{Mumford} for the last equality) and
\begin{equation*}
    \left.\partial_{X'} \frac{E(X,X')}{E(0_-,X')}\right|_{X'=X}
    = \frac1{E(X,0_-)}
\end{equation*}
(see \cite[IIIb, Lemma, p.225]{Mumford}).
This identity is exactly \cite[(A.24)]{GNY2} for $d=1$.
\end{NB}

We next consider the case $r+l$ is odd.
We take the branched double cover $p\colon \widehat C_{\vec{U},l}\to
C_{\vec{U},l}$ given by $W\mapsto X = W^2$ as before. Then $r$, $l$
become $2r$, $2l$ for $\widehat C_{\vec{U},l}$, and hence we have
\eqref{eq:contact5} for $\widehat C_{\vec{U},l}$, with $d$ replaced by
$2d$:
\begin{multline*}
  \sum_{\alpha=2}^r 
  \pd{\log \widehat\Theta_{\widehat E}}{\widehat\xi_\alpha}
  (2d \int^{\infty_+}_{0_-}\vechatom)
  \widehat\omega_\alpha(W)
  + \sum_{\alpha=2}^r 
  \pd{\log \widehat\Theta_{\widehat E}}{\widehat\xi'_\alpha}
  (2d \int^{\infty_+}_{0_-}\vechatom)
  \widehat\omega'_\alpha(W)
\\
  + \pd{\log \widehat\Theta_{\widehat E}}{\widehat\xi'_*}
  (2d \int^{\infty_+}_{0_-}\vechatom)
  \widehat\omega_*(W)
  + 2d\times \widehat\omega_{\infty_+ - 0_-}(W)
  = 2d \frac{\widehat\Psi_{\widehat E}(W,0_-)\widehat\Psi_{\widehat E}(W,\infty_+)}
  {\widehat\Psi_{\widehat E}(0_-,\infty_+)},
\end{multline*}
where we take cycles $A_\alpha$, $B_\alpha$, $A_*$, $B_*$,
$A'_\alpha$, $B'_\alpha$ as in [loc.\ cit., \S A.7] and the
corresponding coordinates $\widehat\xi_\alpha$, $\widehat \xi_*$,
$\widehat\xi'_\alpha$ on $J_0(\widehat C_{\vec{U},l})$.
This holds if $|d|\le (r-l)/2$ as above.
\begin{NB}
Therefore $d = 0$ if $l = r-1$. But then it becomes a trivial identity.

When $l = r-1$, we need to use the more precise limiting behavior
of $\psi_E$, i.e., the coefficient of $X^{-(r-l)/2}$.
\end{NB}%
The right hand side is
\begin{equation*}
  2d \frac{(P(W^2) - Y) dW}{2 WY} 
  = d p^*\left(\frac{(P(X) - Y) dX}{2XY}\right)
\end{equation*}
as [loc.\ cit., the second displayed equation in \S A.6.1]. We also have
\begin{equation*}
  \widehat\omega_{\infty_+-0_-}(W) = \frac12 p^* \omega_{\infty_+-0_-}(X)
\end{equation*}
by definition.

On the other hand, we rewrite the theta function
$\widehat\Theta_{\widehat E}$ by $\Theta_E$ by using [loc.\ cit.,
(A.29) and the second displayed formula in p.1105].
\begin{NB}
Then \cite[p.91 (102)]{Fay} says that there exists a unique half-period
$\left[0,c_*,0\right]\in J_0(\widehat C_{\vec{U}})$ such that
\begin{equation}\label{eq:ThetaCover}
  k_0 \defeq
  \frac{\widehat \Theta_{\left[c,c_*,-c\right]}(\pi^*\xi)}
    {\Theta_c(\xi+\xi_0)\Theta_c(\xi-\xi_0)} 
\end{equation}
is independent of $\xi\in\C^{r-1}$ and a half-integer characteristic $c$
for the curve $C_{\vec{U}}$.

*****************************************

We have
\begin{equation*}
   2d \int^{\infty_+}_{0_-} \vechatom
   = \left[ d \int^{\infty_+}_{0_-} \vec\omega,
     d, 
     -d \int^{\infty_+}_{0_-} \vec\omega\right]
\end{equation*}
and
\begin{equation*}
  \xi_0 = \frac12 \int_{0_-}^{\infty_+} \vec\omega.
\end{equation*}

\end{NB}%
We then get
\begin{multline*}
  \frac12\sum_{\alpha=2}^r 
  \left\{
  \pd{\log \Theta_E}{\xi_\alpha}
   ((d+\frac12) \int^{\infty_+}_{0_-}\vec{\omega})
  +
  \pd{\log \Theta_E}{\xi_\alpha}
   ((d-\frac12) \int^{\infty_+}_{0_-}\vec{\omega})
    \right\}
  \omega_\alpha(X)
\\
  + d\times \omega_{\infty_+ - 0_-}(X)
  = d \frac{(P(X) - Y) dX}{2XY}.
\end{multline*}
From this we get
\begin{multline*}
    0 = 
    \frac{d}{2r}\pd{U_p}{\log\Lambda}
\\
  + \frac1{2\pi\sqrt{-1}}\sum_i \frac12 \left(
    \left.\frac{\partial\log\Theta_E}{\partial \xi^i}\right|_{
      \vec{\xi} = - \frac{d+1/2}{4\pi\sqrt{-1}r}
      \frac{\partial^2F_0}{\partial \log\Lambda\partial\vec{a}}}
  +  \left.\frac{\partial\log\Theta_E}{\partial \xi^i}\right|_{
      \vec{\xi} = - \frac{d-1/2}{4\pi\sqrt{-1}r}
      \frac{\partial^2F_0}{\partial \log\Lambda\partial\vec{a}}}
    \right)
    \pd{U_p}{a^i}
\end{multline*}
as in [loc.\ cit., \S A.6.1]. This is nothing but the sum of
\eqref{eq:contact22} for $d$ replaced by $d + (r+l+1)/2$
and $d + (r+l-1)/2$.
\begin{NB}
  Suppose that we want to have the above for $d=1$ or $-1$, i.e.,
  \eqref{eq:contact2} with $d = (r+l\pm 3)/2$ and $d = (r+l\pm 1)/2$.

  If $p = - (r+l - 1)/2$ (the minimum value), $|p|\le d$ holds only
  the $\pm = +$ case. The other inequality $d\le r$ means
  $l\le r - 3$. Thus $l = r-1$ is excluded.

  If $p = (r - l - 1)/2$ (the maximum value), $d \le r - p = (r + l +
  1)/2$ holds only if $\pm = -$. The other inequality $0 \le d$ means
  $r+l \ge 3$. This is always true, as we assume $r+l$ is odd and
  $l\ge 0$, $r\ge 2$.

  Thus unless $l = r-1$, the above equation is satisfied for
  ${\partial F_0}/{\partial \tau_p}$.
\end{NB}%
Since we have already proved \eqref{eq:contact22} for $(r+l-1)/2$, we have
\eqref{eq:contact22} for $l\le d\le r$.
\begin{NB}
  Take $d = -(r-l-1)/2$ in the above formula. Then
$d + (r+l-1)/2 = l$. The lower bound comes from here. Take
$d = (r-l-1)/2$ in the above formula. Then 
$d + (r+l+1)/2 = r$. The upper bound comes from here.
\end{NB}

\section{Quiver description}\label{sec:quiver}

In this section we review the result of \cite{perv}, rephrase that of
\cite{perv2} in the quiver description, and add a few things on
$\Ext$-groups.

\subsection{Moduli spaces of $m$-stable sheaves}\label{subsec:mstable}
We take vector spaces $V_0$, $V_1$, $W$ with
\begin{equation*}
  r = \dim W, \quad (c_1,[C]) = \dim V_0 - \dim V_1, \quad
  (\ch_2,[\bp]) = -\frac12(\dim V_0+\dim V_1).
\end{equation*}

We consider following datum $X = (B_1,B_2,d,i,j)$
\begin{itemize}
\item $B_1, B_2\in \Hom(V_1, V_0)$, $d\in \Hom(V_0,V_1)$, $i\in
  \Hom(W,V_0)$, $j\in \Hom(V_1,W)$,
\begin{equation*}
\xymatrix{
V_0 \ar@<-1ex>[rr]_{d} && \ar@<-1ex>[ll]_{B_1,B_2} \ar[ll] \ar[ld]^j V_1 \\
& \ar[lu]^i W & \\
}
\end{equation*}
\item $\mu(B_1,B_2,d,i,j) = B_1 d B_2 - B_2 d B_1 + ij = 0$.
\end{itemize}

Let $Q \defeq \mu^{-1}(0)$ be the subscheme of the vector space 
\(
  \Hom(V_1, V_0)^{\oplus 2}\oplus \Hom(V_0,V_1)
  \oplus \Hom(W,V_0) \oplus \Hom(V_1,W)
\)
defined by the equation $\mu=0$.
It is acted by $G \defeq \GL(V_0)\times \GL(V_1)$
\begin{equation*}
   g\cdot (B_1,B_2,d,i,j) =
   (g_0 B_1 g_1^{-1}, g_0 B_2 g_1^{-1}, g_1 d g_0^{-1}, g_0 i, j g_1^{-1}).
\end{equation*}

Let $\zeta = (\zeta_0,\zeta_1)\in\Q^2$.
\begin{Definition}
We say $X = (B_1,B_2,d,i,j)$ is {\it $\zeta$-semistable\/} if
\begin{enumerate}
\item for subspaces $S_0\subset V_0$, $S_1\subset V_1$ such that
$B_\alpha(S_1)\subset S_0$ ($\alpha=1,2$),
$d(S_0)\subset S_1$, $\Ker j\supset S_1$, we have 
\(
    \zeta_0 \dim S_0 + \zeta_1 \dim S_1 \le 0.
\)
\item for subspaces $T_0\subset V_0$,
$T_1\subset V_1$ such that $B_\alpha(T_1)\subset T_0$ ($\alpha=1,2$),
$d(T_0)\subset T_1$, $\Ima i\subset T_0$, we have 
\( 
   \zeta_0 \codim T_0 + \zeta_1 \codim T_1 \ge 0.
\)
\end{enumerate}
We say $X$ is {\it $\zeta$-stable\/} if the inequalities are strict
unless $(S_0,S_1) = (0,0)$ and $(T_0,T_1) = (V_0,V_1)$ respectively.

We say $X_1$, $X_2$ are {\it $S$-equivalent\/} when the closures of
orbits intersect in the $\zeta$-semistable locus of $Q$.
\end{Definition}

By a standard argument, we can see that these come from quotients of
$Q$ by $G$ in the geometric invariant theory. We only explain the
result, see \cite{King-mod} for detail.

Let $\chi\colon G\to \C^*$ be the character given by $\chi(g) = \det
g_0^{-\zeta_0} \det g_1^{-\zeta_1}$, where we assume
$(\zeta_0,\zeta_1)\in \Z^2$ by multiplying a positive integer if necessary.
\begin{NB2}
  Added. 2010/04/05
\end{NB2}%
We have a lift of $G$ action
on the trivial line bundle $Q\times\C$ given by
\(
   (B_1,B_2,d,i,j,z) = (g\cdot(B_1,B_2,d,i,j), \chi(g) z).
\)
Let $L_\zeta$ denote the corresponding $G$-equivariant line
bundle. Then we can consider the GIT quotient
\begin{equation*}
   \bM_\zeta = \operatorname{Proj}\left(\bigoplus_{n \ge 0} A(Q)^{\chi,n}\right),
\end{equation*}
where $A(Q)^{\chi,n}$ is the relative invariants in the coordinate
ring $A(Q)$ of $Q$: $\{ f\in A(Q)\mid f(g\cdot X) = \chi(g)^n f(X)\}$,
which is the space of invariant sections of $L_\zeta^{\otimes n}$.
Then $\bM_\zeta$ is the quotient of $\zeta$-semistable locus modulo
$S$-equivalence relation.
\begin{NB}
$X = (B_1,B_2,d,i,j)$ is called $\zeta$-semistable if
there is an equivariant section $s$ of $L_\zeta^{\otimes n}$ for some
$n > 0$ with $s(X)\neq 0$.
\end{NB}
It contains $\bM_\zeta^{\mathrm{s}}$ of the quotient of $\zeta$-stable
locus modulo the action of $G$.

We have a natural projective morphism $\widehat\pi\colon \bM_\zeta\to
\mu^{-1}(0)\dslash G$, where $\mu^{-1}(0)\dslash G$ is the affine
geometric invariant theory quotient of $\mu^{-1}(0)$ by $G$, i.e.,
$\bM_0$. By \cite[\S1.3]{perv} $\mu^{-1}(0)\dslash G$ is isomorphic to
$M_0$, the Uhlenbeck partial compactification on $\proj^2$.

Now the main result of \cite{perv} says
\begin{Theorem}\label{thm:quiver}
  Let $m\in\Z_{\ge 0}$. Suppose that $\zeta_0 < 0$, $0 > m\zeta_0 +
  (m+1)\zeta_1 \gg -1$. Then we have $\bM_{\zeta} =
  \bM_{\zeta}^{\mathrm{s}}$ and it is bijective to the set of
  isomorphism classes of $m$-stable framed sheaves on $\bp$.
\end{Theorem}

We use this result as the definition of the moduli scheme of
$m$-stable framed sheaves in this paper. Therefore $\bM_{\zeta}$ will
be denoted by $\bM^m$ (or $\bM^m(c)$ when we want to write the Chern
character of sheaves) hereafter.
It was proved in \cite[2.4]{perv} that (a) $d\mu$ is surjective and
(b) the action of $G$ is free on the $\zeta$-stable locus. Therefore
$\bM_{\zeta}$ is a smooth fine moduli scheme.

The construction is given as follows: We consider the complex
\begin{equation}\label{eq:cpx}
\begin{CD}
\begin{matrix}
   V_0 \otimes \shfO(C-\linf)
\\
   \oplus
\\ V_1\otimes\shfO(-\linf)
\end{matrix}
@>{\alpha}>>
\begin{matrix}
   \C^2\otimes V_0\otimes\shfO \\ \oplus \\ \C^2\otimes V_1\otimes\shfO \\
   \oplus \\ W\otimes\shfO
\end{matrix}
@>{\beta}>>
\begin{matrix}
  V_0\otimes \shfO(\linf)
\\
  \oplus
\\
  V_1\otimes\shfO(-C+\linf)
\end{matrix},
\end{CD}
\end{equation}
with
\begin{equation*}
   \alpha = 
   \begin{bmatrix}
     z & z_0 B_1           \\
     w & z_0 B_2           \\
     0 & z_1 - z_0 d B_1   \\
     0 & z_2 - z_0 d B_2   \\
     0 & z_0 j
   \end{bmatrix},
\qquad
   \beta =
   \begin{bmatrix}
    z_2 & - z_1 & B_2 z_0 & -B_1 z_0 & i z_0\\
    dw  & - d z & w       & -z       & 0
   \end{bmatrix}.
\end{equation*}
The equation $\mu(B_1,B_2,d,i,j) = B_1 d B_2 - B_2 d B_1 + ij = 0$ is
equivalent to $\beta\alpha = 0$. The stability condition ensures the
injectivity of $\alpha$ and the surjectivity of $\beta$.  Then the
sheaf corresponding to $(B_1,B_2,d,i,j)$ is defined by $E =
\Ker\beta/\Ima\alpha$. By the definition it is endowed with the framing
$E|_{\linf} \to W\otimes\shfO_\linf$. The $\zeta$-stability is identified with
the $m$-stability.

The inverse construction is given by
\begin{equation*}
   V_0 \defeq H^1(E(-\linf)), \quad V_1 \defeq H^1(E(C-\linf)),
\end{equation*}
and $B_1$, $B_2$, $d$, $i$, $j$ are homomorphisms between cohomology
groups induced from certain natural sections.

From this construction $V_0$, $V_1$ naturally define vector bundles
over $\bM^m$, which are denoted by $\Vcal_0$, $\Vcal_1$
respectively. The above $\alpha$, $\beta$ in \eqref{eq:cpx} are
interpreted as homomorphisms between vector bundles and the universal
family $\cE$ is given by $\Ker\beta/\Ima\alpha$.

When we prove that the $\zeta$-stability corresponds to the
$m$-stability in \defref{def:m-stable}, it is crucial to observe that
the sheaf $\shfO_C(-m-1)$ corresponds to the datum
$V_0 = \C^m$, $V_1 = \C^{m+1}$, $W = 0$, $d=0$ and
\begin{equation*}
  B_1 = \left[\begin{smallmatrix} 1_m & 0
    \end{smallmatrix}\right], \quad
  B_2 = \left[\begin{smallmatrix} 0 & 1_m
    \end{smallmatrix}\right],
\end{equation*}
where $1_m$ is the identity matrix of size $m$. We denote this datum
by $C_m$ as above.

\subsection{Tangent complex}
From the construction the tangent space is the middle cohomology group
of the complex
\begin{equation}\label{eq:tangentcpx}
\begin{CD}
\begin{matrix}
   \Hom(V_0,V_0)
\\
   \oplus
\\ \Hom(V_1,V_1)
\end{matrix}
@>{\iota}>>
\begin{matrix}
   \Hom(V_0,V_1)\\ \oplus \\ \C^2\otimes \Hom(V_1, V_0) \\ 
   \oplus \\ \Hom(W,V_0) \\ \oplus \\ \Hom(V_1,W)
\end{matrix}
@>{d\mu}>>
\begin{matrix}
  \Hom(V_1,V_0)
\end{matrix},
\end{CD}
\end{equation}
with
\begin{equation*}
\begin{split}
  \iota
   \begin{bmatrix}
     \xi_0 \\ \xi_1
   \end{bmatrix}
   = 
   \begin{bmatrix}
     d \xi_0 - \xi_1 d     \\
     B_1 \xi_1 - \xi_0 B_1 \\
     B_2 \xi_1 - \xi_0 B_2 \\
     \xi_0 i \\
      - j \xi_1
   \end{bmatrix},
   \qquad
   (d\mu)
   \begin{bmatrix}
    \widetilde d \\ \widetilde B_1 \\ \widetilde B_2 \\ \widetilde i \\
    \widetilde j
   \end{bmatrix}
   = 
   \begin{aligned}
   & B_1 d \widetilde B_2 + B_1 \widetilde d B_2 + \widetilde B_1 d
   B_2
\\
   & \quad
   - B_2 d \widetilde B_1 - B_2 \widetilde d B_1 - \widetilde B_2 d
   B_1
\\
   & \quad\quad
   + \widetilde i j + i\widetilde j,
   \end{aligned}
\end{split}
\end{equation*}
where $d\mu$ is the differential of $\mu$, and $\iota$ is the
differential of the group action. Remark that $d\mu$ is surjective and
$\iota$ is injective by the above remark if $X$ is $\zeta$-stable.

\subsection{A modified quiver}

We fix the vector space $W$ with $\dim W = r\neq 0$.
We define a new quiver with three vertexes $0$, $1$, $\infty$. We
write two arrows from $1$ to $0$ corresponding to the data $B_1$,
$B_2$, and one arrow from $0$ to $1$ corresponding to the data $d$.
Instead of writing one arrow from $\infty$ to $0$, we write
$r$-arrows. Similarly we write $r$-arrows from $1$ to $\infty$.  And
instead of putting $W$ at $\infty$, we replace it the one dimensional
space $\C$ on $\infty$. We denote it by $V_\infty$.
It means that instead of considering the homomorphism
$i$ from $W$ to $V_0$, we take $r$-homomorphisms
$i_1$,$i_2$,\dots, $i_{r}$ from $V_\infty$ to $V_0$ by taking a base of $W$.
(See Figure~\ref{fig:newquiver}.)
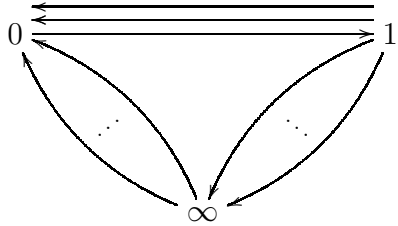
\begin{figure}[htbp]
  \centering
\begin{equation*}
\xymatrix{
0 \ar[rrrr] &&&& \ar@<-1ex>[llll] \ar@<-2ex>[llll] \ar@/^1pc/[lldd]_\ddots
\ar@<-1ex>[llll] \ar@/_1pc/[lldd]  1
\\
&&\\
&& \ar@/^1pc/[lluu]_\iddots
  \ar@/_1pc/[lluu]  \infty & \\
}
\end{equation*}
  \caption{modified quiver}
  \label{fig:newquiver}
\end{figure}

We consider the full subcategory of the abelian category of
representations of the new quiver with the relation, consisting of
representations such that $\dim V_\infty = 0$ or $1$. An object can be
considered as a representation of the original quiver with $\dim W =
0$, or $\dim W = r$, according to $\dim V_\infty = 0$ or $1$. Note
that we do not allow a representation of the original quiver with
$\dim W\neq r,0$.

It is also suitable to modify the stability condition for
representations for the new quiver.
Let $(\zeta_0,\zeta_1,\zeta_\infty)\in\Q^3$.
For a representation $X$ of the modified quiver, let us denote the
underlying vector spaces by $X_0$, $X_1$, $X_\infty$.
We define the {\it rank\/}, {\it degree\/} and {\it slope\/} by
\begin{equation}
  \begin{split}
    & \rank X \defeq \dim X_0+\dim X_1+\dim X_\infty, 
\\
    & \zeta\cdot\vdim X \defeq 
    \zeta_0 \dim X_0 + \zeta_1 \dim X_1 +
    \zeta_\infty \dim X_\infty,
\\    
    & \theta(X) \defeq \frac{\zeta\cdot\vdim X}{\rank X},
  \end{split}
\end{equation}
where $\theta(X)$ is defined only when $X\neq 0$.
We only consider the case $\dim X_\infty = 0$ or $1$ as before. We
say $X$ is {\it $\theta$-semistable\/} if we have
\begin{equation*}
    \theta(S) \le \theta(X)
\end{equation*}
for any subrepresentation $0\neq S$ of $X$. We say $X$ is {\it
  $\theta$-stable\/} if the inequality is strict unless $S = X$. If
$\theta(X) = 0$, $\theta$-(semi)stability is equivalent to
$\zeta$-(semi)stability.
In fact, if a subrepresentation $S$ has $S_\infty = 0$, then
$\theta(S) \le 0$ is equivalent to $\zeta_0\dim S_0 + \zeta_1\dim
S_1\le 0$. If a subrepresentation $T$ has $T_\infty = \C$, then
$\theta(T)\le 0$ is equivalent to $\zeta_0\codim T_0 + \zeta_1\codim
T_1\ge 0$.

The $\theta$-stability is unchanged even if we add $c(1,1,1)$
($c\in\R$) to $(\zeta_0,\zeta_1,\zeta_\infty)$. Therefore once we fix
$\dim X_0$, $\dim X_1$, $\dim X_\infty$, we can always achieve the
condition $\theta(X) = 0$ without the changing stable objects.

\subsection{Wall-crossing}\label{subsec:wall-crossing}

Let us fix a wall $\{ \zeta \mid m\zeta_0 + (m+1)\zeta_1 = 0, \zeta_0
< 0\}$ and a parameter $\zeta^0 = (\zeta^0_0,\zeta^0_1)$ from the
wall. We take $\zeta^+$, $\zeta^-$ sufficiently close to $\zeta^0$
with
\begin{equation*}
    1 \gg m \zeta^+_0 + (m+1)\zeta^+_1 > 0, \qquad
    -1 \ll m \zeta^-_0 + (m+1)\zeta^-_1 < 0.
\end{equation*}
(See Figure~\ref{fig:zeta}.) Then $\zeta^-$ is nothing but the
parameter $\zeta$ appeared in \thmref{thm:quiver} corresponding to the
$m$-stability. On the other hand, we also know that $\zeta^+$
corresponds to the $(m+1)$-stability by the determination of the
chamber structure in \cite[\S2]{perv}.
\begin{figure}[htbp]
\def\JPicScale{.8}
  \centering
\ifx\JPicScale\undefined\def\JPicScale{1}\fi
\psset{unit=\JPicScale mm}
\psset{linewidth=0.3,dotsep=1,hatchwidth=0.3,hatchsep=1.5,shadowsize=1,dimen=middle}
\psset{dotsize=0.7 2.5,dotscale=1 1,fillcolor=black}
\psset{arrowsize=1 2,arrowlength=1,arrowinset=0.25,tbarsize=0.7 5,bracketlength=0.15,rbracketlength=0.15}
\begin{pspicture}(0,0)(95,65)
\psline[linewidth=0.5,arrowscale=2 2]{->}(20,10)(90,10)
\psline[linewidth=0.5,arrowscale=2 2]{->}(80,0)(80,60)
\psline[linewidth=0.4,linestyle=dotted](90,0)(40,50)
\rput(95,10){$\zeta_0$}
\rput(80,65){$\zeta_1$}
\rput(25,55){$\zeta_0+\zeta_1 = 0$}
\psline[linewidth=0.4,fillstyle=vlines](20,34)(90,6)
\rput(19,39){$m\zeta_0 + (m+1)\zeta_1 = 0$}
\rput{0}(50,22){\psellipse[fillstyle=solid](0,0)(1,-1)}
\rput{0}(49,19){\psellipse[fillstyle=solid](0,0)(1,-1)}
\rput{0}(51,25){\psellipse[fillstyle=solid](0,0)(1,-1)}
\rput(44,32){$\zeta^+$}
\rput(37,22){$\zeta^0$}
\rput(36,16){$\zeta^-$}
\psline[linewidth=0.2,fillstyle=solid]{->}(45,31)(50,26)
\psline[linewidth=0.2,fillstyle=solid]{->}(40,22)(49,22)
\psline[linewidth=0.2,fillstyle=solid]{->}(39,16)(48,19)
\end{pspicture}
\caption{wall-crossing}
\label{fig:zeta}
\end{figure}
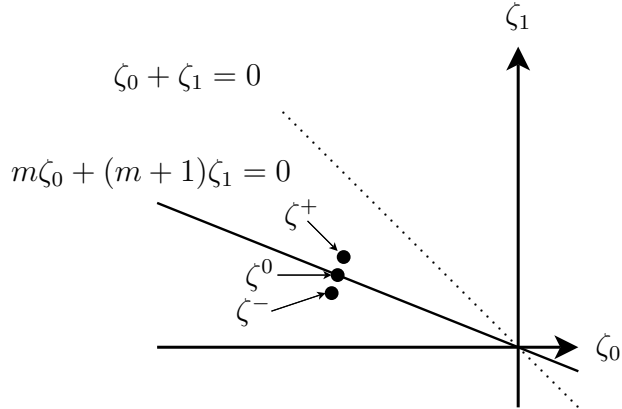

Let us define the scheme $\bM^{m,m+1}$ as the GIT quotient
$\bM_{\zeta^0}$ with respect to the $\zeta^0$-semistability.

From \cite{perv} it follows that $\bM^{m,m+1}$ is the $S$-equivalence
classes of $(m,m+1)$-semistable sheaves set-theoretically, where
\begin{NB2}
  An explanation of \cite{perv} is moved below. Also $(m,m+1)$-stable
  is corrected to $(m,m+1)$-{\it semistable}. 2010/04/03
\end{NB2}%
\begin{Definition}\label{def:m,m+1-stable}
  \textup{\bf (a)} A framed sheaf $(E,\Phi)$ on $\bp$ is called {\it
    $(m,m+1)$-semistable\/} if
  \begin{enumerate}
  \item $\Hom(E,\shfO_C(-m-2)) = 0$,
  \item $\Hom(\shfO_C(-m),E) = 0$, and
  \item $E$ is torsion free outside $C$.
  \end{enumerate}

  \textup{\bf (b)} A framed sheaf $(E,\Phi)$ on $\bp$ is called {\it
    $(m,m+1)$-stable\/} if it is either $\shfO_C(-m-1)$ (with the
  trivial framing) or both $m$-stable and $(m+1)$-stable,
  i.e., we have $\Hom(E,\shfO_C(-m)) = 0$ and $\Hom(\shfO_C(-m),E) =
  0$ instead of (1),(2).

  \textup{\bf (c)} A $(m,m+1)$-semistable framed sheaf $(E,\Phi)$ has
  a filtration $(0 = E^0 \subset E^1\subset \cdots\subset E^N = E)$
  such that $E^i/E^{i-1}$ is $(m,m+1)$-stable with the induced framing
  from $\Phi$.
  We say $(m,m+1)$-semistable framed sheaves $(E,\Phi)$ and
  $(E',\Phi')$ are {\it $S$-equivalent\/} if there exists an
  isomorphism from $\bigoplus_i E^i/E^{i-1}$ to $\bigoplus_j E^{\prime
    j}/E^{\prime j-1}$ respecting the framing in the $(m,m+1)$-stable
  factors.
\end{Definition}

The main result \cite[Th.~1.5]{perv} was stated for $m$-stable framed
sheaves, but it can be generalized to the case of $(m,m+1)$-stable
framed sheaves. It follows from Propositions~4.1,4.2 there that $X$ is
$\zeta^0$-semistable if and only if it satisfies the condition (S2) in
[loc.\ cit., Th.~1.1] and the condition corresponding to (a1-2)
above. Then the remaining arguments are the same.
\begin{NB2}
  The explanation expanded and moved here. 2010/04/03
\end{NB2}%

\begin{NB}
  I changed the definition of $(m,m+1)$-stability, and the following
  becomes meaningless. Feb.24 (HN)

During the study of the stability conditions, we also showed
\begin{Theorem}[\protect{\cite[Th.~2.13, Prop.~5.3]{perv}}]\label{thm:class}
  A sheaf $E$ of rank $0$ with the trivial framing is $(m,m+1)$-semistable
  if and only if $E = \shfO_C(-m-1)^{\oplus p}$ with $p\in\Z_{>0}$.
\end{Theorem}
\end{NB}

Since $\zeta^\pm$-stability implies $\zeta^0$-semistability, we have
natural morphisms $\xi_m\colon \bM^m\to \bM^{m,m+1}$, $\xi_m^+\colon
\bM^{m+1}\to \bM^{m,m+1}$. Thus $\bM^m$, $\bM^{m+1}$ and $\bM^{m,m+1}$
form the diagram \eqref{eq:flip}.

This definition of $\bM^{m,m+1}$ is different from what are given in
\cite{perv2} for ordinary moduli spaces without framing, but they are
the same at least set-theoretically thanks to the construction in
[loc.\ cit., \S\S3.6,3.7]. It is also possible to show directly that
$\xi_m$, $\xi_m^+$ have structures of stratified Grassmann bundles
described there.

From the definition it is clear that $\xi_m$, $\xi_m^+$ are compatible with
projective morphisms $\bM^m\to M_0$, $\bM^{m,m+1}\to M_0$, $\bM^{m+1}\to M_0$,
(all denoted by $\widehat\pi$ before).

\begin{NB}
Recall that we set
\(
   C_m \defeq \shfO_C(-m-1).
\)
\end{NB}

The change of the moduli spaces under the wall-crossing is described
as follows:

\begin{Proposition}[\protect{\cite[3.15]{perv2}}]\label{prop:extension}
\textup{(1)} If $(E^-,\Phi)$ is $m$-stable, we have an exact sequence
\begin{equation*}
   0 \to V\otimes C_m \to E^- \to E' \to 0
\end{equation*}
with $V = \Hom(C_m,E^-)$ such that
\begin{aenume}
\item $E'$ is $(m,m+1)$-stable, and
\item the induced homomorphism $V\to \Ext^1(E',C_m)$ is injective.
\end{aenume}
Conversely if $(E',\Phi)$ is $(m,m+1)$-stable and a subspace $V\subset
\Ext^1(E',C_m)$ is given, $(E^-,\Phi)$, defined by the above
exact sequence, is $m$-stable.

\textup{(2)} If $(E^+,\Phi)$ is $(m+1)$-stable, we have an exact sequence
\begin{equation*}
   0 \to  E' \to E^+ \to U^\vee\otimes C_m \to 0
\end{equation*}
with $U = \Hom(E^+,C_m)$ such that
\begin{aenume}
\item $E'$ is $(m,m+1)$-stable, and
\item the induced homomorphism $U^\vee\to \Ext^1(C_m,E')$ is
  injective.
\end{aenume}
Conversely if $(E',\Phi)$ is $(m,m+1)$-stable and a subspace
$U^\vee\subset \Ext^1(C_m,E')$ is given, $(E^+,\Phi)$,
defined by the above exact sequence, is $(m+1)$-stable.
\end{Proposition}

\subsection{Computation of $\protect{\Ext}$-groups}\label{subsec:Ext}

In this subsection, we continue to fix a wall $m\zeta_0 + (m+1)\zeta_1
= 0$, $\zeta_0 < 0$.

Take $X = (B_1,B_2,d,i,j)$ defined on $V_0$, $V_1$, $W$ such that
$\mu(B_1,B_2,d,i,j) = 0$. We consider the complex
\begin{equation}\label{eq:Ext(X,C_m)}
\begin{CD}
\begin{matrix}
   \Hom(V_0, \C^m)
\\
   \oplus
\\ \Hom(V_1, \C^{m+1})
\end{matrix}
@>{\sigma}>>
\begin{matrix}
   \Hom(V_0, \C^{m+1}) \\ \oplus \\ \C^2\otimes \Hom(V_1, \C^m) \\
   \oplus \\ \Hom(W, \C^m)
\end{matrix}
@>{\tau}>>
\begin{matrix}
  \Hom(V_1,\C^m)
\end{matrix},
\end{CD}
\end{equation}
with
{\allowdisplaybreaks
\begin{equation*}
\begin{split}
   & 
   \sigma
   \begin{bmatrix}
     \xi_0 \\ \xi_1
   \end{bmatrix}
   = 
   \begin{bmatrix}
     & \xi_1 d   \\
     & \xi_0 B_1 - \left[\begin{smallmatrix}
        1_m & 0
      \end{smallmatrix}\right] \xi_1 \\
     &\xi_0 B_2 - \left[\begin{smallmatrix}
        0 & 1_m
      \end{smallmatrix}\right] \xi_1 \\    
     & \xi_0 i    
   \end{bmatrix},
\\
   & 
   \tau
   \begin{bmatrix}
    \widetilde d \\ \widetilde B_1 \\ \widetilde B_2 \\ \widetilde i
   \end{bmatrix}
   = \left[\begin{smallmatrix}
        1_m & 0
      \end{smallmatrix}\right]
    \widetilde d B_2 
    -
   \left[\begin{smallmatrix}
        0 & 1_m
      \end{smallmatrix}\right]
    \widetilde d B_1
    + \widetilde i j
    + \widetilde B_1 d B_2 - \widetilde B_2 d B_1.
\end{split}
\end{equation*}
This} complex is constructed as follows. We take the dual of  
\eqref{eq:cpx}, and replace the part
\(
  \C^2\otimes V_0^*\otimes \shfO\xrightarrow{\left[
      \begin{smallmatrix}
        z & w
      \end{smallmatrix}
    \right]}
    V_0^*\otimes\shfO(-C+\linf)
\)
by $V_0^*\otimes\shfO(C-\linf)$.
We then take the tensor product with $C_m$ and apply
$H^*(\bp,\bullet)$. Therefore when $X$ corresponds to a framed sheaf
$(E,\Phi)$, the cohomology groups of the complex are
$\Ext^\bullet(E,C_m)$.

\begin{Lemma}\label{lem:Ext-}
  Suppose that $X$ corresponds to a framed sheaf $(E,\Phi)$.

  \textup{(1)} $\Hom(E,C_m)\cong \Ker\sigma$,
  $\Ext^1(E,C_m)\cong \Ker\tau/\Ima\sigma$, and
  $\Ext^2(E,C_m)\cong \Coker\tau$.

  \textup{(2)} Suppose further that $X$ is $(m,m+1)$-semistable. Then
  $\tau$ is surjective.
\end{Lemma}

\begin{proof}
(1) These are already explained.

(2) By the Serre duality we have $\Ext^2(E,C_m) =
\Hom(C_{m-1},E)^\vee$. But the right hand side vanishes if $E$ is
$(m,m+1)$-semistable.
\end{proof}

If $X$ is $\zeta^-$-stable (i.e., $(E,\Phi)$ is $m$-stable), we also
have $\Ker\sigma = 0$. But this does not hold in general if $X$ is
only $(m,m+1)$-semistable.

Next we consider the complex
\begin{equation}\label{eq:Ext(C_m,X)}
\begin{CD}
\begin{matrix}
   \Hom(\C^m, V_0)
\\
   \oplus
\\ \Hom(\C^{m+1}, V_1)
\end{matrix}
@>{\sigma}>>
\begin{matrix}
   \Hom(\C^{m}, V_1) \\ \oplus \\ \C^2\otimes \Hom(\C^{m+1}, V_0) \\
   \oplus \\ \Hom(\C^{m+1},W)
\end{matrix}
@>{\tau}>>
\begin{matrix}
  \Hom(\C^{m+1}, V_0)
\end{matrix},
\end{CD}
\end{equation}
with
{\allowdisplaybreaks
\begin{equation*}
\begin{split}
   & \sigma
   \begin{bmatrix}
     \xi_0 \\ \xi_1
   \end{bmatrix}
   = 
   \begin{bmatrix}
     & d \xi_0   \\
     & B_1\xi_1 - \xi_0 \left[\begin{smallmatrix}
        1_m & 0
      \end{smallmatrix}\right] \\
     & B_2\xi_1 - \xi_0\left[\begin{smallmatrix}
        0 & 1_m
      \end{smallmatrix}\right]  \\    
     & j\xi_1
   \end{bmatrix},
\\
   & \tau
   \begin{bmatrix}
    \widetilde d \\ \widetilde B_1 \\ \widetilde B_2 \\ \widetilde j
   \end{bmatrix}
   = B_1 
    \widetilde d 
    \left[\begin{smallmatrix}
        1_m & 0
      \end{smallmatrix}\right]
    -
    B_2
    \widetilde d \left[\begin{smallmatrix}
        0 & 1_m
      \end{smallmatrix}\right]
    + i \widetilde j
    + B_1 d \widetilde B_2 - B_2 d \widetilde B_1.
\end{split}
\end{equation*}
}

\begin{Lemma}\label{lem:Ext+}
  Suppose that $X$ corresponds to a framed sheaf $(E,\Phi)$.

  \textup{(1)} $\Hom(C_m,E)\cong \Ker\sigma$,
  $\Ext^1(C_m,E)\cong \Ker\tau/\Ima\sigma$, and
  $\Ext^2(C_m,E)\cong \Coker\tau$.

  \textup{(2)} Suppose further that $X$ is $(m,m+1)$-semistable. Then
  $\tau$ is surjective.
\end{Lemma}

\begin{proof}
The proof of (1) is the same as in \lemref{lem:Ext-}, so is omitted.

(2) We have
\(
   \Ext^2(C_m,E) \cong \Hom(E,C_{m+1})^\vee
\)
by the Serre duality. But the right hand side vanishes if $E$ is
$(m,m+1)$-semistable.
\end{proof}

If $X$ is $\zeta^+$-stable (i.e., $(E,\Phi)$ is $(m+1)$-stable), we also
have $\Ker\sigma = 0$. But this does not hold in general if $X$ is
only $(m,m+1)$-semistable.

\begin{Corollary}
  Let $Q^{\mathrm{ss}}(\zeta^0)$ be the open subscheme of $Q$
  consisting of $\zeta^0$-semistable, i.e., $(m,m+1)$-semistable
  objects. The differential $d\mu$ is surjective on
  $Q^{\mathrm{ss}}(\zeta^0)$. Hence $Q^{\mathrm{ss}}(\zeta^0)$ is
  smooth.
\end{Corollary}

\begin{proof}
  Since the surjectivity of $d\mu$ is an open condition, it is enough
  to check the assertion when $X$ is a direct sum of $\zeta^0$-stable
  objects. As explained in \defref{def:m,m+1-stable}, we have $X =
  X^0\oplus C_m^{\oplus p}$, where $X^0$ is $\zeta^0$-stable with
  $X_\infty\neq 0$ and $p\ge 0$. Then the tangent complex
  \eqref{eq:tangentcpx} decomposes into four parts, the tangent
  complex for $X^0$, the sum of $p$-copies of the
  complex~\eqref{eq:Ext(X,C_m)} for $X^0$, the sum of $p$-copies of
  the complex~\eqref{eq:Ext(C_m,X)} for $X^0$, and the sum of
  $p^2$-copies of the tangent complex for $C_m$.

  The differential of $d\mu$ for $X^0$ is surjective since $X^0$ is
  $\zeta^0$-stable by \cite[2.4]{perv}. The surjectivity of $\tau$ for
  \eqref{eq:Ext(X,C_m)} and \eqref{eq:Ext(C_m,X)} was proved in
  \lemref{lem:Ext-} and \lemref{lem:Ext+} respectively.
  The surjectivity of $d\mu$ for $C_m$ follows from either
  \cite[2.4]{perv}, \lemref{lem:Ext-} or \lemref{lem:Ext+} since $C_m$
  is $(-1/m,1/(m+1))$-stable by \cite[\S2.2]{perv}.
\end{proof}

\begin{NB}
I moved this subsection to \secref{sec:1st}.

\subsection{Relation to the framed moduli spaces on $\proj^2$}

As we mentioned before, $\bM^0(c)$ is isomorphic to the framed moduli
space $M(p_*(c))$ of torsion free sheaves on $\proj^2$ when $(c_1,[C])
= 0$. We review \cite[\S3.10]{perv2} describing a relation when $0 >
(c_1,[C]) > -r$. Since the condition $0\ge (c_1,[C]) > -r$ is always
achieved by twisting the line bundle $\shfO(C)$, this is the only
remaining situation.

\begin{Proposition}
  Suppose $0 < n \defeq -(c_1,[C]) < r$. There is a variety $\widehat
  N(c,n)$ relating $\bM^1(c)$ and $\bM^1(c-ne)$ through a diagram
\begin{equation*}
\xymatrix@R=.5pc{
& \widehat N(c,n) \ar[ld]_{f_1} \ar[rd]^{f_2} &
\\
\bM^{1}(c) & & \bM^{1}(c-ne)
}
\end{equation*}
satisfying the followings\textup:

\textup{(1)} $f_1$ is surjective and birational.

\textup{(2)} $f_2$ is the Grassmann bundle 
\(
   \Gr(n,\Ext^1_{q_2}(C_0,\cE'))
\)
of $n$-planes in the vector bundle $\Ext^1_{q_2}(C_0,\cE')$ of rank $r$
over $\bM^1(c-ne)$.

\textup{(3)} We have a short exact sequence
\begin{equation*}
   0 \to (\id_{\bp}\times f_2)^*\cE' \to (\id_{\bp}\times f_1)^*\cE 
   \to C_0\boxtimes\mathcal S\to 0.
\end{equation*}
Here $\cE$, $\cE'$ are the universal sheaves for $\bM^1(c)$ and
$\bM^1(c-ne)$ respectively, and $\mathcal S$ is the universal rank $n$ subbundle
of $\Ext^1_{q_2}(C_0,\cE')$ over $\Gr(n,\Ext^1_{q_2}(C_0,\cE'))$.
\end{Proposition}

Remark that $\Ext^i_{q_2}(C_0,\cE) = 0$ for $i=0,2$ by the remark
after \lemref{lem:Ext+}. Hence the rank of $\Ext^1_{q_2}(C_0,\cE)$ is
equal to $r$ by Riemann-Roch.

We have $(c_1(c-ne),[C]) = (c_1,[C]) + n = 0$. Therefore
$\bMm{1}(c-ne)$ becomes $M(p_*(c))$ after crossing the wall $\zeta_1 = 0$.

\begin{NB2}
  I am not sure that we really need this description. It is probably
  enough to explain $\bM^0(c) = \emptyset$ if $(c_1,[C]) < 0$. Then
  the wall-crossing formula will take care.
\end{NB2}%
\end{NB}

\section{Enhanced master spaces}\label{sec:mater spaces}


A naive idea to show the weak wall-crossing formula
(=\thmref{thm:m=0}) is to compare the intersection products on
$\bM^m(c)$ and $\bM^{m+1}(c)$ through the diagram
\eqref{eq:flip}. Such an idea works when $\bM^m(c)$, $\bM^{m+1}(c)$
are replaced by moduli spaces of stable rank $2$ torsion free sheaves
over a surface with $p_g = 0$ with respect to two polarizations
separated by a wall which is good \cite{EG1}. (See also \cite{GNY}.)

However the idea does not work since the morphisms $\xi_m$, $\xi_m^+$
are much more complicated in our current situation, basically because
dimensions of vector spaces $V$, $U$ in \propref{prop:extension} can
be arbitrary. We use a refinement of the idea, due to Mochizuki
\cite{Moc}, which was used to study wall-crossing formula for moduli
spaces of higher ranks stable sheaves. It consists of two parts:
\begin{description}
\item[(Ma)] Consider a {\it pair\/} of a sheaf and a full flag in
the sheaf cohomology group.
\item[(Mb)] Use the fixed point formula for the $\C^*$-equivariant
  cohomology on Thaddeus' master space.
\end{description}

In this section we describe the part {\bf (Ma)}. The results are
straightforward modifications of those in \cite[\S4]{Moc}, possibly
except one in \subsecref{subsec:Grassmann}.

We remark that Yamada also used moduli spaces of pairs of a sheaf and a full flag in the sheaf cohomology group to study wall-crossing of moduli spaces of higher rank sheaves \cite{Ya}.
\begin{NB2}
  Added. 2010/04/07
\end{NB2}%

We continue to fix $m$ and choose parameters $\zeta^0$, $\zeta^\pm$ as
in the previous section. We fix a dimension vector ${\mathbf v} =
(v_0,v_1,1)$ and define $\zeta^0_\infty$ so that the normalization
condition $\zeta^0_0 v_0 + \zeta^0_1 v_1 + \zeta^0_\infty = 0$ is
satisfied.
\begin{NB2}
  Corrected. 2010/04/01
\end{NB2}%

\subsection{Framed sheaves with flags in cohomology groups}
\label{subsec:withflags}

Let $((E,\Phi),F^\bullet)$ be a pair of a framed sheaf and a full flag
$F^\bullet = (0 = F^0 \subset F^1 \subset \cdots \subset F^{N-1}
\subset F^N = V_1(E))$ ($N = \dim V_1(E)$). Let $\ell$ be an integer
between $0$ and $N$.

\begin{Definition}\label{def:m,l-stable}
  An object $((E,\Phi),F^\bullet)$ is called {\it $(m,\ell)$-stable\/}
  if the following three conditions are satisfied
  \begin{enumerate}
  \item $(E,\Phi)$ is $(m,m+1)$-semistable.
  \item For a subsheaf $\mathfrak S\subset E$ isomorphic to
    $C_m^{\oplus p}$ with $p\in\Z_{>0}$, we have
    $V_1(\mathfrak S)\cap F^\ell = 0$.
  \item For a subsheaf $\mathfrak S\subset E$ such that the quotient
    $E/\mathfrak S$ is isomorphic to $C_m^{\oplus p}$ with
    $p\in\Z_{>0}$, we have $F^\ell\not\subset V_1(\mathfrak S)$.
  \end{enumerate}
\end{Definition}

This notion makes a bridge between the $m$-stability and the
$(m+1)$-stability by the following observation:
If $\ell = 0$ (resp.\ $= N$), then $(m,\ell)$-stability of
$((E,\Phi),F^\bullet)$ is equivalent to $m$ (resp.\ $(m+1)$)-stability
of $(E,\Phi)$. (See \propref{prop:extension}.)

\begin{Proposition}[\protect{\cite[4.2.5]{Moc}}]\label{prop:m,l-moduli}
  We have a smooth fine moduli scheme $\tM^{m,\ell}(c)$ of
  $(m,\ell)$-stable objects $((E,\Phi),F^\bullet)$ with $\ch(E) = c$.
  There is a projective morphism $\tM^{m,\ell}(c)\to M_0(p_*(c))$.
  \begin{NB2}
    Second statement added according to the referee's
    suggestion. 2010/04/03
  \end{NB2}%
\end{Proposition}

The proof will be given in the next subsection.

Let $\cE$ be the universal sheaf over $\bp\times \tM^{m,\ell}(c)$ and
$q_1$, $q_2$ be the projection to the first and second factors of
$\bp\times \tM^{m,\ell}(c)$ respectively as before.
As well as the vector bundle 
\(
   \Vcal_1 \equiv \Vcal_1(\cE) 
   \defeq R^1q_{2*}(\cE\otimes q_1^*\shfO(C-\linf)),
\)
\begin{NB}
Coming from the vector space $V_1$ in the quiver description
\end{NB}%
we also have the universal family $\Fcal^\bullet = (0 = \Fcal^0
\subset \Fcal^1 \subset \cdots \subset \Fcal^{N-1} \subset \Fcal^N =
\Vcal_1)$ of flags of vector bundles over $\tM^{m,\ell}(c)$.

\begin{NB}
The following is already explained:

We use the following notation hereafter:
For a sheaf $\mathcal S$ over $\bp\times \tM^{m,\ell}(c)$ (or
$\bp\times \text{other moduli spaces}$), we denote $R^1q_{2*}(\mathcal
S\otimes q_1^*\shfO(C-\linf))$ by $\Vcal_1(\mathcal S)$ or simply by
$\Vcal_1$ if $\mathcal S$ is clear from the context.
\end{NB}

For $\ell = 0$ or $N$, the preceding remark implies that $\tM^{m,0}(c)$,
$\tM^{m,N}(c)$ are the full flag bundles
\(
   \Flag(\Vcal_1,\underline{N})
\)
associated with vector bundles $\Vcal_1$ over $\bM^m(c)$,
$\bM^{m+1}(c)$ respectively. Here $\underline{N} = \{ 1,\dots, N\}$
and the notation $\Flag(\Vcal_1,\underline{N})$ means that the flag is
indexed by the set $\underline{N}$. This notation will become useful
when we consider the fixed points in the enhanced master space.

\subsection{Proof of \propref{prop:m,l-moduli}}

Let us rephrase the $(m,\ell)$-stability in the quiver description.

We consider a pair $(X,F^\bullet) = \left( (B_1,B_2,d,i,j),
  F^\bullet\right)$ of $X\in\mu^{-1}(0)$ and a flag $F^\bullet =
(0 = F^0 \subset F^1 \subset \cdots \subset F^{N-1} \subset F^N
= V_1)$ of $V_1$ with $\dim (F^i/F^{i-1}) = 0$ or $1$.
In \defref{def:m,l-stable} we have assumed that $F^\bullet$ is a full
flag, i.e., $\dim (F^i/F^{i-1}) = 1$, but we slightly generalize it
for a notational simplicity. 

In terms of a quiver representation, \defref{def:m,l-stable} is
expressed as follows:

\begin{Definition}
  For $0\le\ell\le N$, we say $(X,F^\bullet)$ is {\it
    $(m,\ell)$-stable\/} if the following conditions are satisfied:
\begin{enumerate}
\item $X$ is $\zeta^0$-semistable,
\item for a nonzero submodule $0\neq S \subset X$ with
\(
   \zeta^0\cdot \vdim S/\rank S = 0
\)
and $S_\infty = 0$,
we have $F^\ell\cap S_1 = 0$, and
\item for a proper submodule $S \subsetneq X$ with
\(
   \zeta^0\cdot \vdim S/\rank S = 0
\)
and $S_\infty = \C$,
we have $F^\ell\not\subset S_1$.
\end{enumerate}
\end{Definition}

The equivalence of (2),(3) and \ref{def:m,l-stable}(2),(3) is an
immediate consequence of \cite[Th.~2.13, Prop.~5.3]{perv}.

If $\ell=0$ (resp.\ $=N$), then $(m,\ell)$-stability is
equivalent to the $\zeta^-$ (resp.\ $\zeta^+$)-stability of $X$. (See
\propref{prop:extension}.)

The idea of the proof of \propref{prop:m,l-moduli} is to relate the
above condition to an usual stability condition for a linearization on
the product of $\mu^{-1}(0)$ and the flag variety with respect to the
group action of $G$.
It is the tensor product of linearizations on $\mu^{-1}(0)$ and the
flag variety.
For $\mu^{-1}(0)$, we take one as in \subsecref{subsec:mstable}. On
the flag variety, we take $\bigotimes (\det \Fcal^i)^{\otimes (-n_i)}$
for $\bn = (n_1,\dots,n_{N})\in \Q^{N}_{>0}$, where $\Fcal^i$ is the
$i^{\mathrm{th}}$ universal bundle. (See \cite[\S4.2]{Moc}.)
Then we have a natural projective morphism to $\mu^{-1}(0)\dslash G =
M_0(p_*(c))$ as before.
\begin{NB2}
  According to the referee request, we explain the choice of the linearization.
  2010/04/05
\end{NB2}%

The corresponding stability condition is expressed as before, but an
extra term for flags is added to $\theta$ (cf.\ \cite[Ch.~4,
\S4]{MFK}):
\begin{NB2}
  Reference added. 2010/04/07
\end{NB2}
For a nonzero graded submodule $S\subset X$ we define
\begin{equation*}
  \mu_{\zeta,\bn}(S)
  \defeq
  \frac{\zeta\cdot\vdim S + \sum_i n_i \dim (S\cap F^i)}{\rank S}.
\end{equation*}
We say $(X,F^\bullet)$ is {\it $(\zeta,\bn)$-\textup(semi\textup)stable\/} if
\begin{equation*}
   \mu_{\zeta,\bn}(S) (\le) \mu_{\zeta,\bn}(X)
\end{equation*}
holds for any nonzero proper submodule $0\neq S\subsetneq X$. 
Here $(\le)$ means that $\le$ for the semistable case, and $<$ for the
stable case.
We say $(X,F^\bullet)$ is {\it strictly $(\zeta,\bn)$-semistable\/} if
it is $(\zeta,\bn)$-semistable and not $(\zeta,\bn)$-stable.

A standard argument shows the following:
\begin{Lemma}\label{lem:isom}
  If $(X,F^\bullet)$, $(Y,G^\bullet)$ are $(\zeta,\bn)$-stable and
  have the same $\mu_{\zeta,\bn}$-value, a nonzero homomorphism
  $\xi\colon X\to Y$ with $\xi(F^i)\subset G^i$ must be an isomorphism.
\end{Lemma}

\begin{NB}
  \begin{proof}
    Let us consider submodules $\Ker\xi\subset X$, $\Ima\xi\subset
    Y$. From the $\mu_{\zeta,\bn}$-stability, we have
    \begin{equation*}
       \mu_{\zeta,\bn}(\Ker\xi) \le \mu_{\zeta,\bn}(X), \qquad
       \mu_{\zeta,\bn}(\Ima\xi) \le \mu_{\zeta,\bn}(Y)
    \end{equation*}
    and the inequalities are strict unless $\Ker\xi=0$, $\Ima\xi =
    Y$. The first inequality is equivalent to
    \begin{equation*}
      \mu_{\zeta,\bn}(X)\le 
        \frac{
         \zeta\cdot\vdim \Ima\xi + \sum_i n_i \dim (F^i/F^i\cap(\Ker\xi)_1)
       }{\rank \Ima\xi} 
    \end{equation*}
    as $X/\Ker\xi \cong \Ima\xi$. We have an injection $\xi\colon
    F^i/F^i\cap(\Ker\xi)_1 \to G^i\cap\Ima\xi$. Therefore the right
    hand side is less than or equal to
    $\mu_{\zeta,\bn}(\Ima\xi)$. Combining these inequalities with the
    assumption $\mu_{\zeta,\bn}(X) = \mu_{\zeta,\bn}(Y)$, we find that
    all inequalities must be equalities. Therefore $\Ker\xi=0$ and
    $\Ima\xi = Y$, i.e., $\xi$ is an isomorphism.
  \end{proof}
\end{NB}

We take $N = \dim V_1$ so that $F^\bullet$ is a full flag of $V_1$.
Consider the following conditions when $\ell\neq 0$:
{\allowdisplaybreaks
\begin{subequations}\label{eq:cond}
  \begin{align}
    & \zeta_0 v_0 + \zeta_1 v_1 + \zeta_\infty = 0,
\\
    & n_k > \rank X \sum_{i=k+1}^N i n_i > 0 \quad\text{for $k=1,\dots,N-1$},
\\
    & 
    \max_{\substack{S\subset V\\ \zeta^0\cdot \vdim S \neq 0
      }}
    \left| \frac{(\zeta^0 - \zeta) \cdot \vdim S}{\rank S} \right|
    + \sum_{i=1}^{N} i n_i < 
    \min_{\substack{S\subset V\\ \zeta^0\cdot \vdim S \neq 0}}
   \left|\frac{\zeta^0\cdot \vdim S}{\rank S}\right|,
\\
    & 0 < \sum_{i=1}^\ell i n_i - 
    \frac{\rank X}{2m+1} (m\zeta_0 + (m+1)\zeta_1) 
    < n_\ell,
\\
    & \rank X \sum_{i=l+1}^{N} i n_i
    < \min\Bigg(
    \begin{aligned}[t]
    &  \sum_{i=1}^\ell i n_i -
    \frac{\rank X}{2m+1} (m\zeta_0 + (m+1)\zeta_1),
    \\
    & \qquad n_\ell - \sum_{i=1}^\ell i n_i +
    \frac{\rank X}{2m+1} (m\zeta_0 + (m+1)\zeta_1)
    \Bigg),
    \end{aligned}
  \end{align}
\end{subequations}
where} $S\subset V$ runs over the set of all graded subspaces (not
necessarily submodules, as $X$ is not fixed yet) with $\zeta^0\cdot
\vdim S \neq 0$ in (c).
For example, we can take $\zeta$, $\bn$ in the following way: First
note that the right hand side of (c) is a fixed constant independent
of $\zeta$, $\bn$. Therefore it holds if $n_i$ and $\zeta^0-\zeta$ are
sufficiently small. Next note that (d) says that $\zeta^0 - \zeta$, up
to a constant multiplication, is between $(\sum_{i=1}^l in_i) - n_l$
and $\sum_{i=1}^l in_i$. Therefore we can first take $\zeta$ and
$n_1$, \dots, $n_\ell$ so that (b) (for $k=1,\dots,\ell$),(d) and (c)
hold if we set $n_{\ell+1} = \dots = n_N = 0$.
Then we take sufficiently small $n_i$ ($i=\ell+1,\dots, N$) so that (e)
holds. We can choose them very small so that (b) and (c) are not
violated.

\begin{Lemma}[\protect{\cite[4.2.4]{Moc}}]\label{lem:zeta,l}
Assume \eqref{eq:cond}. Then the followings hold:

\textup{(1)} The $(m,\ell)$-stability is equivalent to the
$(\zeta,\bn)$-stability.

\textup{(2)} The $(\zeta,\bn)$-semistability automatically implies
the $(\zeta,\bn)$-stability.
\end{Lemma}

\lemref{lem:isom} also implies that $(\zeta,\bn)$-stable objects have
the trivial stabilizer. Therefore we have \propref{prop:m,l-moduli}
from this lemma.

\begin{NB}
\begin{Corollary}\label{cor:m,l-moduli}
  We have a smooth fine moduli scheme $\tM^{m,\ell}$ of
  $(m,\ell)$-stable objects $(X,F^\bullet)$.
\end{Corollary}
\end{NB}

From the construction the universal family
$\Fcal^\bullet
= (0 = \Fcal^0 \subset \Fcal^1 \subset \cdots \subset \Fcal^{N-1}
\subset \Fcal^N = \Vcal_1)$ is the descent of the universal flag
over $\Flag(V_1,\underline{N})$.
Using descents of $V_0$, $V_1$ (which will be denoted by $\Vcal_0$,
$\Vcal_1$) in the complex \eqref{eq:cpx}, we also get the universal family
$\cE$ over $\bp\times \bM^m(c)$.

\begin{proof}[Proof of \lemref{lem:zeta,l}]
Suppose that $S\subset X$ is a nonzero submodule with
$\zeta^0\cdot\vdim S = 0$, $S_\infty = 0$. Then $\vdim S = p(m,m+1,0)$
for some $p\in \Z_{>0}$. Then $\mu_{\zeta,\bn}(S)\le \mu_{\zeta,\bn}(X)$
means
\begin{equation*}
   \frac{m\zeta_0 + (m+1)\zeta_1}{2m+1}
   + \frac{\sum_{i=1}^{N} n_i \dim (S_1\cap F^i)}{p(2m+1)}
   \le \frac{\sum_{i=1}^{N} i n_i}{\rank X}.
\end{equation*}
By (\ref{eq:cond}d,e) this holds if and only if $S_1\cap F^\ell = 0$.
\begin{NB}
Suppose this inequality holds. Then
  \begin{equation*}
    n_\ell \dim (S_1\cap F^\ell) \le
    \frac{\rank X}{p(2m+1)}\sum_{i=1}^{N} n_i \dim (S_1\cap F^i)
    \le \sum_{i=1}^{N} i n_i
    - \frac{\rank X}{2m+1}(m\zeta_0 + (m+1)\zeta_1)
    < n_\ell.
  \end{equation*}
Therefore $S_1\cap F^\ell = 0$. Conversely if $S_1\cap F^\ell = 0$, then
\begin{equation*}
  \begin{split}
   & \frac{m\zeta_0 + (m+1)\zeta_1}{2m+1}
   + \frac{\sum_{i=1}^{N} n_i \dim (S_1\cap F^i)}{p(2m+1)}
   - \frac{\sum_{i=1}^{N} i n_i}{\rank X}
\\
   =\; & \frac1{\rank X}\left(
        \frac{\rank X}{2m+1}(m\zeta_0 + (m+1)\zeta_1)
        - \sum_{i=1}^{N} n_i \dim F^i
        + \frac{\rank X}{p(2m+1)}
        \sum_{i=\ell+1}^{N} n_i \dim (S_1\cap F^i)
     \right) < 0
  \end{split}
\end{equation*}
\end{NB}%
Moreover the equality never holds.

Next suppose that $S\subset X$ is a proper submodule with
$\zeta^0\cdot\vdim S = 0$, $S_\infty = \C$. Then $\vdim X/S = p(m,m+1,0)$
for some $p\in \Z_{>0}$. Then $\mu_{\zeta,\bn}(S)\le \mu_{\zeta,\bn}(X)$
means
\begin{equation*}
   \frac{m\zeta_0 + (m+1)\zeta_1}{2m+1}
   + \frac{\sum_{i=1}^{N} n_i \dim (F^i/S_1\cap F^i)}{p(2m+1)}
   \ge \frac{\sum_{i=1}^{N} i n_i}{\rank X}.
\end{equation*}
By (\ref{eq:cond}d,e) this holds if and only if $F^i/S_1\cap F^i \neq
0$ for some $i=1,\dots, \ell$, i.e., $F^\ell\not\subset S_1$.
Moreover the equality never holds.

Now let us start the proof.
Suppose that $(X,F^\bullet)$ is $(\zeta,\bn)$-semistable. Then by
(\ref{eq:cond}c) $\mu_{\zeta,\bn}(S)\le \mu_{\zeta,\bn}(X)$ implies
$\zeta^0\cdot \vdim S\le 0$, i.e., $X$ is $\zeta^0$-semistable.
\begin{NB}
  \begin{equation*}
    \begin{split}
      & \frac{\zeta^0\cdot\vdim S}{\rank S}
      \\
      \le\; & \frac{\zeta^0\cdot\vdim S}{\rank S}
    - \left(\mu_{\zeta,\bn}(S) - \mu_{\zeta,\bn}(X)
      \right)
      = \frac{(\zeta^0 - \zeta)\cdot\vdim S}{\rank S}
        - \frac{\sum n_i \dim (S_1\cap F^i)}{\rank S}
        + \frac{\sum i n_i}{\rank X}
        \\
      \le\; &  
    \frac{(\zeta^0 - \zeta) \cdot \vdim S}{\rank S}
    + \frac{\sum i n_i}{\rank X}
     < \min_{\substack{S\subset V\\ \zeta^0\cdot \vdim S > 0}}
   \left(\frac{\zeta^0\cdot \vdim S}{\rank S}\right)
    \end{split}
  \end{equation*}
Therefore $\frac{\zeta^0\cdot \vdim S}{\rank S} > 0$ is not possible.
\end{NB}%
Then the above consideration shows that $(X,F^\bullet)$ is
$(m,\ell)$-stable and $(\zeta,\bn)$-stable.

Conversely suppose $(X,F^\bullet)$ is $(m,\ell)$-stable. We want to
show that $\mu_{\zeta,\bn}(S) < \mu_{\zeta,\bn}(X)$ for any nonzero
proper submodule $S$.
Thanks to (\ref{eq:cond}c), it is enough to check this inequality for
$S$ with $\zeta^0\cdot \vdim S = 0$. 
\begin{NB}
Suppose $\zeta^0\cdot\vdim S < 0$. Then
\begin{equation*}
   \begin{split}
  & \mu_{\zeta,\bn}(S) - \mu_{\zeta,\bn}(X)
  = \frac{\zeta^0\cdot\vdim S}{\rank S}
    - \frac{(\zeta^0 - \zeta)\cdot\vdim S}{\rank S}
        + \frac{\sum n_i \dim (S_1\cap F^i)}{\rank S}
        - \frac{\sum i n_i}{\rank X}
\\
  <\;& \frac{\zeta^0\cdot\vdim S}{\rank S}
    - \frac{(\zeta^0 - \zeta)\cdot\vdim S}{\rank S}
        + \sum i n_i
  \le \frac{\zeta^0\cdot\vdim S}{\rank S}
  + \max_{\substack{S\subset V}}
      \left|\frac{(\zeta^0 - \zeta)\cdot\vdim S}{\rank S}\right|
        + \sum i n_i < 0
  \end{split}
\end{equation*}
\end{NB}%
We have either $S_\infty = 0$ or $S_\infty = \C$, and the above
consideration shows that the inequality holds.
\end{proof}

\begin{Remark}
  A closer look of the argument gives that it is enough to assume
  (\ref{eq:cond}c) for $S\subset V$ satisfying either of
  \begin{enumerate}
  \item $\zeta^0\cdot\vdim S > 0$ and $\mu_{\zeta,\bn}(S)\le \mu_{\zeta,\bn}(X)$,
  \item $\zeta^0\cdot\vdim S < 0$ and $\mu_{\zeta,\bn}(S)\ge \mu_{\zeta,\bn}(X)$.
  \end{enumerate}
\end{Remark}

\subsection{Oriented sheaves with flags in cohomology groups}
\label{subsec:orientedsheaf}
The following variant of objects in the previous subsection will show
up during our analysis for the enhanced master space.

\begin{Definition}\label{def:m,+stable}
(1)
  Let $(E,F^\bullet)$ be a pair of a sheaf and a full flag $F^\bullet
  = (0 = F^0 \subset F^1 \subset \cdots \subset F^{N-1} \subset F^N =
  V_1(E))$ ($N = \dim V_1(E)$).  We say $(E,F^\bullet)$ is {\it
    $(m,+)$-stable\/} if the following two conditions are satisfied:
\begin{aenume}
\item $E\cong C_m^{\oplus p}$ for $p\in\Z_{>0}$.
\item For a proper subsheaf $\mathfrak S\subsetneq E$ isomorphic to
  $C_m^{\oplus q}$ with $q\in\Z_{>0}$, we have
  $V_1(\mathfrak S)\cap F^1 = 0$.
\end{aenume}

(2)
For $(m_0,m_1)\in\Z^2\setminus\{(0,0)\}$, an {\it orientation\/} of
$(E,F^\bullet)$ is an isomorphism $\rho\colon \det H^1(E)^{\otimes m_0}\otimes
\det H^1(E(C))^{\otimes m_1} \xrightarrow{\cong}\C$.
We set $D \defeq mm_0 + (m+1)m_1$.

(3) An oriented $(m,+)$-stable object means an $(m,+)$-stable object
$(E,F^\bullet)$ together with an orientation.
\end{Definition}

We will choose $(m_0,m_1)$ later when we define the enhanced master
space. At this stage we only need that we will have $D > 0$. Since the
orientation will be used frequently, we use the notation
\(
   L(E) \defeq \det H^1(E(-\linf))^{\otimes m_0}\otimes
   \det H^1(E(C-\linf))^{\otimes m_1}
\)
for a sheaf $E$ on $\bp$. (In the above case, $E$ is supported on $C$
and the twisting by $\shfO(-\linf)$ is unnecessary.) If $\cE$ is a
universal family for some moduli stack, we denote the corresponding
line bundle by $\cL(\cE)$. Note that we deal only with those sheaves
given by quiver representations, we have vanishing of $H^0$ and $H^2$.
\begin{NB2}
  The referee seems to suspect that we need to consider also $H^0$,
  $H^2$ for a general sheaf $E$. But we consider only those vanishing
  $H^0$, $H^2$. So we write it explicitly. 2010/04/01
\end{NB2}%

We will show that we have a moduli stack $\tM^{m,+}(pe_m)$ of oriented
$(m,+)$-stable objects with $\ch(E) = pe_m$ with the universal family
$(\cE,\Fcal^\bullet)$ where $\Fcal^\bullet = (0 = \Fcal^0 \subset
\Fcal^1 \subset \cdots \subset \Fcal^{N-1} \subset \Fcal^N = \Vcal_1)$
is a flag of vector bundles over $\tM^{m,+}(pe_m)$.

In the following proposition we identify $\tM^{m,+}(pe_m)$ with a
quotient stack related to the Grassmann variety $\Gr(m+1,p)$ of
$p$-dimensional quotients of $V_1(C_m)^* = H^1(C_m(C))^* = \C^{m+1}$.
Let us fix the notation.
Let $\cQ$ denote the universal quotient bundle over $\Gr(m+1,p)$,
$\det\cQ$ its determinant line bundle, and $(\det\cQ)^{\otimes D}$ its
$D^{\mathrm{th}}$ tensor power. Let $\pi_G\colon ((\det\cQ)^{\otimes
  D})^\times\to \Gr(m+1,p)$ be the associated $\C^*$-bundle.
\begin{NB}
The following will be explained in the statement of \propref{prop:m,+}.
We consider the $\C^*$-action on fibers $(\det\cQ^{\otimes D})^\times$
given $t\cdot u = t^{-pD} u$.
\end{NB}%
Let $\shfO_{\Gr(m+1,p)}\to V_1(C_m)\otimes\cQ$ be the homomorphism
obtained from the universal homomorphism
$V_1(C_m)^*\to \cQ$.

\begin{NB}
The old notation:
Then $H^1(C_m(C))\otimes\cQ$ contains a trivial line subbundle $\shfO$
spanned by the tautological section $\xi\in\Hom(H^1(C_m(C))^*,\cQ) =
H^1(C_m(C))\otimes\cQ$. Let us denote it by $[\xi]$.
\end{NB}
%

\begin{Proposition}\label{prop:m,+}
\textup{(1)} 
  Forgetting $F^i$ for $i\neq 1$, we identify
  $\tM^{m,+}(pe_m)$ with the total space of the flag bundle
  $\Flag(\Vcal_1/\Fcal^1, \underline{N}\setminus\{1\})$, where the base is 
  isomorphic to the quotient stack
  \begin{NB2}
    Editted according to the referee's remark. 2010/04/05
  \end{NB2}%
  \begin{equation*}
     ((\det\cQ)^{\otimes D})^\times/\C^*,
  \end{equation*}
  where $\C^*$ acts by the fiber-wise multiplication with weight $-pD$.

\textup{(2)} Let $\C^*_s = \operatorname{Spec}\C[s,s^{-1}]$ be a copy
of $\C^*$ and let $\C^*\to \C^*_s$ be the homomorphism given by
$t\mapsto t^{-pD} = s$.
It induces an \'etale and finite morphism $h\colon ((\det\cQ)^{\otimes
  D})^\times/\C^* \to ((\det\cQ)^{\otimes D})^\times/\C^*_s =
\Gr(m+1,p)$ of degree $=1/pD$.
The vector bundles $\cV_1$, $\cF^1$ 
and the universal sheaf $\cE$ are related to objects on $\Gr(m+1,p)$ by
\begin{gather*}
   \Vcal_1\otimes (\cF^1)^* = h^*(V_1(C_m)\otimes\cQ),
   \qquad
   (\Fcal^1)^{\otimes -pD} = h^*((\det\Qcal)^{\otimes D}),
\\
   \cE\otimes (\cF^1)^* = (\id\times h)^*(C_m\boxtimes \cQ),
\end{gather*}
and the inclusion $(\cF^1\to \cV_1)$ is
$h^*(\shfO_{\Gr(m+1,p)}\to V_1(C_m)\otimes\cQ)\otimes \id_{\cF^1}$.
\end{Proposition}

The proof will be given in the next subsection.

\begin{NB}
  Check:

We have
\begin{equation*}
  \begin{split}
  & \cL(C_m\boxtimes\cQ\otimes \cF^1) 
  = \det (H^1(C_m)\otimes \cQ\otimes\cF)^{\otimes m_0}
  \otimes \det (H^1(C_m(C))\otimes\cQ\otimes\cF)^{\otimes m_1}
\\
  =\; & \det \cQ^{\otimes D}\otimes\cF^{\otimes pD}
  = \shfO_{((\det\cQ)^{\otimes D})^\times/\C^*}.
  \end{split}
\end{equation*}
\end{NB}

\begin{NB}
  Since $\Fcal^1 ``=" (\det\Qcal)^{\otimes -1/p}$,
it is independent of $(m_0,m_1)$.
\end{NB}

\begin{NB}
  Here is the original version:

  I am not 100\% sure on this claim yet. 
\begin{Proposition}
  There is an \'etale morphism $H\colon \tM^{m,+}(pe_m) \to
  \Flag(\C^{m+1}\otimes\cQ/[\xi], \underline{N}\setminus\{1\})$ to the
  full flag bundle associated with $\C^{m+1}\otimes\cQ/[\xi]$ over
  $\Gr(m+1,p)$ of degree $=p(m m_0 + (m+1)m_1)$.
\end{Proposition}
Here $\C^{m+1}$ appears as $H^1(\shfO_C(-m-2))$. The proof will be
given in \subsecref{subsec:Grassmann}.
\end{NB}

\begin{NB}
  I do not yet understand whether we can say $\tM^{m,+}$ is
  the quotient of $\Flag(\C^{m+1}\otimes\cQ/[\xi],
  \underline{N}\setminus\{1\})$ by a cyclic group of order $=p(m m_0 +
  (m+1)m_1)$ with respect to the trivial action, or not. I believe
  that it is so....
\end{NB}

\subsection{Proof of \propref{prop:m,+}}\label{subsec:Grassmann}
Since $F^i$ with $i > 1$ does not appear in the stability condition,
the moduli stack has a structure of the flag bundle
$\Flag(\Vcal_1/\Fcal^1,\underline{N}\setminus\{1\})$ over
the moduli stack parametrizing $(X,F^1)$.
Here $\Vcal_1$ is a vector bundle over the moduli stack and
$\Fcal^{1}$ is its line subbundle, coming from $V_1$ and $F^{1}$
respectively.

Next we determine the moduli stack parametrizing $(X,F^1)$. Since we
already know $X\cong C_m^{\oplus p}$, the remaining parameter is only
a choice of $F^1$, which is a $1$-dimensional subspace in $V_1(X)
\cong V_1(C_m)\otimes \C^p$.
\begin{NB}
  $V_1(C_m) = H^1(C_m(C)) = H^1(\shfO_C(-m-2)) = \C^{m+1}$.
\end{NB}%
We have an action of the stabilizer $\GL_p(\C)$ of $X$. The above
stability condition means that $F^{1}$, viewed as a nonzero
homomorphism $V_1(C_m)^*\to \C^p$, is surjective. Therefore the
moduli stack is
\begin{equation}\label{eq:GLquot}
    \left(\{ \xi\in \proj(\Hom(V_1(C_m)^*,\C^p)) \mid
      \text{$\xi$ is surjective}\}
      \times \C^*\right)/\GL_p(\C),
\end{equation}
where the action of $\GL_p(\C)$ on $\C^*$ is given by $g\cdot u =
\left(\det g\right)^{D}u$ with $D = m m_0 + (m+1)m_1$. 

We consider \eqref{eq:GLquot} as
\begin{equation}
\label{eq:GLCquot}
    \left(\{ \xi\in \Hom(V_1(C_m)^*,\C^p) \mid \text{$\xi$ is
        surjective}\} \times \C^*\right)/\GL_p(\C)\times\C^*,
\end{equation}
where $\C^*\ni t$ acts by $(\xi,u)\mapsto (t\xi,u)$. If we take the
quotient by $\GL_p(\C)$ first, we get
\(
  L^\times = L\setminus (\text{$0$-section}),
\)
where $L = (\det\cQ)^{\otimes D}$, the $D^{\mathrm{th}}$ tensor power
of the determinant line bundle $\det\cQ$ of the universal quotient
over the $\Gr(m+1,p)$. Since $\C^*\ni t$ acts by
\begin{equation}
  \label{eq:twist}
   (t\xi,u) = t\id_{\C^p}\cdot (\xi, t^{-pD}u),
\end{equation}
it is the fiber-wise multiplication with weight $-pD$ on the quotient
$L^\times$.
Thus the moduli stack parametrizing $(X,F^1)$ is isomorphic to the
quotient stack $[L^\times/\C^*]$.

This action factors through $\rho\colon\C^*\to \C^*$ given by
$t\mapsto s = t^{-pD}$, where the latter action is free on $L^\times$ and
the quotient is $\Gr(m+1,p)$. Let us denote the latter $\C^*$ by
$\C^*_s$. Then the stack $L^\times/\C^*_s$ is
represented by $\Gr(m+1,p)$, as $L^\times$ is a principal
$\C^*_s$-bundle. Since a $\C^*$-bundle induces a
$\C^*_s$-bundle by taking the quotient by $\Ker\rho\cong
\Z/pD$, we have a morphism $H\colon L^\times/\C^*\to
L^\times/\C^*_s = \Gr(m+1,p)$ between stacks.
\begin{NB}
(See \cite[Lemma~3.4]{Gomez}.)  
\end{NB}%
It is \'etale and finite of degree $1/pD$.

\begin{NB}
  Kota's message on Jan. 8:

For a $S$-flat family ${\cal E}$ of coherent sheaves
on $X \times S$ such that
${\cal E}_s={\cal O}_C(-m-1)^{\oplus p}$,
we set
$$
{\cal L}({\cal E}):=
\det p_{S!}({\cal E}\otimes(-m_0{\cal O}_X-m_1{\cal O}_X(C))).
$$
For a line bundle $M$ on $S$,
$P(M)$ denotes the ${\Bbb C}^{\times}$-bundle
associated to $M$.
Then we have a tautological isomorphism
${\cal O} \to M$ on $P(M)$.
Let $H^1({\cal O}_C(-m-2))^* \to {\cal U}$ be the universal quotient
on $Gr(m+1,p)$.
For ${\cal L}:={\cal L}({\cal O}_C(-m-1) \otimes {\cal U})$,
we consider $P({\cal L}) \to Gr(m+1,p)$.
We set
$$
n:=\rank(p_{S!}({\cal O}_C(-m-1) \otimes {\cal U}
\otimes(-m_0{\cal O}_X-m_1{\cal O}_X(C)))
\begin{NB2}
  = pD
\end{NB2}%
.
$$
We have an action of ${\Bbb C}^{\times}$
on the fiber of $P({\cal L}) \to Gr(m+1,p)$
as multiplication by $n$-th power of constant.
Let $[P({\cal L})/{\Bbb C}^{\times}]$ be the
quotient stack by this action.
\begin{NB2}
  That is the category whose objects are principal $\C^*$-bundle
  $\pi\colon P\to B$ together with $\C^\times$-equivariant morphisms
  $P\to P({\cal L})$.
\end{NB2}%

Let $({\cal E},\phi,L)$
be a triple of a $S$-flat family ${\cal E}$, an isomorphism
$\phi:{\cal O}_S \to {\cal L}({\cal E})$
and a subline bundle $L \subset R^1 p_{S*}({\cal E}(C))$.
Since
${\cal E} \cong
{\cal O}_C(-m-1)\boxtimes
p_{S*}({\cal E}((m+1)C))$,
\begin{NB2}
  we have
\end{NB2}%
$R^1 p_{S*}({\cal E}(C)) \cong
H^1({\cal O}_C(-m-2)) \otimes
p_{S*}({\cal E}((m+1)C))$.


By the isomorphism ${\cal O} \to L$ on
$P(L)$,
we have a morphism
$P(L) \to Gr(m+1,p)$ such that
${\cal O} \to L \to R^1 p_{S*}({\cal E}(C))$
is the pullback of ${\cal O} \to H^1({\cal O}_C(-m-2)) \otimes {\cal U}$.
This morphism is decomposed to $P(S) \to S \to Gr(m+1,p)$
and we have a commutative diagram
\begin{equation}
\begin{CD}
P(L) @>{g}>> P({\cal L})\\
@V{\pi}VV @VVV \\
S @>{f}>> Gr(m+1,p),
\end{CD}
\end{equation}
where
$\pi^*(f^*({\cal O}_C(-m-1) \otimes {\cal U})) \cong {\cal E}$
and
$g^*({\cal O}\to {\cal L})={\cal O}\to {\cal L}({\cal E})$.
$g$ is ${\Bbb C}^{\times}$-equivariant and we have a morphism
$S \to [P({\cal L})/{\Bbb C}^{\times}]$.
\end{NB}

Let us identify the pair $(\Fcal^1\subset \Vcal_1)$ of the vector bundle
$\Vcal_1$ and its line subbundle $\Fcal^1$ over the moduli stack in
the description $[L^\times/\C^*]$.
In the description \eqref{eq:GLquot}, it is the descent of the 
restriction of
\(
   \left(
    \shfO_{\proj}(-1)\boxtimes \shfO_{\C^*} \subset 
    V_1(C_m)\otimes\C^p\otimes \shfO_{\proj}\boxtimes \shfO_{\C^*}
   \right) 
\)
with respect to the natural $\GL_p(\C)$-action, where
$\proj = \proj(\Hom(V_1(C_m)^*,\C^p))$.
Then in the description \eqref{eq:GLCquot}, it becomes the descent of
the pair
\(
  \left(
    \shfO_V \subset V_1(C_m)\otimes\C^p\otimes\shfO_V
  \right),
\)
where $\GL_p(\C)$ acts naturally and the $\C^*$-action is twisted by
the weight $-1$ action on the first factor $\shfO_V$. Here
\(
   V = \{ \xi\in \Hom(V_1(C_m)^*,\C^p) \mid \text{$\xi$ is
        surjective}\} \times \C^*.
\)
\begin{NB}
  The convention is as follows: The tautological line bundle $L =
  \shfO(-1)$ over $\proj^n = \C^{n+1}\setminus\{0\}/\C^*$ is given by
\[
   L = \C^{n+1}\setminus\{0\}\times_{\C^*} \C
   = \{ (z, \zeta) \in \C^{n+1}\setminus\{0\}\times\C\} /
    (z,\zeta) \sim (\lambda z, \lambda^{-1} \zeta).
\]
The inclusion $L\subset \shfO^{\oplus n+1}$ is given by
$(z,\zeta)\mapsto z\zeta$. This is well-defined as $(\lambda
z)(\lambda^{-1}\zeta) = z\zeta$. 

We have an isomorphism $L^\times = \C^{n+1}\setminus
\{0\}\times_{\C^*}\C^* \cong \C^{n+1}\setminus \{0\}$ by
$[(z,\zeta)]\mapsto z\zeta$. The fiber-wise $\C^*$-action by
$[(z,\zeta)]\mapsto [(z,\lambda \zeta)]$ is identified with
$\C^{n+1}\setminus\{0\}\ni (z\zeta)\mapsto \lambda (z\zeta)$.  If we
twist the trivial line bundle $\shfO_{L^\times}$ by the weight $-1$
action, it descends to $L^\times\times_{\C^*}\C = L$ by the tautology.
\end{NB}
Finally in the description $[L^\times/\C^*]$, it is the descent of
\begin{equation*}
  \left(
    \shfO_{L^\times} \subset \pi_G^*(V_1(C_m)\otimes\Qcal)
  \right),
\end{equation*}
where the $\C^*$-action is twisted by weight $-1$ on both
$\shfO_{L^\times}$ and $\pi_G^*(V_1(C_m)\otimes\Qcal$.
Here $\pi_G\colon L^\times\to \Gr(m+1,p)$ is the projection.
\begin{NB}
  Recall that $\C^*$ acts on $L^\times$ with the weight $-pD$. If we
  replace it by the weight $1$ action, the descent of
  $\shfO_{L^\times}$ with weight $-1$ is identified with $L$ on
  $[L^\times/\C^*] = \Gr(m+1,p)$, as we remarked above.
  In this sense, it may be reasonable to
  write $\shfO_{L^\times}$ as $\pi_G^* L$.
\end{NB}%
The twist on the second factor $\pi_G^*(V_1(C_m)\otimes\Qcal)$ comes
from the term $t\id_{\C^p}$ in \eqref{eq:twist}.

From the above description of $\Vcal_1$, we have $\Vcal_1 \otimes
(\cF^1)^*= h^* (V_1(C_m) \otimes\Qcal)$. On the other hand,
$(\Fcal^1)^{\otimes -pD}$ is the descent of $\shfO_{L^\times}$ with
the $\C^*$-action twisted by weight $pD$. The action factors through
the $\C^*_s$-action, and it is twisted by weight $-1$. Therefore it
descends to $L$ on $\Gr(m+1,p)$.

\subsection{$2$-stability condition}
This subsection is devoted to preliminaries for a study of enhanced
master spaces.

We consider the following condition on $\bn$:
\begin{equation}
  \label{eq:cond2}
  \sum_{i=1}^{N} k_i n_i \neq 0 \quad
  \text{for any $(k_1,\dots,k_{N})\in \Z^{N}\setminus \{0\}$
    with $|k_i|\le 2N^2$}.
\end{equation}
Our flag $F^\bullet$ of $V_1$ again may have repetition, but assume
$\dim (F^i/F^{i-1}) = 0$ or $1$ as before.

\begin{Lemma}\label{lem:strictlysemistable}
  Assume that $\zeta$ satisfies $(m+1)\zeta_0+(m+2)\zeta_1 < 0$,
  $(m-1)\zeta_0 + m\zeta_1 > 0$, and $\bn$ satisfies \eqref{eq:cond2}.
  If $(X,F^\bullet)$ is strictly $(\zeta,\bn)$-semistable, then there
  exists a submodule $0\neq S\subsetneq X$ such that
  \begin{enumerate}
  \item $\mu_{\zeta,\bn}(S,S_1\cap F^\bullet) = \mu_{\zeta,\bn}(X,F^\bullet)$,
  \item $(S,S_1\cap F^\bullet)$ and $(X/S, F^\bullet/S_1\cap F^\bullet)$
    are $(\zeta,\bn)$-stable.
  \end{enumerate}
  Moreover the submodule $S$ is unique except when $(X,F^\bullet)$ is
  the direct sum $(S,S_1\cap F^\bullet)\oplus (X/S, F^\bullet/S_1\cap
  F^\bullet)$. In this case the another choice of the submodule is
  $X/S$.
\end{Lemma}

\begin{proof}
  Take a submodule $S$ violating the $(\zeta,\bn)$-stability of
  $X$. Then we have (1). Moreover $(S,S_1\cap F^\bullet)$ and $(X/S,
  F^\bullet/S_1\cap F^\bullet)$ are $(\zeta,\bn)$-semistable. We have
  either $S_\infty = 0$ or $(X/S)_\infty = 0$.

  Assume either $(S,S_1\cap F^\bullet)$ or $(X/S, F^\bullet/S_1\cap
  F^\bullet)$ is strictly $(\zeta,\bn)$-semistable. Then we have a
  filtration $0 = X^0 \subsetneq X^1 \subsetneq X^2 \subsetneq X^3 =
  X$ with $\mu_{\zeta,\bn}(X^a/X^{a-1},F_a^\bullet) =
  \mu_{\zeta,\bn}(X,F^\bullet)$ for $a=1,2,3$, where $F_a^\bullet$
  denote the induced filtration on $X^a/X^{a-1}$ from $F^\bullet$.

  Among $X^a/X^{a-1}$ ($a=1,2,3$), one of them has $\C$ and two of
  them have $0$ at the $\infty$-component. Assume $X^1$ has $\C$ at
  the $\infty$-component for brevity, as the following argument can be
  applied to the remaining cases.

  We have $\vdim X^2/X^1 = p_2(m,m+1,0)$, $\vdim X^3/X^2 =
  p_3(m,m+1,0)$ for some $p_2$, $p_3\in \Z_{>0}$. 
  Then
  $\mu_{\zeta,\bn}(X^2/X^1,F^\bullet_2) =
  \mu_{\zeta,\bn}(X^3/X^2,F^\bullet_3)$
  implies
\begin{equation*}
   \sum n_i \left(\frac{\dim(F^i_2)}{p_2}
     - \frac{\dim(F^i_3)}{p_3}\right) = 0.
\end{equation*}
By the assumption \eqref{eq:cond2} we have
\( 
  p_3 \dim(F^i_2) = p_2\dim(F^i_3)
\)
and hence
\begin{equation}
  \label{eq:hosi}
  p_3 \dim(F^i_2/F^{i-1}_2)
     = p_2\dim(F^i_3/F^{i-1}_3)
\end{equation}
for any $i$. On the other hand, we have
\begin{equation*}
  \dim (F^i_1/F^{i-1}_1)
  + \dim(F^i_2/F^{i-1}_2)
  + \dim(F^i_3/F^{i-1}_3) = \dim (F^i/F^{i-1}) = 0 \text{ or } 1
\end{equation*}
for any $i$. Therefore at most one of three terms in the left hand
side can be $1$ and the other terms are $0$. Combined with
\eqref{eq:hosi} this implies $\dim(F^i_2/F^{i-1}_2) = \dim
(F^i_3/F^{i-1}_3) = 0$ for any $i$. Thus we get a contradiction $X^1 =
X^2 = X^3$.

If we have another submodule $S'$ of the same property, the
$(\zeta,\bn)$-stability implies $S\cap S' = 0$ or $S\cap S' = S =
S'$. In the former case we have $S' = X/S$.
\begin{NB}
  Consider $S\cap S'\subset S$ and $\subset S'$, $S/S\cap S'\subset
  X/S'$ and $S'/S\cap S'\subset X/S$.
\end{NB}%
\end{proof}

\begin{Lemma}[\protect{\cite[4.3.9]{Moc}}]\label{lem:stabilizer} Let
  $(\zeta,\bn)$ as in \lemref{lem:strictlysemistable}. If
  $(X,F^\bullet)$ is $(\zeta,\bn)$-semistable, its stabilizer is
  either trivial or $\C^*$. In the latter case, $(X,F^\bullet)$ has
  the unique decomposition $(\Xf,\Ff^\bullet)\oplus (\Xs,\Fs^\bullet)$
  such that both $(\Xf,\Ff^\bullet)$, $(\Xs,\Fs^\bullet)$ are
  $(\zeta,\bn)$-stable, and $\mu_{(\zeta,\bn)}(\Xf) =
  \mu_{(\zeta,\bn)}(\Xs)$. The stabilizer comes from that of the
  factor $(\Xs,\Fs^\bullet)$ with $(\Xs)_\infty = 0$.
\end{Lemma}

\begin{proof}
  Suppose $g$ stabilizes $(X,F^\bullet)$. If $g$ has an eigenvalue
  $\lambda\neq 1$, then we have the generalized eigenspace
  decomposition $(X,F^\bullet) = (\Xf,\Ff^\bullet)\oplus
  (\Xs,\Fs^\bullet)$ with $(\Xf)_\infty = \C$, $(\Xs)_\infty = 0$.  By
  \lemref{lem:strictlysemistable} $(\Xf,\Ff^\bullet)$,
  $(\Xs,\Fs^\bullet)$ are $(\zeta,\bn)$-stable. Since they have the
  same $\mu_{\zeta,\bn}$ and are not isomorphic, there are no nonzero
  homomorphisms between them. Therefore the stabilizer is $\C^*$ in
  this case. The uniqueness follows from that in
  \lemref{lem:strictlysemistable}.

  Next suppose $g$ is unipotent and let $n= g - 1$. Assume $n\neq 0$
  and let $j$ such that $n^j \neq 0$, $n^{j+1} = 0$. We consider
  the submodule $0\neq \Ker n^j \subsetneq X$. From the
  $(\zeta,\bn)$-semistability of $(X,F^\bullet)$ we have
\(
  \mu_{\zeta,\bn}(\Ker n^j,F^\bullet\cap (\Ker n^j)_1)
  = \mu_{\zeta,\bn}(X,F^\bullet).
\)
\begin{NB}
From the proof of \lemref{lem:isom}.
\end{NB}
Therefore $(\Ker n^j,F^\bullet\cap (\Ker n^j)_1)$ and $(X/\Ker n^j,
F^\bullet/F^\bullet\cap (\Ker n^j)_1)$ are $(\zeta,\bn)$-stable by
\lemref{lem:strictlysemistable}. They are not isomorphic since they
have different $\infty$-components. However $n^j\colon X/\Ker n^j\to
\Ker n^j$ is a nonzero homomorphism, and we have a contradiction by
\lemref{lem:isom}.
\end{proof}

\subsection{Enhanced master space}

We continue to fix $c$, $m\in\Z_{\ge 0}$, $\ell\in\underline{N}$. As
we mentioned above, $\bM^m$ is constructed as a GIT quotient of a
common space $Q$ independent of $m$. Then the moduli schemes
$\tM^{m,0}$ and $\tM^{m,\ell}$ will be also constructed as GIT
quotients of
\(
  \widetilde Q = Q\times \Flag(V_1,\underline{N})
\)
by the action of the group $G$ with respect to a common polarization,
but with different lifts of the action. Here $V_1$ is a vector space
of dimension $N$, on which $G$ acts naturally.
And $\Flag(V_1, \underline{N})$ be the variety of full flags in $V_1$.
Let us denote by $L_-$ and $L_+$ the corresponding equivariant line
bundles over $\widetilde Q$ to define $\tM^{m,0}$ and $\tM^{m,\ell}$
respectively. Their descents will be denoted by the same notations.

We consider the projective bundle $\proj(L_-\oplus L_+)\to \widetilde
Q$ with the canonical polarization $\shfO_\proj(1)$.
Here $\proj(L_-\oplus L_+)$ is the space of $1$-dimensional {\it
    quotients\/} of $L_-\oplus L_+$.
\begin{NB}
  The convention of \cite[3.1]{Thaddeus} is as follows:
  The space of sections of $\shfO_\proj(n)$ is $\bigoplus_{j=0}^n
  H^0(L_-^{\otimes j})\otimes H^0(L_+^{\otimes n-j})$.
  Thus $\proj(L_-\oplus L_+)$ is the space of $1$-dimensional {\it
    quotients\/} of $L_-\oplus L_+$.

  I still need to check the definition of the master space. Probably I
  need to change the definition later.
\end{NB}%
We have the natural lifts of the $G$-action to $\proj(L_-\oplus L_+)$
and $\shfO_\proj(1)$. Let
\begin{equation}\label{eq:masterspace}
   \cM \equiv \cM(c) \equiv \cM^{m,\ell}(c)
   \defeq \proj(L_-\oplus L_+)^{\mathrm{ss}}/G
\end{equation}
be the quotient stack of $\shfO_\proj(1)$-semistable objects of
$\proj(L_-\oplus L_+)$ divided by $G$, where $\proj(L_-\oplus
L_+)^{\mathrm{ss}}$ denotes the semistable locus. This is called the
{\it enhanced master space}. This space was introduced in
\cite{Thaddeus} to study the change of GIT quotients under the change
of linearizations.
We have an inclusion $\cM_a \defeq \proj(L_a)^{\mathrm{ss}}/ G
\to \cM$ for $a = \pm$.

The tautological flag of vector bundles over
$\Flag(V_1,\underline{N})$ descends to $\cM$. We denote it by
$\Fcal^\bullet = (0 = \Fcal^0\subset \Fcal^1 \subset \cdots
\subset \Fcal^{N-1}\subset\Fcal^N = \Vcal_1)$.
We also have the universal sheaf $\cE$ over $\bp\times \cM$.

We have a natural $\C^*$-action on $\proj(L_-\oplus L_+)$ given by
$t\cdot[z_-:z_+] = [tz_-:z_+]$ where $[z_-:z_+]$ is the homogeneous
coordinates system of $\proj(L_-\oplus L_+)$ along fibers. It descends
to a $\C^*$-action on $\cM$. We have a natural $\C^*$-equivariant
structure on the universal family $\cE$, $\Fcal^\bullet$.

The following summarizes properties of $\cM$.

\begin{Theorem}\label{thm:master}
  \textup{(1)}
    $\cM$ is a smooth
    Deligne-Mumford stack. There is a projective morphism $\cM\to M_0(p_*(c))$.
    \begin{NB2}
      The second statement added according to the referee's suggestion.
      2010/04/05
    \end{NB2}%

    \textup{(2)}
    The fixed point set of the $\C^*$-action decomposes as
    \begin{equation*}
      \cM^{\C^*} = \cM_+\sqcup \cM_- \sqcup
      \bigsqcup_{\fI\in \mathcal D^{m,\ell}(c)} \cM^{\C^*}(\fI),
    \end{equation*}
    and we have isomorphisms $\cM_+\cong \tM^{m,\ell}(c)$, $\cM_-\cong
    \tM^{m,0}(c)$.
    The universal family $(\cE, \Fcal^\bullet)$ on $\cM$ is restricted
    to ones on $\cM_+ \cong \tM^{m,\ell}(c)$ and $\cM_-\cong
    \tM^{m,0}(c)$ \textup(which were denoted by the same notation
    $(\cE,\Fcal^\bullet)$\textup). And the restriction of the
    $\C^*$-equivariant structure are trivial.

    \textup{(3)} There is a diagram
    \begin{equation}\label{eq:exceptfixed}
  \xymatrix@R=.5pc{ & \mathcal S(\fI)
    \ar[rd]!UL^>>{G} \ar[ld]_{F} & 
\\
\cM^{\C^*}(\fI) & & \tM^{m,\min(\Is)-1}(c_\flat)
      \times \tM^{m,+}(c_\sharp),
}
\end{equation}
where $\mathcal S(\fI)$ is a smooth Deligne-Mumford stack and both
$F$, $G$ are \'etale and finite of degree $1/pD$.
There is a line bundle $L_{\mathcal S}$ over $\mathcal S(\fI)$
with
\(
  L_{\mathcal S}^{\otimes pD} = G^*(\cL(\cE_\flat)^*)
\)
and the restriction of the universal family $(\cE,\Fcal^\bullet)$ over
$\cM$ and the universal families $(\cE_\flat,\Fcal_\flat^\bullet)$,
$(\cE_\sharp,\Fcal_\sharp^\bullet)$ over
$\tM^{m,\min(\Is)-1}(c_\flat)$, $\tM^{m,+}(c_\sharp)$ are related by
\begin{gather*}
   F^* \cE
   \cong G^*(\cE_\flat)\oplus G^*(\cE_\sharp)\otimes L_{\mathcal S},
\qquad
   F^* \Fcal^\bullet \cong G^*(\Fcal_\flat^\bullet)\oplus 
   G^*(\Fcal_\sharp^\bullet)\otimes L_{\mathcal S}.
\end{gather*}
Moreover the restriction of the $\C^*$-equivariant structure on the
universal family $(\cE, \Fcal)$ is trivial on the factor
$(\cE_\flat,\Fcal_\flat)$ and of weight $1/pD$ on
$(\cE_\sharp,\Fcal_\sharp)$ under the above identification.
\end{Theorem}

\begin{NB}
(1)  Since $L_{\mathcal S}
  ``=" (\det R^1q_{2*}(\cE_\flat)^{\otimes -m_0}\otimes
        \det R^1q_{2*}(\cE_\flat\otimes q_1^*(\shfO(C-\linf)))^{\otimes -m_1})^{
          1/pD}$,
it does not change under the replacement $(m_0,m_1)\mapsto (Nm_0,Nm_1)$.

But this is not enough to see that it is independent of $(m_0,m_1)$.

(2) We have
\begin{equation*}
  \mathcal L(F^*(\cE))
   \cong \mathcal L(G^*(\cE_\flat))\otimes
   \mathcal L(G^*(\cE_\sharp)\otimes L_{\mathcal S})
   \cong G^*(\mathcal L(\cE_\flat))\otimes
   G^*(\mathcal L(\cE_\sharp)) \otimes L_{\mathcal S}^{\otimes pD}
   \cong\shfO_{\mathcal S(\fI)},
\end{equation*}
since $\mathcal L(\cE_\sharp)$ is trivial from the definition of
$\tM^{m,+}(c_\sharp)$ and $L_{\mathcal S}^{\otimes pD} =
G^*(\cL(\cE_\flat)^*)$.  This is compatible with the fact that
$\mathcal L(\cE)$ is trivial over the exceptional fixed point
$\cM^{\C^*}(\fI)$.
\end{NB}

We need to explain some notations:
\begin{itemize}
\item $\mathcal D^{m,\ell}(c)$ is the set of {\it decomposition types\/}:
\begin{equation}\label{eq:Sc}
  \mathcal D^{m,\ell}(c)
  \defeq \left\{ \fI = (\If,\Is)
  \left|\,
  \begin{aligned}[c]
  & \underline{N} = \If\sqcup \Is,\ \If,\Is\neq \emptyset,
\\
  & |\Is| = p(m+1) \text{ for $p\in \Z_{>0}$},
  \min(\Is)\le \ell
  \end{aligned}
  \right.\right\}.
\end{equation}
For $\fI\in\mathcal D^{m,\ell}(c)$, we set $c_\sharp = p e_m$, $c_\flat
= c - c_\sharp\in H^*(\bp)$.

\item integers $(m_0,m_1)$ appeared in the definition of an orientation
of a $(m,+)$-stable object (see \defref{def:m,+stable}) will be
determined by the choice of $L_\pm$.
\end{itemize}

The isomorphism 
\(
   F^* \Fcal^\bullet \cong G^*(\Fcal_\flat^\bullet)\oplus 
   G^*(\Fcal_\sharp^\bullet)\otimes L_{\mathcal S}
\)
of universal flags in (3) means as follows:
From the first statement
\(
   F^* \cE
   \cong G^*(\cE_\flat)\oplus G^*(\cE_\sharp)\otimes L_{\mathcal S}
\)
we have a decomposition
\(
  F^*(\Vcal_1(\cE))
  = G^*(\Vcal_1(\cE_\flat))
  \oplus G^*(\Vcal_1(\cE_\sharp)) \otimes L_{\mathcal S}.
\)
Then we have
$F^*(\Fcal^i) = G^*(\Fcal^i_\flat)\oplus G^*(\Fcal^i_\sharp)\otimes L_{\mathcal S}$
where $\Fcal^\bullet_\flat$, $\Fcal^\bullet_\sharp$ are flags indexed
by $\underline{N}$.
If we forget irrelevant factors
$\Fcal^i_\flat$ with $\Fcal^i_\flat = \Fcal^{i-1}_\flat$ and
$\Fcal^j_\sharp$ with $\Fcal^j_\sharp = \Fcal^{j-1}_\sharp$, we get
the universal flags over $\tM^{m,\min(\Is)-1}(c_\flat)$, $\tM^{m,+}(c_\sharp)$.
The above sets $\If$, $\Is$ consist of indexes of relevant factors.
In particular, $(\Fcal^1\subset \Vcal_1)$ appearing in
\propref{prop:m,+} is identified with $(\Fcal^{\min(\Is)}_\sharp\subset
\Fcal^{\max(\Is)}_\sharp = \Fcal^N_\sharp)$. Let us denote
$\Fcal^N_\sharp$ by $\Vcal_1^\sharp$ hereafter.

The fixed point substack $\cM^{\C^*}$ is defined as the zero locus of
the fundamental vector field generated by the $\C^*$-action. Note that
this does not imply that the action of $\C^*$ is trivial on
$\cM^{\C^*}$, but becomes trivial on the finite cover $\C^*_s =
\operatorname{Spec}\C[s,s^{-1}] \to \C^*; s\mapsto s^{pD} = t$.
\begin{NB}
  This is a general result. See the reference
  \cite[Appendix~C]{GP}. But we can see it directly.
\end{NB}%
Therefore the restriction of a $\C^*$-equivariant sheaf to the fixed
point locus is a sheaf tensored by a $\C^*_s$-module. In the
statements (2),(3), we wrote the weights of $\C^*_s$-modules divided
by $pD$, considered formally as weights of {\it rational\/}
$\C^*$-modules.

The proofs of (1),(2),(3) will be given in
\subsecref{subsec:mastersmooth},
\subsecref{subsec:C^*-action},
\subsecref{subsec:decomp}
respectively.

\begin{Remark}
  We can define the master space connecting $\bM^m(c)$ and
  $\bM^{m+1}(c)$ in the same way. However it will be not necessarily a
  Deligne-Mumford stack as a semistable point possibly has a
  stabilizer of large dimension. This is the reason why we, following
  Mochizuki, consider pairs of framed sheaves and flags in cohomology
  groups.
\end{Remark}

\subsection{Smoothness of the Enhanced Master
  Space}\label{subsec:mastersmooth}

\begin{NB}
  The following is a repetition.

We take $N = \dim V_1$ so that $F^\bullet$ is a full flag of
$V_1$. Let $\Flag(V_1, \underline{N})$ be the variety of full flags in
$V_1$.
\end{NB}

Let us write the enhanced master space in the quiver description.
We first take $\zeta^-$ sufficiently close to $\zeta^0$.
For $l=1,\dots, N$, we choose $(\zeta,\bn)$ satisfying
(\ref{eq:cond},\ref{eq:cond2}). We take $\zeta$ so that $|\zeta -
\zeta^0|$, $|\bn|$ are sufficiently smaller than $|\zeta - \zeta^-|$.
We take a large number $k$ so that $k(\zeta^-,\bn)$ and $k(\zeta,\bn)$
are integral. Let $L_-$ (resp.\ $L_+$) be the $G$-equivariant line
bundle over $\mu^{-1}(0)\times\Flag(V_1,\underline{N})$ corresponding
to the stability condition $k(\zeta^-,\bn)$ (resp.\ $k(\zeta,\bn)$). We
consider the projective bundle $\proj(L_-\oplus L_+)$ with the
canonical polarization $\shfO_\proj(1)$.
\begin{NB}
  I still need to understand the definition of the master
  space. Probably I need to change the definition later.
\end{NB}%
We have the natural lifts of the $G$-action to $\proj(L_-\oplus L_+)$
and $\shfO_\proj(1)$. Let
\begin{equation*}
   \cM \equiv \cM(c)
   \defeq \proj(L_-\oplus L_+)^{\mathrm{ss}} / G
\end{equation*}
be the quotient stack of the semistable locus $\proj(L_-\oplus
L_+)^{\mathrm{ss}}$ divided by $G$.
\begin{NB}
The old version:
the GIT quotient $\proj(L_-\oplus L_+)$ with respect
$\shfO_\proj(1)$, where $\proj(L_-\oplus L_+)^{\mathrm{ss}}$ denotes
the semistable locus.   
\end{NB}%
\begin{NB}
  The following is a repetition.

This is called the {\it enhanced master space}.
This space was introduced in \cite{Thaddeus} to study the change of
GIT quotients under the change of linearizations.
\end{NB}%

\begin{NB}
The action is given by
\begin{equation*}
  \begin{split}
    g\cdot \left((B_1,B_2,d,i,j),[z_-:z_+]\right)
   &= \left((g_0 B_1 g_1^{-1}, g_0 B_2 g_1^{-1}, g_1 d g_0^{-1},
     g_0 i, j g_1^{-1}), [ \chi_{\zeta^-}(g) z_-: \chi_{\zeta}(g) z_+ ]
     \right)
   \\
   &= \left((g_0 B_1 g_1^{-1}, g_0 B_2 g_1^{-1}, g_1 d g_0^{-1},
     g_0 i, j g_1^{-1}), [ z_-: \chi_{\zeta}(g) \chi_{\zeta^-}(g)^{-1}z_+ ]
     \right).
  \end{split}
\end{equation*}
\end{NB}


The following was shown in e.g., \cite[\S\S3,4]{Thaddeus}.

\begin{Lemma}\label{lem:Th}
  A point $x$ of $\proj(L_-\oplus L_+) \setminus \left(
    \proj(L_-)\sqcup \proj(L_+)\right)$ is semistable if and only if
  the corresponding $(X,F^\bullet)$ is semistable with respect to a
  $\Q$-line bundle $L_t = L_-^{\otimes (1-t)}\otimes L_+^{\otimes t}$
  for some $t\in [0,1]\cap\Q$.
\end{Lemma}

\begin{Proposition}\label{prop:mastersmooth}
  $\cM$ is a smooth Deligne-Mumford stack.
\end{Proposition}

\begin{proof}
  Let $x$ be a semistable point in $\proj(L_-\oplus L_+)$. Then the
  corresponding point $(X,F^\bullet)$ in
  $\mu^{-1}(0)\times\Flag(V_1,\underline{N})$ is $(\zeta',\bn)$-stable
  for some $\zeta'$ on the segment connecting $\zeta$ and $\zeta^-$
  (\lemref{lem:Th}). We can apply \lemref{lem:stabilizer} as $\zeta'$
  satisfies $(m+1)\zeta_0+(m+2)\zeta_1 < 0$, $(m-1)\zeta_0 + m\zeta_1
  > 0$, and $\bn$ satisfies \eqref{eq:cond2}. Therefore either the
  stabilizer of $(X,F^\bullet)$ is trivial or $(X,F^\bullet)$
  decomposes as $(\Xf,\Ff^\bullet)\oplus (\Xs,\Fs^\bullet)$. Since
  $(\Xs)_\infty = 0$, $\Xs \cong C_m^{\oplus p}$ for some
  $p\in\Z_{>0}$ as explained in \defref{def:m,m+1-stable}. In this
  case the stabilizer is $\C^*$, coming from the automorphisms of
  $(\Xs,\Fs^\bullet)$. Its action on the fiber is given by $t\cdot u =
  t^{pk (m,m+1)\cdot (\zeta - \zeta^-)}u$ for $t\in \C^*$. Therefore
  $x$ only has a finite stabilizer. It is also reduced as the base
  field is of characteristic $0$. Therefore $\cM$ is Deligne-Mumford.
  Since $\proj(L_-\oplus L_+)$ is smooth, $\cM$ is also smooth.
  \begin{NB}
    Kota, do you have a good reference ?
    
    Ans (e-mail, Oct. 17). This is trivial.
  \end{NB}%
\end{proof}

We set
\(
   (m_0,m_1)\defeq k(\zeta - \zeta^-) 
\)
(and hence
\(
  D = k(m,m+1)\cdot (\zeta - \zeta^-))
\)
which was used in \thmref{thm:master}.
From our choices of $\zeta$, $\zeta^-$, we have $D > 0$.

\subsection{$\C^*$-action}\label{subsec:C^*-action}
We have a natural $\C^*$-action on $\proj(L_-\oplus L_+)$ given by
$t\cdot[z_-:z_+] = [tz_-:z_+]$ where $[z_-:z_+]$ is the homogeneous
coordinates system of $\proj(L_-\oplus L_+)$ along fibers. It descends
to a $\C^*$-action on $\cM$, as it commutes with the $G$-action.
Letting $\C^*$ act trivially on $V_0$, $V_1$ and the universal flag
over $\Flag(V_1,\underline{N})$, we have the $\C^*$-equivariant
structure on the universal family $\cE$, $\Fcal^\bullet = (0 = \Fcal^0
\subset \Fcal^1 \subset \cdots \subset \Fcal^{N-1} \subset \Fcal^N =
\Vcal_1)$.

The fixed point substack $\cM^{\C^*}$ is defined as the zero locus of
the fundamental vector field generated by the $\C^*$-action. Note that
this does not imply that the action of $\C^*$ is trivial on
$\cM^{\C^*}$. The action becomes trivial after a finite cover
$\C^*\to \C^*$. 
\begin{NB}
  See the reference \cite[Appendix~C]{GP}.
\end{NB}

We have an inclusion $\cM_a \defeq \proj(L_a)^{\mathrm{ss}} / G
\to \cM$ for $a = \pm$. Then $\cM_a$ is a component of the fixed point
set $\cM^{\C^*}$. From the construction $\cM_a$ is the moduli stack of
objects $(X,F^\bullet)$, which are stable with respect to $L_a$. From
our choice of $(\zeta,\bn)$, we have $\cM_+ \cong \tM^{m,\ell}$ by
\lemref{lem:zeta,l}.
\begin{NB}
  We have a natural surjective homomorphism $(z_-,z_+)\colon L_-\oplus
  L_+ \to \shfO_{\proj}(1)$. On $\cM_\pm$, we have $z_\mp = 0$, and
  $L_\pm \cong \shfO_\proj(1)|_{\cM_\pm}$.

  Since $\shfO_\proj(1)_{[z_-:z_+]} = \C(z_-,z_+)$, we have
  $\cong L_a$.
\end{NB}%
Since $\bn$ is sufficiently smaller than $|\zeta - \zeta^-|$,
$(X,F^\bullet)$ is stable with respect to $L_-$ if and only if $X$ is
$\zeta^-$-stable. Thus we have $\cM_-\cong \tM^{m,0}$.

Next consider a fixed point in $\cM^{\C^*}$ other than $\cM_+\sqcup
\cM_-$. Suppose that a point $x = ((X,F^\bullet), [z_-:z_+])$ in
$\proj(L_-\oplus L_+)^{\mathrm{ss}}\setminus
(\proj(L_-)\sqcup\proj(L_+))$ is mapped to a fixed point in the
quotient $\cM$.
It means that the tangent vector generated by the $\C^*$-action at $x$
is contained in the subspace generated by the $G$-action. In view of
\lemref{lem:stabilizer} this is possible only if $(X,F^*)$ has a
nontrivial stabilizer, and hence decompose as $(\Xf,\Ff^\bullet)\oplus
(\Xs,\Fs^\bullet)$ to the direct sum of two $(\zeta',\bn)$-stable
objects with the equal $\mu_{\zeta',\bn}$ for some $\zeta'$ on the
segment connecting $\zeta$ and $\zeta^-$. (See \lemref{lem:Th} and
\lemref{lem:stabilizer}.) We number the summand so that $(\Xf)_\infty
= \C$.
Therefore $(\Xs)_\infty = 0$ and hence $\Xs \cong
C_m^{\oplus p}$ for some $p\in\Z_{>0}$.
The data $u = z_+/z_-$ corresponds to an isomorphism $L(X) \cong \C$.

Conversely suppose we have such a decomposition $(X,F^\bullet) =
(\Xf,\Ff^\bullet)\oplus(\Xs,\Fs^\bullet)$. Let $V = \Vf\oplus \Vs$ be
the corresponding decomposition of $V$.
We lift the $\C^*$-action on $\cM$ to $\proj(L_-\oplus L_+)^{\mathrm{ss}}$ by
\begin{equation}\label{eq:C*}
   ((X,F^\bullet),[z_-:z_+]) \mapsto
   (\id_{\Vf}\oplus t^{1/pD} \id_{\Vs})\cdot ((X,F^\bullet),[t z_-:z_+]),
\end{equation}
which is well-defined on the covering $\C^*\to \C^*; s\mapsto s^{pD} =
t$, and fixes $((X,F^\bullet),[z_-:z_+]) =
((\Xf,\Ff^\bullet)\oplus(\Xs,\Fs^\bullet),[z_-:z_+])$. 
\begin{NB}
Let $g = \id_{\Vf}\oplus t^{1/pD} \id_{\Vs}$. Then
\begin{equation*}
  \begin{split}
  g\cdot ((\Xf,\Ff^\bullet)
  \oplus (\Xs,\Fs^\bullet),[t z_-:z_+])
  &= ((\Xf,\Ff^\bullet)
  \oplus (\Xs,\Fs^\bullet),[t z_-: \chi_{\zeta}(g)\chi_{\zeta^-}(g)^{-1} z_+])
\\
  &= ((\Xf,\Ff^\bullet)
  \oplus (\Xs,\Fs^\bullet),[t z_-: (t^{1/pD})^{pD} z_+])
\\
  &= ((\Xf,\Ff^\bullet)
  \oplus (\Xs,\Fs^\bullet),[z_-: z_+]).
  \end{split}
\end{equation*}
\end{NB}%
Since this $\C^*$-action is equal to the original one up to the $G$-action,
it is the same on the quotient $\cM$. Therefore
the point $x = ((X,F^\bullet),[z_-:z_+])$ is mapped to a fixed point in $\cM$.

Let
\begin{equation*}
   I_\alpha \defeq \{ i\in \underline{N} \mid
   \dim (F_\alpha^i/F_\alpha^{i-1}) = 1 \}
\end{equation*}
for $\alpha=\flat$, $\sharp$. Then we have the decomposition
$\underline{N} = \If\sqcup \Is$. The datum $(\If,\Is)$ is called the
{\it decomposition type\/} of the fixed point. Since $\dim (\Xs)_1 =
p(m+1)$, we have $|\Is| = p(m+1)$.

\begin{Lemma}[\protect{\cite[4.4.3]{Moc}}]
  $\min(\Is)\le \ell$.
\end{Lemma}

\begin{proof}
  Suppose $\min(\Is) > \ell$. Then \eqref{eq:cond} implies
  $\mu_{\zeta,\bn}(\Xs) < \mu_{\zeta,\bn}(X)$. (See the proof of
  \lemref{lem:zeta,l}.) On the other hand we have
  $\mu_{\zeta^-,\bn}(\Xs) < \mu_{\zeta^-,\bn}(X)$ since $\bn$ is
  sufficiently smaller than $|\zeta - \zeta^-|$. Therefore we cannot
  have $\mu_{\zeta',\bn}(\Xs) = \mu_{\zeta',\bn}(X)$ for any $\zeta'$
  on the segment connecting $\zeta$ and $\zeta^-$.
  \begin{NB2}
    Corrected. 2010/04/02
  \end{NB2}%
  This contradicts with the assumption.
\end{proof}

\begin{NB}
Let $\mathcal D^{m,\ell}(c)$ be the set of decomposition types given in
\eqref{eq:Sc}.
\begin{equation*}
  \mathcal D^{m,\ell}(c) \defeq \left\{ \fI = (\If,\Is)
  \mid \underline{N} = \If\sqcup \Is, \If,\Is\neq \emptyset,
  |\Is| = p(m+1) \text{for $p\in \Z_{>0}$},
  \min(\Is)\le \ell
  \right\}.
\end{equation*}
The datum $\fI = (\If,\Is)$ determines the topological data of
$\Xf$, $\Xs$ as $c - pe_m$, $p e_m$. Let us denote them by $c_\flat$,
$c_\sharp$ respectively.
\end{NB}%

Conversely suppose an object $(X,F^\bullet) = (\Xf,\Ff^\bullet)\oplus
(\Xs,\Fs^\bullet)$ with the decomposition type $(\If,\Is)$ with
$\min(\Is)\le \ell$ is given. We also suppose $\Xs\cong
C_m^{\oplus p}$.
We take a point $x$ of $\proj(L_-\oplus L_+) \setminus \left(
  \proj(L_-)\sqcup \proj(L_+)\right)$ from the fiber over
$(X,F^\bullet)$.
Since we have $\mu_{\zeta,\bn}(\Xs) > \mu_{\zeta,\bn}(X)$ and
$\mu_{\zeta^-,\bn}(\Xs) < \mu_{\zeta^-,\bn}(X)$ by the same argument
as above, we can find $\zeta'$ with $\mu_{\zeta',\bn}(\Xs) =
\mu_{\zeta',\bn}(X)$. Then $x$ is semistable if and only if both
$(\Xf,\Ff)$ and $(\Xs,\Fs)$ are $(\zeta',\bn)$-stable.

\begin{Lemma}[\protect{\cite[4.4.4]{Moc}}]
  \textup{(1)} $(\Xf,\Ff)$ is $(\zeta',\bn)$-stable if and only if
  $(m,\min(\Is) -1)$-stable.
  
  \textup{(2)} $(\Xs,\Fs)$ is $(\zeta',\bn)$-stable if and only if
  $(m,+)$-stable, i.e., $\Xs\cong C_m^{\oplus p}$, and we have
  $S_1\cap \Fs^{\min(\Is)} = 0$ for any proper submodule $S\subsetneq
  \Xs$ of a form $S\cong C_m^{\oplus q}$.
\end{Lemma}

Note that $\dim \Fs^{\min(\Is)} = 1$, $\dim \Ff^{\min(\Is)-1} =
\min(\Is)-1$. So the definitions in \subsecref{subsec:withflags} apply
though $\Ff$, $\Fs$ are flags which possibly have repetitions.

\begin{proof}
  (1) Let $S\subset \Xf$ be a submodule. We need to study the
  stability inequalities when $\zeta^0\cdot\vdim S = 0$. We first
  suppose $S_\infty = 0$. Then the inequality
\(
   \mu_{(\zeta',\bn)}(S) < \mu_{(\zeta',\bn)}(\Xf) = \mu_{(\zeta',\bn)}(\Xs)
\)   
is equivalent to
\begin{equation*}
    \frac{\sum_i n_i \dim (S_1\cap \Ff^i)}{\rank S}
    < \frac{\sum_i n_i \dim \Fs^i}{\rank \Xs}
\end{equation*}
since $\zeta'\cdot \vdim S/\rank S = \zeta'\cdot \vdim \Xs/\rank \Xs
= (2m+1)^{-1}\left(m\zeta'_0+(m+1)\zeta'_1\right)$.
\begin{NB2}
  Corrected and added according to the referee's comment. 2010/04/02
\end{NB2}%
Since $n_i$ ($i \ge \min(\Is)$) is much smaller than $n_{\min(\Is)-1}$
by (\ref{eq:cond}b), we must have $S_1\cap \Ff^{\min(\Is)-1} = 0$ if
the inequality holds.
\begin{NB}
  Suppose $S_1\cap \Ff^{\min(\Is)-1}\neq 0$. Then the LHS is greater
  than or equal to
\[
    \frac{n_{\min(\Is)-1}}{\rank X}.
\]
On the other hand the right hand side is smaller than or equal to
\[
   \sum_{i=\min(\Is)}^N i n_i.
\]
Therefore we have a contradiction to (\ref{eq:cond}b).
\end{NB}
Conversely suppose $S_1\cap \Ff^{\min(\Is)-1} = 0$. Then $S_1\cap
\Ff^{\min(\Is)} = 0$. Thus the inequality holds again by (\ref{eq:cond}b).

Next suppose $S_\infty = \C$. Then the inequality
\(
   \mu_{(\zeta',\bn)}(S) < \mu_{(\zeta',\bn)}(\Xf) = \mu_{(\zeta',\bn)}(\Xs)
\)   
is equivalent to
\begin{equation*}
    \frac{\sum_i n_i \dim (\Ff^i/S_1\cap \Ff^i)}{\rank (\Xf/S)}
    > \frac{\sum_i n_i \dim \Fs^i}{\rank \Xs}.
\end{equation*}
This is equivalent to
\(
    \Ff^{\min(\Is)-1} \not\subset S_1
\)
by the same argument as above. Thus $(\Xf,\Ff)$ is
$(\zeta',\bn)$-stable if and only if $(m,\min(\Is)-1)$-stable.

(2) First note that $\Xs$ must be $\zeta^0$-semistable as $\zeta'$ is
close to $\zeta^0$ and $\bn$ is small. Then $\Xs \cong C_m^{\oplus
  p}$, as explained in \defref{def:m,m+1-stable}. To prove the
remaining part, the same argument as above works.
\end{proof}

If we rephrase what we have observed in terms of sheaves, we get
\begin{Proposition}[\protect{\cite[4.5.2]{Moc}}]\label{prop:exceptfixed}
  $\cM^{\C^*}(\fI)$ is the moduli stack of objects
$(((E_\flat,\Phi),F^\bullet_\flat),(E_\sharp,F^\bullet_\sharp),\rho)$ where
\begin{itemize}
\item $((E_\flat,\Phi),F^\bullet_\flat)$ is $(m,\min(I_\sharp)-1)$-stable,
\item $(E_\sharp,F^\bullet_\sharp)$ is $(m,+)$-stable,
\item $\rho$ is an isomorphism
$L(E_\flat\oplus E_\sharp)\xrightarrow{\cong}\C$.
\end{itemize}
Moreover the restriction of the universal family $(\cE,\Fcal^\bullet)$
on $\cM$ decomposes as 
\begin{gather*}
   \cE = {}^{\cM}\cE_\flat\oplus{}^{\cM}\cE_\sharp,
\qquad
   \Fcal^\bullet = {}^{\cM}\Fcal^\bullet_\flat\oplus {}^{\cM}\Fcal^\bullet_\sharp,
\end{gather*}
where ${}^{\cM}\Fcal^\bullet_\flat$, ${}^{\cM}\Fcal^\bullet_\sharp$
are flags $(0 = {}^{\cM}\Fcal^0_\flat\subset\cdots \subset
{}^{\cM}\Fcal^N_\flat = \Vcal_1({}^{\cM}\cE_\flat))$,
$(0 = {}^{\cM}\Fcal^0_\sharp\subset\cdots \subset
{}^{\cM}\Fcal^N_\sharp = \Vcal_1({}^{\cM}\cE_\sharp))$.

The restriction of the $\C^*$-equivariant structure on the universal
family $(\cE, \Fcal)$ is trivial on the factor
$({}^{\cM}\cE_\flat,{}^{\cM}\Fcal_\flat)$ and of weight $1/pD$ on
$({}^{\cM}\cE_\sharp,{}^{\cM}\Fcal_\sharp)$ under the above
identification.
\end{Proposition}

The last assertion follows from the description of the $\C^*$-action
at \eqref{eq:C*}.
\begin{NB}
  This is different from \cite[4.6.1]{Moc}. This is because Mochizuki
  considers the quotient by $\SL(V)$, and \eqref{eq:C*} is replaced by
  $t^{-1}$ on $\Vs$ and $t^{\rank \Vf/\rank \Vs}$ on $\Vf$. (Notation
  is different, but was appeared in \cite[4.4.1]{Moc}.)
\end{NB}

\subsection{Decomposition into product of two moduli stacks}
\label{subsec:decomp}
Let $\baM$ be the moduli stack of objects
$((E_\flat,\Phi),F^\bullet_\flat,\rho_\flat)$ where
\begin{itemize}
\item $((E_\flat,\Phi),F^\bullet_\flat)$ is $(m,\min(\Is)-1)$-stable,
\item $\rho_\flat$ is an isomorphism $L(E_\flat)
  \xrightarrow{\cong}\C$.
\end{itemize}
We have a natural projection $\baM\to \tM^{m,\min(\Is)-1}$ forgetting
$\rho_\flat$. It is a principal $\C^*$-bundle. 
\begin{NB}
  Consider the associated line bundle $L \defeq \baM \times_{\C^*} \C$
  where $\C^*$ acts $\C$ by $t\cdot z = t^{-1}z$. So $L$ consists of
  equivalence classes $[\rho_\flat,z]$, where $(\rho_\flat,z)\sim (t
  \rho_\flat, tz)$.
  We have an isomorphism $L \cong \det H^1(E_\flat)^{\otimes
    m_0}\otimes \det H^1((E_\flat)(C))^{\otimes m_1}$ given by
\(
  [\rho_\flat,z] \mapsto \rho_\flat^{-1}(z).
\)
This is well-defined as $(t\rho_\flat)^{-1}(tz) = \rho_\flat^{-1}(z)$.
\end{NB}%
On the other hand, let $\tM^{m,+}$ be the moduli stack of oriented
$(m,+)$-stable sheaves with flags as in
\subsecref{subsec:orientedsheaf}. In order to distinguish from $\rho$,
we denote the orientation by $\rho_\sharp$.
Then we have
\begin{equation*}
  \cM^{\C^*}(\fI) \cong (\baM \times \tM^{m,+})/\C^*,
\end{equation*}
where $\C^*$ acts by $\rho_\flat \mapsto t\rho_\flat$,
$\rho_\sharp\mapsto t^{-1}\rho_\sharp$. Let us take a covering
$\C^*_s\to \C^*$ $s\mapsto s^{pD}= t$. Then we have an
\'etale and finite morphism $F\colon (\baM \times
\tM^{m,+})/\C^*_s\to \cM^{\C^*}(\fI)$ of degree
$1/pD$.

The action of $\C^*_s$ on the second factor $\tM^{m,+}$ is
trivial, since it can be absorbed in the isomorphism $s^{-1}\id \colon
E_\sharp\xrightarrow{\cong} E_\sharp$ as
\begin{equation}\label{eq:diagram}
  \begin{aligned}[m]
  \xymatrix{
    E_\sharp \ar[d]^{s^{-1}\id} && 
    L(E_\sharp) 
    = \det H^1(E_\sharp)^{\otimes m_0}\otimes \det H^1((E_\sharp)(C))^{\otimes m_1}
    \ar[r]^>>>>{s^{-pD} \rho_\sharp} \ar[d]^{s^{-pD}\id} & \C \ar@2{-}[d]^{\id}
\\
    E_\sharp && 
    L(E_\sharp) 
    = \det H^1(E_\sharp)^{\otimes m_0}\otimes \det H^1((E_\sharp)(C))^{\otimes m_1}
    \ar[r]_>>>>{\rho_\sharp} & \C.
}
  \end{aligned}
\end{equation}
Therefore we have
\begin{equation*}
   (\baM \times \tM^{m,+})/\C^*_s 
   = \baM/\C^*_s \times \tM^{m,+}.
\end{equation*}
Furthermore we have an \'etale and finite morphism $G\colon
\baM/\C^*_s\to \baM/\C^*$ of degree $1/pD$. But the latter is
nothing but $\tM^{m,\min(\Is)-1}$. Hence we have the diagram in
\thmref{thm:master} with $\mathcal S(\mathfrak I) = (\baM \times
\tM^{m,+})/\C^*_s$.

From \eqref{eq:diagram} the universal sheaf $\cE_\sharp$ over
$\tM^{m,+}$ is twisted by the line bundle over $\baM/\C^*$ associated
with the representation of $\C^*$ with weight $1$.
\begin{NB}
  Here is my convention: The $G$-action on the principal $G$-bundle
  $P$ is from right, the representation $\rho$ is a left
  action. Therefore $(p,v) \sim (pg, \rho(g)^{-1}v)$ is the
  equivalence relation used to define the associated vector bundle.
\end{NB}%
It is a line bundle $L_{\mathcal S}$ such that 
\(
  L_{\mathcal S}^{\otimes pD} = G^* L(E_\flat)^*.
\)
Therefore we have
\(
   F^* \left({}^{\cM}\cE_\sharp\right) = G^*(\cE_\sharp)\otimes L_{\mathcal S}.
\)
On the other hand, we have
\(
   F^* \left({}^{\cM}\cE_\flat\right) = G^*(\cE_\flat).
\)

\subsection{Normal bundle}\label{subsec:normal}

Let us describe the normal bundles $\fN(\cM_\pm)$ of $\cM_\pm$ and
$\fN(\cM^{\C^*}(\fI))$ of $\cM^{\C^*}(\fI)$ in $\cM$
in this subsection. We need to prepare several notations.

Recall first that the covering $\C^*_s\to \C^*; s\mapsto s^{pD} = t$
acts trivially on $\cM^{\C^*}(\fI)$ (while the original $\C^*$
does not). Hence the tangent space at a fixed point has a natural
$\C^*_s$-module structure. We formally consider it as a module
structure of the original $\C^*$ dividing weights by $pD$.

Recall also that the restriction of the universal sheaf $\cE$
decomposes as ${}^\cM\cEf\oplus{}^\cM\cEs$ over $\cM^{\C^*}(\mathfrak
I)$ (see \propref{prop:exceptfixed}). Let $\Ext^\bullet_{q_2}$ denotes
the higher derived functor of the composite functor
$q_{2*}\circ{\mathcal H}om$. Let
\begin{equation}\label{eq:notation_normal}
  \begin{split}
    & \fN({}^\cM\cEf,{}^\cM\cEs) \defeq 
    - \sum_{a=0}^2 (-1)^a \Ext^a_{q_2}({}^\cM\cEf,{}^\cM\cEs),
  \end{split}
\end{equation}
where this is a class in the equivariant $K$-group of
$\cM^{\C^*}(\fI)$. We use similar notation
$\fN({}^\cM\cEs,{}^\cM\cEf)$ exchanging the first and second factors.
Later we will also use $\fN(\bullet,\bullet)$ replacing 
${}^\cM\cEs$, ${}^\cM\cEf$ by similar universal sheaves.
We have already used this notation in \thmref{thm:m=0}.


Let
\begin{equation*}
   \Flag(V_1^\alpha,I_\alpha)
   \defeq \{ \text{a flag $F_\alpha^\bullet$ 
     of $V_1^\alpha$, indexed by $\underline{N}$,
     $F^i_\alpha/F_\alpha^{i-1} = 0$ if and only if $i\notin I_\alpha$}\}
\end{equation*}
for $\alpha=\flat$, $\sharp$. We have an embedding
$\Flag(\Vf_1,\If)\times \Flag(\Vs_1,\Is)\to \Flag(V_1,\underline{N})$
given by $(\Ff^\bullet,\Fs^\bullet)\mapsto \Ff^\bullet\oplus
\Fs^\bullet$. Let $N_0$ denote its normal bundle. It has a natural
$\C^*$-equivariant structure as $\Flag(\Vf_1,\If)\times
\Flag(\Vs_1,\Is)$ is a component of $\C^*$-fixed points in
$\Flag(V_1,\underline{N})$ with respect to the $\C^*$-action induced by
$\C^*\ni t\mapsto \id_{\Vf_1}\oplus t^{1/pD}\id_{\Vs_1}\in \GL(V_1)$.
More precisely, when we write
\begin{equation*}
  N_0 = \bigoplus_{i>j} \Hom(F^i_\flat/F^{i-1}_\flat, F^j_\sharp/F^{j-1}_\sharp)
  \oplus
  \bigoplus_{i>j} \Hom(F^i_\sharp/F^{i-1}_\sharp, F^j_\flat/F^{j-1}_\flat),
\end{equation*}
the first term has weight $1/pD$ and the second term has weight $-1/pD$.
We have an associated vector bundle, denoted also by $N_0$, over
$\cM^{\C^*}(\fI)$, induced from the flag bundle structure
$\widetilde Q/G\to Q/G$ between quotient stacks.

\begin{Theorem}\label{thm:normalbundle}
  \textup{(1)} The normal bundle $\fN(\cM_\pm)$ of $\cM_\pm$ is
  $L_\mp^*\otimes L_\pm$ with the $\C^*$-action of weight $\pm 1$.

  \textup{(2)} The normal bundle $\fN(\cM^{\C^*}(\fI))$ of
  $\cM^{\C^*}(\fI)$ is equivariant $K$-theoretically given by
  \begin{equation*}
    N_0 + \fN({}^\cM\cEf,{}^\cM\cEs)\otimes I_{1/pD}
    + \fN({}^\cM\cEs,{}^\cM\cEf) 
    \otimes I_{-1/pD},
  \end{equation*}
  where $I_n$ denotes the trivial line bundle over
  $\cM^{\C^*}(\fI)$ with the $\C^*$-action of weight $n$.
\end{Theorem}

\begin{NB}
  \begin{equation*}
    \begin{split}
    & \dim \cM = 
    - \chi(\cE,\cE(-\linf)) + \dim \Flag(V_1,\underline{N}) + 1 
\\
    =\; &
    \begin{aligned}[t]
    & - \chi(\cEf,\cEf(-\linf))
     - \chi(\cEf,C_m^{\oplus p}) 
     - \chi(C_m^{\oplus p},\cE_1) - p^2
\\
    &\qquad
    + \dim \Flag(\Vf_1,\If) + \dim \Flag(\Vs_1,\Is) + \rank N_0 + 1
    \end{aligned}
    \end{split}
  \end{equation*}
On the other hand, the dimension of $\cM^{\C^*}(\fI)$ is
\begin{equation*}
  \begin{split}
  & \dim (\text{moduli of $(\zeta^0,l)$-stable objects})
      + \dim \Flag\left[\C^{m+1}\otimes\cQ/[\xi]\to \Gr((\C^{m+1})^*,p)\right]
\\
  =\; & - \chi(\cEf,\cEf(-\linf)) + \dim \Flag(\Vf_1,\If)
  + p(m+1 - p) + \dim \Flag(\Vs_1,\Is) - (p(m+1)-1)
\\
  =\; & - \chi(\cEf,\cEf(-\linf)) + \dim \Flag(\Vf_1,\If)
  + \dim \Flag(\Vs_1,\Is) - p^2 + 1.
  \end{split}
\end{equation*}
Therefore
\begin{equation*}
  \dim \cM - \dim \cM^{\C^*}(\fI)
  = - \chi(\cEf,C_m^{\oplus p}) 
    - \chi(C_m^{\oplus p},\cEf) + \rank N_0.
\end{equation*}
\end{NB}

\begin{proof}
First consider the case of $\cM_\pm = \proj(L_\pm)^{\mathrm{ss}}/G$.
The normal bundle is the descent of the normal bundle
$\proj(L_\pm)\subset \proj(L_-\oplus L_+)$. Then it is $L_\mp^*\otimes
L_\pm$ with the $\C^*$-action of weight $\pm 1$.
\begin{NB}
  The relative tangent bundle is $\Hom(\text{tautological
    subbundle},\text{universal quotient})$. Therefore its restriction
  to $\proj(L_\pm)$ is $\Hom(L_\mp,L_\pm)\cong L_\mp^*\otimes L_\pm$.

  On the other hand, the weight of $\C^*$-action can be calculated
  locally. We can consider $[z_-:z_+]\to [tz_-:z_+]$, the action on
  $\proj^1$. At $[0:1]$, the weight is $1$. At $[1:0]$, the weight is $-1$.
\end{NB}

In the remaining of the proof we consider the case of
$\cM^{\C^*}(\fI)$. The normal bundle is the sum of nonzero
weight subspaces in the restriction of the tangent bundle of $\cM$ to
$\cM^{\C^*}(\fI)$.

Take a point $((X,F^\bullet),[z_-:z_+])\in \proj(L_-\oplus
L_+)^{\mathrm{ss}}$ which descends to a fixed point in
$\cM^{\C^*}(\fI)$. We have a decomposition $(X,F^\bullet) =
(\Xf,\Ff^\bullet)\oplus(\Xs,\Fs^\bullet)$ as above.
Let $V = \Vf\oplus \Vs$ be the corresponding decomposition of $V$.

We lift the $\C^*$-action on $\cM$ to $\proj(L_-\oplus L_+)^{\mathrm{ss}}$
as in \eqref{eq:C*}.
\begin{NB}
by
\begin{equation*}
   ((X,F^\bullet),[z_-:z_+]) \mapsto
   (\id_{\Vf}\oplus t^{1/pD} \id_{\Vs})\cdot ((X,F^\bullet),[t z_-:z_+]),
\end{equation*}
which is well-defined on the covering $\C^*\to \C^*; s\mapsto s^{pD} =
t$, and fixes $((X,F^\bullet),[z_-:z_+]) =
((\Xf,\Ff^\bullet)\oplus(\Xs,\Fs^\bullet),[z_-:z_+])$. 
\begin{NB2}
Let $g = \id_{\Vf}\oplus t^{1/pD} \id_{\Vs}$. Then
\begin{equation*}
  \begin{split}
  g\cdot ((\Xf,\Ff^\bullet)
  \oplus (\Xs,\Fs^\bullet),[t z_-:z_+])
  &= ((\Xf,\Ff^\bullet)
  \oplus (\Xs,\Fs^\bullet),[t z_-: \chi_{\zeta}(g)\chi_{\zeta^-}(g)^{-1} z_+])
\\
  &= ((\Xf,\Ff^\bullet)
  \oplus (\Xs,\Fs^\bullet),[t z_-: (t^{1/pD})^{pD} z_+])
\\
  &= ((\Xf,\Ff^\bullet)
  \oplus (\Xs,\Fs^\bullet),[z_-: z_+]).
  \end{split}
\end{equation*}
\end{NB2}%
Since this $\C^*$-action is equal to the original one up to the $G$-action,
it is the same on the quotient $\cM$.
\end{NB}

The tangent space at the point corresponding to
$((X,F^\bullet),[z_-:z_+])$ is the quotient vector space
\begin{equation*}
   \Ker d\mu \oplus T_{F^\bullet}\Flag(V_1,\underline{N})
   \oplus T_{[z_-:z_+]}\proj^1
   / \Hom(V_0,V_0)\oplus \Hom(V_1,V_1),
\end{equation*}
where $d\mu$ and $\Hom(V_0,V_0)\oplus \Hom(V_1,V_1)\to \Ker d\mu$ are
as in \eqref{eq:tangentcpx}, and $\Hom(V_0,V_0)\oplus \Hom(V_1,V_1)\to
T_{F^\bullet}\Flag(V_1,\underline{N})\oplus T_{[z_-:z_+]}\proj^1$ is
the differential of the $G$-action.
This quotient space has the $\C^*$-module structure induced from the
above lift of the $\C^*$-action.

In the equivariant $K$-group, we can replace this space by
\(
  \Ker d\mu - (\Hom(V_0,V_0)\oplus \Hom(V_1,V_1))
  + T_{F^\bullet}\Flag(V_1,\underline{N})
  + T_{[z_-:z_+]}\proj^1
\)

The expression
\(
  \Ker d\mu - (\Hom(V_0,V_0)\oplus \Hom(V_1,V_1))
\)
is equal to the alternating sum of cohomology groups of the tangent
complex \eqref{eq:tangentcpx} in the $K$-group.
Moreover, the complex \eqref{eq:tangentcpx} decomposes into four
parts, the tangent complex for $\Xf$, the sum of $p$-copies of the
complex~\eqref{eq:Ext(X,C_m)} for $\Xf$, the sum of $p$-copies of the
complex~\eqref{eq:Ext(C_m,X)} for $\Xf$, and the sum of $p^2$-copies
of the tangent complex for $C_m$. Since $\C^*$ acts on $\Vf$ with
weight $0$ and $\Vs$ with weight ${1/pD}$, the first and fourth
parts do not contribute to the normal bundle. The second and third
terms give
\begin{equation*}
  \begin{split}
    & \begin{aligned}[t]
      & \left(
        -\Ext^0_p({}^\cM\cEf,{}^\cM \cEs) + \Ext^1_p({}^\cM\cEf,{}^\cM \cEs)
      \right)\otimes I_{1/pD}
      \\
      & \quad + \left(
        -\Ext^0_p({}^\cM \cEs,{}^\cM\cEf) + \Ext^1_p({}^\cM \cEs,{}^\cM\cEf)
      \right) \otimes I_{-1/pD}
    \end{aligned}
    \\
    =\; &
    \fN({}^\cM\cEf,{}^\cM \cEs)\otimes I_{1/pD}
    +
    \fN({}^\cM \cEs, {}^\cM\cEf)\otimes I_{-1/pD}.
  \end{split}
\end{equation*}

The contribution from $T_{F^\bullet}\Flag(V_1,\underline{N})$ is given
by $N_0$. Thus we have \thmref{thm:normalbundle}. We have no contribution
from $T_{[z_-:z_+]}\proj^1$.
\begin{NB}
  Any choice of $[z_-:z_+]$ (other than $[1:0]$, $[0:1]$) gives a
  fixed point. Thus the fixed point locus contains the fiber $\proj^1$
  (at least locally).
\end{NB}
\end{proof}

\subsection{Relative tangent bundles of flag bundles}

Let $\Trel$ be the bundles over various moduli stacks
induced from the relative tangent bundle of the flag bundle
$\widetilde Q/G\to Q/G$ as in \subsecref{subsec:1st_formula}

On the other hand, we have vector bundles $\Trel^\flat$,
$\Trel^\sharp$ over
$\tM^{m,\min(\Is)-1}(c_\flat)\times\tM^{m,+}(c_\sharp)$ coming from
the tangent bundles of $\Flag(\Vf_1,\If)$ and $\Flag(\Vs_1,\Is)$.
Recall the normal bundle $N_0$ of $\Flag(\Vf_1,\If)\times
\Flag(\Vs_1,\Is)$ in $\Flag(V_1,\underline{N})$ is considered as a
vector bundle over
$\tM^{m,\min(\Is)-1}(c_\flat)\times\tM^{m,+}(c_\sharp)$ (see
\subsecref{subsec:normal}), so we have an exact sequence
\begin{equation}
  \label{eq:reltan1}
   0 \to G^*(\Trel^\flat\oplus \Trel^\sharp) \to
   F^* \Trel|_{\cM^{\C^*}(\fI)} \to G^*(N_0) \to 0
\end{equation}
of $\C^*$-equivariant vector bundles, where $\Trel^\flat\oplus
\Trel^\sharp$ has weight $0$ and $N_0$ has the $\C^*$-equivariant
structure as describe in \subsecref{subsec:normal}.

Recall also that the second factor $\tM^{m,+}(c_\sharp)$ is the flag
bundle
$\Flag(\Vcal_1^\sharp/\Fcal_\sharp^{\min(\Is)},\Is\setminus\{\min(\Is)\})$
over the quotient stack $(\det\Qcal^{\otimes D})^\times/\C^*$ (see
\propref{prop:m,+}). Let $\Trel'$ be the relative tangent bundle of
the fiber. Then we have an exact sequence of vector bundles
\begin{equation}\label{eq:reltan2}
   0 \to \Trel'
     \to \Trel^\sharp
     \to
    \Hom(\Vcal_1^\sharp/\Fcal_\sharp^{\min(\Is)},\Fcal_\sharp^{\min(\Is)}) \to 0
\end{equation}
coming from the fibration
$\Flag(\Vcal_1^\sharp/\Fcal_\sharp^{\min(\Is)}) \to
\Flag(\Vcal_1^\sharp) \to \proj(\Vcal_1^\sharp)$.
\begin{NB}
The following was wrong.
\begin{NB2}
The last term is the pullback of the bundle over $\Gr({m+1},p)$.
\end{NB2}%
\end{NB}
\begin{NB}
  This decomposition will be important when we cut the enhanced master
  space $\cM$ by $e(\Trel^\sharp)$.
\end{NB}

\section{Wall-crossing formula}\label{sec:wallcrossing}

We now turn to Mochizuki method {\bf (Mb)}. (See \cite[\S7.2]{Moc}.)
We will apply the fixed point formula to the equivariant homology
groups $H^{\hT\times\C^*}_*(\cM(c))$ of the enhanced moduli space
$\cM(c)$, which is a module over $H_{\C^*}^*(\pt)\cong \C[\hbar]$.
For the definition of the homology group of a Deligne-Mumford stack,
see \cite{Vistoli}.
\begin{NB}
  But it is just a homology group of the
  underlying scheme \cite{Vistoli}.
  Since it is just an orbifold, or a Satake V-manifold, and in
  particular is an $\Q$-homology manifold. In particular, we have
  Poincar\'e duality theorem on the rational homology groups.
\end{NB}

\subsection{Equivariant Euler class}\label{subsec:Euler}

In the fixed point formula we put an equivariant Euler class in the
denominator and we later consider an equivariant Euler class of a
class in the $K$-group. Let us explain how we treat them in this
subsection.

Let $Y$ be a variety (or a Deligne-Mumford stack) with a trivial
$\C^*$-action.
The Grothendieck group of $\C^*$-equivariant vector bundles decomposes
as
\(
   K_{\C^*}(Y) = K(Y)\otimes_{\Z} R(\C^*)
\)
where $R(\C^*)$ is the representation ring of $\C^*$. Let $I_n$ denote the 
$1$-dimensional representation of $\C^*$ with weight $n$. For a class
$\alpha\in K(Y)$, we set
\begin{equation*}
   e(\alpha\otimes I_n)
   \defeq \sum_{i\ge 0} c_i(\alpha) (n\hbar)^{r(\alpha)-i}
   \in H^*(Y)[\hbar^{-1},\hbar] \defeq
   H^*(Y)\otimes_{\C} \C[\hbar^{-1},\hbar],
\end{equation*}
where
$c_i(\alpha)$ is the $i^{\mathrm{th}}$ Chern class of $\alpha$
and $r(\alpha) = \ch_0(\alpha)$ is the (virtual) rank of $\alpha$.
If $n\neq 0$, this element is invertible.
\begin{NB}
  \begin{equation*}
    \frac1{e(\alpha\otimes I_n)}
    = \frac{(n\hbar)^{-r(\alpha)}}
    {1 + c_1(\alpha)\hbar^{-1} + c_2(\alpha)\hbar^{-2} + \cdots}
    = (n\hbar)^{-r(\alpha)}\left(
       1 - c_1(\alpha)\hbar^{-1} + \cdots
      \right).
  \end{equation*}
\end{NB}%
In general, if $\alpha\in K_{\C^*}(Y)$ is a sum of $\alpha_n\otimes
I_n$ with $n\neq 0$, its equivariant Euler class $e(\alpha)$ is defined
in $H^*(Y)[\hbar^{-1},\hbar]$.

We also consider the case when another group $\hT$ acts on $Y$. Then
$e(\alpha\otimes I_n)$ can be still defined as an element in
$\varprojlim_n H^*(Y\times_{\hT} E_n)[\hbar^{-1},\hbar]$, where
$E_n\to E_n/\hT$ is a finite dimensional approximation of the
classifying space $E\hT\to B\hT$ for $\hT$.
Note $c_i(\alpha) \neq 0$ for possibly infinite $i$'s.
\begin{NB}
Here $\varprojlim_n H^*(Y\times_{\hT} E_n)[\hbar^{-1},\hbar]$ is equal to 
the direct {\it product\/}
\(
   \prod_{k=0}^\infty \left(H^k_{\hT}(Y)[\hbar^{-1},\hbar]\right).
\)

On the other hand, Mochizuki used $R[[\hbar^{-1},\hbar]$, the algebra
of formal power series $\sum a_j \hbar^j$ such that $\{ j > 0\mid
a_j\neq 0\}$ is finite for an algebra $R$.
\end{NB}

\subsection{Wall-crossing formula (I)}\label{subsec:1st_formula}

The projective morphism $\widehat\pi\colon \cM(c)\to M_0(p_*(c))$
induce a homomorphism $\widehat\pi_*\colon H^{\hT\times\C^*}_*(\cM(c))
\to H^{\hT\times\C^*}_*(M_0(p_*(c)))$. Since $\C^*$ acts trivially on
$M_0(p_*(c))$, we have $H^{\hT\times\C^*}_*(M_0(p_*(c)))\cong
H^{\hT}_*(M_0(p_*(c)))\otimes\C[\hbar]$.
For a cohomology class $\bullet$ on $\cM(c)$, we denote the
pushforward $\widehat\pi_*(\bullet\cap [\cM(c)])$ by
$\int_{\cM(c)}\bullet$ as in \secref{sec:1st}. We also use similar
push-forward homomorphisms from homology groups of various moduli
stacks and denote them in similar ways, e.g., $\int_{\cM_\pm(c)}$,
$\int_{\cM^{\C^*}(\fI)}$, etc.
\begin{NB2}
  Editted according to the referee's suggestion. 2010/04/01
\end{NB2}%

Let $e(\Fcal)$ denote the equivariant Euler class of the equivariant
vector bundle $\Fcal$, that is the top Chern class
$c_{\topdeg}(\Fcal)$.  See \subsecref{subsec:Euler} for its
generalization to a $K$-theory class $\Fcal$ in our situation.

Let $\Phi(\cE)$ be as in \subsecref{subsec:statement}.
We denote also by $\Phi(\cE)$ the class on $\tM^{m,\ell}(c)$ given by
the same formula.
On the enhanced master space $\cM(c)$, we can consider the class
defined by the same formula as $\Phi(\cE)$, which is regarded as a
$\C^*$-equivariant class.
\begin{NB}
I abandon this convention.

Let us denote it by $\Phi_t(\cE)$ to emphasize its
$\C^*$-equivariance.
\end{NB}%

Let $\Trel$ be the relative tangent bundle of the flag bundle
$\widetilde Q/G\to Q/G$. Then we have the induced bundle over
$\tM^{m,\ell}(c)$ by the restriction. We denote it also by $\Trel$ for
brevity. We also have the pullback to the enhanced master space
$\cM(c)$, which is again denoted by $\Trel$.  It has a natural
$\C^*$-equivariant structure.
We introduce
\(
   \widetilde\Phi(\cE) \defeq \frac1{v_1(c)!} \Phi(\cE)\cup e(\Trel)
\)
so that we have
\begin{equation*}
    \int_{\tM^{m,0}(c)} \widetilde\Phi(\cE)
    = \int_{\bM^m(c)} \Phi(\cE).
\end{equation*}
Here $v_1(c) = \dim V_1(E)$ for a sheaf $E$ with $\ch(E) = c$.
\begin{NB}
We also set
\(
   \widetilde\Phi_t(\cE) 
   \defeq \frac1{v_1(c)!} \Phi_t(\cE)\cup e(\Trel).
\)
\end{NB}

Then the fixed point formula in the equivariant homology group gives us
\begin{equation*}
  \int_{\bM^m(c)} \Phi(\cE) 
  - \int_{\tM^{m,\ell}(c)} \widetilde\Phi(\cE)
   = \sum_{\fI\in\mathcal D^{m,\ell}(c)}
   \Res_{\hbar=0} \int_{\cM^{\C^*}(\fI)}
      \frac{\widetilde\Phi(\cE)}{e(\fN(\cM^{\C^*}(\fI)))},
\end{equation*}
where `$\Res_{\hbar=0}$' means taking the coefficient of $\hbar^{-1}$.

By \propref{prop:m,+} and Theorem~\ref{thm:master} together with
computations of normal bundles \subsecref{subsec:normal}, we can
rewrite the right hand side to get

\begin{Theorem}\label{thm:1st}
\begin{multline*}
   \int_{\tM^{m,\ell}(c)} \widetilde\Phi(\cE) 
   - \int_{\bM^m(c)} \Phi(\cE)
\\
  = 
  \sum_{\fI\in\mathcal D^{m,\ell}(c)}
   \int_{\tM^{m,\min(\Is)-1}(c_\flat)} 
   \Res_{\hbar=0} \left[\widetilde\Phi(\cEf\oplus
   (C_m\boxtimes\Qcal
      \otimes I_{-1}))
   \Phi'(\cEf)\right],
\end{multline*}
where $\widetilde\Phi(\cEf\oplus (C_m\boxtimes\Qcal \otimes I_{-1}))$
is defined exactly as $\widetilde\Phi(\cE)$ by replacing $\cE$ by
$\cEf\oplus (C_m\boxtimes\Qcal \otimes I_{-1})$ everywhere, and
$\Phi'(\cEf)$ is another
\begin{NB}
  I think this is wrong: Feb.25 (HN)
  multiplicative
\end{NB}%
cohomology class given by
\begin{equation}\label{eq:Phi'}
  \Phi'(\cEf) =
    \frac{(v_1(c_\sharp) - 1)! v_1(c_\flat)!}{v_1(c)!}
     \int_{\Gr(m+1,p)}
      \frac{
      e\left(
        (V_1(C_m)\otimes\Qcal/
        \shfO)^*
        \right)
    }{
      e(\fN(\cEf,C_m)\otimes \Qcal
      \otimes I_{-1})
      \, e(\fN(C_m,\cEf)\otimes \Qcal^*
      \otimes I_{1})
    }.
%
\end{equation}
Here $I_n$ denotes the trivial line bundle with the $\C^*$-action of
weight $n$, and $\fN(\bullet,\bullet)$ is the equivariant $K$-theory
class given by the negative of the alternating sum of $\Ext$-groups
(see \eqref{eq:notation_normal}).
\textup(Note that $\Phi'$ depends on $\fI$.\textup)
\begin{NB}
  It should be independent of $c_\flat$ after we go to
  $\int_{\bM^m(c_\flat)}$.
\end{NB}
\end{Theorem}

The proof will be given in the next subsection.

\subsection{Fixed point formula on the enhanced master space}
\label{subsec:fixedptformula}

Let $\iota_\pm$, $\iota_{\fI}$ be the inclusions of
$\cM^\pm(c)$, $\cM^{\C^*}(\fI)$ into $\cM(c)$. Let
$\iota_\pm^*$, $\iota_{\fI}^*$ be the pullback homomorphisms,
which are defined as $\cM(c)$ is smooth. They will be omitted
from formulas eventually.
Let $e(\fN(\cM_+(c)))$, $e(\fN(\cM^{\C^*}(\fI)))$ be the
equivariant Euler class of the normal bundles $\fN(\cM_+(c))$,
$\fN(\cM^{\C^*}(\fI))$.
The localization theorem in the equivariant cohomology groups says the
following diagram is commutative:
\begin{equation*}
  \begin{CD}
    \varprojlim_n H_*^{\C^*}(\cM(c)\times_{\hT} E_n)
    \otimes_{\C[\hbar]}\C[\hbar^{-1},\hbar]
    @>\cong>>
    \varprojlim_n H_*(\cM(c)^{\C^*}\times_{\hT}E_n)[\hbar^{-1},\hbar]
\\
    @V\int_{\cM(c)}VV 
    @VV\int_{\cM_+(c)}+\int_{\cM_-(c)}
    + \sum_{\fI} \int_{\cM^{\C^*}(\fI)}V
\\
    \varprojlim_n H_*(M_0(p_*(c)\times_{\hT}E_n))[\hbar^{-1},\hbar]
    @=
    \varprojlim_n H_*(M_0(p_*(c)\times_{\hT}E_n))[\hbar^{-1},\hbar],
  \end{CD}
\end{equation*}
where the upper horizontal arrow is given by
\begin{equation*}
   \frac{\iota_+^*}{e(\fN(\cM_+(c)))}
   + \frac{\iota_-^*}{e(\fN(\cM_-(c)))}
   + \sum_{\fI} \frac{\iota_{\fI}^*}
   {e(\fN(\cM^{\C^*}(\fI)))},
\end{equation*}
and $E_n\to E_n/\hT$ is a finite dimensional approximation of $E\hT\to
B\hT$ as above.

\begin{NB}
Let $\Phi(\cE)$ be a multiplicative cohomology class of
slant products of the universal sheaf $\cE$, for example, as in
\eqref{eq:Phi}.
Let us also denote $\Phi(\cE)$ a similar class
on $\tM^{m,\ell}(c)$, and we define 
\( 
   \widetilde\Phi(\cE) \defeq \frac1{v_1(c)!} \Phi\cup e(\Trel)
\)
as in \subsecref{subsec:1st_formula}.
We have similar classes $\Phi(\cE)$, $\widetilde\Phi(\cE)$
on the enhanced master space $\cM(c)$, denoted by the same symbols.
\begin{NB2}
where the subscript $t$ is
put to emphasize their $\C^*$-dependence.
\end{NB2}%
\end{NB}

Let $\mathcal T(1)$ be the trivial line bundle with the $\C^*$-action
of weight $1$. 
\begin{NB}
  I must understand where $\mathcal T(1)$ is defined. Mochizuki
  defined it over an Artin stack, and then pullback. A pullback of a
  trivial bundle is a trivial bundle. Do we need to distinguish where
  it from ?

  I asked Mochizuki this question. He says he was worried about the
  group may not act trivially, and want to ensure that the bundle is
  `really' trivial. I must digest what he means......
\end{NB}
We have
\begin{equation*}
  \int_{\cM(c)}
    \widetilde\Phi(\cE) c_1(\mathcal T(1))
  = \sum_{a=\pm} \int_{\cM_a(c)}
        \frac{\widetilde\Phi(\cE) c_1(\mathcal T(1))}{e(\fN(\cM_a))}
    + \sum_{\fI\in\mathcal D^{m,\ell}(c)}
    \int_{\cM^{\C^*}(\fI)}
      \frac{\widetilde\Phi(\cE) c_1(\mathcal T(1))}
      {e(\fN(\cM^{\C^*}(\fI)))}
\end{equation*}
holds in $\varprojlim
H_*(M_0(p_*(c)\times_{\hT}E_n))[\hbar^{-1},\hbar]$. Since $\mathcal
T(1)$ is a trivial line bundle if we forget the $\C^*$-action, the
left hand side is restricted to $0$ at $\hbar = 0$:
\(
  \left.
  \int_{\cM(c)}
    \widetilde\Phi c_1(\mathcal T(1))
    \right|_{\hbar=0} = 0.
\)
On the other hand, $\left.c_1(\mathcal T(1))\right|_{\cM_a(c)} =
\hbar$ and $\left.c_1(\mathcal T(1))\right|_{\cM^{\C^*}(\fI)}
= \hbar$.
Moreover we have
\[
  \frac1{e(\fN(\cM_a))}
  = a (\hbar - \omega)^{-1} 
  = \frac{a}\hbar \sum_{i=0}^\infty \left(\frac\omega{\hbar}\right)^i
\]
for $a=\pm$, where $\omega = c_1(L_+^*\otimes L_-)$. Combining with
\thmref{thm:master}(2) we get
\begin{equation}\label{eq:fixedptformula}
  \int_{\tM^{m,\ell}(c)} \widetilde\Phi(\cE)
   - \int_{\bM^m(c)} \Phi(\cE) 
   = - \sum_{\fI\in\mathcal D^{m,\ell}(c)}
   \Res_{\hbar=0} \int_{\cM^{\C^*}(\fI)}
      \frac{\widetilde\Phi(\cE)}{e(\fN(\cM^{\C^*}(\fI)))}.
\end{equation}
\begin{NB}
  For a later purpose, we change the sign.
\end{NB}

We now use the diagram \eqref{eq:exceptfixed} to rewrite the integral
in the right hand side of \eqref{eq:fixedptformula}:
\begin{equation*}
  \int_{\cM^{\C^*}(\fI)}
      \frac{\widetilde\Phi(\cE)}{e(\fN(\cM^{\C^*}(\fI)))}
   = {pD} \int_{\mathcal S(\fI)}
      \frac{F^* \widetilde\Phi(\cE)}
      {F^*e(\fN(\cM^{\C^*}(\fI)))}.
\end{equation*}
From \thmref{thm:master}(3) we have
\begin{equation*}
  F^*(\Phi(\cE)) = \Phi(G^*(\cEf)\oplus
  G^*(\cEs)\otimes L_{\mathcal S}\otimes I_{1/pD}).
\end{equation*}
Since $L_{\mathcal S}^{\otimes pD} = G^*(\cL(\cE_\flat)^*)$, we have
\begin{equation*}
  c_1(L_{\mathcal S}) = - \frac1{pD} G^* c_1(\cL(\cE_\flat)).
\end{equation*}
Since $L_{\mathcal S}$ appears as $c_1(L_{\mathcal S})$ in
$\Phi(G^*(\cEs)\otimes L_{\mathcal S}\otimes I_{1/pD})$, we can
formally write
\begin{equation*}
  \Phi(G^*(\cEf)\oplus
  G^*(\cEs)\otimes L_{\mathcal S}\otimes I_{1/pD})
  = G^* \Phi(\cEf\oplus
  \cEs\otimes \cL(\cE_\flat)^{-1/pD}\otimes I_{1/pD}),
\end{equation*}
meaning that we replace $c_1(\cL(\cE_\flat)^{-1/pD})$ by
$-c_1(\cL(\cE_\flat))/pD$.

Similarly from \thmref{thm:normalbundle} and \thmref{thm:master}(3) we have
\begin{equation*}
  \begin{split}
   & F^* e(\fN(\cM^{\C^*}(\fI)))
\\
  =\; & 
  F^*\left(e(\fN({}^\cM\cEf,{}^\cM\cEs)\otimes I_{1/pD})
  e(\fN({}^\cM\cEs,{}^\cM\cEf)\otimes I_{-1/pD})
  e(N_0)\right)
\\
  =\; & 
  G^*\left(
  e(\fN(\cEf,\cEs)\otimes \cL(\cE_\flat)^{-1/pD} \otimes I_{1/pD})
  e(\fN(\cEs,\cEf)\otimes \cL(\cE_\flat)^{1/pD}\otimes I_{-1/pD})
  e(N_0)
  \right).
  \end{split}
\end{equation*}

From (\ref{eq:reltan1},\ref{eq:reltan2}) we have
\begin{equation*}
  \begin{split}
  & F^*(e(\Trel)) = G^*\left(e(\Trel^\flat) e(\Trel^\sharp) e(N_0)\right)
\\
  =\;& G^*\left( e(\Trel^\flat) e(\Trel') e\left(
  \Hom(\Vcal_1^\sharp/\Fcal_\sharp^{\min(\Is)},\Fcal_\sharp^{\min(\Is)})
  \right) e(N_0)\right).
  \end{split}
\end{equation*}
Therefore we get
\begin{equation*}
  \begin{split}
    & {pD} \int_{\mathcal S(\fI)}
      \frac{F^* \widetilde\Phi}
      {F^*e(\fN(\cM^{\C^*}(\fI)))}
\\
    =\; &
    \frac1{v_1(c)!}
    \begin{aligned}[t]
    \int_{\tM^{m,\min(\Is)-1}(c_\flat)\times \tM^{m,+}(c_\sharp)}
    & \frac{\Phi(\cE^\flat \oplus  
      \cEs\otimes \cL(\cE_\flat)^{-1/pD}\otimes I_{1/pD})}
    {
      e(\fN(\cEf,\cEs)\otimes \cL(\cE_\flat)^{-1/pD} \otimes I_{1/pD})
      }
\\
    & \quad \times \frac{
      e(\Trel^\flat) e(\Trel') e\left(
        \Hom(\Vcal_1^\sharp/\Fcal_\sharp^{\min(\Is)},\Fcal_\sharp^{\min(\Is)})
        \right)
    }{
      e(\fN(\cEs,\cEf)\otimes \cL(\cE_\flat)^{1/pD}\otimes I_{-1/pD})
    }.
    \end{aligned}
  \end{split}
\end{equation*}
We use the claim that $\tM^{m,+}(c_\sharp)$ is a flag bundle over
$(\det \Qcal^{\otimes D})^\times/\C^*$ (\propref{prop:m,+}(1)) to rewrite
this further as
\begin{equation*}
  \frac{(v_1(c_\sharp) - 1)!}{v_1(c)!}
  \begin{aligned}[t]
    \int_{\tM^{m,\min(\Is)-1}(c_\flat)\times 
      (\det \Qcal^{\otimes D})^\times/\C^*
    }
    & \frac{\Phi(\cE^\flat\oplus 
      \cEs\otimes \cL(\cE_\flat)^{-1/pD}\otimes
      I_{1/pD})\, e(\Trel^\flat) }
    {
      e(\fN(\cEf,\cEs)\otimes \cL(\cE_\flat)^{-1/pD} \otimes I_{1/pD})
      }
\\
    & \quad\times
    \frac{
      e\left(
        \Hom(\Vcal_1^\sharp/\Fcal_\sharp^{\min(\Is)},\Fcal_\sharp^{\min(\Is)})
        \right)
    }{
      e(\fN(\cEs,\cEf)\otimes \cL(\cE_\flat)^{1/pD}\otimes I_{-1/pD})
    }.
  \end{aligned}
\end{equation*}
We set $\widetilde\Phi(\bullet) \defeq \frac1{v_1(c_\flat)!}\Phi(\bullet)
e(\Trel^\flat)$, use \propref{prop:m,+}(2) to replace $(\det
\Qcal^{\otimes D})^\times/\C^*$ by $\Gr(m+1,p)$, and then plug into
\eqref{eq:fixedptformula} to get
\begin{multline*}
    \int_{\tM^{m,\ell}(c)} \widetilde\Phi(\cE)
  - \int_{\bM^m(c)} \Phi(\cE) 
   = \sum_{\fI\in\mathcal D^{m,\ell}(c)}
  \int_{\tM^{m,\min(\Is)-1}(c_\flat)}     
  \Res_{\hbar=0} \widetilde \Phi(\cEf
  \oplus 
      \cEs\otimes \cL(\cE_\flat)^{-1/pD}\otimes
      I_{1/pD}) \Phi'(\cEf),
\end{multline*}
where
\begin{equation*}
  \begin{split}
  & \Phi'(\cEf) = 
    - \frac{(v_1(c_\sharp) - 1)! v_1(c_\flat)!}{v_1(c)!\, pD}
     \int_{\Gr(m+1,p)} \heartsuit,
\\
  &
  \heartsuit =
  \begin{aligned}[t]
    & \frac{1
      }
    {
      e(\fN(\cEf,C_m)\otimes \Qcal
      \otimes \det\cQ^{-1/p}
      \otimes \cL(\cE_\flat)^{-1/pD} \otimes I_{1/pD})
      }
\\
   & \quad
   \times
   \frac{
     e\left(
        (V_1(C_m)\otimes\Qcal/
        \shfO)^*
        \right)
     }
     {
      e(\fN(C_m,\cEf)\otimes \Qcal^*
      \otimes \det\cQ^{1/p}
      \otimes \cL(\cE_\flat)^{1/pD}\otimes I_{-1/pD})
       }
   .
  \end{aligned}
\end{split}
\end{equation*}
Here 
\begin{itemize}
\item $\int_{\Gr(m+1,p)}$ means the pushforward with respect to the
  projection $\tM^{m,\min(\Is)-1}(c_\flat)\times \Gr(m+1,p)\to
  \tM^{m,\min(\Is)-1}(c_\flat)$.
\item $\det\Qcal^{-1/p}$ is understood as before: we replace
$c_1(\det\Qcal^{-1/p})$ by $-c_1(\det\Qcal)/p$.
\end{itemize}

\begin{NB}
\begin{equation*}
  \frac{(v_1(c_\sharp) - 1)! v_1(c_\flat)!}{v_1(c)!} 
  = \left(v_1(c_\sharp)\times {v_1(c) \choose v_1(c_\flat)}\right)^{-1}.
\end{equation*}
\end{NB}

Let us slightly simplify the formula. First note that $I_{\pm 1/pD}$
appears with $\det\cQ^{\mp 1/p}\otimes \cL(\cE_\flat)^{\mp 1/pD}$. Let
$\omega = - (c_1(\cL(\cE_\flat)) + c_1(\cQ))/pD$. Thus $\Phi'(\cEf)$
is written as
\begin{equation*}
  \sum_{j=-\infty}^\infty A_j (\hbar - \omega)^j.
\end{equation*}
By a direct calculation we have
\begin{equation*}
   \Res_{\hbar=0} (\hbar - \omega)^j =
   \begin{cases}
     1 & \text{if $j=-1$},
     \\
     0 & \text{otherwise}.
   \end{cases}
\end{equation*}
\begin{NB}
  When $j \ge 0$, $(\hbar - \omega)^j$ is a polynomial in
  $\hbar$. Therefore the assertion is clear. Assume $j < 0$. Then
  \begin{equation*}
     (\hbar - \omega)^j = \hbar^j \left(1 - \frac{\omega}{\hbar}\right)^j
       = \hbar^j\left( 1 - j\frac{\omega}{\hbar} + \cdots\right).
  \end{equation*}
  Then the coefficient of $\hbar^{-1}$ is $1$ if $j=-1$, and $0$ otherwise.

  More conceptually, this may be explained as follows. We should consider
  $\Res_{\hbar=0} f(\hbar) = \Res_{1/\hbar = 0} \hbar^2 f(\hbar)$. Then
  $\omega$ is `far away' from $1/\hbar=0$, and $-\omega$ does not affect.
\end{NB}%
Therefore we have
\begin{equation*}
   \Res_{\hbar=0} \sum_{j=-\infty}^\infty A_j (\hbar - \omega)^j
   =    \Res_{\hbar=0} \sum_{j=-\infty}^\infty A_j \hbar^j.
\end{equation*}
This means that we can erase 
$\det\cQ^{\mp 1/p}\otimes \cL(\cE_\flat)^{\mp 1/pD}$
from
$\Phi'(\cEf)$.
(The last simplification appeared in \cite[Proof of Theorem~7.2.4]{Moc})

Next note that nontrivial contributions of the $\C^*$-action appear as
$I_{\pm 1/pD}$.
If we take the covering $\C^*_s\to\C^*; s\mapsto t=s^{-pD}$, $I_{\pm
  1/pD}$ is of weight $\mp 1$ as a $\C^*_s$-module by our
convention. Since that we have a natural isomorphism
$H^*_{\C^*_s}(\pt) = \C[\hbar_s] \cong H^*_{\C^*}(\pt) = \C[\hbar]$ by
$\hbar_s = -\hbar/pD$, we can replace $\hbar$ by $\hbar_s$ noticing
\( 
  \Res_{\hbar=0} f(\hbar) =
  - pD\times \Res_{\hbar_s=0} f(\hbar_s pD).
\)
We use this replacement and then replace back $\hbar_s$ by $\hbar$
again. Therefore we get the formula \eqref{eq:Phi'}.
We have completed the proof of \thmref{thm:1st}.

\subsection{Proof of \thmref{thm:m=0}}\label{subsec:2nd}

The right hand side of the formula in \thmref{thm:1st} can be
expressed by an integral over $\bM^m(c_\flat)$ by using the formula
again.  We continue this procedure recursively, we will get a
wall-crossing formula comparing $\int_{\bM^m(c)}$ and
$\int_{\bM^{m+1}(c)}$. We then get \thmref{thm:m=0}.
The proof has combinatorial nature and is the same as one in
\cite[\S7.6]{Moc}. We reproduce it here for reader's convenience.

In fact, it is enough to consider the case $m=0$, as a general case
follows from $m=0$ since we have an isomorphism $\bM^m(c)
\xrightarrow{\cong} \bM^0(c e^{-m[C]}); (E,\Phi)\mapsto
(E(-mC),\Phi)$.
Then the formula in \thmref{thm:1st} is simplified as we can 
assume $p=1$ as $\Gr(m+1,p) = \Gr(1,p)$ is empty otherwise.
\thmref{thm:m=0} is obtained in this way, but we give a proof for
general $m$.

For $j\in \Z_{>0}$ let
\begin{equation*}
  S_j^m(c)
  \defeq \left\{ 
    \vec{c} = (c_\flat,p_1,\dots,p_j) \in H^*(\bp)\times \Z_{>0}^j
    \left|\,
  (c_\flat,[\linf]) = 0, \ c_\flat + \sum_{i=1}^j p_i e_m = c
  \right.\right\}.
\end{equation*}
We denote the universal family for $\bM^{m}(c_\flat)$ by $\cEf$ as above.
Let $S^m(c) = \bigsqcup_{j=1}^\infty S_j^m(c)$.

For $\vec{p} = (p_1,\dots,p_j)\in \Z_{>0}^j$ we consider the product
of Grassmannian varieties $\prod_{i=1}^j \Gr(m+1,p_i)$. Let
$\Qcal^{(i)}$ be the universal quotient of the $i^{\mathrm{th}}$
factor. We consider the $j$-dimensional torus $(\C^*)^j$ acting
trivially on $\prod_{i=1}^j \Gr(m+1,p_i)$. We denote the
$1$-dimensional weight $n$ representation of the $i^{\mathrm{th}}$
factor by $e^{n \hbar_i}$. (Denoted by $I_n$ previously.) The
equivariant cohomology $H_{(\C^*)^j}^*(\pt)$ of the point is
identified with $\C[\hbar_1,\dots,\hbar_j]$.

\begin{Theorem}\label{thm:2nd}
\begin{equation*}
  \int_{\bM^{m+1}(c)} \Phi(\cE) - \int_{\bM^{m}(c)} \Phi(\cE)
  = 
  \sum_{\vec{c}\in S^m(c)}
   \int_{\bM^{m}(c_\flat)} \Res_{\hbar_{j}=0} \cdots \Res_{\hbar_{1}=0} 
   \Phi(\cEf\oplus 
      \bigoplus_{i=1}^j C_m\boxtimes \Qcal^{(i)}\otimes e^{-\hbar_i})\cup 
   \Psi^{\vec{p}}(\cEf),
\end{equation*}
where $\Psi^{\vec{p}}(\bullet)$ is another cohomology class given by
\begin{gather}\label{eq:Psi}
  \Psi^{\vec{p}}(\bullet) \defeq
  \frac1{(m+1)^j} 
   \prod_{i=1}^j \frac1{\sum_{1\le k\le i} p_k}
          \int_{
        \prod_{i=1}^j \Gr(m+1,p_i)} \heartsuit,
\\
  \heartsuit =
  \begin{aligned}[t]
   & 
   \prod_{i=1}^j
      \frac{
      e\left(
        (V_1(C_m)\otimes\Qcal^{(i)}/\shfO)^*
        \right)
    }{
      e(\fN(\bullet,C_m)\otimes \Qcal^{(i)}
      \otimes e^{-\hbar_i})
      \, e(\fN(C_m,\bullet)\otimes \Qcal^{(i)*}
      \otimes e^{\hbar_i})
    }
\\
   & \qquad \times
   \prod_{1\le i_1\neq i_2 \le j} 
   e(\Qcal^{(i_1)}\otimes \Qcal^{(i_2)*} \otimes e^{-\hbar_{i_1}+\hbar_{i_2}}).
  \end{aligned}
  \notag
\end{gather}
\textup(Note that $\Psi^{\vec{p}}$ depends on $\vec{p} =
(p_1,\dots,p_j)$, but not on $c_\flat$.\textup)
\end{Theorem}

Let us prepare notation before starting the proof.

\begin{NB}
We define $\mathfrak R(\hbar_1,\dots,\hbar_j)$ recursively by
$\mathfrak R(\hbar_1,\dots,\hbar_j) = \mathfrak
R(\hbar_2,\dots,\hbar_j)[[\hbar_1^{-1}, \hbar_1]$.
\end{NB}
We defined a cohomology class $\Psi^{\vec{p}}(\cEf)$ by
the formula \eqref{eq:Psi}, and it is an element in
$H^*_{\hT}(\bp\times M \times \prod_{i=1}^j
\Gr(m+1,p^{(i)}))[\hbar_1^\pm,\dots,\hbar_j^\pm]$ for some moduli stack $M$
with the universal family $\cE$.

Let
\(
  \Dec^{(j)}(c)
\)
be the set of pairs $\fI^{(j)} = (\If^{(j)},\vec{\Is}^{(j)})$
as follows:
\begin{itemize}
\item $\vec{\Is}^{(j)}$ is a tuple
  $(\Is^{(j)},\Is^{(j-1)},\dots,\Is^{(1)})$ of subsets of
  $\underline{v_1(c)}$ such that $|\Is^{(i)}| = p_{i}(m+1)$ for some
  $p_i\in\Z_{>0}$ ($1\le i\le j$).
\item $\min(\Is^{(1)}) > \min(\Is^{(2)}) > \cdots > \min(\Is^{(j)})$.
\item $\If^{(j)}$ is also a subset of $\underline{v_1(c)}$ and we have
  $\underline{v_1(c)} = \If^{(j)}\sqcup\bigsqcup_{i=1}^j \Is^{(i)}$.
\end{itemize}

For $\mathfrak I^{(j)}\in \Dec^{(j)}(c)$ set
\begin{equation*}
   \mathfrak k(\fI^{(j)}) \defeq
  \max\{ i\in \If^{(j)} \mid i < \min(\Is^{(j)})\},
\end{equation*}
where we understand this to be $0$ if there exists no $i\in \If^{(j)}$
with $i < \min(\Is^{(j)})$. We also put
\begin{equation*}
   c_\sharp^{(i)}\defeq p_{i} e_m \ (1\le i \le j), \quad
   c_\flat^{(j)} \defeq c - \sum_{i=1}^j p_i e_m.
\end{equation*}

We have a map $\pi_j\colon \Dec^{(j+1)}(c)\to \Dec^{(j)}(c);
(\If^{(j+1)},\vec{\Is}^{(j+1)})\mapsto (\If^{(j)},\vec{\Is}^{(j)})$
given by
\begin{equation*}
   \If^{(j)} \defeq \If^{(j+1)}\sqcup \Is^{(j+1)},\quad
   \vec{\Is}^{(j)}\defeq (\Is^{(j)},\dots,\Is^{(1)}).
\end{equation*}

Let
\begin{equation*}
   \tM(\fI^{(j)}) \defeq \tM^{m,\ell}
   (c_\flat^{(j)}),
\qquad
   \bM(\fI^{(j)}) \defeq \bM^m(c_\flat^{(j)}),
\end{equation*}
where in the first equality we take the unique order preserving
bijection $\If^{(j)}\cong \underline{v_1(c_\flat^{(j)})} =
\underline{\# \If^{(j)}}$ and take
$\ell\in\underline{v_1(c_\flat^{(j)})}$ corresponding to $\mathfrak
k(\fI^{(j)})$.

For the universal family $\cEf^{(j)}$ for $\tM(\fI^{(j)})$ or
$\bM(\fI^{(j)})$ let
\begin{equation*}
  \Psi^{\fI^{(j)}}(\cEf^{(j)}) \defeq
   \Psi^{\vec{p}}(\cEf^{(j)})
    \frac{v_1(c_\flat^{(j)})! \prod_{i=1}^j (v_1(c_\sharp^{(i)}) - 1)!}
    {v_1(c)!}
    (m+1)^j
    \prod_{i=1}^j 
    \sum_{k=1}^i p_k,
\end{equation*}
where $\vec{p}= (p_1,\dots,p_j)$.

\begin{Lemma}[\protect{\cite[7.6.5]{Moc}}]
For each $j$, we have the formula
\begin{multline}
  \int_{\bM^{m+1}(c)} \Phi(\cE) - \int_{\bM^{m}(c)} \Phi(\cE)
\\
  =
  \sum_{\substack{1\le i < j \\ \fI^{(i)} \in \Dec^{(i)}(c)}}
  \int_{\bM(\fI^{(i)})} \Res_{\hbar_{i}=0} \cdots \Res_{\hbar_{1}=0}
  \Phi(\cEf^{(i)}\oplus \bigoplus_{k=1}^i
   C_m\boxtimes  \Qcal^{(k)}\otimes e^{-\hbar_k}
  ) 
  \Psi^{\fI^{(i)}}(\cEf^{(i)})
\\
  + \sum_{\fI^{(j)}\in \Dec^{(j)}(c)}
  \int_{\tM(\fI^{(j)})}   \Res_{\hbar_{j}=0} \cdots \Res_{\hbar_{1}=0}
  \widetilde\Phi(\cEf^{(j)}\oplus \bigoplus_{k=1}^j
   C_m\boxtimes  \Qcal^{(k)}\otimes e^{-\hbar_k}
  ) 
  \Psi^{\fI^{(j)}}(\cEf^{(j)}).
\end{multline}
\end{Lemma}

\begin{proof}
We prove the assertion by an induction on $j$.
If $j=1$, this is nothing but \thmref{thm:1st} applied to $\ell = v_1(c)$.

Suppose that the formula is true for $j$. We apply \thmref{thm:1st} to get
\begin{multline*}
  \int_{\tM(\fI^{(j)})}
  \Res_{\hbar_{j}=0} \cdots \Res_{\hbar_{1}=0}
  \widetilde\Phi(\cEf^{(j)}\oplus \bigoplus_{k=1}^j
   C_m\boxtimes  \Qcal^{(k)}\otimes e^{-\hbar_k}) 
  \Psi^{\fI^{(j)}}(\cEf^{(j)})
\\
  -
  \int_{\bM(\fI^{(j)})} \Res_{\hbar_{j}=0} \cdots \Res_{\hbar_{1}=0}
  \Phi(\cEf^{(j)}\oplus \bigoplus_{k=1}^j
   C_m\boxtimes  \Qcal^{(k)}\otimes e^{-\hbar_k})
  \Psi^{\fI^{(j)}}(\cEf^{(j)})
\\
  = \sum_{\substack{\fI^{(j+1)}\in\Dec^{(j+1)}(c) \\ 
      \pi_j(\fI^{(j+1)}) = \fI^{(j)}}}
    \int_{\tM(\fI^{(j+1)})} 
        \Res_{\hbar_{j+1}=0} \cdots \Res_{\hbar_{1}=0}\left[
          \widetilde\Phi(\cEf^{(j+1)}\oplus \bigoplus_{k=1}^{j+1}
          C_m\boxtimes  \Qcal^{(k)}\otimes e^{-\hbar_k}
          )
      \Phi'(\cEf^{(j+1)})\right],
\end{multline*}
where
\begin{equation*}
  \begin{split}
  & \Phi'(\cEf^{(j+1)}) =
    \frac{v_1(c_\flat^{(j+1)})! (v_1(c_\sharp^{(j+1)}) - 1)!}{v_1(c_\flat^{(j)})!}
    \int_{\Gr(m+1,p_{j+1})} \heartsuit
\\
  & \heartsuit =
    \frac{
       \Psi^{\fI^{(j)}}(
       \cEf^{(j+1)}\oplus
       C_m\boxtimes\Qcal^{(j+1)}\otimes e^{-\hbar_{j+1}})
      e\left(
        (V_1(C_m)\otimes\Qcal^{(j+1)}/
        \shfO)^*
        \right)
    }{
      e(\fN(\cEf^{(j+1)},C_m)\otimes \Qcal^{(j+1)}
      \otimes e^{-\hbar_{j+1}})
      e(\fN(C_m,\cEf^{(j+1)})\otimes \Qcal^{(j+1)*}
      \otimes e^{\hbar_{j+1}})
    }.
  \end{split}
\end{equation*}
We have 
\(
   \Phi'(\cEf^{(j+1)})
   = \Psi^{\fI^{(j+1)}}(\cEf^{(j+1)})
\)
thanks to the multiplicative property of the Euler class and
\(
   \fN(C_m,C_m) = - \Hom(C_m,C_m) = -\C \id_{C_m}.
\)
Hence the formula holds for $j+1$.
\end{proof}

If $j$ is sufficiently large, $\Dec^{(j)}(c) = \emptyset$. Hence we have
\begin{multline*}
    \int_{\bM^{m+1}(c)} \Phi(\cE) - \int_{\bM^{m}(c)} \Phi(\cE)
\\
    = \sum_{j=1}^\infty \sum_{\fI^{(j)} \in \Dec^{(j)}(c)}
  \int_{\bM(\fI^{(j)})}   \Res_{\hbar_{j}=0} \cdots \Res_{\hbar_{1}=0}
  \Phi(\cEf^{(j)}\oplus \bigoplus_{i=1}^{j}
          C_m\boxtimes  \Qcal^{(i)}\otimes e^{-\hbar_i}) 
  \Psi^{\fI^{(j)}}(\cEf^{(j)}).
\end{multline*}
We have a map $\rho_j\colon\Dec^{(j)}(c)\to S_j(c)$ given by
\begin{equation*}
   \rho_j(\fI^{(j)}) = 
   (c_\flat,p_1,\dots,p_j) = 
   (c_\flat^{(j)},\frac{|\Is^{(1)}|}{m+1},\dots,\frac{|\Is^{(j)}|}{m+1}).
\end{equation*}
Therefore the right hand side is equal to
\begin{multline*}
  \sum_{j=1}^\infty \sum_{\vec{c}\in S_j(c)}
  \# \rho_j^{-1}(\vec{c})
    \frac{v_1(c_\flat)! \prod_{i=1}^j (p_i(m+1) - 1)!}
    {v_1(c)!}
    (m+1)^j
    \prod_{i=1}^j 
    \sum_{k=1}^i p_k
\\
  \times
  \int_{\bM^m(c_\flat)}   \Res_{\hbar_{j}=0} \cdots \Res_{\hbar_{1}=0}
  \Phi(\cEf\oplus \bigoplus_{i=1}^{j}
          C_m\boxtimes  \Qcal^{(i)}\otimes e^{-\hbar_i}) 
  \Psi^{\vec{p}}(\cEf).
\end{multline*}
Thus \thmref{thm:2nd} follows from the following lemma.

\begin{Lemma}[\protect{\cite[7.6.7]{Moc}}]
  \begin{equation*}
      \# \rho_j^{-1}(\vec{c})
    \frac{v_1(c_\flat)! \prod_{i=1}^j (p_i(m+1) - 1)!}
    {v_1(c)!}
    =  
    \frac1{(m+1)^j} 
   \prod_{i=1}^j \frac1{\sum_{1\le k\le i} p_k}
    \begin{NB}
      = W(\vec{p})
    \end{NB}%
   .
  \end{equation*}
\end{Lemma}

\begin{proof}
The set $\rho_j^{-1}(\vec{c})$ is
\begin{equation*}
    \left\{
      (\If^{(j)},\Is^{(j)},\dots,\Is^{(1)}) \left|\,
        \begin{aligned}[c]
          & \underline{v_1(c)} = \If^{(j)}\sqcup \bigsqcup_{i=1}^j \Is^{(i)},
          |\If^{(j)}| = v_1(c_\flat), |\Is^{(i)}| = p_i(m+1),
\\
          & \min(\Is^{(1)}) > \min(\Is^{(2)}) > \dots > \min(\Is^{(j)})
        \end{aligned}
      \right.\right\}.
\end{equation*}
Put $N \defeq v_1(c)$, $N_0 \defeq v_1(c_\flat)$, $N_i \defeq
p_i(m+1)$ ($1\le i\le j$). 

We first choose $\If^{(j)}\subset\underline{N}$. We have ${N \choose
  N_0}$ possibilities. Next we choose $\Is^{(j)}\subset
\underline{N}\setminus \If^{(j)}$. From the second condition, we must
have $\min(\Is^{(j)}) = \min(\underline{N}\setminus \If^{(j)})$. Let
$x$ be this number. Then the remaining choice is
\(
  \Is^{(j)}\setminus \{x\} \subset
  (\underline{N}\setminus \If^{(j)})
  \setminus\{x\}.
\)
We have $N-N_0 -1 \choose N_j-1$ possibilities. Next we choose
$\Is^{(j-1)}\subset\underline{N}\setminus (\If^{(j)}\cup\Is^{(j)})$. We have
$N-N_0-N_j-1\choose N_{j-1}-1$ possibilities. We continue until we choose
$\Is^{(1)}$. Therefore we have
\begin{equation*}
  \begin{split}
   \# \rho_j^{-1}(\vec{c})
   &= 
   {N \choose N_0}\prod_{i=1}^j 
   {N - N_0 - \sum_{k>i} N_k - 1\choose N_i - 1}
\\
   &=
   \frac{N!}{N_0! \prod_{i=1}^j (N_i-1)!}
   \times \prod_{i=1}^j \frac1{\sum_{1\le k\le i} N_k}.
  \end{split}
\end{equation*}
Moreover
\makeatletter
\tagsleft@false
\makeatother
\begin{gather*}
  \prod_{i=1}^j \frac1{\sum_{1\le k\le i} N_k}
  = \prod_{i=1}^j \frac1{\sum_{1\le k\le i} p_k(m+1)}
  \tag*{\qedhere}.
\end{gather*}
\end{proof}
\makeatletter
\tagsleft@true
\makeatother

\begin{NB}
  First consider \thmref{thm:1st} with $m=0$.  We have $c_\sharp =
  e_0$ is the only possibility. Therefore
\begin{equation*}
  \mathcal D^{0,\ell}(c) = \{
  \mathfrak I = (\underline{N}\setminus \{ i\}, \{i\})
  \mid i\le \ell\} \cong \underline{\ell}.
\end{equation*}

\begin{Theorem}
\begin{equation*}
   \int_{\tM^{0,\ell}(c)} \widetilde\Phi(\cE) 
   - \int_{\bM^0(c)} \Phi(\cE)
  = 
  \sum_{i=1}^{\ell}
   \int_{\tM^{m,i-1}(c-e_0)} 
   \widetilde\Phi(\cEf)\cup
   \Res_{\hbar=0} \Phi'(\cEf),
\end{equation*}
where
\begin{equation*}
  \begin{split}
  & \Phi'(\cEf) =
    \frac1{v_1(c)}
    \frac{
      \Phi(C_0\boxtimes I_{-1})
    }{
      e(\fN(\cEf,C_0)\otimes I_{-1})
      \, e(\fN(C_0,\cEf)\otimes I_{1})
    }.
  \end{split}
\end{equation*}
\end{Theorem}

\begin{Theorem}
\begin{equation*}
  \int_{\bM^{1}(c)} \Phi(\cE) - \int_{\bM^{0}(c)} \Phi(\cE)
  = 
  \sum_{j=1}^\infty
   \int_{\bM^{0}(c - je_0)} \Phi(\cEf)\cup 
   \Res_{\hbar_{j}=0} \cdots \Res_{\hbar_{1}=0}
   \Psi^{j}(\cEf),
\end{equation*}
where
\begin{gather*}
  \Psi^{j}(\bullet) \defeq
    \frac1{j!}
    \left[
      \prod_{i=1}^j
      \frac{
      \Phi(C_0\boxtimes e^{-\hbar_i})
    }{
      e(\fN(\bullet,C_0)
      \otimes e^{-\hbar_i})
      \, e(\fN(C_0,\bullet)
      \otimes e^{\hbar_i})
    }
    \times \prod_{1\le i_1\neq i_2 \le j} 
   ({-\hbar_{i_1}+\hbar_{i_2}})
   \right].
\end{gather*}
\end{Theorem}

Since we have an isomorphism $\bM^m(c) \xrightarrow{\cong} \bM^0(c
e^{-m[C]}); (E,\Phi)\mapsto (E(-mC),\Phi)$, we get \thmref{thm:m=0}.
\end{NB}

\begin{NB}
\subsection{Some formulas}

For a cohomology class $\varphi\in H_{\C^*\times\hT}^*(X\times\cM)$ we have
\begin{equation*}
  \begin{split}
  & \varphi \ch(C_m)/[\bC] 
\\
  = \;
  & 
  \ve_1 \ve_2 \frac{e^{m\ve_1}(\frac{1 - e^{\ve_1}}{\ve_1})
     - e^{m\ve_2}(\frac{1 - e^{\ve_2}}{\ve_2})}{\ve_1 - \ve_2}
   \varphi/[\bC]
   - \frac{e^{m\ve_1}(1 - e^{\ve_1}) - e^{m\ve_2}(1 - e^{\ve_2})}
  {\ve_1-\ve_2}\varphi/[C],
\\
  & \varphi \ch(C_m)/[C] 
\\
  = \;
  &
  \ve_1 \ve_2 \frac{e^{m\ve_1}(1-e^{\ve_1}) - e^{m\ve_2}(1-e^{\ve_2})}
  {\ve_1 - \ve_2}
  \varphi/[\bC]
  -
  \frac{\ve_1 e^{m\ve_1}(1-e^{\ve_1}) - \ve_2 e^{m\ve_2}(1-e^{\ve_2})}
  {\ve_1 - \ve_2}
  \varphi/[C].
  \end{split}
\end{equation*}

\begin{proof}
First note that
\begin{equation*}
  \begin{split}
    & \varphi \psi /[\bC] 
    = \varphi_{p_1} \psi_{p_1} \frac1{\ve_1(\ve_2 - \ve_1)}
    + \varphi_{p_2} \psi_{p_2} \frac1{(\ve_1-\ve_2)\ve_2}
    = \frac{(A + \ve_1 B) \varphi_{p_1}}{\ve_1(\ve_2-\ve_1)}
    + \frac{(A + \ve_2 B) \varphi_{p_2}}{(\ve_1-\ve_2)\ve_2}
\\
    =\; & A \varphi/[\bC] + B\varphi/[C]
  \end{split}
\end{equation*}
for
\begin{equation*}
   A \defeq \frac{ \ve_1 \psi_{p_2} - \ve_2 \psi_{p_1}}{\ve_1 - \ve_2},
   \quad
   B \defeq \frac{ \psi_{p_1} - \psi_{p_2}}{\ve_1 - \ve_2}.
\end{equation*}

We have
\(
   c_1(\shfO(C))|_{p_1} = \ve_1,
\)
\(
   c_1(\shfO(C))|_{p_2} = \ve_2.
\)
From the exact sequence
\(
   0 \to \shfO(-C) \to \shfO \to \shfO_C\to 0, 
\)
we get
\begin{gather*}
  \left.\ch(C_m)\right|_{p_a} = e^{(m+1)\ve_a} - e^{m\ve_a}.
\end{gather*}
Therefore we have the first assertion.

Similarly
\begin{equation*}
  \begin{split}
    & \varphi \psi /[C] 
    = \varphi_{p_1} \psi_{p_1} \frac1{\ve_2 - \ve_1}
    + \varphi_{p_2} \psi_{p_2} \frac1{\ve_1-\ve_2}
    = \frac{(A + \ve_1 B) \varphi_{p_1}}{\ve_1(\ve_2-\ve_1)}
    + \frac{(A + \ve_2 B) \varphi_{p_2}}{(\ve_1-\ve_2)\ve_2}
\\
    =\; & A \varphi/[\bC] + B\varphi/[C]
  \end{split}
\end{equation*}
for
\begin{equation*}
   A \defeq  \ve_1 \ve_2 \frac{\psi_{p_2} - \psi_{p_1}}{\ve_1 - \ve_2}
   \quad
   B \defeq \frac{ \ve_1\psi_{p_1} - \ve_2\psi_{p_2}}{\ve_1 - \ve_2}.
\end{equation*}
\end{proof}
Remark that $\ve_1\ve_2 \bullet/[\bC] = \bullet/ p^*[0]$, where $0$ is
the origin of $\C^2$.
\begin{NB2}
We have $[0]=\ve_1\ve_2\in H^4_{\hT}(\C^2)$. More precisely, 
the pullback of $\ve_1\ve_2\in H^4_{\hT}(\pt)$ by $\C^2\to \pt$.
Therefore $p^*[0] = \ve_1\ve_2\in H^4_{\hT}(\bC)$
\[
  \varphi/p^*[0] = \left(\frac{\ve_1\ve_2 \varphi_{p_1}}{\ve_1(\ve_2 - \ve_1)}
    + \frac{\ve_1\ve_2 \varphi_{p_2}}{(\ve_1-\ve_2)\ve_2}\right)
  = \ve_1\ve_2 \varphi/[\bC].
\]
\end{NB2}
Setting $\varphi = 1$, we get
\begin{equation*}
  \begin{split}
    \ch(C_m)/[C] &= 
  \frac{e^{m\ve_1}(1-e^{\ve_1}) - e^{m\ve_2}(1-e^{\ve_2})}
  {\ve_1 - \ve_2},
\\
    \ch(C_m)/[\bC] &= 
  \frac{e^{m\ve_1}(\frac{1 - e^{\ve_1}}{\ve_1})
     - e^{m\ve_2}(\frac{1 - e^{\ve_2}}{\ve_2})}{\ve_1 - \ve_2}.
  \end{split}
\end{equation*}
Therefore
\begin{equation*}
  \begin{split}
   &\Phi(C_m\boxtimes e^{-\hbar_i})
     = 
     \exp\left[
       \sum_{a=1}^\infty \left\{
        t_a \ch_{a+1}(C_m\boxtimes e^{-\hbar_i})/[C]
        + \tau_a \ch_{a+1}(C_m\boxtimes e^{-\hbar_i})/[\bC]
      \right\}\right]
\\
   = \; & 
   \exp\left[
       \sum_{a=1}^\infty \left\{
        t_a \left[
          \frac{e^{m\ve_1}(1-e^{\ve_1}) - e^{m\ve_2}(1-e^{\ve_2})}
          {\ve_1 - \ve_2}
            e^{-\hbar_i}
        \right]_{a}
        +
        \tau_a \left[
            \frac{e^{m\ve_1}(\frac{1 - e^{\ve_1}}{\ve_1})
              - e^{m\ve_2}(\frac{1 - e^{\ve_2}}{\ve_2})}{\ve_1 - \ve_2}
            e^{-\hbar_i}
        \right]_{a}
      \right\}\right],
  \end{split}
\end{equation*}
where $[\bullet]_a$ means degree $a$ part of $\bullet$. 
When $\ve_1$, $\ve_2\to 0$, this converges to
\begin{equation*}
  \left.\Phi(C_m\boxtimes e^{-\hbar_i})\right|_{\ve_1,\ve_2 = 0}
  \begin{NB2}
  = \exp\left[
    \sum_{a=1}^\infty \left\{
      - \left(t_a + (m + \frac12)\tau_a\right)
      \left[ e^{-\hbar_i} \right]_a
      \right\}
    \right]
  \end{NB2}
  = \exp\left[
    \sum_{a=1}^\infty \left\{
      (-1)^{a-1} \left(t_a + (m + \frac12)\tau_a\right)
      \frac{\hbar_i^a}{a!}
      \right\}
    \right]
\end{equation*}

We have
\begin{equation*}
  \begin{split}
  & \left(\varphi \ch(C_m)\td{\bp}\right)/[\bp]
\\
  =\; & - \ve_1\ve_2 \frac{e^{\ve_1+\ve_2}(e^{m\ve_1} - e^{m\ve_2})}
   {e^{\ve_1} - e^{\ve_2}} \varphi/[\bC]
   + \frac{e^{\ve_1+\ve_2}(\ve_1 e^{m\ve_1} - \ve_2 e^{m\ve_2})}
    {e^{\ve_1} - e^{\ve_2}}\varphi/[C].
  \end{split}
\end{equation*}

\begin{proof}
  There is no contribution from fixed points in $\linf$ since
  $\ch(C_m)$ vanishes outside $C$. Thus we use the above argument.

  Since
\begin{gather*}
  \td{\bp}|_{p_1} = \frac{\ve_1}{1 - e^{-\ve_1}}
  \frac{{\ve_2-\ve_1}}{1 - e^{-\ve_2+\ve_1}}
  = \frac{\ve_1(\ve_2-\ve_1)}{(1 - e^{-\ve_1})(1 - e^{-\ve_2+\ve_1})},
\\
  \td{\bp}|_{p_2} = \frac{\ve_2}{1 - e^{-\ve_2}}
  \frac{{\ve_1-\ve_2}}{1 - e^{-\ve_1+\ve_2}}
  = \frac{(\ve_1-\ve_2)\ve_2}{(1 - e^{-\ve_2})(1 - e^{-\ve_1+\ve_2})},
\end{gather*}
we have
\begin{equation*}
  \begin{split}
  & \left(\ch(C_m)\td{\bp}\right)_{p_1}
  = \frac{\ve_1(\ve_2-\ve_1) e^{m\ve_1}(e^{\ve_1} - 1)}
  {(1 - e^{-\ve_1})(1 - e^{-\ve_2+\ve_1})}
  = \frac{e^{(m+1)\ve_1}\ve_1(\ve_2-\ve_1)}{1 - e^{-\ve_2+\ve_1}},
\\
  & \left(\ch(C_m)\td{\bp}\right)_{p_2}
  = \frac{e^{(m+1)\ve_2}(\ve_1-\ve_2)\ve_2}{1 - e^{-\ve_1+\ve_2}}.
  \end{split}
\end{equation*}
Therefore
\begin{equation*}
  \begin{split}
   A &= \frac1{\ve_1 - \ve_2}
   \left(
     \frac{\ve_1\ve_2(\ve_1-\ve_2)e^{(m+1)\ve_2}}{1 - e^{-\ve_1+\ve_2}}
     - \frac{\ve_1\ve_2(\ve_2-\ve_1)e^{(m+1)\ve_1}}{1 - e^{-\ve_2+\ve_1}}
   \right)
   = - \frac{\ve_1\ve_2 e^{\ve_1+\ve_2}(e^{m\ve_1} - e^{m\ve_2})}
   {e^{\ve_1} - e^{\ve_2}},
\\
   B &= \frac1{\ve_1 - \ve_2}
   \left(
     \frac{\ve_1(\ve_2-\ve_1)e^{(m+1)\ve_1}}{1 - e^{-\ve_2+\ve_1}}
     - \frac{(\ve_1-\ve_2)\ve_2 e^{(m+1)\ve_2}}{1 - e^{-\ve_1+\ve_2}}
   \right)
   = \frac{e^{\ve_1+\ve_2}(\ve_1 e^{m\ve_1} - \ve_2 e^{m\ve_2})}{e^{\ve_1} - e^{\ve_2}}
  \end{split}
\end{equation*}
\end{proof}

Setting $\varphi = 1$ and replacing $m$ by $m+1$, we get
\begin{equation*}
   \ch(H^1(C_m(C))) = 
   \ve_1\ve_2 e^{\ve_1+\ve_2} \frac{e^{(m+1)\ve_1} - e^{(m+1)\ve_2}}
   {e^{\ve_1} - e^{\ve_2}}.
\end{equation*}
The Euler class $e(\Hom(H^1(C_m(C))\otimes \cQ/[\xi],[\xi]))$ can be
expressed in terms of Chern classes of $\cQ$ from this formula in
principle.

We also have
\begin{equation*}
  \begin{split}
  & \ch(\fN(\cEf,C_m))
  = -\ch(\cEf)^\vee\ch(C_m)\td{\bp}/[\bp] 
\\
  =\; & \ve_1\ve_2 e^{\ve_1+\ve_2} \frac{e^{m\ve_1} - e^{m\ve_2}}
   {e^{\ve_1} - e^{\ve_2}} \ch(\cEf)^\vee/[\bC]
   - e^{\ve_1+\ve_2} \frac{\ve_1 e^{m\ve_1} - \ve_2 e^{m\ve_2}}
    {e^{\ve_1} - e^{\ve_2}}\ch(\cEf)^\vee/[C],
  \end{split}
\end{equation*}
where $\vee$ is the involution defined by $1$ on $H^{2(\mathrm{even})}$
and $-1$ on $H^{2(\mathrm{odd})}$.
Similarly we have
\begin{equation*}
  \begin{split}
  & \ch(\fN(C_m,\cEf))
  = -\ch(\cEf)\ch(C_m)^\vee\td{\bp}/[\bp] 
\\
  =\; & - \ve_1\ve_2 e^{\ve_1+\ve_2} \frac{e^{(1-m)\ve_1} - e^{(1-m)\ve_2}}
   {e^{\ve_1} - e^{\ve_2}} \ch(\cEf)/[\bC]
   + e^{\ve_1+\ve_2} \frac{\ve_1 e^{(1-m)\ve_1} - \ve_2 e^{(1-m)\ve_2}}
    {e^{\ve_1} - e^{\ve_2}}\ch(\cEf)/[C].
  \end{split}
\end{equation*}
\begin{NB2}
  we have
\begin{equation*}
  \begin{split}
  & \left(\ch(C_m)^\vee \td{\bp}\right)_{p_1}
  = - \frac{e^{-m\ve_1}\ve_1(\ve_2-\ve_1)}{1 - e^{-\ve_2+\ve_1}}
  = - \left(\ch(C_{1-m}) \td{\bp}\right)_{p_1},
\\
  & \left(\ch(C_m)^\vee \td{\bp}\right)_{p_2}
  = - \frac{e^{-m\ve_2}(\ve_1-\ve_2)\ve_2}{1 - e^{-\ve_1+\ve_2}}
  = - \left(\ch(C_{1-m})\td{\bp}\right)_{p_2}.
  \end{split}
\end{equation*}
\end{NB2}%

We have
\begin{equation*}
  \begin{split}
   \left.\ch(\fN(\cEf,C_m))\right|_{\ve_1=\ve_2=0}
  &= m \ch(\cEf)^\vee/p^*[0] - \ch(\cEf)^\vee/[C],
\\
  \left.\ch(\fN(C_m,\cEf))\right|_{\ve_1=\ve_2=0}
  &= (m-1) \ch(\cEf)/p^*[0] + \ch(\cEf)/[C].
  \end{split}
\end{equation*}
Therefore
\begin{equation*}
  \begin{split}
   & \left. e(\fN(\cEf,C_m)\otimes e^{-\hbar_i})\right|_{\ve_1=\ve_2=0}
   = \left[
    \left(m \ch(\cEf)^\vee/p^*[0] - \ch(\cEf)^\vee/[C]\right) e^{-\hbar_i}
   \right]_{mr(c_\flat) + (c_1(c_\flat),[C])},
\\   
  & \left. e(\fN(C_m,\cEf)\otimes e^{\hbar_i})\right|_{\ve_1=\ve_2=0}
  = \left[
  \left((m-1) \ch(\cEf)/p^*[0] + \ch(\cEf)/[C]\right) e^{\hbar_i}
  \right]_{(m+1)r(c_\flat) + (c_1(c_\flat),[C])}.
  \end{split}
\end{equation*}

\subsection{Tensor product}

We have the following identity of symmetric polynomials (\cite[\S4,
Example~5]{Macdonald})
\begin{equation*}
   \prod_{i=1}^m \prod_{j=1}^n (1 + x_i + y_j)
   = \sum_{\mu\subset \lambda \subset (n^m)}
   d_{\lambda,\mu} s_\mu(x) s_{\tilde\lambda'}(y),
\end{equation*}
where $s_\bullet(\bullet)$ is a Schur polynomial,
\begin{equation*}
  d_{\lambda,\mu} = \det
  \begin{pmatrix}
    \lambda_i + n - i \choose \mu_j + n - j
  \end{pmatrix}_{1\le i,j\le n}
\end{equation*}
and
\(
   \tilde\lambda' = (m - \lambda_n',\dots,m-\lambda_1').
\)
This gives us the formula for the Chern classes of tensor products:
\begin{equation*}
  c(E\otimes F)
  = \sum_{\mu\subset \lambda \subset (n^m)}
   d_{\lambda,\mu} s_\mu(E) s_{\tilde\lambda'}(F),
\end{equation*}
where $s_\mu(E)$ means that we write $s_\mu(x)$ as a polynomial in
elementary symmetric functions $e_i(x)$ and replace it by $c_i(E)$,
and similarly for $s_{\tilde\lambda'}(F)$.

\subsection{Effects of $\otimes\shfO(C)$}

We have
\begin{equation*}
  \begin{split}
   & \Phi(\cE(C))
   = 
     \exp\left[
      \sum_{a=1}^\infty \left\{
        t_a \ch_{a+1}(\cE(C))/[C] + \tau_a \ch_{a+1}(\cE(C))/[\bC]
      \right\}
    \right]
\\
  =\; &
     \exp\Biggl[
      \sum_{a=1}^\infty \Biggl\{
        \sum_{b=0}^{a+1} 
        \begin{aligned}[t]
        & \left(
        - t_a
        \left[
          \ve_1\ve_2 \frac{e^{\ve_1} - e^{\ve_2}}{\ve_1 - \ve_2} 
        \right]_{a+1-b}
        +
        \tau_a
        \left[
          \frac{\ve_1 e^{\ve_2} - \ve_2 e^{\ve_1}}{\ve_1 - \ve_2} 
        \right]_{a+1-b}
        \right)
          \ch_{b}(\cE)/[\bC]
\\
        & \quad + 
        \left(
          t_a
          \left[
            \frac{\ve_1 e^{\ve_1} - \ve_2 e^{\ve_2}}{\ve_1 - \ve_2} 
          \right]_{a+1-b}
          + \tau_a
        \left[
          \frac{e^{\ve_1} - e^{\ve_2}}{\ve_1 - \ve_2} 
        \right]_{a+1-b}
        \right)
          \ch_{b}(\cE)/[C]
      \Biggr\}
    \Biggr]
        \end{aligned}
  \end{split}
\end{equation*}
as
\begin{equation*}
  \begin{split}
   & \varphi \ch(\shfO(C))/[\bC]
   = \frac{\ve_1 e^{\ve_2} - \ve_2 e^{\ve_1}}{\ve_1 - \ve_2} 
   \varphi/[\bC]
   + \frac{e^{\ve_1} - e^{\ve_2}}{\ve_1 - \ve_2} 
   \varphi/[C],
\\
   & \varphi \ch(\shfO(C))/[C]
   = - \ve_1\ve_2 \frac{e^{\ve_1} - e^{\ve_2}}{\ve_1 - \ve_2} 
   \varphi/[\bC]
   + \frac{\ve_1 e^{\ve_1} - \ve_2 e^{\ve_2}}{\ve_1 - \ve_2} 
   \varphi/[C].
  \end{split}
\end{equation*}

Since we formally need $\ch_0(\cE)$, $\ch_1(\cE)$, we include
$t_{-1}$, $t_0$, $\tau_{-1}$, $\tau_0$, and replace
$\sum_{a=1}^\infty$ to $\sum_{a=-1}^\infty$.
Then
\begin{equation*}
  \begin{split}
  & \Phi(\cE(C))
\\
  =\; &
  \exp\Biggl[
      \sum_{b=-1}^\infty \Biggl\{
        \sum_{a=b}^{\infty} 
        \begin{aligned}[t]
        & \left(
        - t_a
        \left[
          \ve_1\ve_2 \frac{e^{\ve_1} - e^{\ve_2}}{\ve_1 - \ve_2} 
        \right]_{a-b}
        +
        \tau_a
        \left[
          \frac{\ve_1 e^{\ve_2} - \ve_2 e^{\ve_1}}{\ve_1 - \ve_2} 
        \right]_{a-b}
        \right)
          \ch_{b+1}(\cE)/[\bC]
\\
        & \quad + 
        \left(
          t_a
          \left[
            \frac{\ve_1 e^{\ve_1} - \ve_2 e^{\ve_2}}{\ve_1 - \ve_2} 
          \right]_{a-b}
          + \tau_a
        \left[
          \frac{e^{\ve_1} - e^{\ve_2}}{\ve_1 - \ve_2} 
        \right]_{a-b}
        \right)
          \ch_{b+1}(\cE)/[C]
      \Biggr\}
    \Biggr]
        \end{aligned}
\\
  =\; &
  \Phi(\cE) \exp\Biggl[ 
   \sum_{a=-1}^\infty \tau_a \ch_{a+1}(\cE)/[C] 
  \Biggr]
\\
  & \quad \times
  \exp\Biggl[
      \sum_{b=-1}^\infty \Biggl\{
        \sum_{a=b+1}^{\infty} 
        \begin{aligned}[t]
        & \left(
        - t_a
        \left[
          \ve_1\ve_2 \frac{e^{\ve_1} - e^{\ve_2}}{\ve_1 - \ve_2} 
        \right]_{a-b}
        +
        \tau_a
        \left[
          \ve_1 \ve_2
          \frac{\frac{1-e^{\ve_1}}{\ve_1} - \frac{1 - e^{\ve_2}}{\ve_2}}
          {\ve_1 - \ve_2} 
        \right]_{a-b}
        \right)
          \ch_{b+1}(\cE)/[\bC]
\\
        & \quad + 
        \left(
          t_a
          \left[
            \frac{\ve_1 e^{\ve_1} - \ve_2 e^{\ve_2}}{\ve_1 - \ve_2} 
          \right]_{a-b}
          + \tau_a
        \left[
          \frac{e^{\ve_1} - e^{\ve_2}}{\ve_1 - \ve_2} 
        \right]_{a-b}
        \right)
          \ch_{b+1}(\cE)/[C]
      \Biggr\}
    \Biggr]
        \end{aligned}
  \end{split}
\end{equation*}

\underline{Matters}:

From 
\(
   0 \to \shfO \to \shfO(C) \to \shfO_C(-1) \to 0,
\)
we have
\begin{equation*}
   0  \to H^0(\cE\otimes^L \shfO_C(-1))
   \to H^1(\cE(-\linf)) \to H^1(\cE(C-\linf)) \to 
   H^1(\cE\otimes^L \shfO_C(-1)) \to 0.
\end{equation*}
We have
\(
  H^1(\cE\otimes^L \shfO_C(-1)) = 0.
\)
\begin{NB2}
  This is because $d\colon V_0\to V_1$ is surjective. Is there a
  direct proof ?
\end{NB2}%
Thus we have
\begin{equation*}
  e(H^1(\cE-\linf)) = e(H^1(\cE(C-\linf))) e(H^0(\cE\otimes^L \shfO_C(-1))).
\end{equation*}
We have
\begin{equation*}
  \left(\ch(\cE)\ch(C_0)\td{\bp}\right)/[\bp] 
  = \frac{e^{\ve_1+\ve_2}(\ve_1 - \ve_2)}
    {e^{\ve_1} - e^{\ve_2}}\ch(\cE)/[C].
\end{equation*}
\end{NB}


\end{document}